\documentclass[twoside,11pt]{article}

\usepackage{jmlr2e}

\usepackage[normalem]{ulem}
\usepackage{amsmath}
\usepackage{color}
\usepackage{amsfonts}
\usepackage{mathtools}
\newcommand{\norm}[1]{\left\lVert#1\right\rVert}
\usepackage{lipsum}
\usepackage{verbatim}
\usepackage{algorithm}
\usepackage{enumitem}
\usepackage{float}
\usepackage{xcolor}% you could also use the color package
\usepackage{colortbl}
\usepackage{stackrel}
\usepackage{extarrows}
\usepackage{tcolorbox}
\usepackage{xcolor}
\usepackage{natbib}
\usepackage{appendix}

\DeclarePairedDelimiter\ceil{\lceil}{\rceil}
\DeclarePairedDelimiter\floor{\lfloor}{\rfloor}

\newcommand{\highlightRED}[1]{%
  \colorbox{red!0}{$\displaystyle#1$}}
  \newcommand{\highlightBLUE}[1]{%
  \colorbox{blue!0}{$\displaystyle#1$}}
  \newcommand{\highlightGRAY}[1]{%
  \colorbox{gray!0}{$\displaystyle#1$}}

\usepackage{tikz}
\makeatletter
\newcommand\mathcircled[1]{%
  \mathpalette\@mathcircled{#1}%
}
\newcommand\@mathcircled[2]{%
  \tikz[baseline=(math.base)] \node[draw,circle,inner sep=1pt] (math) {$\m@th#1#2$};%
}

\newcounter{AlphabetCounter}

\newtheorem{assumption}[AlphabetCounter]{\bf Assumption}

\ShortHeadings{Decentralized Dictionary Learning Over Time-Varying Digraphs}{Daneshmand, Sun, Scutari, Facchinei, and Sadler}
\firstpageno{1}

\begin{document}

\title{Decentralized Dictionary Learning Over Time-Varying Digraphs}

 \author{\name Amir Daneshmand \email adaneshm@purdue.edu \\
\name Ying Sun \email sun578@purdue.edu \\
\name Gesualdo Scutari \email gscutari@purdue.edu \\
       \addr School of Industrial Engineering\\
       Purdue University\\
       West-Lafayette, IN, USA
       \AND
       \name Francisco Facchinei \email facchinei@diag.uniroma1.it \\
       \addr Department of Computer, Control, and Management Engineering\\
       University of Rome ``La Sapienza''\\
       Rome, Italy
       \AND
       \name Brian M. Sadler \email brian.m.sadler6.civ@mail.mil \\
       \addr U.S. Army Research Laboratory
       \\ Adelphi, MD, USA
       }

\editor{??\vspace{-0.3cm}}

\maketitle

\begin{abstract}%   <- trailing '%' for backward compatibility of .sty file
This paper studies Dictionary Learning  problems wherein the learning task is distributed over a multi-agent network, modeled as a   time-varying directed graph.
This formulation is relevant, for instance, in Big Data scenarios where massive amounts of data are collected/stored in different locations (e.g., sensors, clouds) and  aggregating and/or processing all data in a fusion center might be inefficient or unfeasible, due to resource limitations, communication overheads or privacy issues.
We develop a unified decentralized algorithmic framework for this class of  \emph{nonconvex} problems, {which is proved to converge to  stationary solutions at a sublinear rate}.  
The new method hinges on Successive Convex Approximation  techniques, coupled with a decentralized tracking mechanism    aiming at locally estimating the gradient of the smooth part of the sum-utility. 
To the best of our knowledge, this is the first provably convergent decentralized algorithm  for Dictionary Learning and, more generally, bi-convex problems  over (time-varying) (di)graphs.
\end{abstract}


\section{Introduction and Motivation\label{sec:intro}}
%\subsection{Literature and Motivation}
This paper introduces,  analyzes, and tests numerically the first {\em provably convergent distributed method} for a fairly general class of Dictionary Learning (DL) problems.  More specifically, we study  the problem of  finding a matrix    $\mathbf{D}\in\mathbb{R}^{M\times K}$ (a.k.a. the dictionary), %, with $K$ \emph{atoms}  (corresponding to the columns of $\mathbf{D}$) 
by which the data matrix $\mathbf{S}\in\mathbb{R}^{M\times N}$ can be represented through a matrix $\mathbf{X}\in\mathbb{R}^{K\times N}$,  with a favorable structure  on $\mathbf{D}$ and $\mathbf{X}$ (e.g., sparsity). 
 We target  scenarios where computational resources and  data  are not centrally available, but distributed over a group of  $I$ agents, which can communicate through a (possibly) \emph{time-varying, directed}  network; see Fig. \ref{Schematic_digraph}. 
Each agent  $i\in\{1,2,\ldots,I\}$ owns one block $\mathbf{S}_i\in\mathbb{R}^{M\times n_i}$ of the  data $\mathbf{S}\triangleq [\mathbf{S}_1,\ldots,\mathbf{S}_I]$, with  $\sum_{i=1}^I n_i=N$. Partitioning the representation matrix $\mathbf{X}\triangleq [\mathbf{X}_1,\ldots,\mathbf{X}_I]$ according to  $\mathbf{S}$,  with  $\mathbf{X}_i\in\mathbb{R}^{K\times n_i}$, 
 the class of \emph{distributed} DL problems we aim at studying  reads 
\begin{figure}[t]
\label{Schematic_digraph}
\center
\includegraphics[scale=0.35]{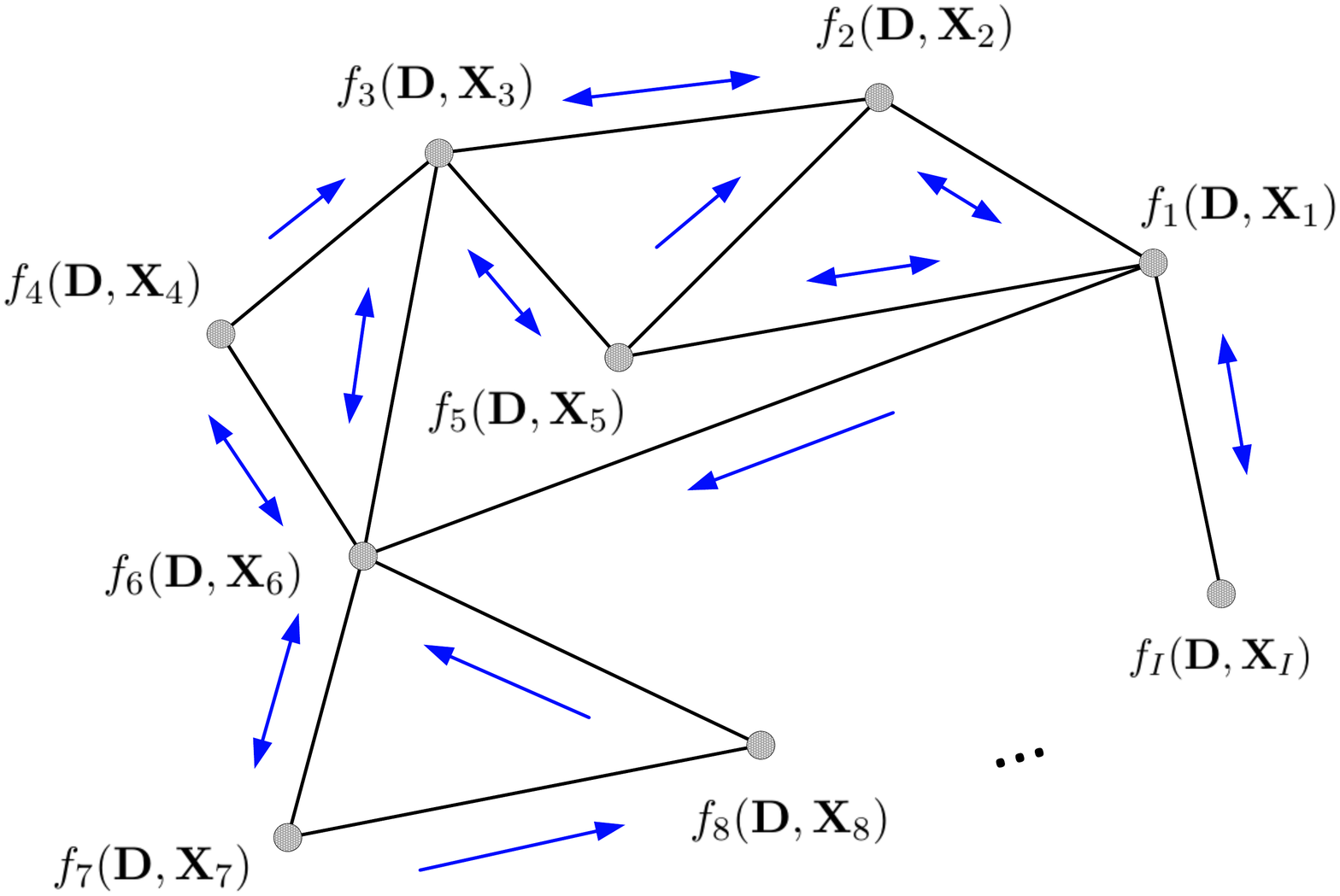}
\caption{Directed network topology}
\end{figure}
\begin{equation}
\begin{aligned}
\min_{\mathbf{D},\{\mathbf{X}_i\}_{i=1}^I}\quad& U(\mathbf{D},\mathbf{U})\triangleq\underset{\triangleq F(\mathbf{D},\mathbf{X})}{\underbrace{\sum_{i=1}^{I}f_{i}(\mathbf{D},\mathbf{X}_{i})}}+\sum_{i=1}^{I}g_{i}(\mathbf{X}_{i})+G(\mathbf{D})
\\
\mathrm{s.t.}\quad\,\,\quad &\mathbf{D} \in {\cal D},~\mathbf{X}_i\in\mathcal{X}_i,\quad i=1,\ldots ,I,
\end{aligned}
\label{eq:P1}\tag{P}
\end{equation}
 where  $f_i:\mathcal{D}\times \mathcal{X}_i\rightarrow\mathbb{R}$ is the fidelity function of agent $i$, which  measures the mismatch between the data $\mathbf{S}_i$ and the (local) model; {this function is assumed to be smooth and \emph{biconvex} (i.e., convex in  $\mathbf{D}$ for fixed  $\mathbf{X}_i$, and vice versa)}; $G:\mathcal{D}\rightarrow\mathbb{R}$ and $g_i:\mathcal{X}_i\rightarrow\mathbb{R}$ are (possibly non-smooth) convex functions, which are generally used to impose extra structure on the solution (e.g.,  low-rank or sparsity); and $\mathcal{D}\subseteq\mathbb{R}^{M\times K}$ and $\mathcal{X}_i\subseteq\mathbb{R}^{K\times n_i}$ are some  closed convex sets. To avoid scaling ambiguity in the model,  $\mathcal{D}$ is assumed to be bounded, without loss of generality. \setcounter{equation}{1}%, e.g., to discard infinite solutions when $t\rightarrow\infty$ in  $\mathbf{S}_i\approx (t\mathbf{D})(\frac{1}{t}\mathbf{X}_i)$. {
Since all $f_i$'s share the common variable $\mathbf{D}$, we call it a \emph{shared} variable and, by the same token,  $\mathbf{X}_i$'s are termed \emph{private} variables. Note that, in this distributed setting, agent $i$  knows only its own functions $f_i$ (and $g_i$) but not $\sum_{j\neq i}f_j$. Hence, agents aim to cooperatively  solve Problem~\ref{eq:P1} leveraging local communications with their neighbors.   %are distribu%Define also $F(\mathbf{D},\mathbf{X})\triangleq \sum_{i=1}^I  f_i(\mathbf{D},\mathbf{X}_i)$.

Problem \ref{eq:P1} encompasses several DL-based formulations of practical interest, corresponding to different choices of the fidelity functions, regularizers, and feasible sets;  examples include the \emph{elastic net} \citep{elasticNet_Hastie_2005} sparse DL, sparse PCA \citep{Shen_Huang08}, non-negative matrix factorization and low-rank approximation \citep{Hastie_book15}, supervised  DL \citep{Mairal_NIPS08}, sparse singular value decomposition \citep{Lee_Shen_Huang_Marron10}, non-negative sparse coding \citep{Hoyer_SNMF}, principal component pursuit \citep{R-PCA}, robust non-negative sparse matrix factorization,  and  discriminative label consistent learning \citep{ZhuolinJiang_2011}. More details on  explicit customizations of the general model \ref{eq:P1} can be found in Sec.~\ref{sec:distributed-DL}.
 
 Our distributed setting  is motivated by several data-intensive applications in several fields, including    signal processing and machine learning, and   network systems  (such as clouds, cluster computers, networks of sensor  vehicles, or  autonomous robots) wherein the sheer volume and spatial/temporal disparity of scattered data, energy constraints, and/or privacy issues,  render centralized processing and storage infeasible or inefficient. Also, time-varying communications
arise, for instance, in mobile wireless networks (e.g., ad-hoc networks),
wherein nodes are mobile and/or communicate through  fading channels.
Moreover, since nodes generally transmit at different power and/or communication
channels are not symmetric, directed links are a natural assumption.
 
 Our goal is to design a provably convergent  \emph{decentralized} method  for  Problem \ref{eq:P1}, over \emph{time-varying} and  \emph{directed} graphs.  To the best of our knowledge this is an open problem, as documented next. \vspace{-0.2cm}
 
 \subsection{Challenges and related works}

The design of distributed algorithms for \ref{eq:P1} faces the following challenges:
%  The distributed   setting we study is however difficult to analyze for the following reasons:
\begin{enumerate}[label=(\roman*)]
\item Problem \ref{eq:P1} is \emph{non-convex} and \emph{non-separable} in the optimization 
variables; 
\item Each agent $i$ owns exclusively $\mathbf{S}_i$ and thus  can only compute  its own 
function $f_i$;
\item Each $f_i$ depends on a common set of variables$-$the dictionary $\mathbf{D}-$shared 
among all the agents,  as well as the private variables $\mathbf{X}_i$. Shared and private 
variables need to be treated differently. In fact, in several applications, the size of private 
variables is  much larger than that of the  shared ones; hence, broadcasting   agents' private 
variables over the network would result in an unaffordable   communication overhead;
\item The gradient   of each  $f_i$   is in general \emph{neither  bounded nor globally Lipschitz } on the  
feasible region. This represents a challenge in the design of provably convergent distributed 
algorithms, as boundedness of the gradient is  a standard assumption in the analysis of   most
distributed schemes for   nonconvex problems;
\item $G$ and $g_i$'s are  \emph{nonsmooth};
 
\item {The graph is  \emph{directed},  \emph{time-varying}; no other      
structure is assumed (such as  star or ring topology, etc.), but some long term connectivity properties (cf. Assumption \ref{B_strongly_connectivity}).}
\end{enumerate}

\noindent \emph{Centralized} methods for the solution of Problem~\ref{eq:P1} (or some closely related variants) have  been  extensively studied and prominent examples are \citep{Aharon_KSVD_2006,Mairal_Bach_2010,RazaviyaynTsengLuo14}. However, we are not aware of any \emph{distributed} algorithm that can address  challenges 
i)-vi) (even some subsets of them), 
as documented next. 

\noindent \textbf{Ad-hoc heuristics}: Several attempts have been made to extend  centralized approaches for DL problems to a distributed setting (\emph{undirected, static} graphs), under more or less restrictive assumptions; examples include primal methods  \citep{Raja_CKSVD_2013,chainais2013distributed,Wai_EXTRA_AO_2015}  and   (primal/)dual-based  ones
\citep{ChenTSP15,Liang_ADMM_2014,Theodoridis_ICASSP15}. While these schemes represent good heuristics, their theoretical
convergence remains an open question, and  numerical results are contradictory. For instance, some schemes are shown not to converge while some others fail to reach asymptotic agreement among the local copies of the dictionary; see, e.g. \citep{chainais2013distributed}.  

{
Recently and independently from our conference work \citep{daneshmand_Asilomar16},  
\citet{Prox_PDA_GlobalSIP}  proposed a distributed primal-dual-based method for   a 
class of  dictionary learning problems  related, but different from  Problem \ref{eq:P1}. 
Specifically, they considered:  quadratic loss functions $f_i$,  with a quadratic regularization 
on the dictionary (i.e., $G=0$), and norm ball constraints on the private variables. The 
network is modeled as a fixed undirected graph. Asymptotic convergence  of the scheme  to 
stationary solutions  is proved, but no rate analysis is reported. We remark that  the scheme 
 in \citep{Prox_PDA_GlobalSIP},    in 
order to establish  convergence,  requires some penalty parameters to go to infinity, which makes the method numerically not attractive.}

\noindent \textbf{Distributed nonconvex optimization}: Since the DL problem~\ref{eq:P1} is an instance of non-convex  optimization problems, {we briefly discuss   here the few  works in the literature on    distributed methods for   non-convex   optimization   \citep{Bianchi_Jakubowicz,Tatarenko2015,DoF,NEXT,sun2016distributed,pmlr-v70-hong17a,scutari-sun19}; we group these papers as  follows. The schemes in \citep{Bianchi_Jakubowicz,Tatarenko2015, DoF,pmlr-v70-hong17a}, while substantially different,  are all applicable  to {\it smooth, unconstrained} optimization, with  \citep{Bianchi_Jakubowicz,DoF} handling also {\it compact} constraints and \citep{Tatarenko2015} implementable on (time-varying) digraphs.} The distributed algorithms   in \citep{NEXT,sun2016distributed,scutari-sun19} can handle   objectives  with additive   {\it nonsmooth} convex  functions, with \citep{sun2016distributed,scutari-sun19} applicable to (time-varying) digraphs.  

All the above schemes    \emph{cannot} 
adequately deal  with   \emph{private} (i.e., $\mathbf{X}_i$'s)  {\it and} shared variables (i.e., $
\mathbf{D}$), which are a key feature of Problem \ref{eq:P1}. {Furthermore, convergence 
therein is proved under the assumption that the gradient of (the smooth part of) the objective 
function is \emph{globally Lipschitz continuous}, a property 
that we do not assume and that is not satisfied in many of the applications we consider.}
%is not true for  $\nabla F(\mathbf{D},\mathbf{X})$ in Problem \ref{eq:P1}.  
The design of provably convergent distributed algorithms for \ref{eq:P1} remains an open problem, let alone rate guarantees.

\vspace{-0.2cm}

\subsection{Major contributions}\vspace{-0.1cm}
In this paper, we propose the \emph{first provably convergent} distributed algorithm for the general class of DL problems \ref{eq:P1}, addressing  all challenges i)-vi).  The proposed approach uses a  general convexification-decomposition technique that hinges on recent (centralized) Successive Convex Approximation  methods \citep{ScutFacchSonPalPang13,Scutari-Facchinei-SagratellaTSP15}.  This technique is    coupled  with a {perturbed push-sum consensus
scheme  preserving the \emph{feasibility} of the iterates} and a tracking mechanism
aiming at estimating locally the gradient of $\sum_i f_i$. Both communication
and tracking protocols are implementable on \emph{time-varying} undirected
or \emph{directed} graphs (B-strongly connected). 
The scheme is proved to  converge
to stationary solutions of Problem~\ref{eq:P1}, under mild assumptions on the step-size employed by the algorithm; {a sublinear convergence rate is also established}.     
On the technical side, we contribute to the literature of distributed algorithms for bi-convex  (nonsmooth) constrained optimization by putting forth a new non-trivial convergence analysis that, for the first time, {i) avoids the   assumption  that  the gradients  $\nabla f_i$ are  globally Lipschitz; and ii)  deals with {\it private} and shared optimization variables.} Numerical experiments show that the proposed
schemes compare favorably with ad-hoc  algorithms, proposed for special instances of Problem~\ref{eq:P1}.\vspace{-0.2cm}

\subsection{Paper Organization}
The rest of the paper is organized as follows. The problem and network setting are introduced in Sec.~\ref{sec:distributed-DL}, along with some motivating applications. Sec.~\ref{sec:alg-design} presents the algorithm   and its convergence properties; the proofs of our results are given in the Appendix, Sec.~\ref{appendix}. Extensive numerical experiments showing the effectiveness  of the proposed scheme are discussed in   Sec. \ref{num_results} whereas   Sec.~\ref{conclusions} draws some conclusions. \vspace{-0.2cm}

\subsection{Notation}
\label{notation_section} 
Throughout the paper we  use the following notation.   We denote by $\mathbb{N}_+$ the set of non-negative integers. {Given $x\in\mathbb{R}$,   $\ceil{x}$ (resp. $\floor{x}$) denotes the smallest (resp. the largest) integer greater (resp. smaller) than or equal to $x$.} Vectors are denoted by bold lower-case 
letters (e.g., $\mathbf{x}$) whereas matrices are denoted by bold capital letters (e.g., $\mathbf{A}
$). The $k$-th canonical vector is denoted by  $\mathbf{e}_k$.  
The inner product between two real matrices, $\mathbf{A}$ and $\mathbf{B}$, is denoted by   $
\left\langle \mathbf{A},\mathbf{B}\right\rangle\triangleq \mathrm{tr}(\mathbf{A}^\intercal
\mathbf{B})$, where $\text{tr}(\bullet)$ is the trace operator{; $ \mathbf{A}  \otimes \mathbf{B}$ denotes the Kronecker product.}
  {Given the real matrix $\mathbf{A}$, with $ij$-entries denoted by $A_{ij}$, we will use the  
  following    matrix norms:   the Frobenius norm $||\mathbf{A}||_F\triangleq\sqrt{\sum_{i,j}|
  A_{ij}|^2}$; the $L_{1,1}$ norm $||\mathbf{A}||_{1,1}\triangleq\sum_{i,j}|A_{ij}|$; the $L_{2,\infty}
  $ norm $||\mathbf{A}||_{2,\infty}\triangleq\max_i\,\sqrt{\sum_jA_{ij}^2}$; the $L_{\infty,\infty}$ 
  norm $||\mathbf{A}||_{\infty,\infty}=\max_{i,j}|A_{ij}|$; and the spectral norm 
  $||\mathbf{A}|| _2\triangleq\sigma_{\mathrm{max}}(\mathbf{A})$}, where $
  \sigma_{\mathrm{max}}(\mathbf{A})
  $ denotes the  maximum singular value of $\mathbf{A}$. 
 The matrix quantities $\nabla_D f_i(\mathbf{D},\mathbf{X}_i)$ and $\nabla_{X_i} f_i(\mathbf{D},
 \mathbf{X}_i)$  are the gradients of $f_i$ with respect to $\mathbf{D}$ and $\mathbf{X}_i$, 
 evaluated  at $(\mathbf{D},\mathbf{X}_i)$, respectively, with the partial derivatives arranged 
 according to the patterns of $\mathbf{D}$ and $\mathbf{X}_i$, respectively. The same convention 
 is adopted  for subgradients of $g_i$ and $G$, that are therefore written as matrices of the same 
 dimensions of $\mathbf{X}_i$ and $\mathbf{D}$, respectively.   {Table \ref{Table_of_Notation} summarizes the main notation and symbols used in the paper}. 
 
Because of the nonconvexity of Problem \ref{eq:P1}, we aim at computing stationary solutions of  \ref{eq:P1}, defined as follows:  a  tuple $({\mathbf{D}^\ast}, \mathbf{X}^\ast)$, with  $\mathbf{X}^\ast\triangleq [\mathbf{X}_1^\ast, \dots, \mathbf{X}_I^\ast]$ is a stationarity solution of  \ref{eq:P1} if the following holds: $\mathbf{D}^\ast\in \mathcal{D}$, $\mathbf{X}_i^\ast\in \mathcal{X}_i$, $i=1,\ldots, I$, and 
{\begin{equation}
\begin{aligned}
 &\Big\langle\nabla_D F({\mathbf{D}^\ast},{\mathbf{X}^\ast}), \mathbf{D} -\mathbf{D}^\ast\Big\rangle+G(\mathbf{D}) -G(\mathbf{D}^\ast) \, \geq\, 0, & \forall \,
 \mathbf{D} \in { \cal D},
 \\
&\left\langle\nabla_{X_i} f_i(\mathbf{D}^\ast,\mathbf{X}_i^\ast),\mathbf{X}_i-{\mathbf{X}}_i^\ast\right\rangle+ g_i(\mathbf{X}_i)-g_i(\mathbf{X}_i^\ast)\, \geq \, 0, &  \forall\mathbf{X}_i\in\mathcal{X}_i, \; i = 1,\ldots, I.
\end{aligned}
\label{stationarity_cond}
\end{equation}	}
 
\begin{table}[H]
	\center \footnotesize  
	\renewcommand{\arraystretch}{1.2}
	\begin{tabular}{|c || c | c |c |}
		\hline
		\textbf{Symbol}  & \textbf{Definition} & \textbf{Member of} & \textbf{Reference}
		\\
		\hline
		\hline
		{$F(\mathbf{D},\mathbf{X})$} & $\sum_{i=1}^I f_i(\mathbf{D},\mathbf{X}_i)$ &  $\mathbb{R}^{M\times K}\times\mathbb{R}^{K\times N}\rightarrow\mathbb{R}$ &	\eqref{eq:P1}
		\\
		\hline
		{$U(\mathbf{D},\mathbf{X})$} & $F(\mathbf{D},\mathbf{X})+\sum_{i=1}^I g_i(\mathbf{X}_i)+G(\mathbf{D})$ &  $\mathbb{R}^{M\times K}\times\mathbb{R}^{K\times N}\rightarrow\mathbb{R}$ &\eqref{eq:P1}	
		\\
		\hline
		\parbox{1em}{$\mathbf{S}_i$} & Local data matrix &  $\mathbb{R}^{M\times n_i}$ &
		\\
		\hline	
		\parbox{1em}{$\mathbf{S}$} & $[\mathbf{S}_1,\mathbf{S}_2,\ldots,\mathbf{S}_I]$ &  $\mathbb{R}^{M\times N}$ &	
		\\
		\hline
		\parbox{1em}{$\mathbf{D}$} & Dictionary matrix variable &  $\mathcal{D}\subseteq\mathbb{R}^{M\times K}$ &
		\\
		\hline
		\parbox{1em}{$\mathbf{D}_{(i)}$} & Local copy of $\mathbf{D}$ of agent $i$ &  $\mathcal{D}\subseteq\mathbb{R}^{M\times K}$ &
		\\
		\hline
		\parbox{1em}{$\mathbf{D}^\nu_{(i)}$} & $\mathbf{D}_{(i)}$ at iteration $\nu$&  $\mathcal{D}\subseteq\mathbb{R}^{M\times K}$ &
		\\
		\hline
		\parbox{1em}{$\mathbf{D}^\nu$} & $ [\mathbf{D}_{(1)}^{\nu\intercal}, \mathbf{D}_{(2)}^{\nu\intercal},\ldots,\mathbf{D}_{(I)}^{\nu\intercal}]^\intercal$ & $\mathbb{R}^{M\times KI}$ & \eqref{Dnu_Xnu_Dbar_Def}
		\\
		\hline 
		\parbox{1em}{$\widetilde{\mathbf{D}}_{(i)}^\nu$} & Solution of subproblem \eqref{D_tilde_subproblem} & $\mathcal{D}\subseteq\mathbb{R}^{M\times K}$ & \eqref{D_tilde_subproblem}
		\\
		\hline
		\parbox{1em}{$\overline{\mathbf{D}}^\nu$} & $(1/I)\sum_{i=1}^I  \mathbf{D}_{(i)}^\nu$ &$\mathcal{D}\subseteq\mathbb{R}^{M\times K}$ & \eqref{Dnu_Xnu_Dbar_Def}
%		\\
%		\hline 
%		\parbox{1em}{$\widehat{\mathbf{D}}_{\text{avg}}^\nu$}  & Solution of subproblem \eqref{D_avg_def}& $\mathbb{R}^{M\times K}$  &\eqref{D_avg_def}
%		\\
%		\hline 
%		\parbox{1em}{$\boldsymbol{\Delta}_D^\nu$}  & $\widehat{\mathbf{D}}_{\text{avg}}^\nu-\overline{\mathbf{D}}^\nu$& $\mathbb{R}^{M\times K}$  &\eqref{Delta_Dnu}	
		\\
		\hline
		\parbox{1em}{$\mathbf{U}^\nu_{(i)}$} &  Local update of dictionary variable &  $\mathcal{D}\subseteq\mathbb{R}^{M\times K}$ & \eqref{U_update}
		\\
		\hline
		\parbox{1em}{$\mathbf{X}_i$} & Local matrix variable &  $\mathcal{X}_i\subseteq\mathbb{R}^{K\times n_i}$ &
		\\
		\hline
		\parbox{1em}{$\mathbf{X}$} & $ [\mathbf{X}_1,\mathbf{X}_2,\ldots,\mathbf{X}_I]$  &  $\mathcal{X}\subseteq\mathbb{R}^{K\times N}$ &
		\\
		\hline
		\parbox{1em}{$\mathbf{X}_i^\nu$} & $\mathbf{X}_i$ at iteration $\nu$ &  $\mathcal{X}_i\subseteq\mathbb{R}^{K\times n_i}$ &
		\\
		\hline
		\parbox{1em}{$\mathbf{X}^\nu$} & $\mathbf{X}$ at iteration $\nu$ : $ [\mathbf{X}_1^\nu,\mathbf{X}_2^\nu,\ldots,\mathbf{X}_I^\nu]$   &  $\mathcal{X}\subseteq\mathbb{R}^{K\times N}$ & \eqref{Dnu_Xnu_Dbar_Def}
%		\\
%		\hline
%		\parbox{1em}{$\widehat{\mathbf{X}}_{\text{avg},i}^\nu$} &  Solution of subproblem \eqref{X_avg_def}  &  $\mathcal{X}_i\subseteq\mathbb{R}^{K\times n_i}$ &	\eqref{X_avg_def}	
%		\\
%		\hline
%		\parbox{1em}{$\widehat{\mathbf{X}}_{\text{avg}}^\nu$} &  $[\widehat{\mathbf{X}}_{\text{avg},1}^\nu,\widehat{\mathbf{X}}_{\text{avg},2}^\nu,\ldots,\widehat{\mathbf{X}}_{\text{avg},I}^{\nu}]$ &  $\mathcal{X}\subseteq\mathbb{R}^{K\times N}$ &	\eqref{Delta_Xnu}		
%		\\
%		\hline
%		\parbox{1em}{$\boldsymbol{\Delta}_X^\nu$} &  $\widehat{\mathbf{X}}_{\text{avg}}^\nu-\mathbf{X}^\nu$ &  $\mathcal{X}\subseteq\mathbb{R}^{K\times N}$ &	\eqref{Delta_Xnu}			
		\\
		\hline
		\parbox{1em}{$\boldsymbol{\Theta}^{\nu}_{(i)}$} & Gradient-tracking variable  &  $\mathcal{D}\subseteq\mathbb{R}^{M\times K}$ & \eqref{Theta_update}
		\\
		\hline
		\parbox{1em}{$\mathbf{A}^\nu$} &  $(a_{ij}^\nu)_{i,j=1}^I-$ consensus weights at time $\nu$ & $\mathbb{R}^{I\times I}$ & Assumption \ref{A_matrix_Assumptions}
%		\\
%		\hline
%		\parbox{1em}{$\phi_i^\nu$} & Rescaling variable  & $\mathbb{R}$ & \eqref{D_weighted_avg}
%		\\
%		\hline 
%		$L_{\nabla{X_i}}(\mathbf{D})$  & Lipschitz constant of $\nabla_{X_i} f_i(\mathbf{D},\bullet)$ & $\mathbb{R}$	& 
		\\
		\hline 	
	\end{tabular}
	\caption{Table of notation}
	\label{Table_of_Notation}
\end{table}

\section{Problem Setup and Motivating  Examples}\vspace{-0.1cm}
In this section, we first discuss  the assumptions underlying our model and then provide several examples of possible applications.
\label{sec:distributed-DL} 
\subsection{Problem Assumptions}
\label{formualtion}
We consider Problem \ref{eq:P1} under the following assumptions.

\begin{assumption}[On Problem \ref{eq:P1}]~
\label{Problem_Assumptions}\vspace{-0.1cm}
\begin{itemize}
	\item[\textbf{(\ref{Problem_Assumptions}1)}]   {Each}  $f_i:\mathcal{O}\times \mathcal{O}_i\rightarrow\mathbb{R}$ is $\mathcal{C}^2$, {lower bounded,} and biconvex, where $\mathcal{O}\supseteq\mathcal{D}$ and   $\mathcal{O}_i\supseteq\mathcal{X}_i$  are convex open sets;
		\item[\textbf{(\ref{Problem_Assumptions}2)}] Given $\mathbf{D}\in\mathcal{D}$,  {each} $\nabla_{X_i} f_i(\mathbf{D},\bullet)$ is Lipschitz continuous on $\mathcal{X}_i$, with Lipschitz constant $L_{\nabla{X_i}}(\mathbf{D})$. Furthermore, {each} $L_{\nabla{X_i}}:\mathcal{D}\rightarrow\mathbb{R}$ is continuous;
	\item[\textbf{(\ref{Problem_Assumptions}3)}]   $\!\mathcal{D}$ is compact and convex; and each  $\mathcal{X}_i$ is closed and convex (not necessarily bounded);
	\item[\textbf{(\ref{Problem_Assumptions}4)}]  $G:\mathcal{O} \rightarrow\mathbb{R}$ is   convex (possibly non-smooth);%, where  $\mathcal{O}\supseteq\mathcal{D}$ is a convex open set;
	
\item[\textbf{(\ref{Problem_Assumptions}5)}] For all $i=1,\ldots, I$, either i)  $\mathcal{X}_i$ is {compact} and   $g_i:\mathcal{O}_i\rightarrow\mathbb{R}$ is convex; or ii)   $g_i$ is $\mu_i$-strongly convex. %\footnote{If $\mathcal{X}_i$ is bounded, this assumption on $g_i$ can be removed for the price of weaker convergence results (cf. Theorem \ref{th:conver})}
\end{itemize}
\end{assumption}

\noindent 
 The above assumptions  are quite mild and are satisfied by several problems of practical interest; see Sec. \ref{examples} for several concrete examples.

\subsection*{Network topology}\label{Network_topology}
We study Problem \ref{eq:P1} in   the following network setting.
Time is slotted and in  each time-slot $\nu$ the network of the $I$ agents is modeled as a digraph $\mathcal{G}^{\nu}=(\mathcal{V},\mathcal{E}^\nu)$, where   $\mathcal{V}=\{1,\ldots,I\}$ is the set of agents   and  $\mathcal{E}^\nu$ is the set of edges (communication links); we use $(i,j)\in \mathcal{E}^\nu$ to indicate that there is a directed link from node $i$ to node $j$.   The \emph{in-neighborhood} of agent $i\in\mathcal{V}$ at time $\nu$ is defined as  $\mathcal{N}_i^{\rm in}[\nu]=\{j\in\mathcal{V}|(j,i)\in\mathcal{E}^\nu\}\cup\{i\}$ (see  Fig. \ref{in_neighborhood_set}) whereas its \emph{out-neighborhood} is $\mathcal{N}_i^\textrm{out}\left[\nu\right]=\{j\in\mathcal{V}|\left(i,j\right)\in\mathcal{E}^\nu\}\cup\{i\}$. In words, agent $i$ can receive information from its in-neighborhood members, and  send information to its out-neighbors. The \emph{out-degree} of agent $i$ is defined as $d_i^\nu \triangleq  \left|\mathcal{N}_i^\textrm{out}\left[\nu\right]\right|$, where $\left|\bullet\right|$ denotes the cardinality of a set. If the graph is undirected, the set of in-neighbors and out-neighbors coincide; in such a case we just write $\mathcal{N}_i$ to denote the set of neighbors of agent $i$.  When the network is static, all the above quantities do not depend on the iteration index $\nu$; hence, we will drop the superscript ``$\nu$''.  To let information propagate over the network,  we assume that the sequence $\{\mathcal{G}^\nu\}_{\nu}$ possesses some ``long-term'' connectivity property, as stated next.

\begin{figure}[t]
\center\vspace{-0.7cm}
\includegraphics[scale=0.21]{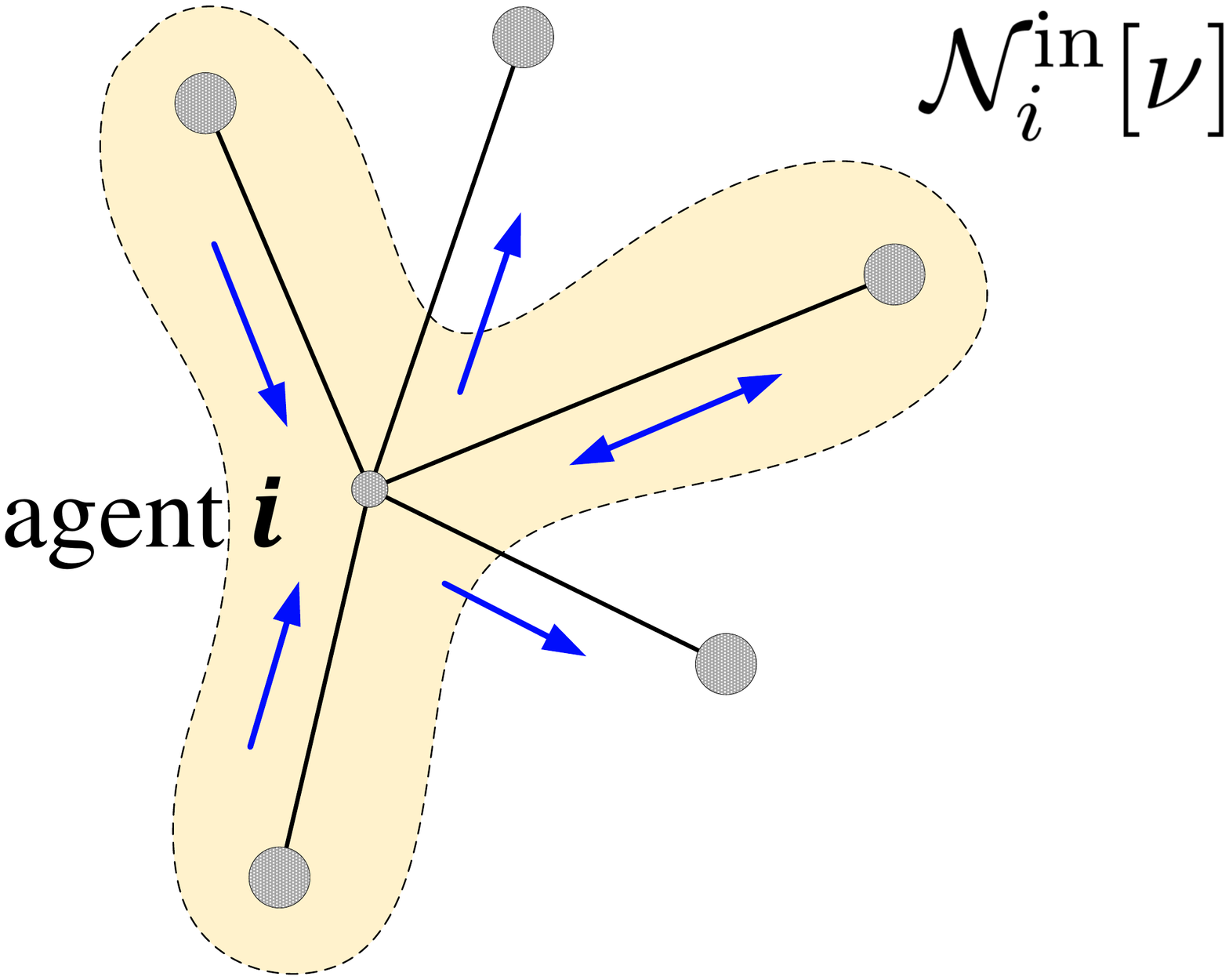}\vspace{-0.3cm}
\caption{Illustration of in-neighborhood set of agent $i$ at time $\nu$.}
\label{in_neighborhood_set}\vspace{-0.4cm}
\end{figure}

\begin{assumption}[B-strong connectivity]
\label{B_strongly_connectivity}
The graph sequence $\{\mathcal{G}^\nu\}_{\nu}$ is $B$-strongly connected, i.e., there exists an (arbitrarily large) integer $B>0$  (unknown to the agents) such that the graph with edge set $\cup_{t=kB}^{(k+1)B-1} \mathcal{E}^t$  is strongly connected,  for all $k\geq 0$.
\end{assumption}
Notice that  this condition  is quite mild and   widely used     in the literature to analyze convergence of  distributed algorithms over time-varying networks. Generally speaking, it permits strong connectivity to occur over time windows of length $B$, so that  information can propagate   from every node  to every other node in the network. 
Assumption~\ref{B_strongly_connectivity} is  satisfied in several practical scenarios. % including sensor networks, cloud infrastructures, etc.  
For instance, commonly used settings in cloud computing infrastructures  are   star, ring, tree, hypercube, or n-dimensional mesh (Torus) topologies, which all  satisfy  Assumption \ref{B_strongly_connectivity}. It is worth  mentioning that the multi-hop network topologies of these structures are migrating towards \emph{high-radix mesh and Torus}, since they are scalable, low-energy consuming,  and much cheaper than other topologies, like \emph{fat-tree} topologies  \citep{Kim_PhDThesis_Stanford_2008}. These type of connected networks are generally time-invariant and undirected, and  clearly  they  satisfy  Assumption \ref{B_strongly_connectivity}.

\subsection{Motivating examples}  \label{sec:examples}
We conclude this section discussing some practical instances of Problem \ref{eq:P1}, all satisfying Assumption~\ref{Problem_Assumptions}, which show the generality of the proposed model. \vspace{-0.2cm}
\label{examples}
\subsection*{Elastic net sparse DL \citep{Tosic_Frossard_2011,elasticNet_Hastie_2005}}
\label{elastic_net_example}
Sparse approximation of a signal with an adaptive dictionary is one of the most studied DL problems \citep{Tosic_Frossard_2011}. When an elastic net sparsity-inducing regularizer is used \citep{elasticNet_Hastie_2005}, the problem can be written as
\begin{equation}
\begin{aligned}
\min_{\mathbf{D},\{\mathbf{X}_i\}_{i=1}^I}\quad& \sum_{i=1}^I~\left\{\frac{1}{2}\,\norm{\mathbf{S}_i-\mathbf{D}\mathbf{X}_i}_F^2+\lambda\norm{\mathbf{X}_i}_{1,1}+\frac{\mu}{2}\norm{\mathbf{X}_i}_F^2\right\}
\\
\mathrm{s.t.}\qquad &\mathbf{D} \in {\cal D},~\mathbf{X}_i\in\mathbb{R}^{K\times n_i},\quad i=1,2,\ldots,I,
\end{aligned}
\label{eq:sparse_DL}
\end{equation}
where $\mathcal{D}\triangleq\{\mathbf{D}\,:\,||\mathbf{D}\mathbf{e}_k||_2\leq\alpha, \,k=1,2,\ldots,K\}$, and  $\alpha,\,\lambda,\,\mu>0$ are the tuning parameters. Problem \eqref{eq:sparse_DL} is an instance of \ref{eq:P1}, with $f_i(\mathbf{D},\mathbf{X}_i)=({1}/{2})\cdot\norm{\mathbf{S}_i-\mathbf{D}\mathbf{X}_i}_F^2$, $g_i(\mathbf{X}_i)=\lambda\norm{\mathbf{X}_i}_{1,1}+\frac{\mu}{2}\norm{\mathbf{X}_i}_F^2$, $G(\mathbf{D})=0$, and $\mathcal{X}_i=\mathbb{R}^{K\times n_i}$. It is not difficult to check that  \eqref{eq:sparse_DL} satisfies Assumption \ref{Problem_Assumptions}, and the Lipschitz constant in \ref{Problem_Assumptions}2 is given by  $L_{\nabla X_i}(\mathbf{D})=(\sigma_{\mathrm{max}}(\mathbf{D}))^2$. 
%Note that  the elastic net regularizer $g_i$ is preferred to the plain $\ell_1$ regularization  in 
%\eqref{eq:sparse_DL}, since it better preserves group patterns in the variables, especially when 
%there are highly correlated variables \citep{elasticNet_Hastie_2005}.

\subsection*{Supervised DL \citep{Mairal_NIPS08}}
\label{Supervised_DL}
Consider a classification problem with training set $\{\mathbf{s}_n,y_n\}_{n=1}^N$, where $\mathbf{s}_n$ is the feature vector with associated   binary label $y_n$. The discriminative DL problem aims at simultaneously learning a dictionary $\mathbf{D}^{(1)}\in\mathbb{R}^{M\times K}$ such that  $\mathbf{s}_n = \mathbf{D}^{(1)}\mathbf{x}_n$, for some sparse  $\mathbf{x}_n\in\mathbb{R}^K$, and finding a bilinear classifier $\zeta_n(\mathbf{D}^{(2)},\mathbf{x}_n,\mathbf{s}_n)\triangleq \mathbf{s}_n^\intercal\,\mathbf{D}^{(2)}\,\mathbf{x}_n$  that best separates the coded data with distinct labels  \citep{Mairal_NIPS08}. Assume that each agent $i$ owns $\{\left(\mathbf{s}_n,y_n\right)\,:\,n\in\mathcal{S}_i\}$, with  $\{\mathcal{S}_i\}_{i=1}^{I}$ being a partition of $\{1,\ldots,N\}$,
then the discriminative DL reads \vspace{-0.1cm}
\begin{equation}
\begin{aligned}
\min_{\substack{\mathbf{D}^{(1)},\mathbf{D}^{(2)}\\ \{\mathbf{x}_n\}_{n=1}^N}}\quad& \sum_{i=1}^{I}\sum_{n\in\mathcal{S}_i} \left[\ell\left( y_n\zeta_n\left(\mathbf{D}^{(2)},\mathbf{x}_n,\mathbf{s}_n\right)\right)+\frac{1}{2}\norm{\mathbf{s}_n-\mathbf{D}^{(1)}\mathbf{x}_n}_2^2+g_n(\mathbf{x}_n)\right]
\\
\mathrm{s.t.}\qquad &\mathbf{D}^{(1)} \in {\cal D}^{(1)},~\mathbf{D}^{(2)}\in\mathcal{D}^{(2)},\ {\mathbf{x}_n\in\mathbb{R}^K},\quad n=1,2,\ldots,N,
\end{aligned}\vspace{-0.1cm}
\label{eq:SDL}
\end{equation}
where $\ell(x)\triangleq \log(1+e^{-x})$ is the \emph{logistic} loss function; and  $g_n\left(\mathbf{x}_n\right) \triangleq \lambda \norm{\mathbf{x}_n}_1 + ({\mu}/{2})\cdot \norm{\mathbf{x}_n}_2^2$ is the elastic net regularizer. The dictionary $\mathbf{D}^{(1)}$ and classifier parameter $\mathbf{D}^{(2)}$ are constrained to belong to the convex compact sets $\mathcal{D}^{(1)}$ and $\mathcal{D}^{(2)}$, respectively.
Problem \eqref{eq:SDL} is an instance of Problem \ref{eq:P1}, with  $\mathbf{D} \triangleq \left[\mathbf{D}^{(1)}, \mathbf{D}^{(2)}\right]$, $\mathbf{S}_i\triangleq \left[\mathbf{s}_n\right]_{n\in\mathcal{S}_i}$ and $\mathbf{X}_i\triangleq \left[\mathbf{x}_n\right]_{n\in\mathcal{S}_i}$. Note that Assumption \ref{Problem_Assumptions} is satisfied, and  the Lipschitz constant in \ref{Problem_Assumptions}2  is given by 
%\begin{equation}
%\nonumber
$ L_{\nabla X_i}(\mathbf{D})=({1}/{4})\cdot \|y_i\cdot\mathbf{s}_i^\intercal\,\mathbf{D}^{(2)}\|_2^2+(\sigma_{\mathrm{max}}(\mathbf{D}^{(1)}))^2.$%\end{equation}

\subsection*{DL for low-rank plus sparse representation \citep{Bouwmans20171}}\vspace{-0.1cm}
The low-rank plus sparse decomposition problems cover many applications in signal processing and machine learning \citep{Bouwmans20171}, including matrix completion, image denoising, deblurring, superresolution, and Principal Component Pursuit (PCP) \citep{R-PCA}.
Consider the  bi-linear model $\mathbf{S}\approx \mathbf{L}+\mathbf{H}\mathbf{Q}\mathbf{U}$:  the data matrix $\mathbf{S}$ is decomposed as the superposition of   a low-rank matrix $\mathbf{L}$ (capturing the correlations among data) and  $\mathbf{H}\mathbf{Q}\mathbf{U}$, where {$\mathbf{Q}\in\mathbb{R}^{M\times \tilde{K}}$} is an over-complete dictionary (capturing the representative modes of the data), {$\mathbf{U}\in\mathbb{R}^{\tilde{K}\times N}$} is a sparse matrix (representing the data parsimoniously), and  $\mathbf{H}$ is a given  \emph{degradation} matrix, which accounts for tasks such as denoising, superresolution, and deblurring.  
To enforce $\mathbf{L}$ to be low-rank, we employ the nuclear norm $\norm{\mathbf{L}}_*$ regularizer, which can be equivalently  rewritten as $\norm{\mathbf{L}}_* = \inf\{\frac{1}{2}\norm{\mathbf{P}}_F^2+\frac{1}{2}\norm{\mathbf{V}}_F^2\,:\, \mathbf{L}=\mathbf{P}\mathbf{V}\}$, where  $\mathbf{P}\in\mathbb{R}^{M\times L}$, $\mathbf{V}\in\mathbb{R}^{L\times N}$, and $L\ll \min(M,N)$ \citep{Srebro2005,Recht_SIAM_2010}.
Partitioning $\mathbf{V}$ and $\mathbf{U}$ according to $\mathbf{S}$, i.e.,  $\mathbf{S}_i = \mathbf{P}\mathbf{V}_i + \mathbf{H}\mathbf{Q}\mathbf{U}_i$, the problem reads
\begin{equation}
\begin{aligned}[c l]
\min_{\substack{\mathbf{P},\mathbf{Q}, (\mathbf{V}_i,\mathbf{U}_i)_{i=1}^I}}\quad& \sum_{i=1}^I~\Bigg[ \frac{1}{2}~\Big\|\mathbf{S}_i-\left[\mathbf{P}~ \mathbf{H}\mathbf{Q}\right]
{\begin{bmatrix}
\mathbf{V}_i
\\
\mathbf{U}_i
\end{bmatrix}} 
\Big\|_F^2
\\
&\qquad\qquad\qquad\quad+\frac{\zeta}{2I}\left( \norm{\mathbf{P}}_F^2+I\cdot\norm{\mathbf{V}_i}_F^2\right)
+\lambda\norm{\mathbf{X}_i}_{1,1}+\frac{\mu}{2}\norm{\mathbf{X}_i}_F^2\Bigg]
\\
\mathrm{s.t.} \quad\qquad&\mathbf{D} \in {\cal D}, \quad{\mathbf{X}_i\in\mathbb{R}^{(L+\tilde{K})\times n_i},\quad i=1,2,\ldots,I,}
\end{aligned}\label{eq:LRPS2}
\end{equation}
\noindent where $\mathcal{D}$ is some compact set; $\zeta>0$ is a constant used to promote the low-rank structure on $\mathbf{L}$ while  sparsity on $\mathbf{X}$ is enforced by the elastic net regularization, with constants $\lambda,\,\mu>0$. Problem  \eqref{eq:LRPS2} is clearly an instance of Problem \ref{eq:P1} wherein $f_i$ is the quadratic loss,  {and $[\mathbf{P},\mathbf{Q}]$ and $ 
[\mathbf{V}_i^\intercal, \mathbf{U}_i^\intercal]^\intercal$ are the shared and private variables {($K=L+\tilde{K}$)}, respectively.} Assumption~\ref{Problem_Assumptions} is satisfied, and  the Lipschitz constant in \ref{Problem_Assumptions}2 is given by $ L_{\nabla X_i}(\mathbf{D})=(\sigma_{\mathrm{max}}(\mathbf{\mathbf{P} \mathbf{H}\mathbf{Q}}))^2$.

A variant of this problem, which still is a particular case of Problem \eqref{eq:LRPS2}, is obtained by
replacing   the quadratic loss function  with the smoothed Huber function to achieve robustness against outliers \citep{Aravkin_2014}.\vspace{-0.2cm}

\subsection*{Sparse SVD/PCA \citep{Lee_Shen_Huang_Marron10, MAL-055, Mairal_Bach_2010}}\vspace{-0.1cm}
\label{SSVD}
Computing the SVD of a set of data with sparse singular vectors (Sparse SVD) is the foundation of many  applications in multivariate analysis, e.g., biclustering \citep{Lee_Shen_Huang_Marron10}.
As proposed in \citep{Mairal_Bach_2010}, Problem \ref{eq:P1} can be used to accomplish this task by imposing sparsity on the factors $\mathbf{D}$ and $\mathbf{X}$ of $\mathbf{S}$. More specifically, we have\vspace{-0.1cm}
\begin{equation}
\begin{aligned}
\min_{\mathbf{D},(\mathbf{X}_i)_{i=1}^I}\quad& \sum_{i=1}^I~\left\{\frac{1}{2}~\norm{\mathbf{S}_i-\mathbf{D}\mathbf{X}_i}_F^2+\lambda_X\norm{\mathbf{X}_i}_{1,1}+\frac{\mu_X}{2}\norm{\mathbf{X}_i}_F^2\right\} 
+\lambda_D\norm{\mathbf{D}}_{1,1}+\frac{\mu_D}{2}\norm{\mathbf{D}}_F^2
\\
\mathrm{s.t.}\qquad &\mathbf{D} \in {\cal D} \triangleq \{\mathbf{D}\in\mathbb{R}^{M\times K}\,:\,||\mathbf{D}||_{2,\infty}\leq\alpha\},\quad  {\mathbf{X}_i\in\mathbb{R}^{K\times n_i},
\quad i=1,2,\ldots,I,}
\end{aligned}
\label{eq:SSVD}
\end{equation}
where $\lambda_D,\,\lambda_X,\,\mu_D,\,\mu_X,\,\alpha>0$ are given constants.
% {and $\mathcal X_i$ accounts for possible constraints on $\mathbf{X}_i$.}
 Problem \eqref{eq:SSVD}  is an instance of \ref{eq:P1}, with
 $f_i(\mathbf{D},\mathbf{X}_i)=({1}/{2})\cdot ||\mathbf{S}_i-\mathbf{DX}_i||^2_F$;  $G(\mathbf{D})=\lambda_D\,||\mathbf{D}||_{1,1}+({\mu_D}/{2})\cdot||\mathbf{D}||_F^2$, and $g_i(\mathbf{X}_i)=\lambda_X\,||\mathbf{X}_i||_{1,1}+({\mu_X}/{2})\cdot\,||\mathbf{X}_i||^2_F$.
  Note that orthonormality of factors are relaxed for sake of simplicity.
A related  formulation, termed Sparse PCA, has also been used  in \citep{MAL-055}.
 It is not difficult to show that Assumption \ref{Problem_Assumptions} is satisfied, and the Lipschitz constant in \ref{Problem_Assumptions}2  is given by $ L_{\nabla X_i}(\mathbf{D})=(\sigma_{\mathrm{max}}(\mathbf{D}))^2$.\vspace{-0.2cm}

 \subsection*{Non-negative Sparse Coding (NNSC) \citep{Hoyer_SNMF}}
\label{NNSC}\vspace{-0.1cm}
Non-negative Matrix Factorization (NMF) was primarily proposed by \citep{lee1999learning} as a better alternative to the classic SVD in learning localized features of image datasets, such as face images. The formulation enforces non-negativity of the entries of  $\mathbf{D}$ and $
\mathbf{X}$. This has been shown to  empirically lead to sparse solutions; however no explicit control on  sparsity is employed in the model. To overcome this  shortcoming, 
\citep{Hoyer_SNMF} proposed a  non-negative sparse coding (NNSC) formulation which extends NMF by adding a sparsity-inducing penalty function of $\mathbf{X}$. The problem reads %to 
%further control the sparsity. 
\begin{equation}
\begin{aligned}
\min_{\mathbf{D},(\mathbf{X}_i)_{i=1}^I}\quad& \sum_{i=1}^I~\left\{\frac{1}{2}\,\norm{\mathbf{S}_i-\mathbf{D}\mathbf{X}_i}_F^2+
\lambda\, \norm{\mathbf{X}_i}_{1,1} +  
\frac{\mu}{2}\,\norm{\mathbf{X}_i}_F^2
\right\} 
\\
\mathrm{s.t.}\qquad &\mathbf{D} \in {\cal D} \triangleq {\{\mathbf{D}\in\mathbb{R}_+^{M\times K}~|~||
\mathbf{D}||_{2,\infty}\leq \alpha\}},~\mathbf{X}_i\in\mathbb{R}^{K\times n_i}_+,\quad i=1,2,\ldots,I,
\end{aligned}
\label{eq:NNSC}
\end{equation}
  for some  $\lambda,\,\mu,\,\alpha>0$. Problem~\eqref{eq:NNSC} is another instance of  \ref{eq:P1},
 with 
 $f_i(\mathbf{D},\mathbf{X}_i)=({1}/{2})\cdot||\mathbf{S}_i-\mathbf{D}
\mathbf{X}_i||_F^2$, $g_i(\mathbf{X}_i)=\lambda \norm{\mathbf{X}_i}_{1,1} +  
({\mu}/{2})\cdot\norm{\mathbf{X}_i}_F^2$, $G(\mathbf{D})=0$, $\mathcal{D}=\{\mathbf{D}\in\mathbb{R}_+^{M\times K}~|~||
\mathbf{D}||_{2,\infty}\leq \alpha\}$, and $\mathcal{X}_i=\mathbb{R}_+^{K\times n_i}$.  Assumption 
\ref{Problem_Assumptions} is satisfied, and the Lipschitz constant in \ref{Problem_Assumptions}2  is given by $ L_{\nabla X_i}
(\mathbf{D})=(\sigma_{\mathrm{max}}(\mathbf{D}))^2$. %\vspace{-0.2cm}

\section{Algorithmic Design}\label{sec:alg-design}
We introduce now our algorithmic framework. To shed light on the core idea behind the proposed scheme, we begin introducing an
informal and constructive description of the algorithm, followed
by its formal description along with its convergence properties.

%To deal with the lack of global knowledge of the objective function,  %Let us start with the apparent shortcoming of distributed settings, which is: there is no \emph{central} shared memory, and the agents do not have \emph{immediate} access to the global variable $\mathbf{D}$, but only to their neighbors' information. A natural technique to deal with this issue 
%a natural approach is letting 
{Each agent $i$ controls its private variable  $\mathbf{X}_i$  and   maintains     a local copy of the shared variables $\mathbf{D}$, denoted by $\mathbf{D}_{(i)}$, along with an auxiliary variable $\boldsymbol{\Theta}_{(i)}$;  we anticipate that  $\boldsymbol{\Theta}_{(i)}$ aims at  {\it locally} estimating the gradient sum $\sum_j \nabla_{D} f_j(\mathbf{D}_{(i)},\mathbf{X}_j)$, an information that is not available at agent $i$'s side. The value of these variables at iteration $\nu$ is denoted by $\mathbf{X}_i^\nu$, $\mathbf{D}_{(i)}^\nu$, and $\boldsymbol{\Theta}_{(i)}^\nu$, respectively. Roughly speaking, the update of these variables is designed so that asymptotically i) all the $\mathbf{D}_{(i)}$ will be   consensual,  i.e.,   $\mathbf{D}_{(i)}=\mathbf{D}_{(j)}$, $\forall i\neq j$; and ii)  the tuples $(\mathbf{D}_{(i)}, (\mathbf{X}_j)_{j=1}^I)$ will be  a stationary solutions of  Problem~\ref{eq:P1}.  This is accomplished throughout the following two steps, which are performed iteratively and in parallel across the agents.  }

\begin{comment}

 These variables are iteratively updated  and controls its private variable  $\mathbf{X}_i$. %Note that we only use subscript $(i)$ to subindex the variables which are intended to reach a consensus on a common value; hence, we subindex $\mathbf{D}_{(i)}$ and $\mathbf{X}_i$ differently.
The optimization variables  $(\mathbf{D}_{(i)},\mathbf{X}_i)$ need to be updated so that %The goal is that each agent handle a part of the main problem, and collaboratively with their neighbors, they successively update $(\mathbf{D}_{(i)},\mathbf{X}_i)$ so that 
asymptotically:  i) all $\mathbf{D}_{(i)}$'s reach a consensus, i.e.,   $\mathbf{D}_{(i)}=\mathbf{D}_{(j)}$, $\forall i\neq j$; ii) and each tuple $(\mathbf{D}_{(i)}, (\mathbf{X}_j)_{j=1}^I)$ is  a stationary solutions of  Problem~\ref{eq:P1}. To achieve these goals, agents  face two main challenges: the \emph{non-convexity} of \ref{eq:P1} and the lack of \emph{global knowledge} of the  $f_i$'s. We deal with these issues by leveraging Successive Convex Approximation techniques (Step 1 below) and a {suitably designed perturbed push-sum  consensus  protocol} (Step 2), as described next.
\end{comment}

%\noindent \textbf{Step 1:\,Local updates:}
\subsection*{Step 1:\,Local Optimization}
The nonconvexity of $f_i$ together with the lack of knowledge of $\sum_{j\neq i}f_j$ in $F$ prevents agent $i$ to solve directly  Problem~\ref{eq:P1}  {with respect to $(\mathbf{D}_{(i)},\mathbf{X}_i)$}.
Since $f_i$ is  \emph{bi-convex}  in $(\mathbf{D}_{(i)},\mathbf{X}_i)$,    a  natural approach is  then  to update $\mathbf{D}_{(i)}$ and $\mathbf{X}_i$ in an {\it alternating} fashion by solving a {\it local} approximation  of \ref{eq:P1}.
Specifically, at  iteration $\nu$, given the iterates $\mathbf{X}_i^\nu$, $\mathbf{D}_{(i)}^\nu$, and $\boldsymbol{\Theta}_{(i)}^\nu$, agent $i$ fixes $\mathbf{X}_i=\mathbf{X}_i^\nu$ and  solves the following strongly convex problem in $\mathbf{D}_{(i)}$ : 
{\begin{equation}
\widetilde{\mathbf{D}}_{(i)}^\nu\!\triangleq \underset{\mathbf{D}_{(i)}\in\mathcal{D}}{\text{argmin}}~\tilde{f}_i\big(\mathbf{D}_{(i)};\mathbf{D}_{(i)}^\nu,\mathbf{X}_i^\nu\big)+\left\langle I\cdot
 {\boldsymbol{\Theta}}_{(i)}^\nu-\nabla_D f_i(\mathbf{D}^{\nu}_{(i)},\mathbf{X}^{\nu}_i),\mathbf{D}_{(i)}-\mathbf{D}_{(i)}^\nu\right\rangle+G\left(\mathbf{D}_{(i)}\right),
\label{D_tilde_subproblem}
\end{equation} }
 where $\tilde{f}_i(\bullet; \mathbf{D}_{(i)}^\nu,\mathbf{X}_i^\nu)$ is a suitably chosen strongly convex approximation of $f_i(\bullet,\mathbf{X}_i^\nu)$  at   $(\mathbf{D}_{(i)}^\nu,\mathbf{X}_i^\nu)$ {(cf. Assumption~\ref{Ass-surrogates}, Sec.~\ref{subsec:discussion}); and ${\boldsymbol{\Theta}}_{(i)}^\nu$, as anticipated,   {is used to  track the gradient of $F$, with   
 %the gradient average $(1/I) \sum_{j=1}^I\nabla_D f_j(\mathbf{D}_{(i)}^\nu,\mathbf{X}_j^\nu)%$, i.e., 
 $\lim_{\nu\to \infty}\|I\cdot {\boldsymbol{\Theta}}_{(i)}^\nu-\sum_{j=1}^I\nabla_D f_j(\mathbf{D}_{(i)}^\nu,\mathbf{X}_j^\nu)\|=0$;} which would lead to \vspace{-0.1cm}\begin{equation}\lim_{\nu\to \infty}\left\|\left(I\cdot  {\boldsymbol{\Theta}}_{(i)}^\nu-\nabla_D f_i(\mathbf{D}^{\nu}_{(i)},\mathbf{X}^{\nu}_i)\right)-\sum_{j\neq i}\nabla_D f_j(\mathbf{D}_{(i)}^\nu,\mathbf{X}_j^\nu)\right\|=0.\label{eq:tracking-property}\end{equation} This sheds light on the role of the linear term in (\ref{D_tilde_subproblem}): it can be regarded as a proxy of the sum-gradient $\sum_{j\neq i}\nabla_D f_j(\mathbf{D}_{(i)}^\nu,\mathbf{X}_j^\nu)$, which is not available at agent $i$'s side. In Step 2 below we show how to update  ${\boldsymbol{\Theta}}_{(i)}^\nu$  using only local information, so that (\ref{eq:tracking-property}) holds. }
 
Given  $\widetilde{\mathbf{D}}_{(i)}^\nu$, a step-size is employed in the update of $\mathbf{D}_{(i)}$, generating the iterate $\mathbf{U}_{(i)}^\nu$:
\begin{equation}
\label{U_update}
\mathbf{U}_{(i)}^\nu=\mathbf{D}_{(i)}^\nu+\gamma^\nu(\widetilde{\mathbf{D}}_{(i)}^\nu- \mathbf{D}_{(i)}^\nu),
\end{equation}
where $\gamma^\nu$ is the step-size,  to be  properly chosen (see Assumption \ref{assumption:step-size}, Sec. \ref{subsec:discussion}). 

Let us now consider the update of the private variables  $\mathbf{X}_i$. Fixing $\mathbf{D}_{(i)}=\mathbf{U}_{(i)}^\nu$, 
agent $i$ computes the new update $\mathbf{X}_i^{\nu+1}$ by solving the following strongly convex optimization problem:
\begin{equation}
\label{eq:Xupdate}
\mathbf{X}_i^{\nu+1}\triangleq \underset{\mathbf{X}_i\in\mathcal{X}_i}{\text{argmin}}\quad \tilde{h}_i(\mathbf{X}_i;
\mathbf{U}_{(i)}^\nu,\mathbf{X}_i^\nu)+g_i(\mathbf{X}_i),
\end{equation}
where $\tilde{h}_i(\bullet;
\mathbf{U}_{(i)}^\nu,\mathbf{X}_i^\nu)$ is a strongly convex function of $\mathbf{X}_i$, approximating $f_i(\mathbf{U}_{(i)}^\nu,\bullet)$  at   $(\mathbf{U}_{(i)}^\nu,\mathbf{X}_i^\nu)$;   {  see  Assumption~\ref{Ass-surrogates} (cf.  Sec.~\ref{subsec:discussion})  for specific instances of  $\tilde{h}_i$.}

\subsection*{Step 2:\,Local Communications}
Let us design now  a local communication  mechanism  ensuring asymptotic consensus over the local copies  $\mathbf{D}_{(i)}$'s and   property  (\ref{eq:tracking-property}). {To do so, we build on the (perturbed) push-sum protocol  {proposed in} \citep{sun2016distributed} (see also \cite{Kempe2003}). Specifically, an extra scalar variable $\phi_i$ is introduced at each agent's side to deal with the directed nature of the graph; given $\phi_i^\nu$ and  $\mathbf{U}_{(j)}^\nu$ from its in-neighbors  $j\in \mathcal{N}_i$, each agent $i$  updates its own local estimate $\mathbf{D}_{(i)}^\nu$ and $\phi_i^\nu$ according to:} \vspace{-0.2cm}
 \begin{equation}
{\phi^{\nu+1}_i =\sum_{j\in\mathcal{N}^{\rm in}_i[\nu]}a_{ij}^\nu\,\phi^\nu_j\quad \text{and}\quad\mathbf{D}_{(i)}^{\nu+1}=\frac{1}{\phi_i^{\nu +1}}\sum_{j\in\mathcal{N}^{\rm in}_i[\nu]}a^\nu_{ij}\,\phi_j^\nu\mathbf{U}_{(j)}^\nu.}
\label{D_weighted_avg}
\end{equation}
{where  $a_{ij}^\nu$'s  are some weights (to be properly chosen, see Assumption~\ref{A_matrix_Assumptions}, Sec.~\ref{subsec:discussion});   and $\phi_i^0=1$, for all $i=1,\ldots, I$.} %$\phi_i^\nu$ is an extra scalar variable that is updated according to

Note that the updates in  \eqref{D_weighted_avg}    can be    implemented locally: all agents only need to (i) {send their local variable $\mathbf{U}_{(j)}^\nu$ and the scalar weight $a_{ij}^\nu\,\phi_j^\nu$ to their neighbors}; and (ii) collect locally the information coming from the neighbors.

%\begin{comment}
%{\color{red}A similar scheme can be put forth  to update $\boldsymbol{\widetilde{\Pi}}_i^{\nu}$'s in \eqref{D_tilde_subproblem}. Let us first rewrite   $\boldsymbol{\Pi}_i^\nu$ [defined in \eqref{D_hat_subproblem}] as
%\begin{equation}
%\label{Pi_update_idea}
%\boldsymbol{\Pi}_i^\nu=I \cdot\underbrace{\left(\frac{1}{I}\displaystyle \sum_{j=1}^I\nabla_D f_j(\mathbf{D}^\nu_{(i)},
%\mathbf{X}^\nu_j)\right)}_{\triangleq \boldsymbol{\Theta}_i^\nu}-\nabla_D f_i(\mathbf{D}^\nu_{(i)}, \mathbf{X}^\nu_i).
%\end{equation}
%
% 
%   and leveraging the introduced communication protocol in \eqref{phi_update}-\eqref{D_weighted_avg}, we propose   to compute $\boldsymbol{\widetilde{\Pi}}_i^{\nu}$ mimicking \eqref{Pi_update_idea}, i.e.,
%\begin{equation}
%\label{Pi_update}
%\boldsymbol{\widetilde{\Pi}}_i^{\nu+1}=I\cdot\boldsymbol{\widetilde{\Theta}}^{\nu+1}_{(i)}-\nabla_D f_i(\mathbf{D}^{\nu+1}_{(i)},\mathbf{X}^{\nu+1}_i),
%\end{equation}
%}
%where $\boldsymbol{\widetilde{\Theta}}^{\nu+1}_{(i)}$ is an extra (matrix) variable maintained by  agent $i$, and instrumental to track $\boldsymbol{\Theta}_i^\nu$  in \eqref{Pi_update_idea}. To this end, we update $\boldsymbol{\widetilde{\Theta}}^{\nu+1}_{(i)}$ according to the 
%\end{comment}

To update the $\boldsymbol{\Theta}^{\nu}_{(i)}$ variables we leverage the 
  gradient tracking mechanism first introduced in  \citep{NEXT},  coupled with the push-sum consensus scheme \citep{sun2016distributed}, resulting in the following {perturbed push-sum scheme}:  
\begin{equation}
\boldsymbol{\Theta}^{\nu+1}_{(i)}= \frac{1}{\phi_i^{\nu+1}}\sum_{j\in\mathcal{N}_i^{\rm in}[\nu]}a_{ij}^\nu \phi_j^\nu\boldsymbol{{\Theta}}_{(j)}^{\nu}
+\frac{1}{\phi^{\nu+1}_i} \left(\nabla_D f_i(\mathbf{D}_{(i)}^{\nu+1},\mathbf{X}_i^{\nu+1})-\nabla_D f_i(\mathbf{D}_{(i)}^{\nu},\mathbf{X}_i^{\nu})\right),
\label{Theta_update}
\end{equation}
with  $\boldsymbol{{\Theta}}_{(i)}^0\triangleq \nabla_D f_i(\mathbf{D}_{(i)}^0,\mathbf{X}_i^0)$, for all $i=1,\ldots, I$. The update \eqref{Theta_update} follows similar logic as that  of $\mathbf{D}_{(i)}^\nu$ in \eqref{D_weighted_avg},   with the difference that     \eqref{Theta_update} contains a perturbation    [the second term in the RHS of \eqref{Theta_update}], which employs $\boldsymbol{{\Theta}}_{(i)}^{\nu}$  {and ensures the} desired tracking properties (otherwise $\boldsymbol{{\Theta}}_{(i)}^{\nu}$ would converge to the average of their initial values).  Note that \eqref{Theta_update} can be performed locally by agent $i$, following the same procedure as described for   \eqref{D_weighted_avg}.%-\eqref{phi_update}. %One can show that if $\mathbf{D}_{(i)}^\nu$'s and $\boldsymbol{\widetilde{\Theta}}^\nu_{(i)}$'s are consensual (a fact that is proved in Theorem \ref{th:conver}), then $\|\widetilde{\boldsymbol{\Pi}}_{i}^\nu-\sum_{j\neq i}\nabla_D f_j( \mathbf{D}_{(i)}^\nu,\mathbf{X}_j^\nu)\|_F \underset{\nu\rightarrow\infty}{\longrightarrow}0$.}

Combining the above steps, we can now  formally introduce the proposed distributed algorithm for the  DL problems \ref{eq:P1}, as described in Algorithm \ref{alg1}, and termed  \emph{D$^4$L} (Decentralized Dictionary Learning over Dynamic Digraphs) \emph{Algorithm}.  %We discuss next  the main assumptions underlying the choices of the free parameters in  the D$^4$L Algorithm, followed by    its convergence properties.

\begin{algorithm}[t]
~
\\
$\textbf{Initialization}:$ set $\nu=0$ and $
\phi^0_i=1,~\mathbf{D}_{(i)}^0\in\mathcal{D},~\mathbf{X}_i^0\in\mathcal{X}_i$,
$\boldsymbol{{\Theta}}_{(i)}^0=\nabla_D f_i(\mathbf{D}^0_{(i)},\mathbf{X}^0_i),$
\\
$~\qquad\qquad\qquad\quad$ for all $i=1,2,\ldots,I$.
\\
\texttt{S1.} If $(\mathbf{D}_{(i)}^\nu,\mathbf{X}_i^\nu)$ satisfies a suitable stopping criterion: \texttt{STOP};\smallskip
\\
\texttt{S2.} \textbf{Local Optimization:} Each agent $i$ computes:
\begin{enumerate}[label=(\alph*)]
\item $\widetilde{\mathbf{D}}_{(i)}^\nu$ and $\mathbf{U}_{(i)}^\nu$ according to \eqref{D_tilde_subproblem} and \eqref{U_update};
\item $\mathbf{X}_i^{\nu+1}$ according to \eqref{eq:Xupdate};
\end{enumerate}
\texttt{S3.} \textbf{Local Communications:} Each agent $i$ collects data
from its current neighbors and updates:
\begin{enumerate}[label=(\alph*)]
\item $\phi^{\nu+1}_i$ and $\mathbf{D}_{(i)}^{\nu+1}$ according to \eqref{D_weighted_avg};
\item $\boldsymbol{{\Theta}}^{\nu+1}_{(i)}$ according to  \eqref{Theta_update};
\end{enumerate}
\texttt{S4.} Set $\nu+1\to\nu$, and go to \texttt{S1}.
\\
~
\caption{: Decentralized Dictionary Learning over Dynamic Digraphs (D$^4$L)}
\label{alg1}
\end{algorithm}

\subsection{ {Algorithmic} Assumptions }
\label{subsec:discussion}
Before stating the {main} convergence result for the D$^4$L  Algorithm, we discuss the main assumptions governing the choices of  the free parameters of the algorithm, namely: the surrogate functions $\tilde{f}_i$ and $\tilde{h}_i$,   the step-size $\gamma^\nu$, and  the consensus weights $(a_{ij}^\nu)_{i,j=1}^I$. %,  and   the  coefficients $(\tau_{X,i}^\nu)_{i=1}^I$ and $(\tau_{D,i}^\nu)_{i=1}^I$.

\smallskip 
\subsubsection{On the choice of $\tilde{f}_i$ and $\tilde{h}_i$.}
%\noindent\textbf{On the choice of $\tilde{f}_i$ and $\tilde{h}_i$.}

The surrogate functions  {are chosen to satisfy the following assumption.}
 
 \begin{assumption}[On  $\tilde{f}_i$ and $\tilde{h}_i$]  
	\label{Ass-surrogates} { Given $\mathbf{D}_{(i)}^\nu$ and $\mathbf{X}_i^\nu$, $\tilde{f}_i(\bullet;\mathbf{D}_{(i)}^\nu,\mathbf{X}_i^\nu)$ in (\ref{D_tilde_subproblem}) is either \begin{equation}
\label{eq:ftilde1}
\tilde{f}_i(\mathbf{D}_{(i)};\mathbf{D}_{(i)}^\nu,\mathbf{X}_i^\nu)=f_i(\mathbf{D}_{(i)},
\mathbf{X}_i^\nu)+\frac{\tau_{D,i}^\nu}{2}\, ||\mathbf{D}_{(i)}-\mathbf{D}_{(i)}^\nu||_F^2,
\end{equation}}
	{ or 
\begin{equation}
\label{eq:ftilde2}
\tilde{f}_i(\mathbf{D}_{(i)};\mathbf{D}_{(i)}^\nu,\mathbf{X}_i^\nu)=\left\langle \nabla_D 
f_i(\mathbf{D}^\nu_{(i)}, \mathbf{X}_i^\nu),\mathbf{D}_{(i)}-\mathbf{D}_{(i)}^\nu\right\rangle+
\frac{\tau_{D,i}^\nu}{2}\, \big\|\mathbf{D}_{(i)}-\mathbf{D}_{(i)}^\nu\big\|_F^2, \vspace{-0.1cm}
\end{equation}
where $\tau_{D,i}^\nu$ is a positive scalar satisfying   Assumption \ref{freeVars_assumptions}.}

{Given $\mathbf{U}_{(i)}^\nu$ and $\mathbf{X}_i^\nu$,  $\tilde{h}_i(\bullet;\mathbf{U}_{(i)}^\nu,\mathbf{X}_i^\nu)$ in (\ref{eq:Xupdate}) is either 
\begin{equation}
\label{eq:htilde1}
\tilde{h}_i(\mathbf{X}_i;
\mathbf{U}_{(i)}^\nu,\mathbf{X}_i^\nu)\triangleq f_i(\mathbf{U}_{(i)}^\nu,
\mathbf{X}_i)+ \frac{\tau_{X,i}^\nu}{2}\norm{\mathbf{X}_i-\mathbf{X}_i^\nu}_F^2,
\end{equation}
or 
\begin{equation}
\label{eq:htilde2}
\tilde h_i(\mathbf{X}_i;\mathbf{U}_{(i)}^\nu,\mathbf{X}_i^\nu)= \left\langle\nabla_{X_i} f_i(\mathbf{U}_{(i)}^\nu,\mathbf{X}_i^\nu),\mathbf{X}_i-\mathbf{X}_i^\nu\right\rangle+\frac{\tau_{X,i}^\nu}{2}\big\|\mathbf{X}_i-\mathbf{X}_i^\nu\big\|_F^2,
\end{equation}
where $\tau_{X,i}^\nu$ is a positive scalar satisfying   Assumption \ref{freeVars_assumptions}.}
\end{assumption}
 
\begin{assumption}[On    $\tau_{X,i}^\nu$ and $\tau_{D,i}^\nu$]
\label{freeVars_assumptions} 
The parameters     $(\tau_{X,i}^\nu)_{i=1}^I$ and $(\tau_{D,i}^\nu)_{i=1}^I$ are chosen such that\vspace{-0.2cm}
\begin{description}
%\item[\textbf{(\ref{freeVars_assumptions}1)}]  $\{\gamma^\nu\}_\nu$  satisfies:
%$\gamma^\nu\in(0,1]$,  for all $\nu$; $\sum_{\nu=0}^\infty \gamma^\nu=\infty$; and $\sum_{\nu=0}^\infty\left(\gamma^\nu\right)^2<\infty$;
\item[\textbf{(\ref{freeVars_assumptions}1)}]  $\{\tau^\nu_{D,i}\}_\nu$ and $\{\tau^\nu_{X,i}\}_\nu$ satisfy\vspace{-0.2cm}
\begin{equation}
0<\inf_{\nu}\tau^\nu_{D,i}\leq \sup_{\nu}\tau^\nu_{D,i}<+\infty,
\label{tau_D_condition}
\end{equation}
and\vspace{-0.3cm}
\begin{eqnarray}
& \sup_{\nu}\tau^\nu_{X,i}<+\infty,\medskip\nonumber\\
&\tau^\nu_{X,i}\geq\frac{1}{2}L_{\nabla{X_i}}(\mathbf{U}_{(i)}^\nu)+\epsilon,\quad \forall\nu\geq 
1,
\label{tau_X_cond_c}
\end{eqnarray}
for all $i=1,2,\ldots,I$, where $\epsilon>0$ is an arbitrarily small constant,  and   $L_{\nabla{X_i}}
$ is defined in Assumption \ref{Problem_Assumptions}2.
\item[\textbf{(\ref{freeVars_assumptions}2)}] Stronger convergence results [cf. Theorem~
\ref{th:conver}] can be obtained if, {under Assumption  \ref{Problem_Assumptions}5(ii)},  the sequences $\{\tau^\nu_{D,i}\}_\nu$ and $\{\tau^\nu_{X,i}
\}_\nu$, in addition to D1, also satisfy \vspace{-0.1cm}
\begin{equation}
\label{extra_condition_tau_D}
 \sum_{t=0}^\infty\left|\tau^{t+1}_{D,i}-\tau^{t}_{D,i}\right|<\infty,
\end{equation}
and\vspace{-0.1cm}
\begin{equation}
\label{extra_condition_tau_x}
\limsup_{\nu}\left|\tau^\nu_{X,i}-\tau^{\nu-1}_{X,i}\right|<\mu,
\end{equation}
where $\mu\triangleq \min_i \mu_i$ and  $\mu_i$ is the strongly convexity constant of $f_i$.
\end{description}
\end{assumption}

\noindent \textbf{Discussion.} Several comments are in order. 

\noindent $\bullet$ \textit{On the choice of $\tilde{f}_i$ and $\tilde{h}_i$}. 
Since $f_i$ (resp. $h_i$)  is convex in $\mathbf{D}_{(i)}$ (resp. $\mathbf{X}_{i}$), (\ref{eq:ftilde1}) [resp. (\ref{eq:htilde1})] is a  natural choice for the surrogate $\tilde{f}_i$ (resp. $\tilde{h}_i$): the structure of  $f_i$ (resp. $h_i$) is preserved while  a quadratic term is added to make the overall surrogate strongly convex.   The non-smooth strongly convex subproblems \eqref{D_tilde_subproblem} and \eqref{eq:Xupdate} resulting from (\ref{eq:ftilde1}) and  (\ref{eq:htilde1}) can be solved using standard solvers, e.g., projected subgradient methods. When dealing with large-scale instances, effective methods are also  \citep{Scutari-Facchinei-SagratellaTSP15,Dan-Facch-Kung-ScutTSP15}.

The alternative surrogates  $\tilde{f}_i$ and $\tilde{h}_i$ as given in \eqref{eq:ftilde2} and \eqref{eq:htilde2}, respectively, are based on the   the linearization of the original $f_i$ and $h_i$. This option is motivated by the fact that, for  specific instances of $f_i$ and $h_i$, \eqref{eq:ftilde2} and \eqref{eq:htilde2} lead to subproblems \eqref{D_tilde_subproblem} and \eqref{eq:Xupdate}  whose solution can be computed in closed form. 
 %  However, when it comes to specific loss functions $f_i$  and penalty functions $G$ and $g_i$, the alternative choices () and () can appropriate choices for  $\tilde{f}_i$ and $\tilde{h}_i$ can be employed to abate the computational complexity of solving the resulting subproblems \eqref{D_tilde_subproblem} and \eqref{eq:Xupdate}.
 %If instead one uses as surrogate $\tilde{h}_i$ the linearization of $f_i$ ,that is,
 For instance, consider the \emph{elastic net sparse DL} problem   \eqref{eq:sparse_DL} in Sec. \ref{examples}, where $f_i(\mathbf{D},\mathbf{X}_i)=\frac{1}{2}||\mathbf{S}_i-\mathbf{D}\mathbf{X}_i||_F^2$; $G(\mathbf{D})=0$; and  $g_i(\mathbf{X}_i)=\lambda\,||\mathbf{X}_i||_1+\frac{\mu}{2}\,||\mathbf{X}_i||^2_F$, with   $\lambda,\mu>0$.
By using \eqref{eq:ftilde2}, the resulting subproblem \eqref{D_tilde_subproblem} admits  the following closed form solution:\vspace{-0.2cm}
{\begin{equation}
\label{closed_form_dictionary_lin} 
\widetilde{\mathbf{D}}_{(i)}^\nu=P_\mathcal{D}\left[\mathbf{D}_{(i)}^\nu-\frac{I}{\tau_{D,i}^\nu}\boldsymbol{\Theta}_{(i)}^\nu\right].
\end{equation}}
Referring to the sparse coding subproblem \eqref{eq:Xupdate}, if  $\tilde{h}_i$ is chosen according to \eqref{eq:htilde1},   computing the update  $\mathbf{X}_i^{\nu+1}$ results in solving a LASSO problem. If instead one uses the surrogate  in (\ref{eq:htilde2}), the  solution of  \eqref{eq:Xupdate} can be computed in closed form as%\vspace{-0.1cm}
\begin{equation}
\mathbf{X}_i^{\nu+1}
=\frac{\tau^\nu_{X,i}}{\mu+\tau^\nu_{X,i}}\mathcal{T}_{\frac{\lambda}{\tau^\nu_{X,i}}} \left(\mathbf{X}^\nu_i-\frac{1}{\tau^\nu_{X,i}}\nabla_{X_i} f_i(\mathbf{U}_{(i)}^\nu,\mathbf{X}_i^\nu) \right),
\label{closed_form_X_linearization}
\end{equation}
where  $\mathcal{T}$ is  the soft-thresholding operator  $\mathcal T_\theta(x)\triangleq \max(|x|-\theta,0)\cdot \text{sign}(x)$   [with $\text{sign}(\cdot)$ denoting the sign function], applied to the matrix argument component-wise.

\smallskip 

\noindent$\bullet$ \textit{On the choice of $\tau_{X,i}^\nu$ and $\tau_{D,i}^\nu$.}  These coefficients must satisfy Assumption \ref{freeVars_assumptions}. Roughly speaking,   \ref{freeVars_assumptions}1   ensures that $(\tau_{X,i}^\nu)_{i=1}^I$ and $(\tau_{D,i}^\nu)_{i=1}^I$ are bounded (both from below and above) while  \ref{freeVars_assumptions}2 guarantees that these parameters are asymptotically  ``stable". 
  A trivial choice for $\tau_{D,i}^\nu$ satisfying both \eqref{tau_D_condition} and \eqref{extra_condition_tau_D} is $\tau_{D,i}^\nu=c$, for some $c>0$;    some practical rules for $\tau_{X,i}^\nu$ satisfying both  \eqref{tau_X_cond_c} and \eqref{extra_condition_tau_x} are the following:
\begin{enumerate}
\item[(a)] Use a constant $\tau_{X,i}^\nu$, that is,
\begin{equation}\nonumber
\tau^\nu_{X,i}=\max_{\mathbf{D}\in\mathcal{D}}\bigg[\max\bigg(\sigma_{\mathrm{max}}\big(\nabla^2_{X_i} f_i(\mathbf{D},\mathbf{X}_i^\nu)\big),\tilde{\epsilon}\bigg)\bigg],
\end{equation}
for some $\tilde{\epsilon}>0$. The above value can be,  however, much larger than any $\sigma_{\mathrm{max}}(\nabla^2_{X_i} f_i(\mathbf{U}_{(i)}^\nu,$ $\mathbf{X}_i^\nu))$, which can slow down the practical convergence of the algorithm;

\item[(b)] A less conservative choice is to satisfy \eqref{tau_X_cond_c} iteratively, while guaranteeing that $\tau_{X,i}^\nu$ is uniformly positive:
\begin{equation}\label{tau_x_second_choice}
\tau^\nu_{X,i}=\max(L_{\nabla{X_i}}(\mathbf{U}_{(i)}^\nu),\tilde{\epsilon}),
\end{equation}
where $\tilde{\epsilon}$ is any positive (possibly small) constant;
\item[(c)]
 A generalization of (b) is
\begin{equation}
\nonumber
\tau^\nu_{X,i}\in\left[\max(L_{\nabla{X_i}}(\mathbf{U}_{(i)}^\nu),\tilde{\epsilon}),\,L_{\nabla{X_i}}(\mathbf{U}_{(i)}^\nu)+{\tilde{\mu}}\right],
\end{equation}
for some $\tilde{\epsilon}$ and $\tilde{\mu}$  such that $0<\tilde{\epsilon} \leq \tilde{\mu}<\mu$.
\end{enumerate}

\begin{remark}
While the choices   (a)-(c)  above clearly satisfy \eqref{tau_X_cond_c}, it can be shown that    \eqref{extra_condition_tau_x} also holds, as a consequence of   the continuity of $L_{\nabla{X_i}}(\cdot)$ and Proposition \ref{consVar_props} (cf. Appendix \ref{Prelim_results}). 
\vspace{-0.1cm}
\end{remark} Note that all the above rules do not require any coordination among the agents, but are implementable in a fully distributed manner, using only local information.    

{
\subsubsection{On the choice of $\gamma^\nu$}
 The step-size can be chosen according to  the following  {assumption}. 
\begin{assumption}[On  $\gamma^\nu$]
\label{assumption:step-size} 
   $\{\gamma^\nu\}_\nu$  satisfies:
$\gamma^\nu\in(0,1]$,  for all $\nu$; $\sum_{\nu=0}^\infty \gamma^\nu=\infty$; and $\sum_{\nu=0}^\infty\left(\gamma^\nu\right)^2<\infty$.\end{assumption}}
The above assumption is the standard diminishing-rule; see, e.g., \citep{Bertsekas_ParallelMethods_1997}.  Here, we only recall one rule{, satisfying Assumption \ref{assumption:step-size},}  that we found very effective in our experiments, namely  \citep{Scutari-Facchinei-SagratellaTSP15}:  $\gamma^\nu=\gamma^{\nu-1}(1-\epsilon_0 \gamma^{\nu-1})$ with $\gamma^0\in (0,1]$ and $\epsilon_0\in (0,1/\gamma^0)$.

 \subsubsection{On the choice of the weigh coefficients   $\{a_{ij}^\nu\}$.} \label{subsec_weights}

We denote by   $\mathbf{A}^\nu$ the matrix whose entries are  the weights $a_{ij}^\nu$'s, i.e., $[\mathbf{A}^\nu]_{i,j}=a_{ij}^\nu$. This matrix is chosen so that  the following conditions are satisfied.

\begin{assumption}[On the weighting matrix]
\label{A_matrix_Assumptions}
Given the digraph $\mathcal{G}^{\nu}=(\mathcal{V},\mathcal{E}^\nu)$, each matrix $\mathbf{A}^\nu$, with $[\mathbf{A}^\nu]_{ij}=a_{ij}^\nu$, satisfies
	\begin{description}
		\item[\textbf{(\ref{A_matrix_Assumptions}1)}]  $a_{ii}^\nu\geq \kappa>0$ for all $i=1,\ldots,I$;
		\item[\textbf{(\ref{A_matrix_Assumptions}2)}] $a_{ij}^\nu\geq \kappa>0$, if $\left(j,i\right)\in \mathcal{E}^\nu$; and $a_{ij}^\nu=0$ otherwise;
		\item[\textbf{(\ref{A_matrix_Assumptions}3)}] $\mathbf{A}^\nu$ is column stochastic, i.e., $\mathbf{1} ^\intercal\mathbf{A}^\nu= \mathbf{1}^\intercal$.
	\end{description}
\end{assumption}
%Note that $\mathbf{A}^\nu$ is not required to be doubly-stochastic, but only column-stochastic and this can easily be guaranteed by each agent independently.  

{When the graph $\mathcal G^\nu$ is directed, a valid choice of   $\mathbf{A}^\nu$         is   \citep{Kempe2003}:  $a_{ij}^\nu=1/d_j^\nu$ if $j\in \mathcal{N}^{\rm in}_i[\nu]$, and $a_{ij}^\nu=0$ otherwise, where $d_j^\nu$ is the out-degree of agent $j$ at time $\nu$. 
The resulting  communication  protocols  \eqref{D_weighted_avg}--\eqref{Theta_update} can be easily implemented in a distributed fashion:  
each agent i) broadcasts its local variable normalized by its current out-degree; and ii) collects locally the information coming from its neighbors. When the graph is undirected, several options are available in the literature, including:  the Laplacian,   Metropolis-Hastings, and maximum-degree weights; see, e.g.,  \citep{xiao2005scheme}.} \vspace{-0.3cm} 

%Note that the above choice of $\mathbf{A}^\nu$  is the same as that used in the   push-sum protocol  \citep{Kempe2003}. However, there is a fundamental difference between \citep{Kempe2003} and the proposed consensus scheme:  push-sum-like protocols
%do not preserve feasibility of the iterates, and  thus cannot be readily used  
% to  solve distributed \emph{constrained} optimization
%problems.   In fact, the few works in the literature   using push-sum in the context
%of distributed optimization [see \citep{Nedic2015} and references therein]  all consider only \emph{unconstrained} problems.

\section{Convergence of D$^4$L}
\label{convergence_alg}
In this section, we provide the main convergence results for the D$^4$L Algorithm. 
  	 {	We begin introducing some definitions, instrumental to state our results. Let   \begin{equation}
  	 	\label{Dnu_Xnu_Dbar_Def}
  	 	\mathbf{D}^\nu\triangleq [\mathbf{D}_{(1)}^{\nu\intercal}, \mathbf{D}_{(2)}^{\nu\intercal},\ldots,\mathbf{D}_{(I)}^{\nu\intercal}]^\intercal,\quad \mathbf{X}^\nu\triangleq [\mathbf{X}_1^\nu,\mathbf{X}_2^\nu,\ldots,\mathbf{X}_I^\nu],\quad \text{and}\quad  \overline{\mathbf{D}}^\nu\triangleq \frac{1}{I}\sum_{i=1}^I\mathbf{D}_{(i)}^\nu.
  	 \end{equation}    }

{Given the sequence $\{(\mathbf{D}^\nu,\mathbf{X}^\nu)\}_\nu$ generated by   the D$^4$L Algorithm, convergence is stated measuring the distance of the   sequence $\{(\overline{\mathbf{D}}^\nu,\mathbf{X}^\nu\}_\nu$ from optimality as well as the consensus disagreement among the local variables $\mathbf{D}^\nu_{(i)}$'s. Distance from stationarity is measured by  
\begin{equation}
\label{stat_merit_subproblems}
\Delta^\nu\triangleq \max(\Delta_D(\overline{\mathbf{D}}^\nu,\mathbf{X}^\nu),\Delta_X(\overline{\mathbf{D}}^\nu,\mathbf{X}^\nu))
\end{equation}
where \begin{equation}
\begin{aligned}
\Delta_D(\overline{\mathbf{D}}^\nu,\mathbf{X}^\nu)\triangleq ||\widehat{\mathbf{D}}(\overline{\mathbf{D}}^\nu,\mathbf{X}^\nu)-\overline{\mathbf{D}}^\nu||_{\infty,\infty},\quad  \Delta_X(\overline{\mathbf{D}}^\nu,\mathbf{X}^\nu)\triangleq||\widehat{\mathbf{X}}(\overline{\mathbf{D}}^\nu,\mathbf{X}^\nu)-\mathbf{X}^\nu||_{\infty,\infty},
\end{aligned}
\end{equation} 
with the functions  $\widehat{\mathbf{D}}(\bullet,\bullet)$   
and $\widehat{\mathbf{X}}(\bullet,\bullet)$ defined as:
\begin{align}%\label{cons_stat_merits}
	&\widehat{\mathbf{D}}(\overline{\mathbf{D}}^\nu,\mathbf{X}^\nu) \triangleq \underset{\mathbf{D}'\in\mathcal{D}}
	{\mathrm{argmin}} \left\langle \nabla_D F(\overline{\mathbf{D}}^\nu,\mathbf{X}^\nu),\mathbf{D}'-\overline{\mathbf{D}}^\nu\right\rangle+\frac{\hat{\tau}_D}{2}\norm{\mathbf{D}'-\overline{\mathbf{D}}^\nu}^2+G(\mathbf{D}'),
	\label{D_avg_def}  
	\\
	&\widehat{\mathbf{X}}(\overline{\mathbf{D}}^\nu,\mathbf{X}^\nu)\triangleq [\widehat{\mathbf{X}}_1(\overline{\mathbf{D}}^\nu,\mathbf{X}_1^\nu),\ldots,\widehat{\mathbf{X}}_I(\overline{\mathbf{D}}^\nu,\mathbf{X}_I^\nu)],\nonumber
	\\
	&\quad\quad\text{ with }\widehat{\mathbf{X}}_i(\overline{\mathbf{D}}^\nu,\mathbf{X}_i^\nu)\triangleq {\underset{\mathbf{X}'_i\in\mathcal{X}_i}{\mathrm{argmin}}} 
	\left\langle\nabla_{X_i} f_i(\overline{\mathbf{D}}^\nu,\mathbf{X}_i^\nu),\mathbf{X}'_i-\mathbf{X}_i^\nu\right\rangle+\frac{\hat{\tau}_{X}}{2}\norm{\mathbf{X}'_i-\mathbf{X}_i^\nu}^2+g_i(\mathbf{X}'_i),
	\label{X_avg_def}
\end{align}
for some given constants $\hat{\tau}_D>0$ and $\hat{\tau}_{X}>0$.  Note that  $\Delta_D(\bullet,\bullet)$   
and $\Delta_X(\bullet,\bullet)$ are valid merit functions,   in the sense that  { they are continuous and  $\Delta_D(\overline{\mathbf{D}}^\infty,\mathbf{X}^\infty)=\Delta_X(\overline{\mathbf{D}}^\infty,\mathbf{X}^\infty)=0$ if and only if  $(\overline{\mathbf{D}}^\infty,\mathbf{X}^\infty)$ is a stationary solution of Problem~\eqref{eq:P1}  \citep{Scutari-Facchinei-SagratellaTSP15}}.}

 The consensus error at iteration $\nu$ is measured by the function   \begin{equation}
 e^\nu\triangleq ||\mathbf{D}^\nu-\mathbf{1}\otimes\overline{\mathbf{D}}^\nu||_{\infty,\infty}.\label{eq:consensus_error}     \vspace{-0.1cm}	
   \end{equation}
Asymptotic convergence of D$^4$L to stationary solutions of $\eqref{eq:P1}$ is stated in  Theorem \ref{th:conver} below while the convergence rate is studied in Theorem \ref{th:rate}.

{\begin{theorem}
\label{th:conver}
\it Given Problem \ref{eq:P1} under Assumption \ref{Problem_Assumptions},     %1--\ref{Problem_Assumptions}5(i), 
let 
   {$\big\{\big(\mathbf{D}^\nu,\mathbf{X}^\nu\big)\big\}_\nu$} be the sequence generated by the D$^4$L Algorithm for a given initial point $\big(\mathbf{D}^0,\mathbf{X}^0\big)$ and under Assumptions \ref{B_strongly_connectivity}, \ref{Ass-surrogates}, \ref{freeVars_assumptions}1, \ref{assumption:step-size},   \ref{A_matrix_Assumptions}. 
 Then,
\begin{enumerate}[label=(\alph*)]
\item \texttt{[Consensus]}:   {All   $\mathbf{D}_{(i)}^\nu$'s are asymptotically   consensual, i.e.,}  $\lim_{\nu\rightarrow\infty}e^\nu=0$;
\item \texttt{[Convergence]}: i) $\{(\overline{\mathbf{D}}^\nu,\mathbf{X}^\nu)\}_\nu$ is bounded; ii) $\{U(\overline{\mathbf{D}}^\nu,\mathbf{X}^\nu)\}_\nu$   converges to a finite value;   iii)   $\lim_{\nu\rightarrow\infty}\Delta_X(\overline{\mathbf{D}}^\nu,\mathbf{X}^\nu)=0$; and iv) $\liminf_{\nu\rightarrow\infty}\Delta_D(\overline{\mathbf{D}}^\nu,\mathbf{X}^\nu)=0$. Therefore,   $\{(\overline{\mathbf{D}}^\nu,\mathbf{X}^\nu)\}_\nu$   has at least one  limit point which is a stationary solution of    \ref{eq:P1}. %and 2)  the sequence of objective values $\{U(\overline{\mathbf{D}}^\nu,$ $\mathbf{X}^\nu)\}_\nu$ converges to  a value achieved by $U$ at a  stationary solution;
\end{enumerate}
If, in particular,  Assumption   \ref{Problem_Assumptions}5(ii) holds and  \ref{freeVars_assumptions}1 is reinforced  by \ref{freeVars_assumptions}2,  then convergence in (b) can be strengthened as follows:
\begin{enumerate}
\item [(b')]  Case (b) holds and %  $\{U(\overline{\mathbf{D}}^\nu,\mathbf{X}^\nu)\}_\nu$  converges to a finite value; $\{(\overline{\mathbf{D}}^\nu,\mathbf{X}^\nu)\}_\nu$ is bounded; and 
$\lim_{\nu\rightarrow\infty}\Delta^\nu=0$, implying that     all the  limit points of  $\{(\overline{\mathbf{D}}^\nu,\mathbf{X}^\nu)\}_\nu$ are stationary solutions of    \ref{eq:P1}.
\end{enumerate}
\end{theorem}
\begin{proof}
The proof is quite involved and is given in Appendix  \ref{sub:Proof_th_conver}.\vspace{-0.3cm}
\end{proof}}

The above theorem states two main convergence results {under
Assumptions \ref{B_strongly_connectivity}, \ref{Ass-surrogates}, \ref{freeVars_assumptions}1, \ref{assumption:step-size},   \ref{A_matrix_Assumptions}:  i) existence of at least a  subsequence  of $(\overline{\mathbf{D}}^\nu,\mathbf{X}^\nu)$  converging to a stationary solution of Problem \ref{eq:P1}; and  ii)  asymptotic  consensus of all  $\mathbf{D}_{(i)}^\nu$  to a common value $\overline{\mathbf{D}}^\nu$. If Assumption   \ref{Problem_Assumptions}5(ii) is also assumed and  \ref{freeVars_assumptions}1 is reinforced  by \ref{freeVars_assumptions}2 the stronger results in (b') can be proven, showing that  every limit point is a stationary solution. Note that from a practical point of view the weaker result guaranteeing existence 
 of at least a  subsequence   converging to a stationary solution is perfectly satisfying, since it guarantees that the algorithm can be terminated after a finite number of iterations with an approximate solution.}
\vspace{-0.1cm}

{\begin{theorem}
	\label{th:rate}
	\it Consider either settings of Theorem \ref{th:conver}, with the additional assumption that 
	the  step-size sequence $\{\gamma^\nu\}_\nu$ is non-increasing. For any given $\epsilon>0$, let ${T}_{D,\epsilon}\triangleq\min\{\nu\in \mathbb{N}_+: \Delta_D(\overline{\mathbf{D}}^\nu,\mathbf{X}^\nu)\leq \epsilon\}$ and   ${T}_{X,\epsilon}\triangleq\min\{\nu\in \mathbb{N}_+\,:\, \Delta_X(\overline{\mathbf{D}}^\nu,\mathbf{X}^\nu)\leq \epsilon\}$. 
	Then,
	\begin{enumerate}[label=(\alph*)]
		\item \texttt{[Rate of consensus error]}:
		\begin{equation}
		\vspace{-.5cm}
		e^\nu=  \mathcal{O}\left(\gamma^{\ceil{\theta \nu}}\right),		\vspace{-.3cm}
		\label{Cons_err_rate}
		\end{equation}
		for  every   $\theta\in(0,1)$;
%		\textcolor{red}{if easily possible, I did not read the proofs yet when writing this, so this is a memo, a word should be told about how the constants in big O go, when $\theta$ approaches 1?}
		\item \texttt{[Rate of optimization errors]}:   
		\begin{equation}
		{T}_{X,\epsilon}= \mathcal{O}\left(\frac{1}{\epsilon^{2}}\right).
				\label{Rate_TeX}
		\end{equation}
		Let $\gamma^\nu=K/\nu^p$ with  some constant  $K>0$ and $p\in(1/2,1)$. Then,
		\begin{equation}
		{T}_{D,\epsilon}= \mathcal{O}\left(\frac{1}{\epsilon^{2/(1-p)}}\right).
		\label{Rate_TeD}
		\end{equation}
	\end{enumerate}
\end{theorem}
\begin{proof}
See Appendix  \ref{th:rate_proof}
\end{proof}}

{%Theorem \ref{th:rate} (a) guarantees that the consensus error decays with order of $\gamma^{\ceil{\theta \nu}}$, in other words, eq. \eqref{Cons_err_rate} says that there exist $K_e>0$ and  $\bar{\nu}\in\mathbb{N}_+$ such that $e^\nu\leq K_e\gamma^{\ceil{\theta \nu}}$ holds for all $\nu\geq \bar{\nu}$, while an observation from the proof of \eqref{Cons_err_rate} is that, fixing $\bar{\nu}$ for all $\theta\in(0,1)$,  $K_e\rightarrow\infty$ as $\theta\rightarrow 1$.
%And the subsequent case (b) guarantees that the first time the merit functions $\Delta_D$ and $\Delta_X$ go below the required accuracy $\epsilon$ grows roughly as $1/\epsilon^2$ (termed as sublinear rate).  

We  remark that, while a convergence rate has been established in the literature (see, e.g., \citet{Razaviyayn2014ParallelSC}) for certain {\it centralized} algorithms applied to special classes of DL problems, Theorem   \ref{th:rate} 
represents the first    rate result  for a {\it distributed} algorithm tackling the class  of DL problems  \ref{eq:P1}.}

\section{Numerical Experiments}\label{num_results}
In this section, we   test numerically  our algorithmic framework  on several classes of problems, namely: (i) Image denoising, (ii) Biclustering, (iii) Sparse PCA, and (iv) Non-negative sparse coding. {
We recall that D$^4$L is the first provably convergent distributed algorithm for Problem \ref{eq:P1}}; comparisons are thus not simple.  To give the sense of the performance of D$^4$L, in our experiments,%  the  order to get a feel for its numerical performance and to compare it to existing methods,

{
(i) when available,  we implemented,  {\em centralized} algorithms tailored to the specific problems under consideration and used the results as {\em benchmarks};}

{
(ii) for undirected graphs, we extended the (distributed)    Prox-PDA-IP \citep{Prox_PDA_GlobalSIP} algorithm  to the simulated instances of Problem P  (generalizations of this method to directed graphs seem not possible);}

{
(iii) for both undirected and directed graphs,  we implemented a suitable version of  the Adapt-Then-Combine (ATC) Algorithm
\citep{chainais2013distributed}. Note that  ATC has no formal convergence proof, and is originally designed to handle only undirected graphs, but we managed to make a sensible extension of this method to directed graphs too, by using some of  the ideas developed in  this paper.}

 All codes are written in MATLAB 2016b, and implemented on a computer  with Intel Xeon (E5-1607 v3) quad-core 3.10GHz processor and 16.0 GB of DDR4 main memory. 

\begin{figure}[t]\vspace{-0.8cm}%\hspace{0.0cm}
	\center
	\includegraphics[scale=0.5]{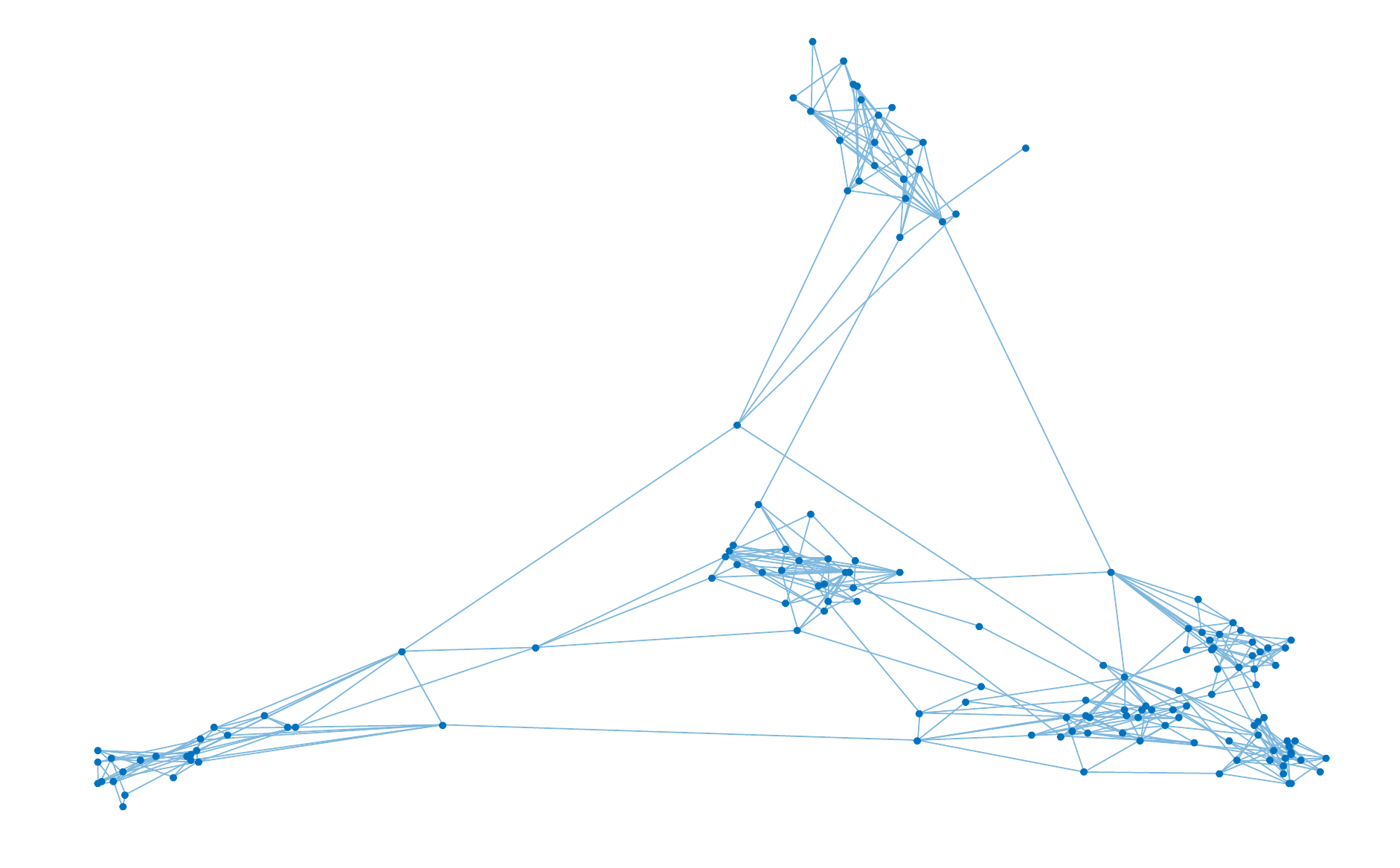}\vspace{-0.4cm}
	\caption{Topology of {a} simulated (sparse) network} \vspace{-0.4cm}
	\label{fig:network_schem}
\end{figure}

\subsection{Image Denoising }
\label{image_denoising_sec}
\noindent\textbf{Problem formulations:} We consider denoising a $512\times 512$ pixels image of a fishing boat \citep{USC_SIPI}$-$see  Fig.~\ref{fig:image_quality_comparison}(a). We simulate a cluster computer network  composed of   $150$ nodes (computers). Denoting by $\mathbf{F}_0$ and $\mathbf{F}$ the noise-free and corrupted image, respectively, the SNR (in dB) is defined as  $\text{SNR}\triangleq 20\cdot\log({||\text{vec}(\mathbf{F}_0)||_2}/{\sqrt{\text{MSE}}})$ while the Peak SNR (in dB) is defined as $\text{PSNR}\triangleq 20\cdot\log({\max_{i} (\text{vec}(\mathbf{F}_0))_i}/{\sqrt{\text{MSE}}})$, where $\text{MSE}$ is the Mean-Squared-Error between $\mathbf{F}_0$ and $\mathbf{F}$. 
The fishing boat image is corrupted by additive white Gaussian noise, so that $\text{SNR}=15$ dB and    $\text{PSNR}=20.34$ dB.%\footnote{\textcolor{red}{Given a noise free image (vector) $\mathbf{F}_0$ and its approximation $\mathbf{F}$, $\text{SNR}\triangleq 20\cdot\log(\frac{||\mathbf{F}_0||_2}{\sqrt{MSE}})~db$, and $PSNR\triangleq 20\cdot\log(\frac{\max_{i} (\mathbf{F}_0)_i}{\sqrt{MSE}})~db$, where $MSE$ is the Mean-Squared-Error between $\mathbf{F}_0$ and $\mathbf{F}$.}} 

To perform the denoising task,  we consider the elastic net sparse DL formulation \eqref{eq:sparse_DL}. 
We extract 255,150 square sliding $s\times s$ pixel patches ($s=8$) and aggregate the vectorized extracted patches in a single data matrix $\mathbf{S}$ of size $64\times 255,150$. The size of the dictionary is $s^2\times s^2=64\times 64$; the data matrix is equally distributed across the $150$ nodes, resulting   in   sparse representation matrices $\mathbf{X}_i$  of size  $64\times 1701$ ($K= 64$ and $n_i = 1701$). The  total number of optimization variables is then  $16,333,696$. {The free parameters     $\lambda$, $\mu$ and $\alpha$ in \eqref{eq:sparse_DL} are set to $\lambda=1/s$, $\mu=\lambda$ and $\alpha=1$, respectively.}

 \noindent\textbf{Algorithms and tuning:} We tested:  i) two instances of the D$^4$L Algorithm, corresponding to   two alternative choices of the surrogate functions; {ii)  the Prox-PDA-IP algorithm  \citep{Prox_PDA_GlobalSIP}, adapted  to problem \eqref{eq:sparse_DL} (only on undirected networks); iii) the ATC algorithm \citep{chainais2013distributed};} and iv) the \emph{centralized} K-SVD algorithm  \citep{Elad_KSVDdenoising_TIP_2010} (\texttt{KSVD-Box v13} package), used as a benchmark. More specifically, the two instances of the D$^4$L Algorithm are:\vspace{-0.1cm}
\begin{itemize}
\item  \texttt{Plain D$^4$L}: $\tilde h_i$ is chosen as in \eqref{eq:htilde1} (the original function) and $\tilde{f}_i$ as  in \eqref{eq:ftilde2};
\item \texttt{Linearized D$^4$L}: $\tilde h_i$ is  given by \eqref{eq:htilde2} (first-order approximation) and $\tilde{f}_i$ is given by \eqref{eq:ftilde2}.\vspace{-0.1cm}
\end{itemize} The rest of the parameters in both instances of D$^4$L is set as:   $\gamma^\nu=\gamma^{\nu-1}(1-\epsilon \gamma^{\nu-1})$, with  $\gamma^0=0.5$ and $\epsilon=10^{-2}$;  $\tau^\nu_{D,i}=10$; and $\tau^\nu_{X,i}=\max(L_{\nabla{X_i}}(\mathbf{U}_{(i)}^\nu),1)$ [cf.~\eqref{tau_x_second_choice}].

{ {Our adaptation} of the Prox-PDA-IP algorithm       to   Problem \eqref{eq:sparse_DL} is summarized in Algorithm \ref{Prox-PDA-IP}. The difference with the original version in \citep{Prox_PDA_GlobalSIP} are: i)  the elastic net penalty is used in the objective function for the $\mathbf{X}_{i}$'s variables, instead of the $\ell_1$-norm and $\ell_2$-norm ball constraints; and ii)  the variables $\mathbf{D}_{(i)}$'s are constrained in $\mathcal{D}\triangleq\{\mathbf{D}\,:\,||\mathbf{D}\mathbf{e}_k||_2\leq\alpha, \,k=1,2,\ldots,K\}$  rather than using the $\ell_2$-norm regularization in the objective function.}  
{The other symbols used in  Algorithm \ref{Prox-PDA-IP} are: i)  the  incidence matrix of $\mathcal{G}$, denoted by  $\mathbf{M}=(M_{ei})_{e,i}\in\mathbb{R}^{E\times I}$, with $E\triangleq |\mathcal{E}|$; ii) the matrices $\boldsymbol{\Omega}_e^\nu\in\mathbb{R}^{M\times K}$,  $e=1,...,E$, which are the $\nu$-th iterate  of the dual matrix variables   $\boldsymbol{\Omega}_e\in\mathbb{R}^{M\times K}$, as   introduced in the original Prox-PDA-IP; and iii)   $\{\beta^\nu\}_{\nu\in\mathbb{N}_+}$ is the  increasing penalty parameter, set to  $\beta^\nu=0.002\nu$}.
{
\begin{algorithm}[t] {
		~
		\\
		{$\textbf{Initialization}: 
			\mathbf{D}_{(i)}^0\in\mathcal{D},~\mathbf{X}_i^0\in\mathcal{X}_i$,
			$\boldsymbol{\Omega}^0=\mathbf{0}$;
			\\
			\texttt{S1.} If $(\mathbf{D}_{(i)}^\nu,\mathbf{X}_i^\nu)_i$ satisfies stopping criterion: \texttt{STOP};\smallskip\vspace{0.1cm}
			\\
			\texttt{S2.} Each agent $i$ computes $\theta_i^\nu=||\mathbf{D}_{(i)}^\nu\mathbf{X}_i^\nu-\mathbf{S}_i||^2_F$ and: 
			\vspace{-.1cm}
			\begin{enumerate}[label=(\alph*)]
				\item $\displaystyle\mathbf{X}_i^{\nu+1} = \underset{\mathbf{X}_i\in\mathbb{R}^{K\times n_i}}{\text{argmin}} f_i(\mathbf{D}_{(i)}^\nu,\mathbf{X}_i)+g_i(\mathbf{X}_i)+\frac{\beta^{\nu+1}\theta_i^\nu}{2}||\mathbf{X}_i-\mathbf{X}_i^\nu||_F^2+\frac{\beta^{\nu+1}}{2}||\mathbf{D}_{(i)}^\nu(\mathbf{X}_i-\mathbf{X}_i^\nu)||_F^2$;
				\item  $\displaystyle{\mathbf{D}}_{(i)}^{\nu+1} = \underset{\mathbf{D}_{(i)}\in\mathcal{D}}{\text{argmin}}  ~~f_i(\mathbf{D}_{(i)},\mathbf{X}_i^{\nu+1})+\textstyle\sum_{e=1}^EM_{ei}\left\langle\boldsymbol{\Omega}_e^\nu,\mathbf{D}_{(i)}\right\rangle$ 
				\vspace{-0.3cm}
				\item[] \hspace{4cm}$+\beta^{\nu+1} \left(d_i||\mathbf{D}_{(i)}||_F^2- \big\langle\mathbf{D}_{(i)},(d_i-1)\mathbf{D}_{(i)}^\nu+\textstyle\sum_{j\in\mathcal{N}_i}\mathbf{D}_{(j)}^\nu\big\rangle\right)$;
				\item  $\boldsymbol{\Omega}^{\nu+1}_e =\boldsymbol{\Omega}^{\nu}_e+\beta^{\nu+1}\sum_{i=1}^IM_{ei}\mathbf{D}_{(i)}^{\nu+1},\quad\forall e=(i,j)\in\mathcal{E}$;
			\end{enumerate}
			\texttt{S3.} Set $\nu+1\to\nu$, and go to \texttt{S1}.}
		\\
		~
		\caption{: Prox-PDA-IP algorithm \citep{Prox_PDA_GlobalSIP} \label{Prox-PDA-IP}}
	}
\end{algorithm}}

All the algorithms are initialized to the same value:   $\mathbf{D}_{(i)}^0$'s coincide  with randomly  (uniformly) chosen   columns of $\mathbf{S}_{(i)}$'s whereas all $\mathbf{X}_i^0$'s are set to  zero.    

 While the subproblems solved at each iteration $\nu$ in Linearized D$^4$L admit a closed-form$-$see \eqref{closed_form_X_linearization} and \eqref{closed_form_dictionary_lin}$-$in both Plain D$^4$L and ATC, the update of the dictionary   has   the closed form  expression \eqref{closed_form_dictionary_lin}, but the update of   the private variables calls for the solution of a LASSO problem  (cf. Sec.~\ref{subsec:discussion}). %;  the update of the dictionary in the two algorithms  has instead a closed form expression, given by \eqref{closed_form_dictionary_lin}.  
 For both Plain D$^4$L and ATC, the LASSO subproblems at iteration $\nu$ are solved using the (sub)gradient algorithm, with the following tuning. A diminishing step-size is used, set to  $\gamma^r=\gamma^{r-1}(1-\epsilon \gamma^{r-1})$, where $\gamma^0=0.9$,  $\epsilon=10^{-3}$, and $r$ denotes the inner iteration index. A warm start is used for the subgradient algorithm: the initial points  are set   to  $\mathbf{X}^\nu_i$, where $\nu$ is the iteration index of the outer loop.  We terminate the subgradient algorithm in the inner loop when $J^{r}_{i}\leq 10^{-6}$, with %$J^{\nu,r}_{i}\triangleq \norm{\mathbf{X}^{\nu,r}_i-\frac{\tau_0}{\mu+\tau_0}\mathcal{T}_{\frac{\lambda}{\tau_0}} \left(\mathbf{X}^{\nu,r}_i-\frac{1}{\tau_0}\left(\nabla_{X_i} f_i(\mathbf{U}_{(i)}^\nu,\mathbf{X}^{\nu,r}_i)+\tau_{X,i}^\nu\left(\mathbf{X}^{\nu,r}_i-\mathbf{X}^{\nu}_i\right) \right)\right)}_{\infty,\infty}$ and $\tau_0=1$}  
 $$J^{r}_{i}\triangleq \norm{\mathbf{X}^{\nu,r}_i-\frac{s}{1+s}\,\mathcal{T}_{\frac{1}{s}} \left(\mathbf{X}^{\nu,r}_i-\left(\nabla_{X_i} f_i(\mathbf{U}_{(i)}^\nu,\mathbf{X}^{\nu,r}_i)+\tau_{X,i}^\nu\left(\mathbf{X}^{\nu,r}_i-\mathbf{X}^{\nu}_i\right) \right)\right)}_{\infty,\infty},$$
 %\textcolor{magenta}{\it [Let us use the previous style of measure of stationarity, so it would be more consistent (from aspect of used parameter $\tau_0$) with \eqref{SSVD_optimality_inner_loop} and \eqref{SSVD_optimality_inner_loop2}. Because we can not use parameter $s$ in the  merit functions of next simulations.]} 
 where $\mathbf{X}^{\nu,r}_i$ denotes the value of $\mathbf{X}_i$ at the $r$-th inner iteration and  outer iteration $\nu$; and  $\mathcal T_\theta(x)\triangleq \max(|x|-\theta,0)\cdot \text{sign}(x)$  is  the soft-thresholding operator, applied to the matrix argument componentwise. In all our simulations, we observed that  the  above accuracy was reached within  30 (inner) iterations of the subgradient algorithm. 
 
 {In  the Prox-PDA-IP scheme, Step \texttt{S2} (cf. Algorithm \ref{Prox-PDA-IP}) calls for the solution of two subproblems, including a LASSO problem. As for Plain D$^4$L and ATC,  we used the (projected) (sub)-gradient algorithm (with the same diminishing step-size rule) to solve the subproblems;   we terminated  the inner loop  when the length between two consecutive iterates of the (projected) (sub)-gradient algorithm goes for the first time  below $10^{-6}$.}

We simulated both undirected and directed static graphs. {In the former case, there is no need of the $\phi$-variables and, in the second equation of \eqref{D_weighted_avg} [and (\ref{Theta_update})], the terms   $(\phi_j^\nu \, a^\nu_{ij})/\phi_i^{\nu +1}$ reduce to $a_{ij}$. The weights   $a_{ij}$ are chosen according to the  {Metropolis-Hasting} rule  \citep{Xiao}; the resulting    matrix  $\mathbf{A}^\nu=[a_{ij}]_{ij}$     is  thus  time-invariant and  doubly stochastic.  }When the graph is directed, we use the update of the $\phi_i^\nu$'s as in \eqref{D_weighted_avg},  with the weights  $a_{ij}^\nu$   chosen according to the  {push-sum protocol} \citep{Kempe2003} (cf. Sec.~\ref{subsec_weights}).

\noindent \textbf{Convergence speed and quality of the reconstruction:}
  In the first set  of simulations, we considered an undirected graph composed of 150 nodes, clustered in 6 groups of 25 (see Fig. \ref{fig:network_schem}). Starting from this topology, we kept adding random edges till a connected  graph was obtained. Specifically,  
 an arc is added  between two nodes in  the same cluster (resp.  different clusters) with probability $p_1=0.2$ (resp. $p_2=2\times 10^{-3}$).\\  
 \indent In Fig.~\ref{fig:obj_consErr_stationarity}
we plot the objective function value [subplot on the left], the consensus disagreement $e^\nu$ as in (\ref{eq:consensus_error}) [subplot in the center], and the distance from stationarity $\Delta^\nu$ as in (\ref{D_avg_def})  [subplot on the right] versus the {\em number of message exchanges}, achieved by Plain D$^4$L, Linearized D$^4$L, Prox-PDA-IP, and ATC. {Note that  the number of messages exchanged in the  ATC algorithm  at iteration $\nu$ coincides with   $\nu$ whereas  for  Prox-PDA-IP and the 
D$^4$L schemes is $2\nu$ (recall that  the latter schemes employ two steps of communications per iteration). }
  The figures clearly show that  both versions of D$^4$L {are much faster than Prox-PDA-IP  and ATC  (or, equivalently, they require fewer information exchanges). Moreover,  ATC does not seem  to  reach a consensus on the local copies of the dictionary, while Prox-PDA-IP and D$^4$L schemes reach an agreement quite soon.}
%The faster behavior of  the proposed schemes seems mainly due to the gradient tracking mechanism.
  In Fig.~\ref{fig:image_quality_comparison}, we plot the   reconstructed images along with their PSNR and MSE, obtained by the algorithms, when terminated after 1000 message exchanges. The figures  clearly show superior performance of D$^4$L {over its competitors}. Also,  the values of PSNR and MSE achieved by  D$^4$L are comparable with those obtained by (centralized) K-SVD  (\texttt{KSVD-Box v13} package).\\
\begin{figure}[t]
\center\vspace{-0.7cm}
\includegraphics[scale=0.35]{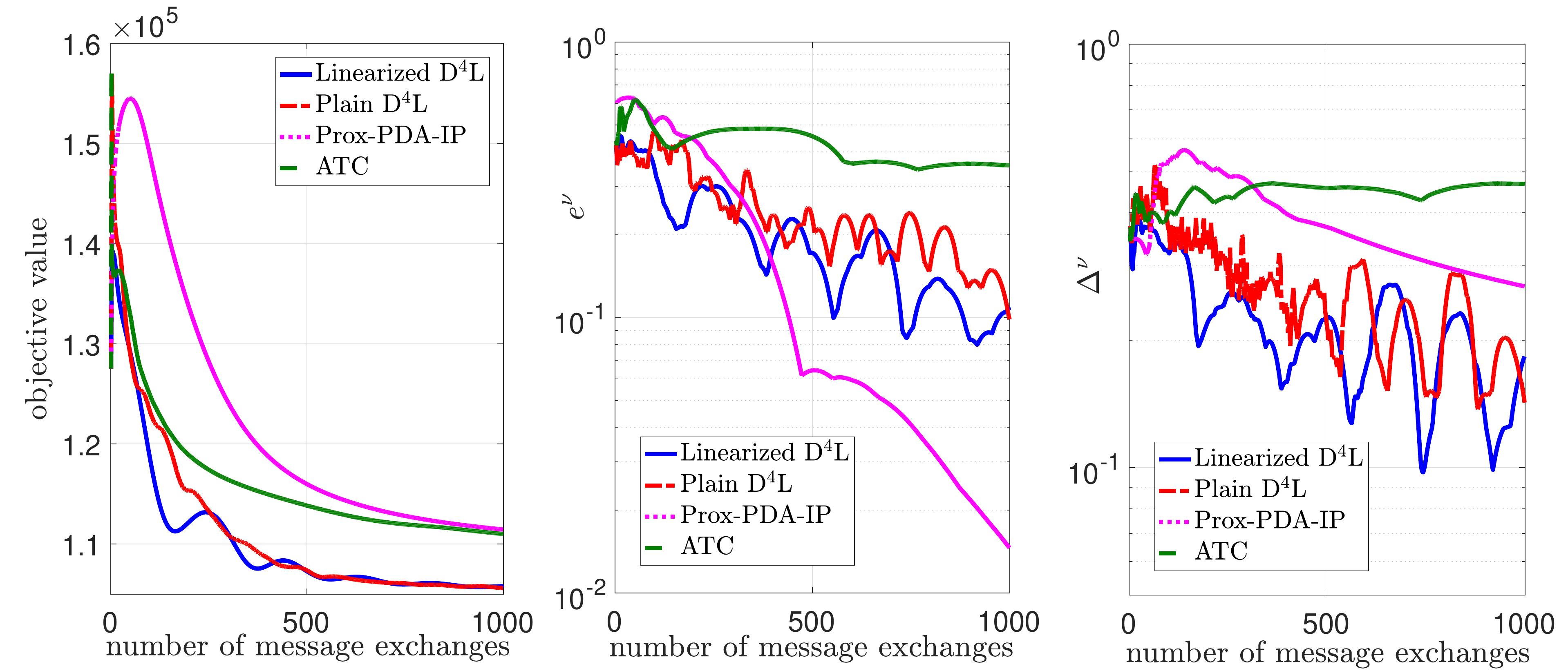}\vspace{-0.3cm}
\caption{{Denoising problem -- D$^4$L, Prox-PDA-IP and ATC algorithms:}   objective value [subplot on the left],  consensus disagreement [subplot in the center], and distance from stationarity $\Delta^\nu$ [cf.~\eqref{stat_merit_subproblems}] [subplot on the right] vs. number of message exchanges.}
\label{fig:obj_consErr_stationarity}\vspace{-0.5cm}
\end{figure}
\begin{figure}[t]  
 \centering  \includegraphics[scale=0.55]{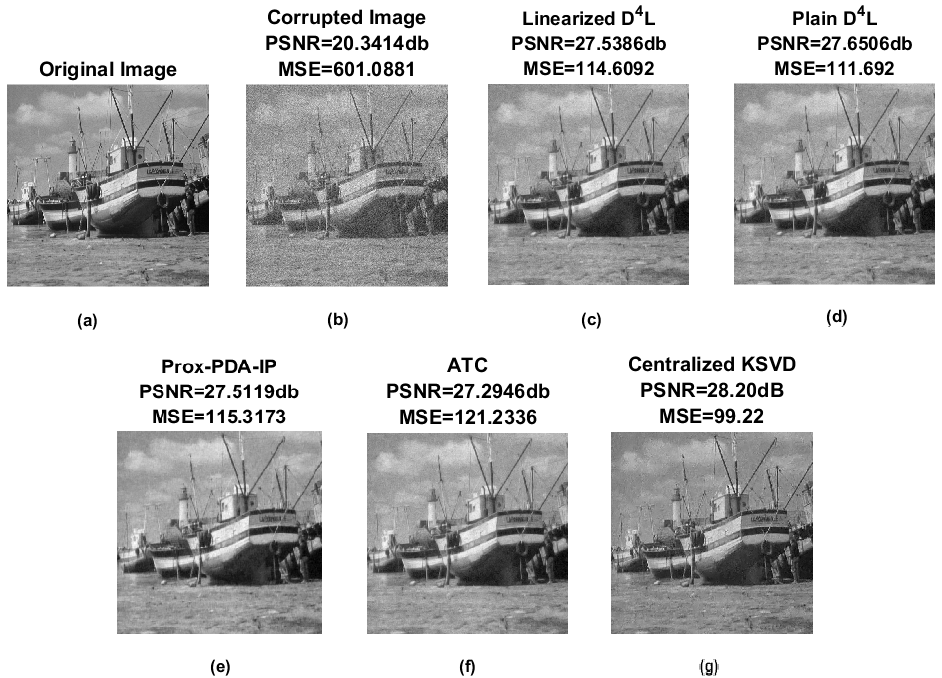}\vspace{-0.3cm}
\caption{Denoising outcome. (a): original image; (b):  corrupted image;   (c)-(f): denoising  achieved by D$^4$L, Prox-PDA-IP and ATC  terminated after 1000 message exchanges; and (g):  denoising achieved by centralized K-SVD  (\texttt{KSVD-Box v13}).}
\label{fig:image_quality_comparison}\vspace{-0.5cm}
\end{figure}
\begin{figure}[t]\vspace{-0.6cm}\centering
\includegraphics[scale=0.55]{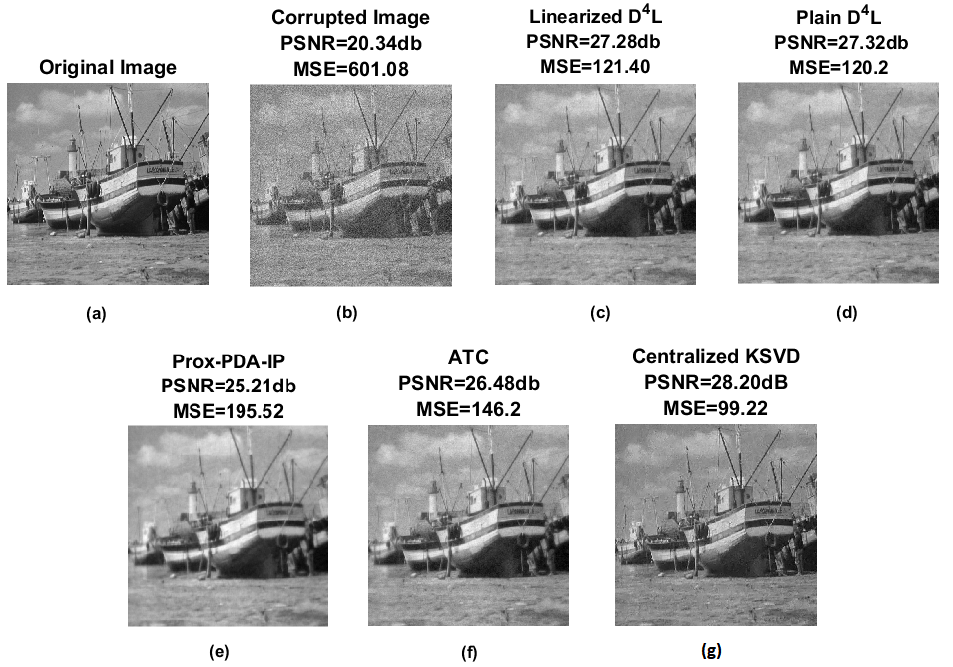}\vspace{-0.3cm}
\caption{Denoising outcome. (a): original image; (b):  corrupted image;   (c)-(f): denoising  achieved by D$^4$L, Prox-PDA-IP and ATC  terminated after 200 message exchanges; and (g):  denoising achieved by centralized K-SVD  (\texttt{KSVD-Box v13}).}\vspace{-0.5cm}
\label{fig:image_quality_comparison2}
\end{figure}
A closer look at Fig.~\ref{fig:obj_consErr_stationarity} shows that a significant decay on the objective function occurs in the first 200 message exchanges. It is then interesting to check the quality of the reconstructed images, achieved by the algorithms if  terminated  then.  
In %Table \ref{image_quality_table} and 
Fig.~\ref{fig:image_quality_comparison2}, we report the images and values of PSNR and MSE obtained by terminating {the schemes after 200 message exchanges (we also plot the  benchmark obtained by  K-SVD,  run till optimality). }
%\begin{table}[H]
%\center
% \begin{tabular}{c c c}
%\hline
%\textbf{Algorithm} & \multicolumn{2}{c}{\textbf{terminated after}}
%\\
%&200 message exchanges & 1000 message exchanges\\
%\hline
%Linearized D$^4$L &27.28db/121.4  & 27.53db/114.6
% \\
%Plain D$^4$L & 27.32db/120.2  &  27.65db/111.69
% \\
%ATC & 26.48db/146.2 & 27.29db/121.23
%\\
%\hline
%\end{tabular}
%\caption{Quality of reconstruction (PSNR/MSE)}
%\label{image_quality_table}
%\end{table}
 The figure shows that both versions of D$^4$L attain high quality solutions even if terminated after few message exchanges {while ATC and Prox-PDA-IP}  lag behind. This means that, in practice,  there is no need to run D$^4$L till very low values of  $e^\nu$ and  $\Delta^\nu$ are achieved.  

Since the algorithms do not have the same {cost-per-iteration,}   to get further insights into the performance of these schemes, we also compare them in terms of  running time. In Table \ref{avg_time_comparison_table}, we  report the \emph{averaged elapsed time} to execute one iteration of all algorithms.  We considered the same setting as in the previous figures, but  we terminated all algorithms after  273 seconds, which corresponds to the time  for the fastest algorithm (i.e. Linearized D$^4$L)  to perform 200 message exchanges [cf. Fig.~\ref{fig:image_quality_comparison2}]. The associated reconstructed images are shown in Fig.~\ref{fig:comparison_overall_shortTimeWait}. Once again, {these results clearly show  that  the linearized D$^4$L scheme significantly  outperforms Prox-PDA-IP and ATC}. Also,  Linearized D$^4$L performs considerably better than Plain D$^4$L, when terminated early; the explanation is  in Table \ref{avg_time_comparison_table} which shows that the time of one iteration of the former algorithm is much shorter than that of Plain D$^4$L.
\begin{table}[H]
\center
\begin{tabular}{c c}
\hline
\textbf{Algorithm}  & \textbf{Average Time per Iteration (sec)}
 \\
 \hline
Linearized D$^4$L  & 2.862
 \\
 Plain D$^4$L &  11.328
 \\
 Prox-PDA-IP & 30.98
 \\
ATC &9.838
\\
\hline
\end{tabular}
\caption{D$^4$L vs. Prox-PDA-IP and ATC: Average computation time per iteration}
\label{avg_time_comparison_table}\vspace{-0.2cm}
\end{table}

\begin{figure}[h]\vspace{-0.3cm}
\centering 
\includegraphics[scale=0.35]{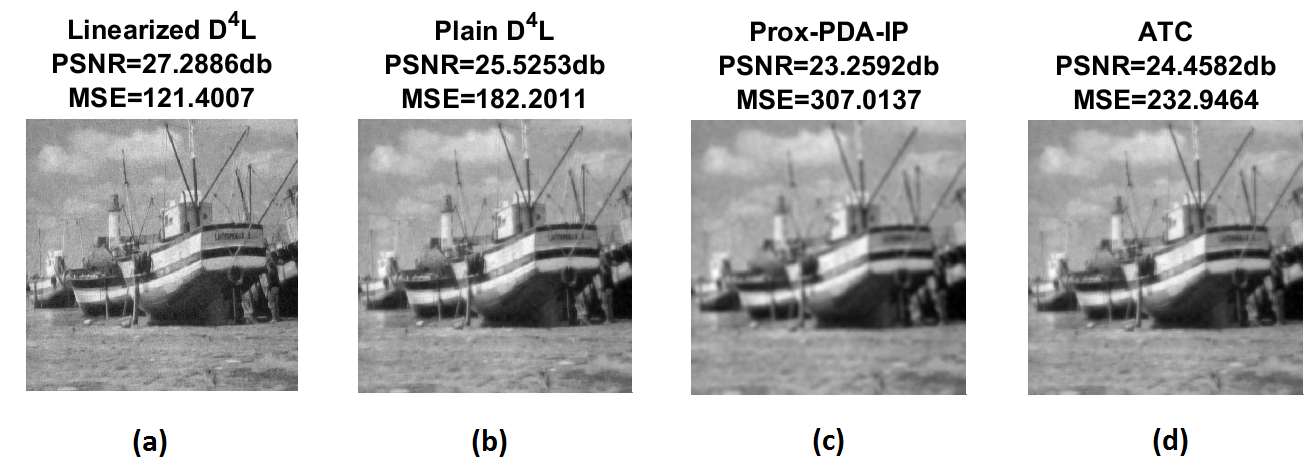}\vspace{-0.3cm}
\caption{Denoising outcome.    (a): Linearized D$^4$L; (b): Plain D$^4$L; (c): Prox-PDA-IP; (d) ATC; all terminated after  after 273 seconds run-time (corresponding to 200 message exchanges of the Linearized D$^4$L).}
\label{fig:comparison_overall_shortTimeWait}\vspace{-0.5cm}
\end{figure}
%\begin{figure}[t]
%\center
%\includegraphics[scale=0.32]{figs/denoising_figs_directedNet/plots_overall-eps-converted-to.pdf}
%\caption{Comparison in terms of (left) objective function (center) consensus error (right) proximity to stationarity}
%\label{fig:obj_consErr_stationarity_directed}
%\end{figure}
%
%\begin{table}[H]
%\center
% \renewcommand{\arraystretch}{2}% Wider
% \begin{tabular}{c c c c}
%\rowcolor{gray!15}
%\cellcolor{gray!0}  & \parbox{3cm}{\center Linearized D$^4$L\\~} & \parbox{3cm}{\center Plain D$^4$L\\~} & \parbox{3cm}{\center ATC\\~} \\ [0.5ex]
%\cellcolor{gray!15}  \parbox{4cm}{\center 200 message exchanges\\~} & \parbox{1.8cm}{PSNR=27.28db\\MSE=121.4} & \parbox{1.8cm}{PSNR=27.32db\\MSE=120.2} & \parbox{1.8cm}{PSNR=26.48db\\MSE=146.2}
%\\
%\cellcolor{gray!15}  \parbox{4.2cm}{\center 1000 message exchanges\\~} & \parbox{1.8cm}{PSNR=27.53db\\MSE=114.6} & \parbox{1.8cm}{PSNR=27.65db\\MSE=111.69} & \parbox{1.8cm}{PSNR=27.29db\\MSE=121.23} \\  [1ex]
%\end{tabular}
%\caption{Comparison of quality of reconstruction}
%\label{image_quality_table_directed}
%\end{table}

\noindent \textbf{Impact of the graph topology and connectivity:} We study now  the influence of the topology and graph connectivity  on the performance of the algorithms.  We consider  \emph{directed, static} graphs. 
We generated   5 instances of  digraphs, with different connectivity, according to the following procedure. There are    $500$ nodes ($I=500$), which are clustered in $n_c=50$ clusters,  each of them containing $10 = I/n_c$ nodes.  Each node has an outgoing arc to another node in the same cluster with probability $p_1$  while $p_2$ is the probability of an outgoing arc to a node in a different cluster.   We chose the values of  $p_1$ and $p_2$, as   in Table 
\ref{network_setting_casewise_0}; we simulated three     scenarios, namely:
N1 corresponds to a ``highly''  connected network, N3 describes a ``poorly'' connected scenario, and  N2 is  an  intermediate case.   
For each scenario, we generated 5 random instances (if a generated graph was not strongly connected we discarded it and generated a new one) and then 
ran Plain and Linearized D$^4$L and ATC  on the resulting 15 graphs. 
Recall that   ATC was not designed to work on directed networks. We thus  modified it  by using our new consensus protocol (but not the gradient 
tracking mechanism); we term it  {\emph{Modified} ATC}. %The connectivity weights $\{a_{ij}^\nu\}_{i,j=1}^I$ are generated according to push-sum protocol \citep{Kempe2003} as we 
%discussed in Sec. \ref{subsec:discussion}.

%We consider 
%Problems on digraphs are in general more difficult than those on graphs and many methods designed for undirected graphs do not work on directed ones.

\begin{figure}\vspace{-0.5cm}
\center
\includegraphics[scale=0.35]{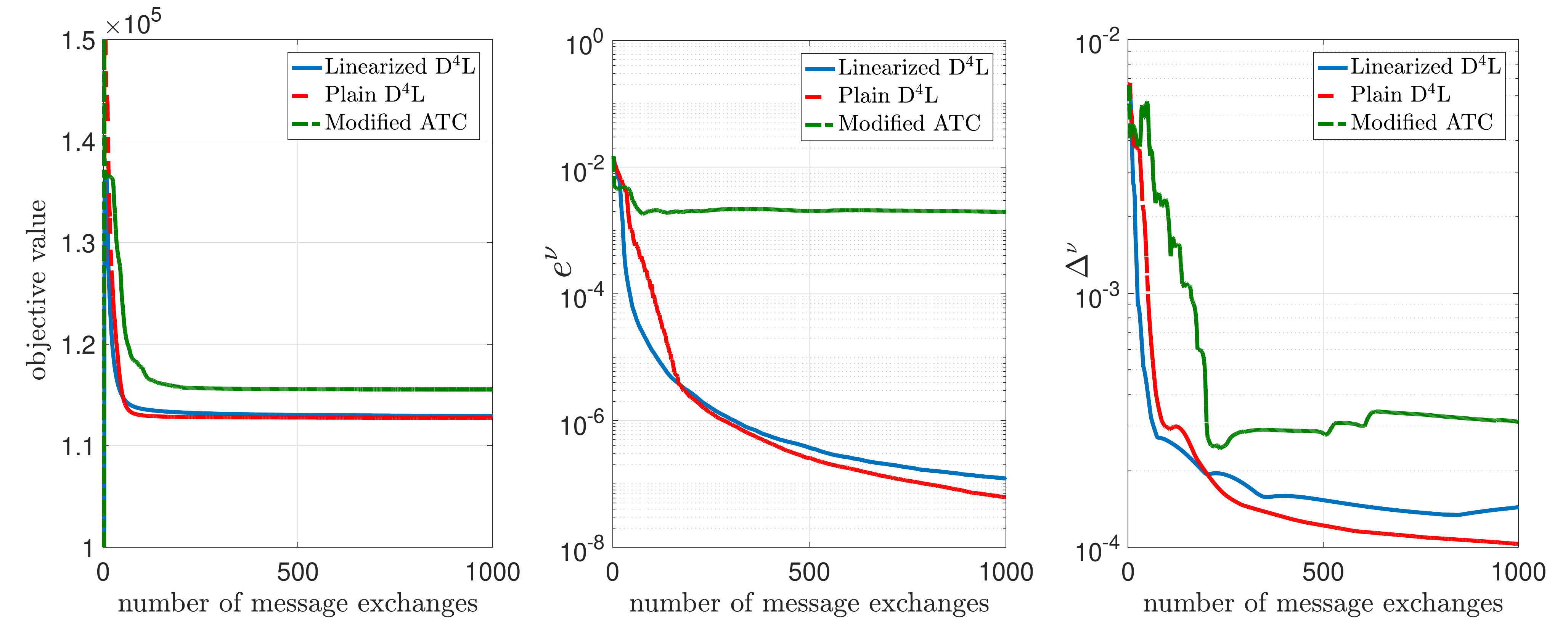}
\\
\centering (a)
\\
\center
\includegraphics[scale=0.35]{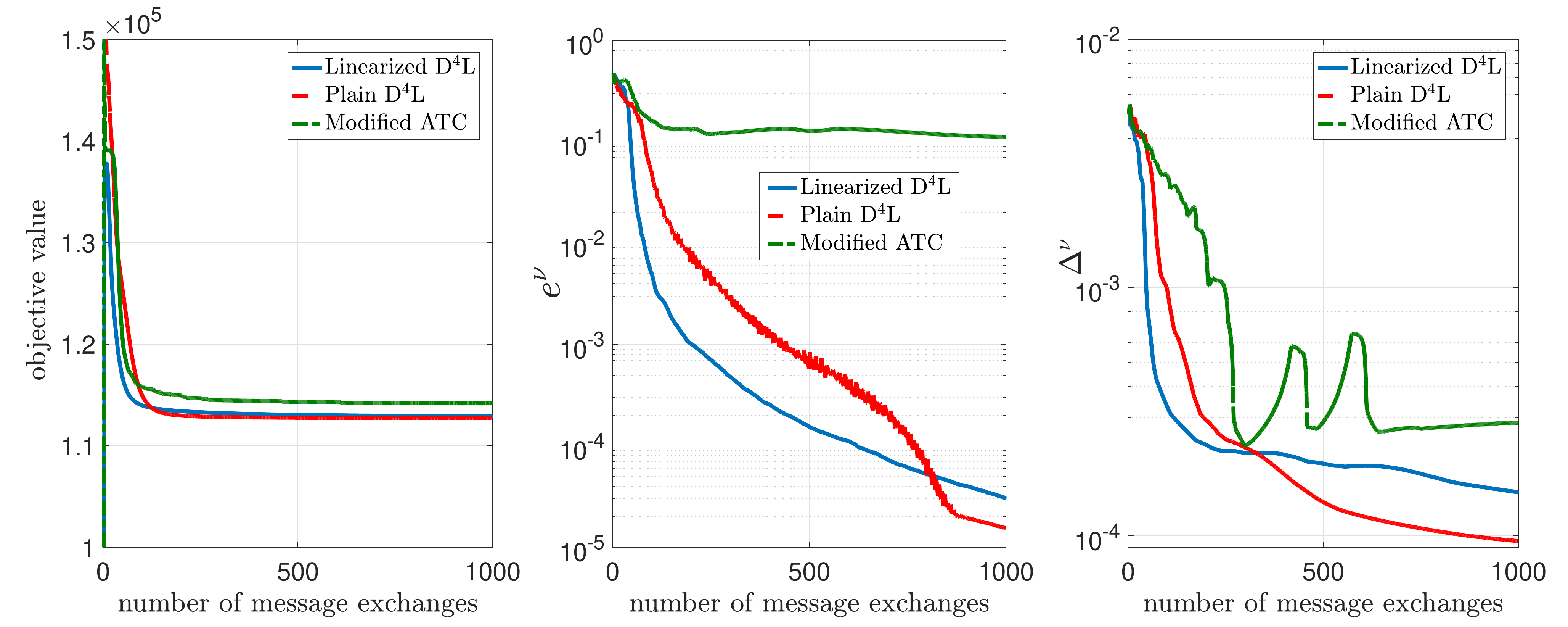}
\\
\centering (b)
\\\center
\includegraphics[scale=0.35]{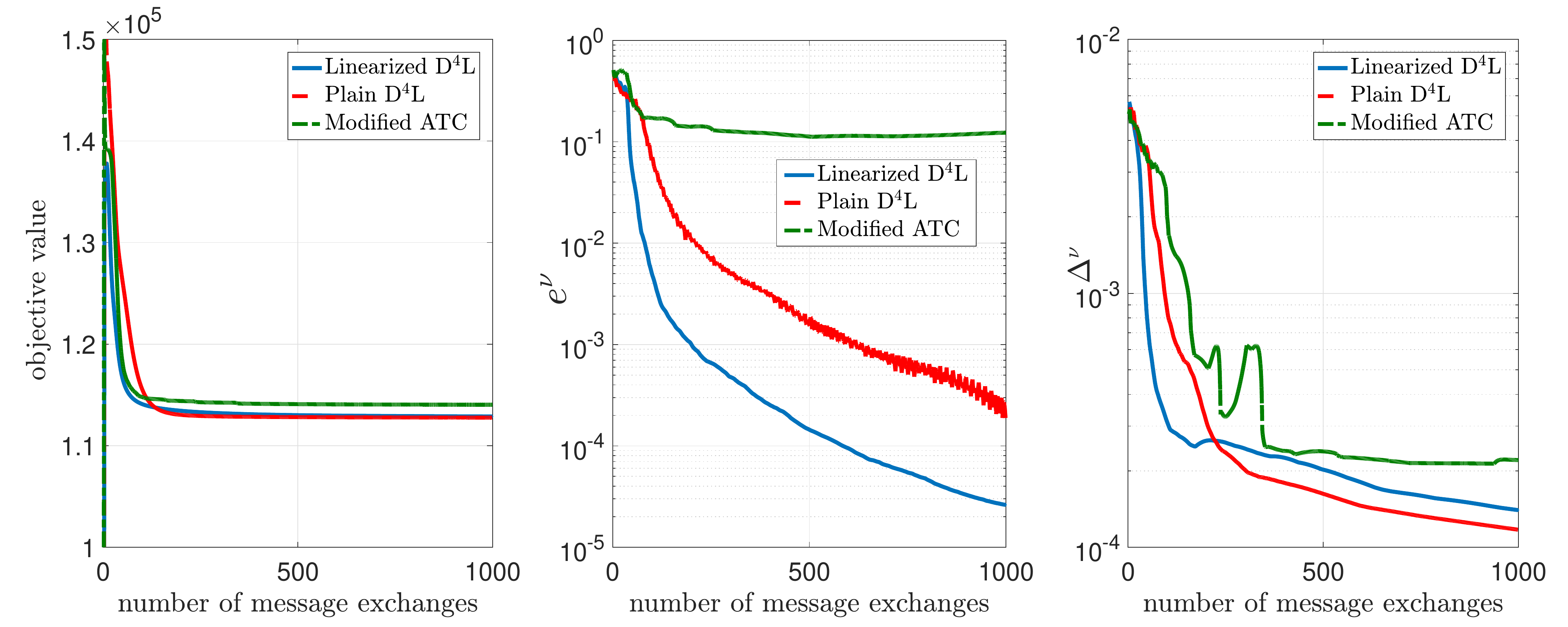}
\\
\centering (c)
\vspace{-0.2cm}\\
\caption{{Denoising problem--D$^4$L and  {Modified ATC} algorithms:} objective value [subplots on the left],  consensus disagreement [subplots in the center], and distance from stationarity $\Delta^\nu$ [cf. \eqref{stat_merit_subproblems}] [subplots on the right] vs. number of message exchanges. Comparison  over three network settings [cf.~Table~\ref{network_setting_casewise_0}]:  N1 [subplots (a)],  N2 [subplots (b)], and  N3 [subplots (c)].}
\label{fig:denoising_consensus_VS_inexactness}
\end{figure}

\begin{table}[h]
\center
\begin{tabular}{c c c c c}
\hline
\textbf{Network \#} &  $I$ &$n_c$& $p_1$& $p_2$ \\
\hline
~~~ N1 ~~~~~~& 500 &50  &0.9  &0.9  \\
~~~ N2 ~~~~~~& 500 &50  &0.1  &0.01  \\
~~~ N3 ~~~~~~& 500 &50  &0.05  &0.01  \\
  \hline
\end{tabular}
\caption{Network setting}
\label{network_setting_casewise_0}
\end{table}

 In Fig.~\ref{fig:denoising_consensus_VS_inexactness} we   plot the \emph{average} value of the objective function   [subplot on the left], the consensus disagreement $e^\nu$ [subplot in the center], and the distance from stationarity $\Delta^\nu$   [subplot on the right],  achieved by Plain D$^4$L, Linearized D$^4$L, and Modified ATC, versus the   number of message exchanges, for   the three scenarios N1 [subplot (a)], N2 [subplot (b)] and N3 [subplot (c)]. The average is taken over the  aforementioned  5 digraph realizations.  While also this batch of tests confirms the better behavior of D$^4$L schemes over ATC, it is interesting to observe that there seems to be little influence of the degree of connectivity on the behavior of Linearized and Plain D$^4$L.  The only aspect for which a reasonable influence can be seen is on consensus.  In fact, with respect to consensus, Linearized D$^4$L seems to improve over Plain D$^4$L,  when connectivity decreases. This has a natural  interpretation. 
   Plain D$^4$L solves much more accurate subproblems at each iteration and this is, in some sense useless, especially in early iterations, when information has not spread across the network. It seems clear that the less connected the graph,  the more time information  needs to spread. Therefore, in scenario N1, the two methods are almost equivalent and, looking at consensus error, we see that initially Linearized D$^4$L is better than Plain D$^4$L, but soon, as information spreads, Plain D$^4$L becomes, even if slightly, better than  Linearized D$^4$L. The same behavior can be observed for scenario N2, but this time the initial advantage of Linearized D$^4$L is larger and the switching point is reached much later. This is consistent with the fact that information needs more time 
 to spread and therefore solving the accurate subproblem is not advantageous. If one passes to N3, where connectivity is very loose, there is no switching point within the first 1000 message exchanges. 

\subsection{Biclustering}
\label{Biclustering_numerical_results}  Biclustering  has been shown to be useful in several applications, including biology,  information retrieval, and data mining; see, e.g., \citep{biclustering_survey}.
%Many clustering approaches have been proposed for the analysis of gene expression data obtained from microarray experiments. However, the application of standard clustering methods to this problem generally leads to unsatisfactory results because of the existence of experimental conditions where the activity of genes is uncorrelated. Similar results also hold when clustering of conditions is attempted. To deal with this issue the  simultaneous clustering on the rows and columns of the data matrix has been put forward; this approach has been termed biclustering. Biclustering has shown useful in biology, but, more generally in  information retrieval and data mining.

\noindent \textbf{Problem Formulation}: We consider a Biclustering problem %which has been shown to be useful in several applications, including biology,  information retrieval, and data mining. For more details see \citep{biclustering_survey}. 
%To perform this task, we consider sparse Singular Value Decomposition (SSVD) formulation  
in the form \eqref{eq:SSVD}, applied to genetic information. We solved the problem simulating  a  networked computer cluster  composed of 500 nodes (see Table \ref{network_setting_casewise_0}). 
The genetic data is borrowed from \citep{Lee_Shen_Huang_Marron10} (centered and normalized): 
the  data matrix $\mathbf{S}$ of size $56\times 12,625$ ($M=56$ and $N=12,625$) contains microarray gene expressions of 56 patients (rows);  each patient is either identified to be normal (Normal) or belonging to one of the following three types of lung cancer: pulmonary carcinoid tumors (Carcinoid), colon metastases (Colon), and small cell carcinoma (SmallCell). We  considered the \emph{unsupervised} instance of the problem, meaning that none of  the a-priori information about the type of  patients' cancer is used to perform biclustering. Following the numerical experiments of \citep{Lee_Shen_Huang_Marron10}, we seek rank-$3$ sparse matrices $\mathbf{X}_i$, and the data matrix $\mathbf{S}$ is equally distributed across the 500 nodes, %\footnote{The data matrix is padded with zero columns to make it equally-distributable across the nodes}, 
resulting thus in $K=3$ and $n_i=26$. The total number of variables is then $39,168$. The other parameters are set as follows:  $\alpha=1$, $\lambda_X=\mu_X=0.1$, and $\lambda_D=\mu_D=0.1$.

\noindent\textbf{Algorithms and tuning:}
%\textcolor{green}{Maybe we should say why we do not test the linearized algorithm D4L (I guess becuase the linearized version has subproblems without closed-form solution?) Also, there is certain amount of repetition in the description and naming of the other algorithms (here and in other following parts), however, repetita iuvant, and it might be not too bad to keep the repetition, I leave this to your judgement. One more thing, we should explain somewhere here that since we only consider directed graphs, we do not test the other algorithm. Many of these comments apply also to the following problem settings}
{We tested the instance of  D$^4$L where $\tilde{f}_i$ and $\tilde{h}_i$ are chosen according to  \eqref{eq:ftilde1} and \eqref{eq:htilde1}, respectively. The rationale behind this choice is to exploit the extra structure of the original function $f_i$, plus, in case of $\tilde{f}_i$, there is no certain benefit in using the linear approximation  \eqref{eq:ftilde2} as it does \emph{not} lead to any closed form solution of the subproblem \eqref{D_tilde_subproblem}.} {Note that the Prox-PDA-IP scheme is not applicable here since the network is directed.} 

The other parameters of the algorithm are set to: $\gamma^\nu=\gamma^{\nu-1}(1-\epsilon \gamma^{\nu-1})$, with  $\gamma^0=0.2$ and $\epsilon=10^{-2}$; and  $\tau^\nu_{D,i}=100$ and $\tau^\nu_{X,i}=\max(L_{\nabla{X_i}}(\mathbf{U}_{(i)}^\nu),100)$. We term such an instance of  D$^4$L   \emph{Plain D$^4$L}. We compared Plain D$^4$L with the following algorithms:  {i) (a modified version of) the distributed  ATC algorithm  \citep{chainais2013distributed}, where the optimization of $\mathbf{D}$ is adjusted to solve \eqref{eq:SSVD} (the elastic-net penalty is added), and the consensus mechanism is modified with our new consensus protocol  to handle directed network topologies;   we termed this instance  \emph{Modified ATC};} and ii) the centralized SSVD algorithm proposed in \cite{Lee_Shen_Huang_Marron10} (implemented using the MATLAB code provided by the authors), to benchmark the results obtained by the distributed algorithms.  All the distributed algorithm are initialized  setting each  $\mathbf{X}_i^0=\mathbf{0}$, and each $\mathbf{D}_{(i)}^0$ equal to some randomly chosen columns of $\mathbf{S}_i$. 

In D$^4$L, the subproblems \eqref{D_tilde_subproblem} and \eqref{eq:Xupdate} at iteration $\nu$ do not have a closed form solution; they are solved using the projected (sub)gradient algorithm, with diminishing step-size $\gamma^r=\gamma^{r-1}(1-\epsilon \gamma^{r-1})$, where $\gamma^0=0.9$, $\epsilon=10^{-3}$, and $r$ denotes the inner iteration index. A warm start is used  for the projected subgradient algorithm; the initial points are set to $\mathbf{D}_{(i)}^\nu$ and $\mathbf{X}^\nu_i$ in problems \eqref{D_tilde_subproblem} and \eqref{eq:Xupdate}, respectively, where $\nu$ is the iteration index of the outer loop.  We terminate the projected subgradient algorithm solving \eqref{D_tilde_subproblem} and \eqref{eq:Xupdate}  when $J^{r}_{D,i}\triangleq \|\widehat{\mathbf{D}}^{\nu,r}_{(i)}-\mathbf{D}^{\nu,r}_{(i)}\|_{\infty,\infty}\leq 10^{-6}$   and $J^{r}_{X,i}\triangleq \|\widehat{\mathbf{X}}^{\nu,r}_i-\mathbf{X}^{\nu,r}_i\|_{\infty,\infty},\leq 10^{-6}$, respectively, where \vspace{-0.3cm}
{\begin{equation}
\begin{aligned}
\widehat{\mathbf{D}}^{\nu,r}_{(i)}&\triangleq\underset{\mathbf{D}_{(i)}\in\mathcal{D}}{\text{argmin}}\left\langle \nabla_{D} f_i(\mathbf{D}_{(i)}^{\nu,r},\mathbf{X}^{\nu}_i)+I\boldsymbol{\Theta}_{(i)}^\nu-\nabla_{D} f_i(\mathbf{D}_{(i)}^{\nu},\mathbf{X}^{\nu}_i)+\tau_{D,i}^\nu\big(\mathbf{D}^{\nu,r}_{(i)}-\mathbf{D}^{\nu}_{(i)}\big),\mathbf{D}_{(i)}-\mathbf{D}_{(i)}^{\nu,r} \right\rangle\vspace{-0.2cm}
\\
&\qquad\qquad\qquad\qquad\qquad\qquad\qquad\qquad\qquad\qquad\qquad+\frac{100}{2}\norm{\mathbf{D}_{(i)}-\mathbf{D}_{(i)}^{\nu,r}}^2+G\left(\mathbf{D}_{(i)}\right),\medskip 
\\
\widehat{\mathbf{X}}^{\nu,r}_i&\triangleq\underset{\mathbf{X}_i\in\mathbb{R}^{K\times n_i}}{\text{argmin}}\left\langle \nabla_{X_i} f_i(\mathbf{U}_{(i)}^\nu,\mathbf{X}^{\nu,r}_i)+\tau_{X,i}^\nu\left(\mathbf{X}^{\nu,r}_i-\mathbf{X}^{\nu}_i\right),\mathbf{X}_i-\mathbf{X}_i^{\nu,r} \right\rangle \nonumber\\&\qquad\qquad\qquad\qquad\qquad\qquad\qquad\qquad\qquad\qquad\qquad  +\frac{100}{2}\norm{\mathbf{X}_i-\mathbf{X}_i^{\nu,r}}^2+g_i\left(\mathbf{X}_i\right),
\end{aligned}
%\label{SSVD_optimality_inner_loop}
\nonumber\vspace{-0.1cm}
\end{equation}}
with $\mathbf{D}^{\nu,r}_{(i)}$ and $\mathbf{X}^{\nu,r}_i$ denoting the value of $\mathbf{D}_{(i)}$ and $\mathbf{X}_i$ at the $\nu$-th outer and  $r$-th inner iteration, respectively. In all our simulations, the above accuracy was reached within 50 (inner) iterations of the projected subgradient algorithm.

\noindent \textbf{Convergence speed and quality of the reconstruction:} We simulated 3 \emph{directed static} network topologies, namely:  N1-N3, as given in Table \ref{network_setting_casewise_0}.
In Fig.~\ref{fig:biclustering_plots} we  plot he \emph{average} value of the objective function   [subplot on the left], the consensus disagreement $e^\nu$ [subplot in the center], and the distance from stationarity $\Delta^\nu$   [subplot on the right],  achieved by Plain D$^4$L  and Modified ATC, versus the number of message exchanges, for  the three scenarios N1 [subplot (a)], N2 [subplot (b)] and N3 [subplot (c)]. The average is taken over  5 digraph realizations. Fig. \ref{fig:biclustering_plots} shows that Plain D$^4$L algorithm attains satisfactory merit values in all network scenarios, while Modified ATC fails to reach consensus/convergence, even in highly connected networks. The poor performance of Modified ATC seem   mainly due to the incapability of locking the consensus.

\begin{figure}\vspace{-0.5cm}
\center
\includegraphics[scale=0.35]{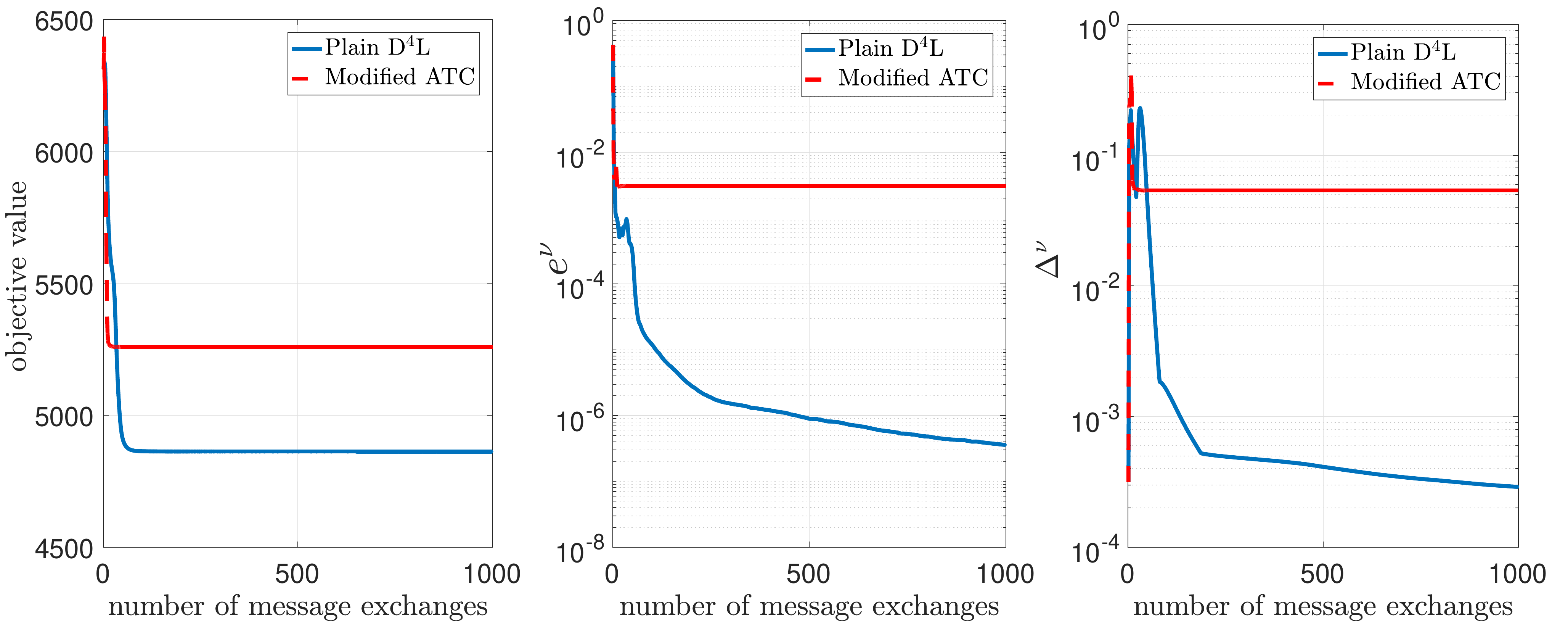}
\\
\centering (a)
\\
\center
\includegraphics[scale=0.35]{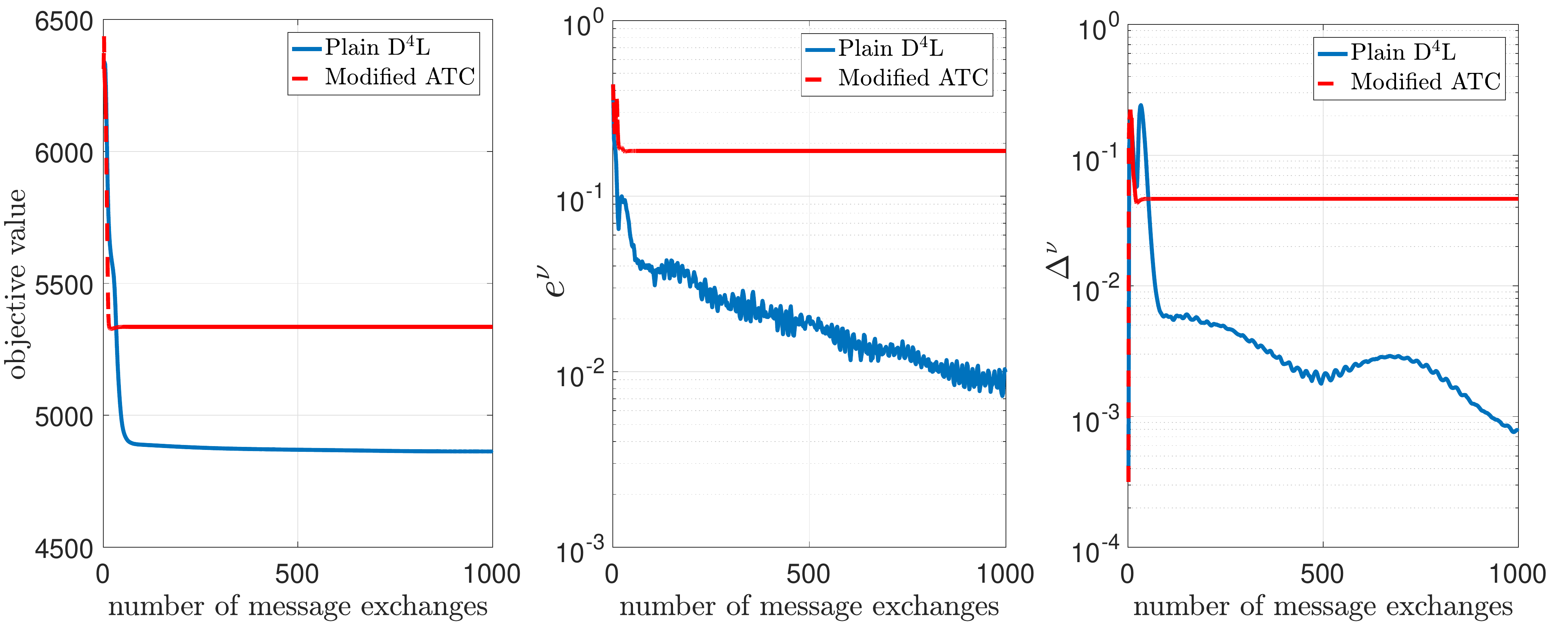}
\\
\centering (b)
\\\center
\includegraphics[scale=0.35]{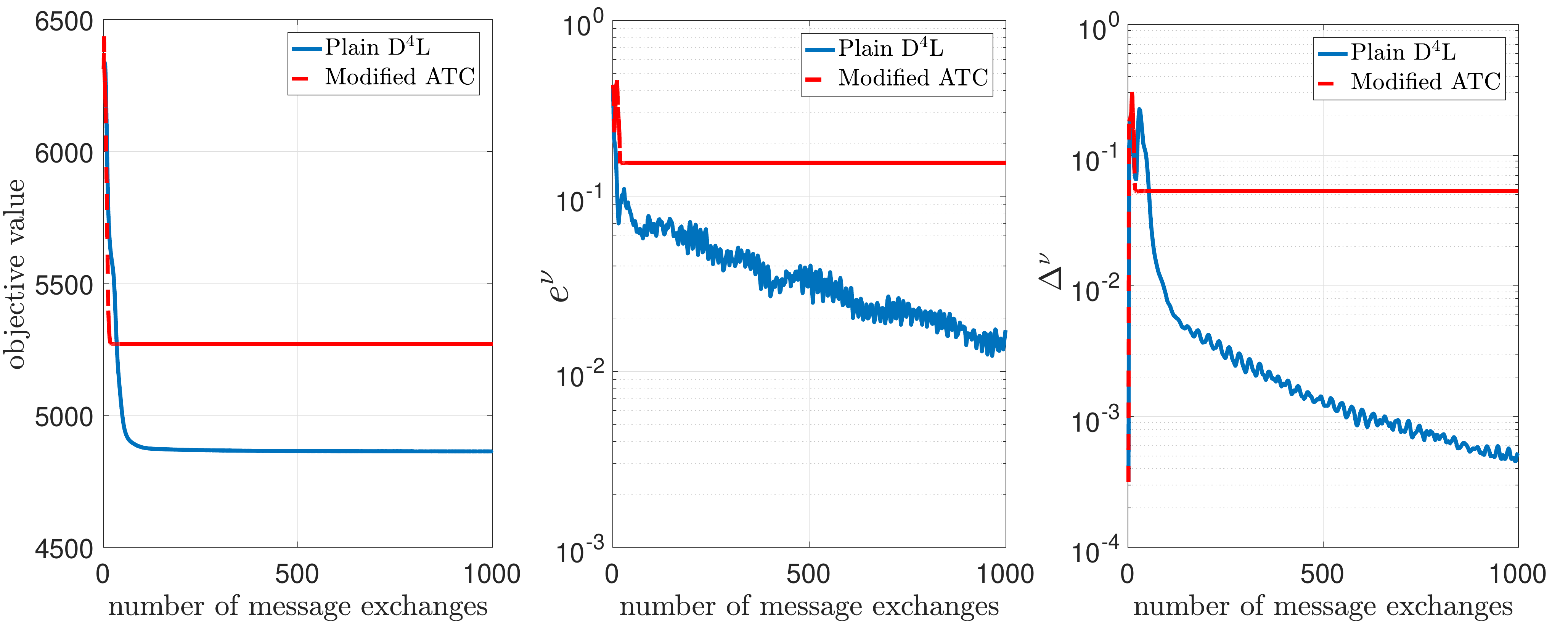}
\\
\centering (c)
\vspace{-0.2cm}\\
\caption{{Biclustering problem--Plain D$^4$L and Modified ATC algorithms:} objective value [subplots on the left],  consensus disagreement [subplots in the center], and distance from stationarity $\Delta^\nu$ [cf. \eqref{stat_merit_subproblems}] [subplots on the right] vs. number of message exchanges. Comparison  over three network settings [cf.~Table~\ref{network_setting_casewise_0}]:  N1 [subplots (a)],  N2 [subplots (b)], and  N3 [subplots (c)].}
\label{fig:biclustering_plots}
\end{figure}

In order to assess the quality of the solutions achieved by the three algorithms, we employ the following procedure. Given the limit point (up to the fixed accuracy)  $\mathbf{D}^\infty$ of the algorithm under consideration,
 patients' information is  in form of (unlabeled clusters of) data points $\{\mathbf{D}_{m,:}^\infty\}_{m=1}^{56}$, where  $\mathbf{D}_{m,:}^\infty$ denotes the $m$-th row of  $\mathbf{D}^\infty$ and   represents an individual patient. %However,  no labels are associated to patients' information $\mathbf{D}^\infty$. 
 In order to compare $\mathbf{D}^\infty$ with the labeled ground truth, we need to tag labels to the clustered points of $\mathbf{D}^\infty$. To do so,  we run the  K-means clustering algorithm on  $\{\mathbf{D}_{m,:}^\infty\}_{m=1}^{56}$. Specifically, we first run  K-means 100 times  and,  in each run,  we  perform a preliminary clustering   using  10\% of the points (randomly chosen). Then,  among the 100 obtained clustering configurations, we picked  the one  with  the smallest ``within-cluster sum  of point-to-centroid distances''.\footnote{Given a clustering partition $\{\mathcal{C}_i\}_{i=1}^4$, the 
 ``within-cluster sum  of point-to-centroid distances'' measures the quality of the k-means clustering, and is defined as $\sum_{i=1}^4\sum_{j\in\mathcal{C}_i}||\mathbf{D}_{j,:}^\infty-\overline{\mathbf{D}}_{\mathcal{C}_i}^\infty||^2$, where   $\overline{\mathbf{D}}_{\mathcal{C}_i}\triangleq\frac{1}{|\mathcal{C}_i|}\sum_{j\in\mathcal{C}_i}\mathbf{D}_{j,:}^\infty$ and  $|\mathcal{C}_i|$ denotes the cardinality of the set $\mathcal{C}_i$.}  
 Finally, we assign to each cluster the label associated with the most populated type of cancer in the cluster. Denoting the ground truth classes by $\{\mathcal{C}_i\}_{i=1}^4$ (recall that there are 4 classes/types 
of cancer), where each $\mathcal{C}_i$ consists of the group of patients with the same type of cancer, and by  $\{\widetilde{\mathcal{C}}_i\}_{i=1}^4$  the   clustering obtained  by the procedure described above  applied to the outcome   $\mathbf{D}^\infty$ of the simulated algorithms,  % the classes assigned and using the ground truth labeling $\{\mathcal{C}_i\}_{i=1}^4$, 
we   measure the quality of the clustering by  the \emph{Jaccard index},   defined as \vspace{-0.2cm}
\begin{equation}
\nonumber
J=\frac{\left|\bigcup_i\left(\mathcal{C}_i\cap\widetilde{\mathcal{C}}_i\right)\right|}{\sum_i\left|\mathcal{C}_i\cup\widetilde{\mathcal{C}}_i\right|}.\vspace{-0.2cm}
\end{equation} 
%where $|\bullet|$ denotes the cardinality of a set.
Clearly $0\leq J\leq 1$, and the higher the index value, the better the quality of the clustering. 

In Table \ref{avg_Jaccard_table}, we report the average   and   Maximum Absolute Deviation (MAD) %\footnote{\textcolor{red}{Maximum Absolute Deviation of $n$ given scalars $\{x_i\}_{i=1}^n$ is defined as $\max_i|x_i-\bar{x}|$, where $\bar{x}$ is the average.}} 
of the Jaccard indices from their average, computed over the aforementioned 5 realizations of the three graph topologies, as   in Table~\ref{network_setting_casewise_0} (see also Fig.~\ref{fig:biclustering_plots}). The values in the table  clearly show that Plain D$^4$L  achieves better results than  those produced by Modified ATC or centralized methods. Moreover,  the value of the Jaccard index from  Plain D$^4$L does not depend on the specific network topology. which is not the case for Modified ATC.

\begin{table}[H]
\center
\begin{tabular}{c c c c}
\hline
\textbf{Network \#}  & \textbf{Plain D$^4$L} & \textbf{Modified ATC} & \textbf{\cite{Lee_Shen_Huang_Marron10}}
 \\
 \hline
N1 &  0.8983/0 & 0.7778/0 & --
 \\
N2 &  0.8983/0 & 0.7045/0.3218 & --
 \\
N3 &  0.8983/0 & 0.7892/0.0172 & --
\\
Centralized &  -- & -- & 0.7231/--
\\
\hline
\end{tabular}
\caption{Biclustering problem$-${Average/MAD of Jaccard indices over 5 realizations of digraphs.} }
\label{avg_Jaccard_table}\vspace{-0.5cm}
\end{table}

\subsection{Non-negative Sparse Coding (NNSC) and Sparse PCA (SPCA)}
\noindent\textbf{Problem Formulation:} We consider the Non-negative Sparse Coding (NSC) formulation \eqref{eq:NNSC} \citep{Hoyer_SNMF} and the Sparse PCA problem  \eqref{eq:SSVD} \citep{Mairal_Bach_2010}. %Both procedures aim to exploit the latent structure of data and to reduce its dimensionality in order to facilitate further analytics. 
For both formulations, we run experiments  using the following two datasets:
\begin{itemize}
\item MIT-CBCL face database \#1 \citep{sung96}: a pool of $N=2,429$ vectorized face images of size $19\times 19$ pixels each (i.e. $M=361$);
\item The VOC 2006 database \citep{Everingham10}:  a pool of $N=10,000$ vectorized natural image patches of size $16\times 16$ pixels each (i.e. $M=256$).
\end{itemize}

\noindent Consistently with \citep{Mairal_Bach_2010},  the free parameters  are set as:
\begin{itemize}
\item 	NNSC  \eqref{eq:NNSC}:   $K=49$, $\lambda=\mu=1/\sqrt{M}$, and $\alpha=1$;    % and $n_i=\ceil*{N/I}$;
\item Sparse PCA   \eqref{eq:SSVD}:  $K=49$, $\lambda_X=\mu_X=1/\sqrt{M}$, $\lambda_D=\mu_D=1/\sqrt{M}$, and $\alpha=1$. %and $n_i=\ceil*{N/I}$. 
\end{itemize} 
The total number of variables for the above optimization problems are   136,710 for the  MIT-CBCL dataset, and 502,544  for the  VOC 2006 dataset.

We simulated the communication network as  \emph{static directed} graphs of size $I$, clustered in $n_c$ groups, where each node has an outgoing arc to another node in the same cluster with probability $p_1$, while $p_2$ is the probability of an outgoing arc to a node in a different cluster.
We run our tests over 6 different network scenarios, with various  size $I$ and probability pair $(p_1,p_2)$, as  given in Table \ref{network_setting_casewise}.  Note that if $N/I$ is not an integer, we pad zero columns to the data matrix $\mathbf{S}$ so that all the agents own equal-size partitions $\mathbf{S}_i$'s, thus $n_i=\ceil*{N/I}$ in both problems \eqref{eq:SSVD} and \eqref{eq:NNSC}.

\begin{table}[h]\vspace{-0.1cm}
\center
\begin{tabular}{c c c c c}
\hline
\textbf{Network \#} &  $I$ &$n_c$& $p_1$& $p_2$ \\
\hline
~~~ N4 ~~~~~~& 10 &2  &0.9  &0.3  \\
~~~ N5 ~~~~~ & 10 &2  &0.2  &0.1  \\
~~~ N6 ~~~~~ & 50 &5  &0.9  &0.3  \\
~~~ N7 ~~~~~ & 50 &5  &0.2  &0.1  \\
~~~ N8 ~~~~~ & 250 &10  &0.9  &0.3  \\
~~~  N9 ~~~~~ & 250 &10  &0.2  &0.1  \\
  \hline
\end{tabular}
\caption{Network setting for the NNSC and Sparse PCA problems.}
\label{network_setting_casewise}\vspace{-0.6cm}
\end{table}

\vspace{-0.1cm}
\subsubsection{Non-negative Sparse Coding} 
\noindent\textbf{Algorithms and tuning:} We test the Plain D$^4$L, with   $\tilde{f}_i$ and $\tilde{h}_i$  chosen according to  \eqref{eq:ftilde2} and  \eqref{eq:htilde1}, respectively.  The other parameters of the algorithm are set to: $\gamma^\nu=\gamma^{\nu-1}(1-\epsilon \gamma^{\nu-1})$, with  $\gamma^0=0.2$ and $\epsilon=10^{-2}$; and  $\tau^\nu_{D,i}=10$ and $\tau^\nu_{X,i}=\max(L_{\nabla{X_i}}(\mathbf{U}_{(i)}^\nu),10)$.
 We compare the proposed scheme with a modified version of ATC, equipped with our new consensus protocol, implementable on  directed networks. All the distributed algorithm are initialized  setting   $\mathbf{X}_i^0=\mathbf{0}$ and $\mathbf{D}_{(i)}^0$ equal to some randomly chosen columns of $\mathbf{S}_i$. Both Plain D$^4$L and Modified ATC call for solving a LASSO problem in updating the private variables  (cf. Sec.~\ref{subsec:discussion});  the update of the dictionary has instead a closed form expression, see \eqref{closed_form_dictionary_lin}.  For both Plain D$^4$L and Modified ATC, the LASSO subproblems at iteration $\nu$ are solved using the  projected (sub)gradient algorithm with diminishing step-size $\gamma^r=\gamma^{r-1}(1-\epsilon \gamma^{r-1})$, where $\gamma^0=0.9$, $\epsilon=10^{-3}$, and $r$ denoting the inner iteration index. We terminate the projected subgradient algorithm in the inner loop when $J_{X,i}^{r}\triangleq \|\widehat{\mathbf{X}}^{\nu,r}_i-\mathbf{X}^{\nu,r}_i\|_{\infty,\infty}\leq 10^{-4}$, where  \vspace{-0.2cm}
\begin{equation}
\widehat{\mathbf{X}}^{\nu,r}_i\triangleq\underset{\mathbf{X}_i\in\mathcal{X}_i}{\text{argmin}}\left\langle \nabla_{X_i} f_i(\mathbf{U}_{(i)}^\nu,\mathbf{X}^{\nu,r}_i)+\tau_{X,i}^\nu\left(\mathbf{X}^{\nu,r}_i-\mathbf{X}^{\nu}_i\right),\mathbf{X}_i-\mathbf{X}_i^{\nu,r} \right\rangle+\frac{1}{2}\norm{\mathbf{X}_i-\mathbf{X}_i^{\nu,r}}^2+g_i\left(\mathbf{X}_i\right),\vspace{-0.2cm}
%\label{SSVD_optimality_inner_loop2}
\nonumber
\end{equation}
 and  $\mathbf{X}^{\nu,r}_i$ denotes the value of  $\mathbf{X}_i$ at the $\nu$-th outer and  $r$-th inner iteration. In all our simulations, the above accuracy was reached within 30 (inner) iterations of the projected subgradient algorithm.

\noindent\textbf{Convergence speed and quality of the reconstruction:} {We run the Plain D$^4$L and Modified ATC algorithms over different network settings, as listed in Table \ref{network_setting_casewise}, and we terminated them after 1500 message exchanges.  We replicated the tests for 5 independent realizations and we reported the \emph{average} of the final values of the objective function, the consensus disagreement, and the distance from stationarity in Table~\ref{NNSC_CBCL} and Table~\ref{NNSC_VOC}, for the MIT-CBCL and VOC 2006 datasets, respectively.}
 In Fig. \ref{fig:NNSC_CBCL} and Fig. \ref{fig:NNSC_VOC} (for MIT-CBCL and VOC 2006 datasets, respectively), we plot the \emph{average} value (over the aforementioned 5 graph realizations) of the objective function [subplot on the left], the consensus disagreement $e^\nu$ [subplot in the center], and the distance from stationarity $\Delta^\nu$   [subplot on the right], versus number of message exchanges, for the two extreme network scenarios N4 [subplot (a)] and N9 [subplot (b)].
These results show that the proposed  Plain D$^4$L significantly outperforms   the Modified ATC algorithm. They also show a remarkable stability of Plain D$^4$L with respect to the simulated network graphs, which is not observed for the Modified ATC, whose performance  deteriorates significantly going from N4 to N9. 
\begin{table}[H]
\center
\begin{tabular}{c c c c }
\hline
\textbf{Network \#} & \textbf{objective value} &\textbf{consensus disagreement} & \textbf{distance from stationarity}\\
\hline
~~~ N4 ~~~~~ & 169.8/171.9 &9.17e-7/2.9e-4 &4.7e-4/8.8e-2  \\
~~~ N5 ~~~~~ & 169.9/172.2 &4.5e-6/6.7e-4  &5.3e-4/9.8e-2    \\
~~~ N6 ~~~~~ & 169.8/177.2 &2.4e-7/1.9e-4  &5.1e-4/6.3e-2    \\
~~~ N7 ~~~~~ & 169.9/177.3 &1.1e-6/6.3e-4  &5.5e-4/6.1e-2    \\
~~~ N8 ~~~~~ & 169.8/191.0 & 2.1e-7/1.2e-4  &5.1e-4/2.2e-2  \\
~~~ N9 ~~~~~ & 169.8/190.9 &5.8e-7/3.1e-4 &6.6e-4/1.1e-2  \\
  \hline
\end{tabular}
\caption{NNSC problem (MIT-CBCL dataset)--Plain D$^4$L/Modified ATC algorithms: objective value, consensus disagreement, and distance from stationarity obtained after 1500 message exchanges.}
\label{NNSC_CBCL}\vspace{-0.5cm}
\end{table}
 \begin{table}[H]
\center
\begin{tabular}{c c c c }
\hline
\textbf{Network \#} & \textbf{objective value} &\textbf{consensus disagreement} & \textbf{distance from stationarity}\\
\hline
~~~ N4 ~~~~~ & 845.2/848.8 & 2.9e-6/1.3e-4 & 1.1e-3/4.1e-3  \\
~~~ N5 ~~~~~ & 845.9/850.1 & 1.5e-5/5.3e-4 &1.5e-3/6.1e-3    \\
~~~ N6 ~~~~~ & 844.8/879.5 & 5.7e-7/2.2e-4  &1.3e-3/4.4e-2    \\
~~~ N7 ~~~~~ & 844.5/879.3 & 1.9e-6/1.2e-3 & 1.3e-3/3.9e-2  \\
~~~ N8 ~~~~~ & 844.8/941.2 & 5.6e-7/1.5e-4  & 1.6e-3/4.8e-2  \\
~~~ N9 ~~~~~ & 845.0/941.8 & 1.5e-6/3.6e-4 & 1.1e-3/5.1e-2  \\
  \hline
\end{tabular}
\caption{NNSC problem (VOC 2006 dataset)--Plain D$^4$L/Modified ATC algorithms: objective value, consensus disagreement, and distance from stationarity obtained after 1500 message exchanges.}
\label{NNSC_VOC}\vspace{-0.7cm}
\end{table}
\begin{figure}[t]
\includegraphics[scale=0.34]{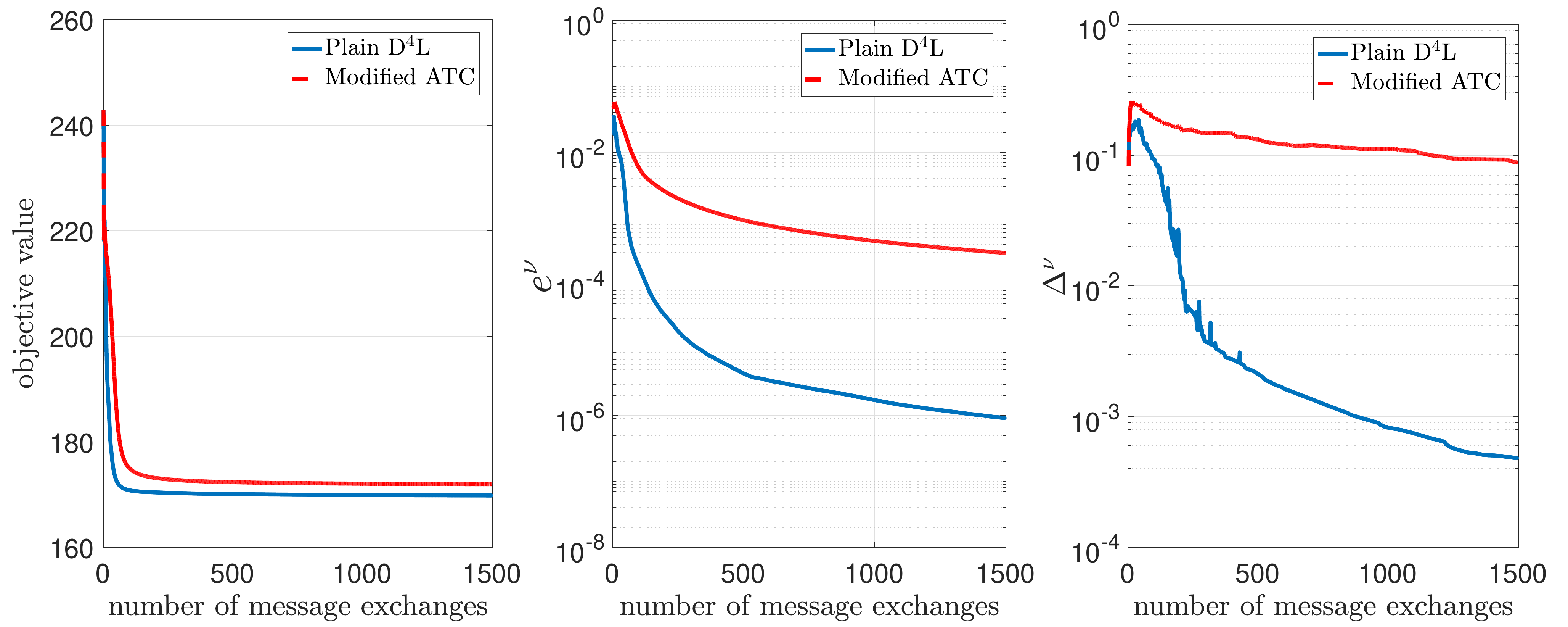}
\\
\centering(a)
\\
\hspace{-.4cm}\includegraphics[scale=0.34]{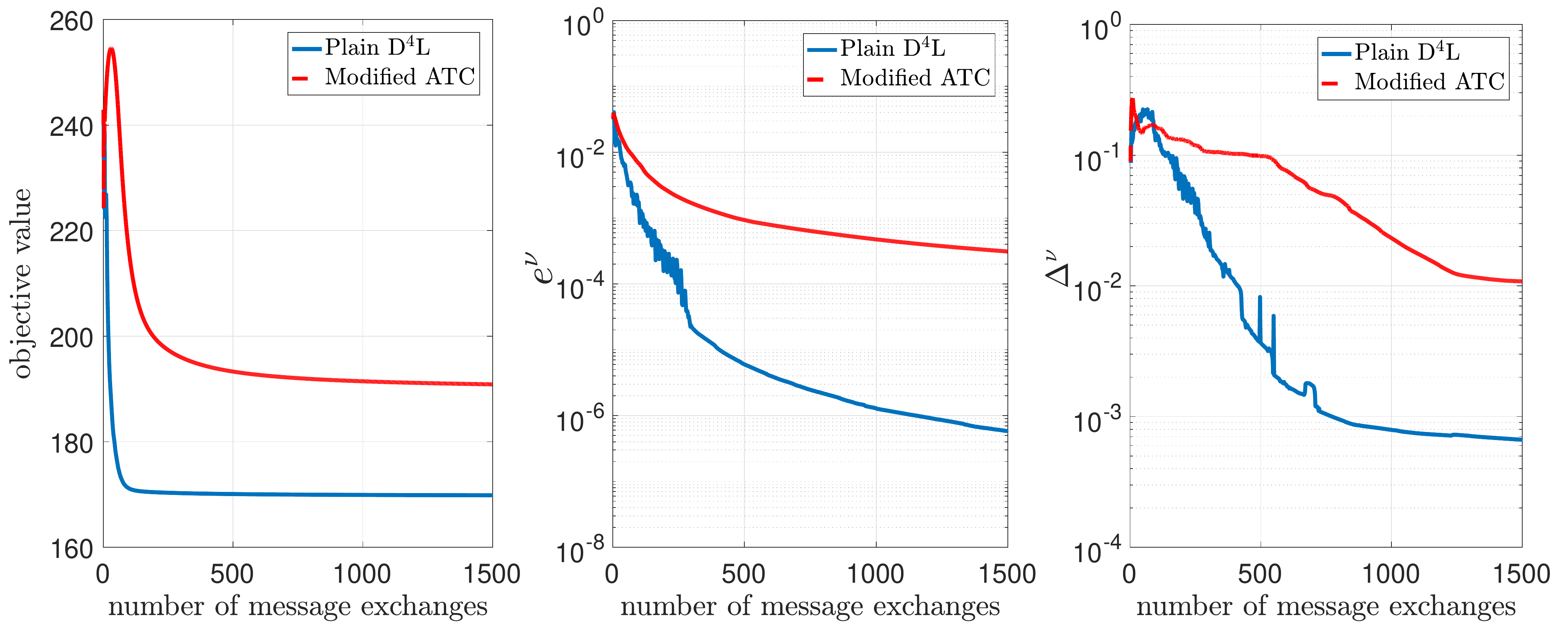}
\\
\centering (b)
\caption{NNSC problem (MIT-CBCL dataset)--Plain D$^4$L and Modified ATC algorithms: objective value [subplots on the left],  consensus disagreement [subplots in the center], and distance from stationarity $\Delta^\nu$ [cf. \eqref{stat_merit_subproblems}] [subplots on the right] vs. number of message exchanges. Comparison  over the  network settings   N4 [subplots (a)] and  N9 [subplots (b)] (cf.~Table~\ref{network_setting_casewise}).}\vspace{-0.4cm}
\label{fig:NNSC_CBCL}
\end{figure}
 \begin{figure}[t]
\includegraphics[scale=0.34]{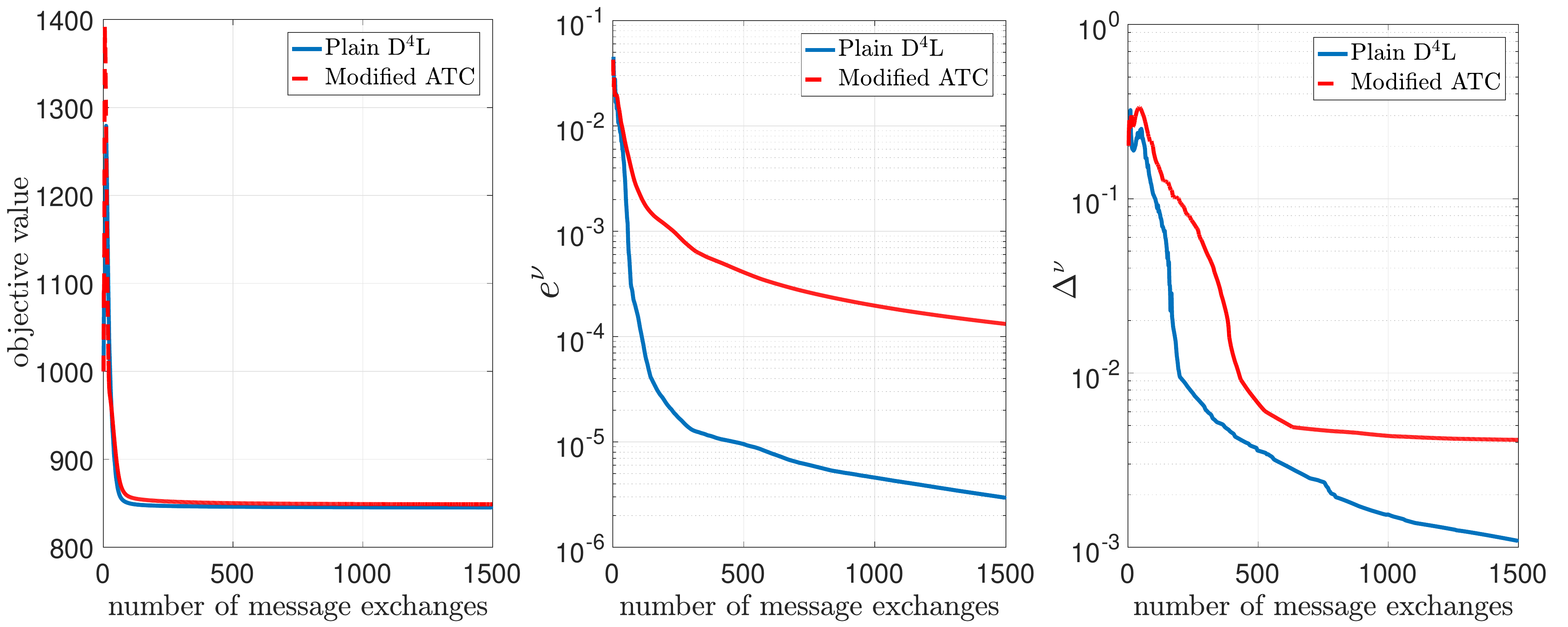}
\\
\centering(a)
\\
\hspace{-.4cm}\includegraphics[scale=0.34]{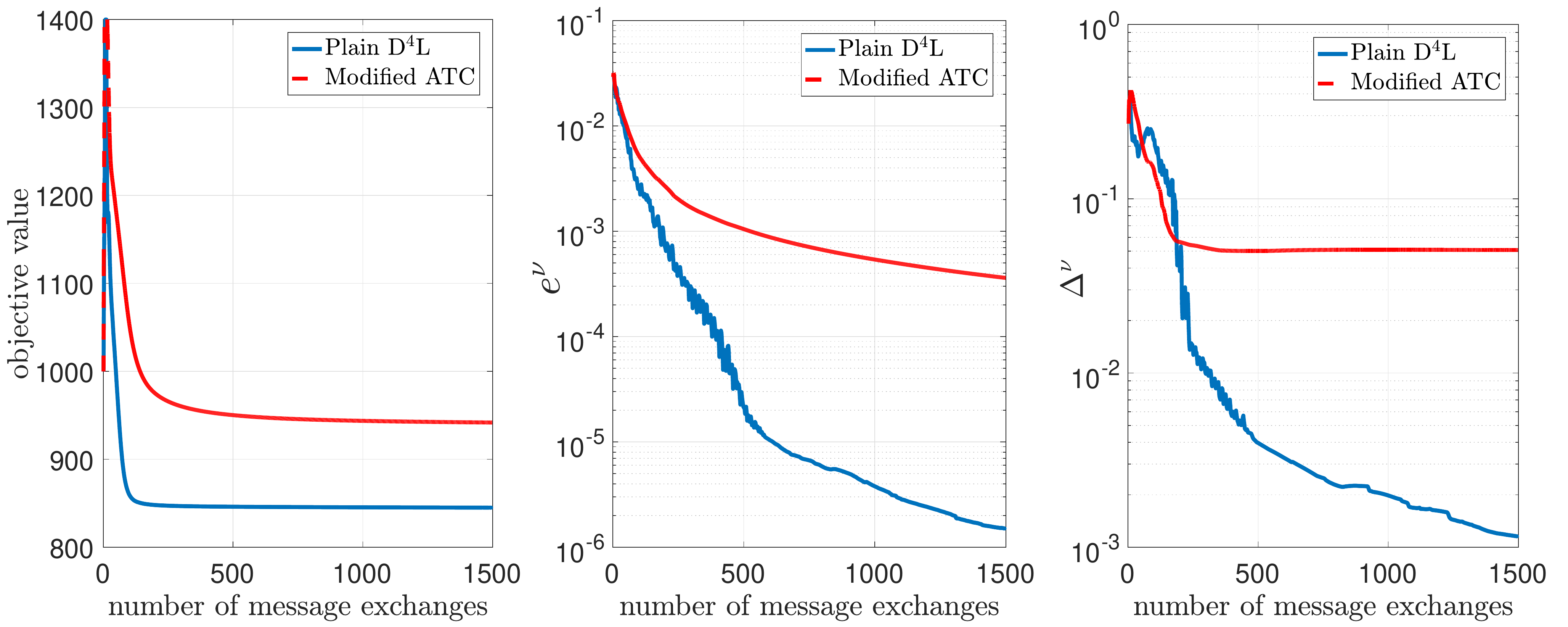}
\\
\centering (b)
\caption{NNSC problem (VOC 2006 dataset)--Plain D$^4$L and Modified ATC algorithms: objective value [subplots on the left],  consensus disagreement [subplots in the center], and distance from stationarity $\Delta^\nu$ [cf. \eqref{stat_merit_subproblems}] [subplots on the right] vs. number of message exchanges. Comparison  over the  network settings   N4 [subplots (a)] and  N9 [subplots (b)] (cf.~Table~\ref{network_setting_casewise}).}
\label{fig:NNSC_VOC}
\vspace{-0.5cm}
\end{figure}

\subsubsection{Sparse Principal Component Analysis (SPCA)}
%\noindent\textbf{Problem Formulation:} To perform SPCA, we consider the formulation \eqref{eq:SSVD} provided in Sec. \ref{examples}. In our tests, the free parameters are set to be the followin:g  $\lambda_X=\mu_X=1/\sqrt{M}$, $\lambda_D=\mu_D=1/\sqrt{M}$ and $\alpha=1$, and $n_i=\ceil*{N/I}$. Note that the dimensions of the problem (for each dataset) is exactly the same as those in the previous NNSC example. 

\noindent\textbf{Algorithms and tuning:} We test the same D$^4$L version as used in the Biclustering experiments (cf. Subsec. \ref{Biclustering_numerical_results}), i.e.,   $\tilde{f}_i$ and $\tilde{h}_i$ are chosen according to  \eqref{eq:ftilde1} and \eqref{eq:htilde1}, respectively; we set   $\gamma^\nu=\gamma^{\nu-1}(1-\epsilon \gamma^{\nu-1})$, with  $\gamma^0=0.2$ and $\epsilon=10^{-2}$; and  $\tau^\nu_{D,i}=10$ and $\tau^\nu_{X,i}=\max(L_{\nabla{X_i}}(\mathbf{U}_{(i)}^\nu),10)$. We term it Plain D$^4$L. We compare Plain D$^4$L  with a  modified version of the ATC algorithm,    which has been adapted to solve \eqref{eq:SSVD} and  equipped with our new consensus protocol to handle directed network topologies. All the distributed algorithm are initialized, setting   $\mathbf{X}_i^0=\mathbf{0}$ and $\mathbf{D}_{(i)}^0$ equal to some randomly chosen columns of $\mathbf{S}_i$.
The subproblems \eqref{D_tilde_subproblem} and \eqref{eq:Xupdate} at iteration $\nu$ are solved using the projected (sub)gradient algorithm; the setting is the same as that used in the Biclustering problem (cf. Subsec.~\ref{Biclustering_numerical_results}). We terminate the projected subgradient algorithm in the inner loop when $J^{r}_{D,i}\leq 10^{-4}$ (in solving subproblem  \eqref{D_tilde_subproblem}) and $J^{r}_{X,i}\leq 10^{-4}$ (in solving subproblem \eqref{eq:Xupdate}), where  $J^{r}_{D,i}$ and $J^{r}_{X,i}$ are defined as those in Subsec.~\ref{Biclustering_numerical_results}. In all our simulations, the above accuracy was reached within 30 (inner) iterations of the projected subgradient algorithm. 

\noindent\textbf{Convergence speed and quality of the reconstruction:} We test the Plain D$^4$L and  the Modified ATC in different network settings, as listed in  Table \ref{network_setting_casewise}.  {The setting of the experiments and the averaging procedure of the reported values is the same of those used for the  NNSC problem.} The results of our experiments are reported in  Table~\ref{SPCA_CBCL} and Figure~\ref{fig:SPCA_CBCL} for the MIT-CBCL dataset; and in Table~\ref{SPCA_VOC} and Figure~\ref{fig:SPCA_VOC} for the VOC 2006 dataset. 
The behaviors or the algorithms are very   similar to  those described in the NNSC case and confirm all previous observations.\vspace{-0.2cm}

\begin{figure}[h]
\includegraphics[scale=0.34]{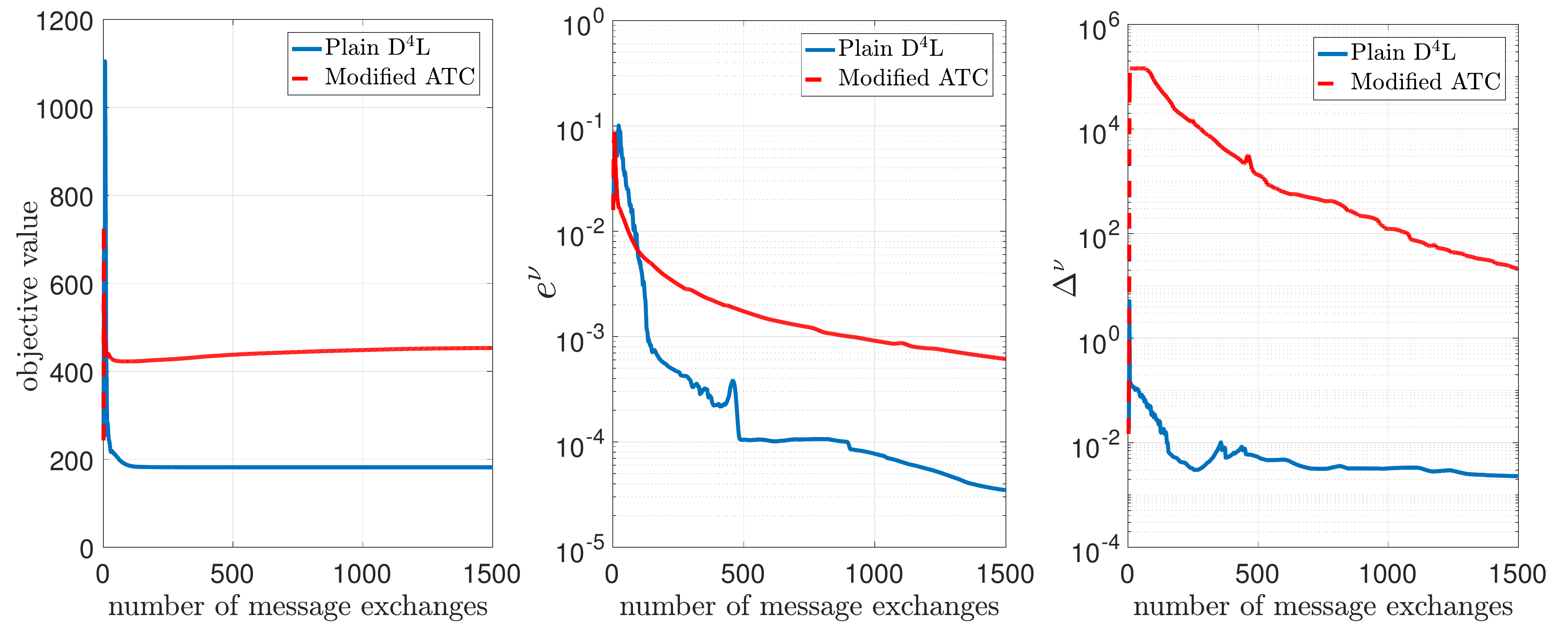}
\\
\centering(a)
\\
\hspace{-.4cm}\includegraphics[scale=0.34]{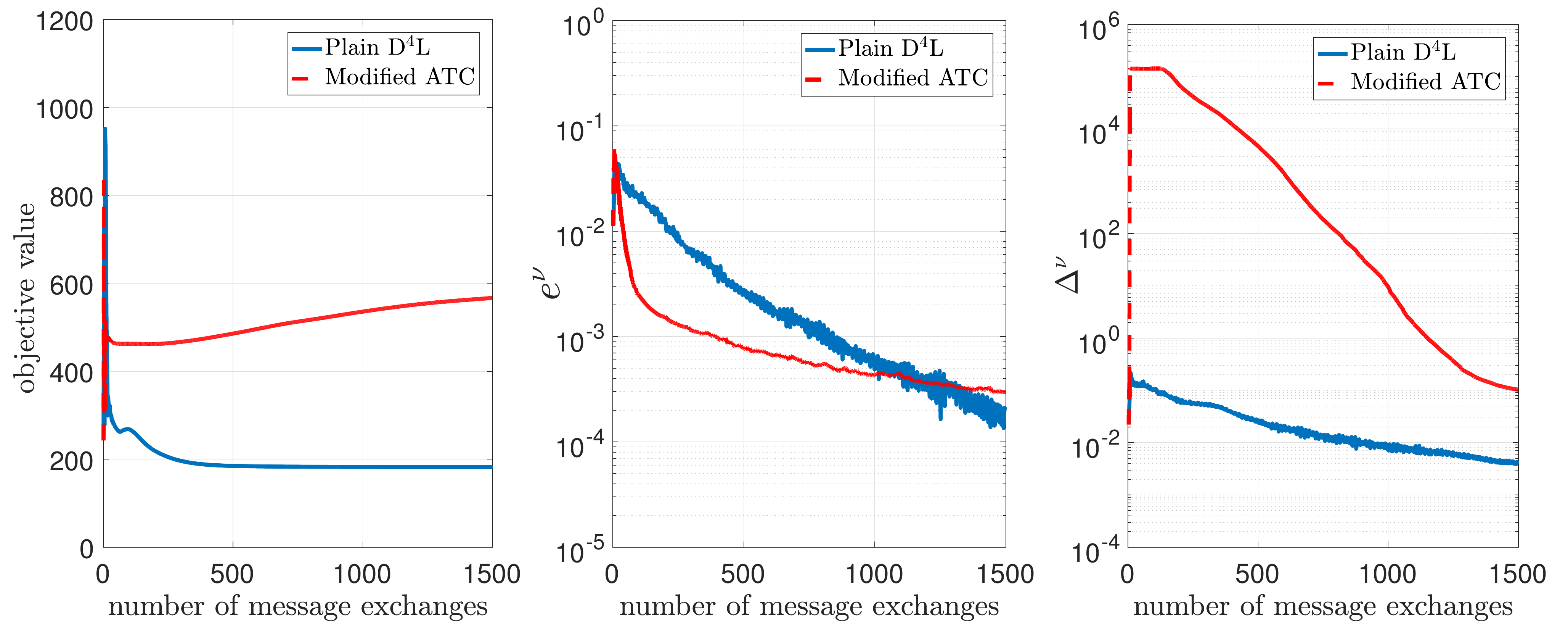}
\\
\centering (b)
\caption{SPCA problem (MIT-CBCL dataset)--Plain D$^4$L and Modified ATC algorithms: objective value [subplots on the left],  consensus disagreement [subplots in the center], and distance from stationarity $\Delta^\nu$ [cf. \eqref{stat_merit_subproblems}] [subplots on the right] vs. number of message exchanges. Comparison  over the  network settings   N4 [subplots (a)] and  N9 [subplots (b)] (cf.~Table~\ref{network_setting_casewise}).}
\label{fig:SPCA_CBCL}
\end{figure}

\begin{table}[H]
\center
\begin{tabular}{c c c c }
\hline
\textbf{Network \#} & \textbf{objective value} &\textbf{consensus disagreement} & \textbf{distance from stationarity}\\
\hline
~~~ N4 ~~~~~ & 181.9/453.1 &3.5e-5/6.1e-4 &2.2e-3/21.39  \\
~~~ N5 ~~~~~ & 185.1/446.4 &2.6e-5/1.5e-3  &1.3e-3/136.2    \\
~~~ N6 ~~~~~ & 182.7/517.3 &5.2e-6/4.3e-4 &1.3e-3/1.2e-1    \\
~~~ N7 ~~~~~ & 186.3/512.0 &2.3e-5/9.2e-4 &1.4e-3/1.18    \\
~~~ N8 ~~~~~ & 181.9/566.0 &4.0e-5/1.7e-4  &2.6e-3/1.5e-1  \\
~~~ N9 ~~~~~ & 182.6/566.9 &1.7e-4/2.9e-4 &4.2e-3/1.0e-1  \\
  \hline
\end{tabular}
\caption{SPCA problem (MIT-CBCL dataset)--Plain D$^4$L/Modified ATC algorithms: objective value, consensus disagreement, and distance from stationarity obtained after 1500 message exchanges.}
\label{SPCA_CBCL}
\end{table}

\begin{figure}[t]
\includegraphics[scale=0.34]{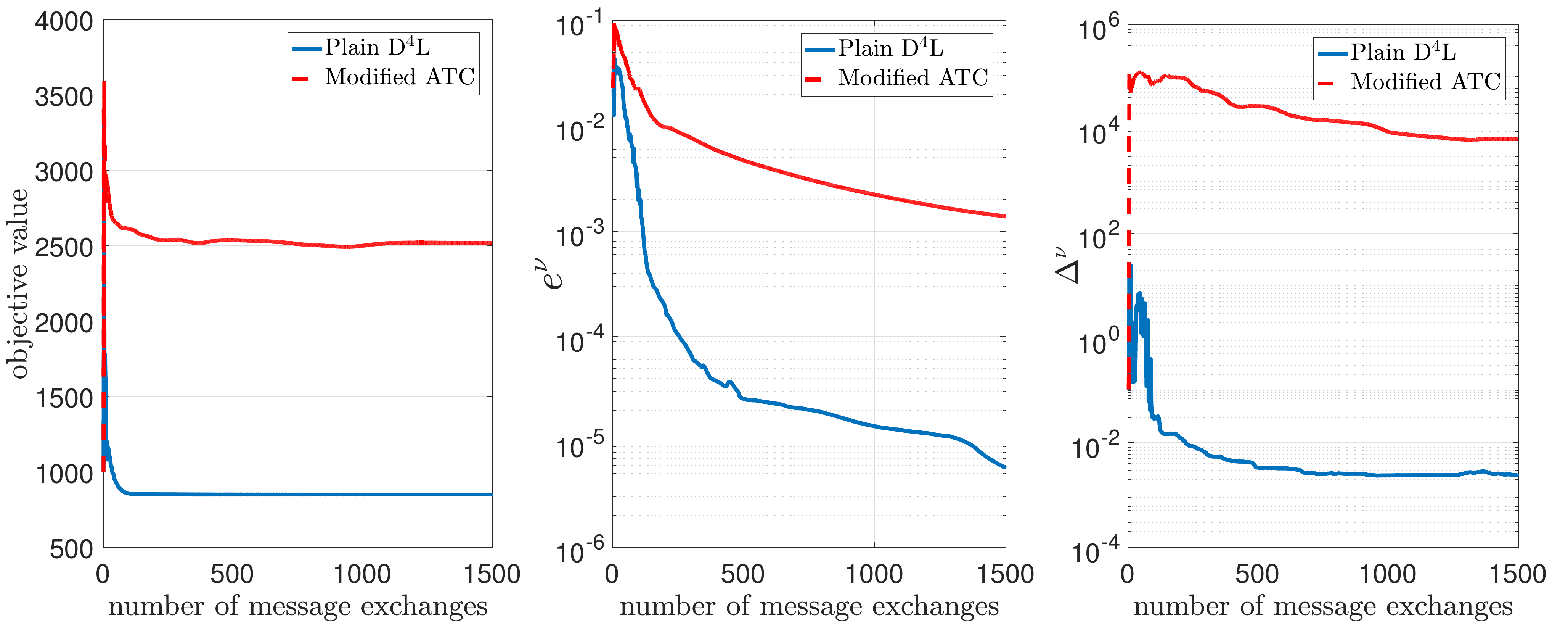}
\\
\centering(a)
\\
\hspace{-.4cm}\includegraphics[scale=0.34]{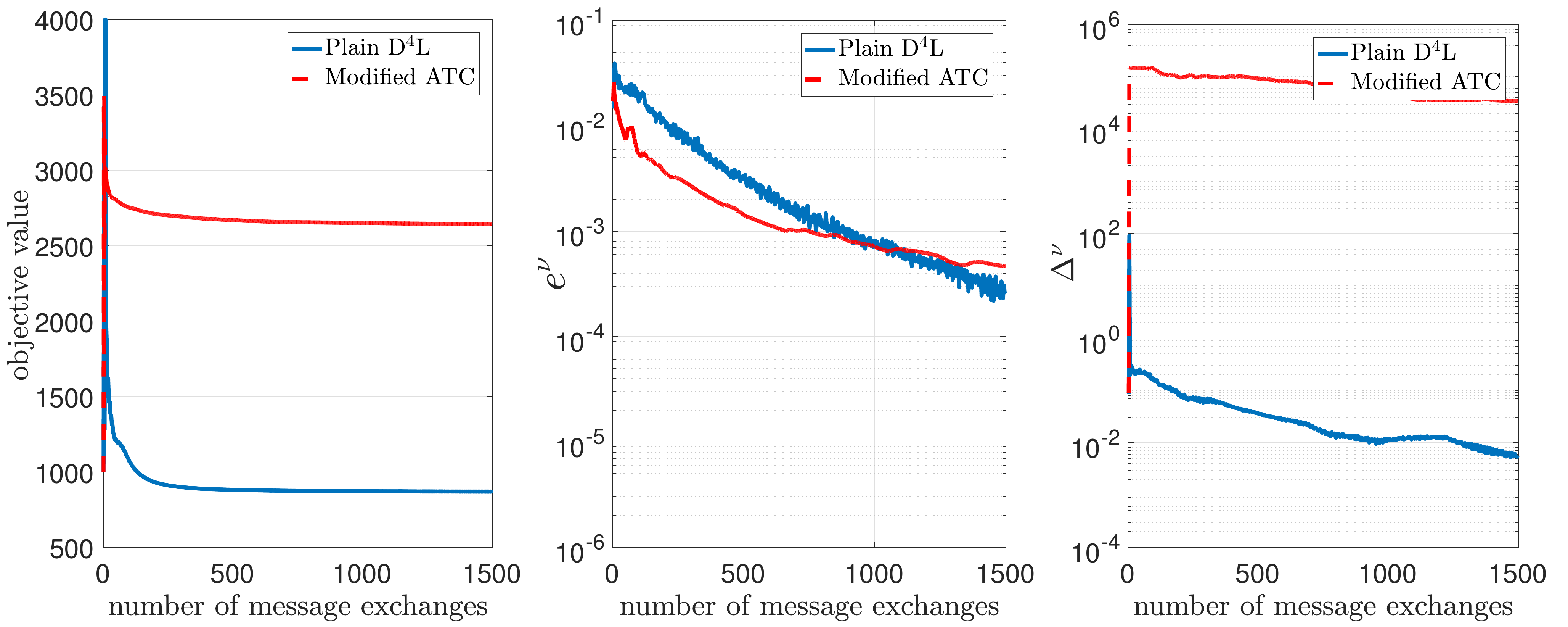}
\\
\centering (b)
\caption{SPCA problem (VOC 2006 dataset)--Plain D$^4$L and Modified ATC algorithms: objective value [subplots on the left],  consensus disagreement [subplots in the center], and distance from stationarity $\Delta^\nu$ [cf. \eqref{stat_merit_subproblems}] [subplots on the right] vs. number of message exchanges. Comparison  over the  network settings   N4 [subplots (a)] and  N9 [subplots (b)] (cf.~Table~\ref{network_setting_casewise}).}
\label{fig:SPCA_VOC}
\end{figure}

\begin{table}[H]
\center
\begin{tabular}{c c c c }
\hline
\textbf{Network \#} & \textbf{objective value} &\textbf{consensus disagreement} & \textbf{distance from stationarity}\\
\hline
~~~ N4 ~~~~~ & 849.3/2516.0 & 5.7e-6/1.3e-3 & 2.3e-3/6532      \\
~~~ N5 ~~~~~ & 866.0/2533.6 & 1.8e-5/3.5e-3 & 2.0e-3/1.3e+4    \\
~~~ N6 ~~~~~ & 864.8/2621.0 & 7.2e-6/4.3e-4 & 3.4e-3/1.7e+4    \\
~~~ N7 ~~~~~ & 859.1/2624.1 & 5.3e-5/1.1e-3 & 2.4e-3/1.6e+4    \\
~~~ N8 ~~~~~ & 872.7/2643.2 & 6.6e-5/2.3e-4 & 1.1e-2/2.9e+4    \\
~~~ N9 ~~~~~ & 869.1/2642.0 & 2.3e-4/4.6e-4 & 5.6e-3/3.4e+4    \\
  \hline
\end{tabular}
\caption{SPCA problem (VOC 2006 dataset)--Plain D$^4$L/Modified ATC algorithms: objective value, consensus disagreement, and distance from stationarity obtained after 1500 message exchanges.}
\label{SPCA_VOC}  
\end{table}

\section{Conclusions}\label{conclusions} 
This paper studied a fairly general class of distributed dictionary learning problems over time-varying multi-agent networks, with arbitrary topologies. We proposed the first decentralized algorithmic framework$-$the D$^4$L Algorithm$-$with provable convergence for this class of problems. Numerical experiments showed promising performance of our scheme with respect to state-of-the-art distributed methods. 

\acks{The work of Daneshmand, Sun, and Scutari has been  supported by the 
	the National Science Foundation (NSF grants CIF 1564044 and CAREER Award No. 1555850)
	and the  the Office of Naval Research (ONR  Grant N00014-16-1-2244). The work of Facchinei has been partially supported   by the Italian Ministry of Education, Research and University, under the PLATINO (PLATform for INnOvative services in future internet) PON project (Grant Agreement no. PON01$\_$01007).}

\begin{appendix}
\section{Appendix}
\label{appendix}
{In this section we prove the major results of the paper,   Theorems \ref{th:conver} and \ref{th:rate}.   We begin   rewriting the D$^4$L Algorithm in an equivalent vector-matrix form (cf. Sec.~\ref{notation_section_appendix}), which is more suitable for the analysis.    Theorem~\ref{th:conver} and Theorem \ref{th:rate} are then proved in Sec.~\ref{sub:Proof_th_conver} and Sec. \ref{th:rate_proof}, respectively. Some miscellanea results supporting the main proofs are collected in Appendix \ref{Prelim_results}. Table \ref{Table_of_Notation_appendix} below summarizes  the   symbols appearing in the proofs.}

\begin{table}[H]
	\center \footnotesize  
	\renewcommand{\arraystretch}{1.2}
	\begin{tabular}{|c || c | c |c |}
		\hline
		\textbf{Symbol}  & \textbf{Definition} & \textbf{Member of} & \textbf{Reference}
		\\
		\hline
		\hline
		\parbox{1em}{$\overline{\mathbf{D}}_{\boldsymbol{\phi}^\nu}$} & $(1/I)\sum_{i=1}^I \phi_i^\nu \,\mathbf{D}_{(i)}^\nu$ &$\mathcal{D}\subseteq\mathbb{R}^{M\times K}$ & \eqref{average_notation}
		\\
		\hline
		\parbox{1em}{$\overline{\mathbf{U}}_{\boldsymbol{\phi}^\nu}$} & $(1/I)\sum_{i=1}^I \phi_i^\nu \,\mathbf{U}_{(i)}^\nu$ &$\mathcal{D}\subseteq\mathbb{R}^{M\times K}$ & \eqref{average_notation}
		\\
		\hline
		\parbox{1em}{$\widetilde{\mathbf{D}}^{\nu}$}  &$[\widetilde{\mathbf{D}}_{(1)}^{\nu\intercal},\widetilde{\mathbf{D}}_{(2)}^{\nu\intercal},\ldots,\widetilde{\mathbf{D}}_{(I)}^{\nu\intercal}]^\intercal$ & $\mathbb{R}^{M\times KI}$  & \eqref{phi_notation}
		\\
		\hline 
		\parbox{1em}{${\mathbf{U}}^{\nu}$} &$[{\mathbf{U}}_{(1)}^{\nu\intercal},{\mathbf{U}}_{(2)}^{\nu\intercal},\ldots,{\mathbf{U}}_{(I)}^{\nu\intercal}]^\intercal$ & $\mathbb{R}^{M\times KI}$  &\eqref{phi_notation}
		\\
		\hline
		\parbox{1em}{$\boldsymbol{{\Theta}}^{\nu}$} & $[\boldsymbol{{\Theta}}_{(1)}^{\nu\intercal},\boldsymbol{{\Theta}}_{(2)}^{\nu\intercal},\ldots,\boldsymbol{{\Theta}}_{(I)}^{\nu\intercal}]^\intercal$  &  $\mathbb{R}^{M\times K I}$ &  \eqref{phi_notation}
		\\
		\hline
		\parbox{1em}{$\mathbf{G}^\nu$} & $[\nabla_D f_1(\mathbf{D}_{(1)}^\nu,\mathbf{X}^\nu_1)^\intercal,  ~\ldots~, ~\nabla_D f_I(\mathbf{D}_{(I)}^\nu,\mathbf{X}^\nu_I)^\intercal]^\intercal$ & $\mathbb{R}^{M\times K I}$ &\eqref{phi_notation}
		\\
		\hline
		\parbox{1em}{$\mathbf{W}^\nu$} & $(w_{ij}^\nu)_{i,j=1}^I= (a^\nu_{ij}\,\phi_j^\nu/\phi_i^{\nu +1})_{i,j=1}^I=(\boldsymbol{\Phi}^{\nu+1})^{-1}\mathbf{A}^\nu\boldsymbol{\Phi}^\nu$ & $\mathbb{R}^{I\times I}$ &  \eqref{W_mat_A_relat} 
		\\
		\hline 
		\parbox{1em}{$\mathbf{W}^{\nu:l}$} & $\mathbf{W}^{\nu}\cdot\mathbf{W}^{\nu-1}\cdots\mathbf{W}^{l},\quad \nu>l$ & $\mathbb{R}^{I\times I}$ & \eqref{P_def}
		\\
		\hline
		\parbox{1em}{$\widehat{\mathbf{W}}^\nu$} & $\mathbf{W}^\nu\otimes\mathbf{I}_M$ & $\mathbb{R}^{MI\times MI}$ & \eqref{P_def_kroneck}
		\\
		\hline 
		\parbox{1em}{$\widehat{\mathbf{W}}^{\nu:l}$} & $\mathbf{W}^{\nu:l}\otimes\mathbf{I}_M, \quad \nu>l$ & $\mathbb{R}^{MI\times MI}$ & \eqref{P_def}
		\\
		\hline
		\parbox{1em}{$\boldsymbol{\phi}^\nu$} &  $ [\phi_1^\nu,\phi_2^\nu,\ldots,\phi_I^\nu]^\intercal$ & $\mathbb{R}^I$ & \eqref{phi_notation}
		\\ 
		\hline
		\parbox{1em}{$\boldsymbol{\Phi}^\nu$} & $ \mathrm{diag}\left(\boldsymbol{\phi}^\nu\right)$ & $\mathbb{R}^{I\times I}$ & \eqref{phi_notation}
		\\
		\hline
		\parbox{1em}{$\widehat{\boldsymbol{\Phi}}^\nu$} & $\boldsymbol{\Phi}^\nu\otimes\mathbf{I}_M$ & $\mathbb{R}^{MI\times MI}$ & \eqref{P_def_kroneck}
		\\
		\hline 
		\parbox{1em}{$\mathbf{J}_{\boldsymbol{\phi}^\nu}$} & $(1/I)\mathbf{1}\boldsymbol{\phi}^{\nu \intercal}$ & $\mathbb{R}^{I\times I}$ &\eqref{J_def}
		\\
		\hline 
		\parbox{1em}{$\widehat{\mathbf{J}}_{\boldsymbol{\phi}^\nu}$} & $\mathbf{J}_{\boldsymbol{\phi}^\nu} \otimes\mathbf{I}_M$ & $\mathbb{R}^{MI\times MI}$ &\eqref{J_def}	
%		\\
%		\hline 
%		\parbox{1em}{$\underline{\epsilon}_\phi$} and $\bar{\epsilon}_\phi$ & Lowerbound and upperbound of $\{\phi_i^\nu\}_{i,\nu}$ & $\mathbb{R}$	& \eqref{delta_def} 
		%		\\
		%		\hline 
		%		$L_{\nabla{X_i}}(\mathbf{D})$  & Lipschitz constant of $\nabla_{X_i} f_i(\mathbf{D},\bullet)$ & $\mathbb{R}$	& 
		\\
		\hline 	
	\end{tabular}
	\caption{Table of notation (appendix)}
	\label{Table_of_Notation_appendix}
\end{table}	

\subsection{D$^4$L in vector-matrix form}
\label{notation_section_appendix}
We rewrite Algorithm \ref{alg1} in a more convenient vector-matrix form. To do so, we introduce the following notation. Recalling  
{the definitions of    
 $\widetilde{\mathbf{D}}^\nu_{(i)}$ [cf. \eqref{D_tilde_subproblem}],  $\mathbf{U}_{(i)}^\nu$ [cf. \eqref{U_update}], $\phi_i^\nu$ [cf. (\ref{D_weighted_avg})],  
 and $\boldsymbol{{\Theta}}^\nu_{(i)}$ [cf. \eqref{Theta_update}], define the corresponding aggregate quantities\begin{equation}
	\label{phi_notation}
\begin{aligned}
 & \widetilde{\mathbf{D}}^{\nu}\triangleq [\widetilde{\mathbf{D}}_{(1)}^{\nu\intercal},\widetilde{\mathbf{D}}_{(2)}^{\nu\intercal},\ldots,\widetilde{\mathbf{D}}_{(I)}^{\nu\intercal}]^\intercal,\\ 
&\mathbf{U}^\nu\triangleq [\mathbf{U}_{(1)}^{\nu\intercal}, \mathbf{U}_{(2)}^{\nu\intercal}, \ldots, \mathbf{U}_{(I)}^{\nu\intercal}]^\intercal,\\
&\boldsymbol{\phi}^\nu\triangleq [\phi_1^\nu,\phi_2^\nu,\ldots,\phi_I^\nu]^\intercal,\quad 
\boldsymbol{\Phi}^\nu \triangleq \mathrm{diag}\left(\boldsymbol{\phi}^\nu\right)\\
&\boldsymbol{ {\Theta}}^\nu\triangleq[\boldsymbol{ {\Theta}}_{(1)}^{\nu\intercal},\boldsymbol{ {\Theta}}_{(2)}^{\nu\intercal},\ldots,\boldsymbol{ {\Theta}}_{(I)}^{\nu\intercal}]^\intercal,
\\
& \mathbf{G}^\nu \triangleq\bigg[\nabla_D f_1(\mathbf{D}_{(1)}^\nu,\mathbf{X}^\nu_1)^\intercal, ~\nabla_D f_2(\mathbf{D}_{(2)}^\nu,\mathbf{X}^\nu_2)^\intercal, ~\ldots~, ~\nabla_D f_I(\mathbf{D}_{(I)}^\nu,\mathbf{X}^\nu_I)^\intercal\bigg]^\intercal,
\end{aligned}
\end{equation}}
where   $\mathrm{diag}(\mathbf{x})$ is a diagonal matrix whose diagonal entries are the  components of the vector $\mathbf{x}$. Combining the weights $a_{ij}^\nu$ and the variables $\phi_i^\nu$ in the update of $\mathbf{D}_{(i)}^\nu$ [cf. (\ref{D_weighted_avg})] in the single coefficient  
\begin{equation*}
	w_{ij}^\nu\triangleq  \dfrac{a_{ij}^\nu\phi_j^\nu}{\sum_k a_{ik}^\nu\phi_k^\nu},
\end{equation*}
we define the weight matrix $\mathbf{W}^\nu$, whose entries are $[\mathbf{W}^\nu]_{i,j} =w_{ij}^\nu$. Note that $\mathbf{W}^\nu$ has the same zero-pattern of $\mathbf{A}^\nu$, and the following   properties hold (the latter under Assumption \ref{A_matrix_Assumptions})\begin{equation}
\mathbf{W}^\nu=\big(\boldsymbol{\Phi}^{\nu+1}\big)^{-1}\mathbf{A}^\nu\boldsymbol{\Phi}^\nu,\quad\mathrm{and}\quad\mathbf{W}^\nu\mathbf{1}=\mathbf{1}.
\label{W_mat_A_relat}
\end{equation} 
Finally, we define the following ``augmented'' matrices
\begin{equation}
	\widehat{\mathbf{W}}^\nu\triangleq\mathbf{W}^\nu\otimes\mathbf{I}_M\quad\text{and}\quad  \widehat{\boldsymbol{\Phi}}^\nu\triangleq \boldsymbol{\Phi}^\nu\otimes\mathbf{I}_M,
\label{P_def_kroneck}
\end{equation}
where $\mathbf{I}_M$ is the $M$-by-$M$ identity matrix.

Using the above notation,   the main iterates of the D$^4$L Algorithm, i.e., \eqref{U_update},   \eqref{D_weighted_avg}, and \eqref{Theta_update}, can be rewritten in compact form as\begin{align}
&~~\quad\qquad\mathbf{U}^\nu = \mathbf{D}^\nu+\gamma^\nu (\widetilde{\mathbf{D}}^\nu-\mathbf{D}^\nu),
\label{U_update_stackingMat}
\\
&~\qquad\qquad\qquad\boldsymbol{\phi}^{\nu+1} = \mathbf{A}^\nu\boldsymbol{\phi}^\nu,
\label{phi_update_stacked}
\\
&\qquad\qquad\qquad\mathbf{D}^{\nu+1} = \widehat{\mathbf{W}}^\nu\mathbf{U}^\nu,
\label{D_weighted_avg_stackingMat}
\\
&\boldsymbol{ {\Theta}}^{\nu+1} = \widehat{\mathbf{W}}^\nu \boldsymbol{ {\Theta}}^{\nu}+\big(\widehat{\boldsymbol{\Phi}}^{\nu+1}\big)^{-1}\left(\mathbf{G}^{\nu+1}-\mathbf{G}^\nu\right).
\label{Theta_update_stackingMat}
\end{align}

{ Instrumental to the analysis of the consensus disagreement are the following weighted   average quantities:}
\begin{equation}
\label{average_notation}
\overline{\mathbf{D}}_{\boldsymbol{\phi}^\nu}\triangleq\dfrac{1}{I}\sum_{i=1}^I \phi_i^\nu \,\mathbf{D}_{(i)}^\nu,\quad  \overline{\mathbf{U}}_{\boldsymbol{\phi}^\nu} \triangleq\dfrac{1}{I}\sum_{i=1}^I \phi_i^\nu\,\mathbf{U}_{(i)}^\nu.
\end{equation}
Using \eqref{D_weighted_avg}, \eqref{average_notation} and the column stochasticity of $\mathbf{A}^\nu$ (cf. Assumption \ref{A_matrix_Assumptions}3), it is not difficult to check that 
\begin{align}
\overline{\mathbf{D}}_{\boldsymbol{\phi}^{\nu+1}} = \overline{\mathbf{U}}_{\boldsymbol{\phi}^\nu},
\label{D_bar_eq_U_bar}
\end{align}
which, together with \eqref{U_update}, leads to the following dynamics for $\overline{\mathbf{D}}_{\boldsymbol{\phi}^{\nu}}$:
\begin{align}
\overline{\mathbf{D}}_{\boldsymbol{\phi}^{\nu+1}} = \overline{\mathbf{D}}_{\boldsymbol{\phi}^\nu} +\frac{\gamma^\nu}{I} \sum_{i=1}^I &\phi_i^\nu\left(\widetilde{\mathbf{D}}^\nu_{(i)}-\mathbf{D}_{(i)}^\nu\right).
\label{D_bar_update}
\end{align}

\subsection{Proof of Theorem \ref{th:conver}}
\label{sub:Proof_th_conver}  
{We prove Theorem \ref{th:conver} in the following order (consistent with the statements in the theorem):
\begin{description}
	\item[Step 1--Asymptotic consensus \&  related properties:] We prove that consensus is asymptotically achieved, that is,  $\lim_{\nu\rightarrow\infty}e^\nu=0$ [statement (a)], along with some properties on related quantities, which will be used in the other steps--see Sec.~\ref{subsection-proof-step1};
	\item[Step 2--Boundedness of the iterates:] We prove that the sequence  $\big\{\big(\mathbf{D}^\nu,\mathbf{X}^\nu\big)\big\}_\nu$ generated by the algorithm   is bounded [statement (b-i)]--see Sec.~\ref{subsection-proof-step2};
	\item[Step 3--Decrease of $\{U(\overline{\mathbf{D}}^\nu,\mathbf{X}^\nu)\}_\nu$:] We study the properties    of  
$\{U(\overline{\mathbf{D}}^\nu,\mathbf{X}^\nu)\}_\nu$     [statement (b-ii)]--see Sec.~\ref{subsection-proof-step3}; \item[Step 4--Vanishing X-stationarity:] We prove that  $\lim_{\nu\rightarrow\infty}\Delta_X(\overline{\mathbf{D}}^\nu,\mathbf{X}^\nu)=0$ [statement (b-iii)]--see Sec.~\ref{subsection-proof-step4};
\item[Step 5--Vanishing liminf  D-stationarity:] We prove  $\liminf_{\nu\rightarrow\infty}\Delta_D(\overline{\mathbf{D}}^\nu,\mathbf{X}^\nu)=0$ [statement (b-iv)]--see Sec.~\ref{subsection-proof-step5};
\item[Step 6--Vanishing D-stationarity:] Finally, we prove  $\lim_{\nu\rightarrow\infty}\Delta_D(\overline{\mathbf{D}}^\nu,\mathbf{X}^\nu)=0$ [statement (b')]--see Sec.~\ref{subsection-proof-step6}.
\end{description}  }

\subsubsection{Step 1--Asymptotic consensus and  related  properties}
   \label{subsection-proof-step1}
   \noindent \textbf{1) Preliminaries:}
To analyze the dynamics of the consensus disagreement,  we first  introduce the following product matrices and their augmented counterparts: Given      $\mathbf{W}^\nu$ [cf. (\ref{W_mat_A_relat})], and $\nu, l\in \mathbb{N}_+$, with $\nu\geq l$, let  
\begin{align}
&~\mathbf{W}^{\nu:l}\triangleq\begin{cases}
\mathbf{W}^{\nu}\cdot\mathbf{W}^{\nu-1}\cdots\mathbf{W}^{l}, & \nu>l,\\
\mathbf{W}^{\nu}, & \nu=l,\\
\mathbf{0}_{I}, & \nu<l,
\end{cases}
\label{P_def} 
\quad \text{and}\quad \widehat{\mathbf{W}}^{\nu:l}\triangleq\mathbf{W}^{\nu:l}\otimes\mathbf{I}_M. %\quad\widehat{\mathbf{W}}^\nu\triangleq\mathbf{W}^\nu\otimes\mathbf{I}_M,
%\label{P_def_kroneck}
\end{align}
  
%where $\mathbf{W}^{\nu:\nu}\triangleq\mathbf{W}^\nu$. 
%We   define similar quantities as     \eqref{P_def} and \eqref{P_def_kroneck}  for the matrix $\mathbf{A}^\nu$, namely: $\mathbf{A}^{\nu:l}$, $\widehat{\mathbf{A}}^{\nu:l}$, and $\widehat{\mathbf{A}}^{\nu}$.
Define also the weight-average matrices 
\begin{equation}\label{J_def}
	\mathbf{J}_{\boldsymbol{\phi}^\nu}\triangleq\frac{1}{I}\mathbf{1}\boldsymbol{\phi}^{\nu \intercal},\quad\widehat{\mathbf{J}}_{\boldsymbol{\phi}^\nu}\triangleq\mathbf{J}_{\boldsymbol{\phi}^\nu} \otimes\mathbf{I}_M.
\end{equation}
 For notational simplicity, when $\boldsymbol{\phi}^\nu=\mathbf{1}$, we use $\mathbf{J}$ instead of $\mathbf{J}_{\boldsymbol{\phi}^\nu}$ and $\widehat{\mathbf{J}}\triangleq \mathbf{J}\otimes\mathbf{I}_M$. Using   \eqref{W_mat_A_relat} and \eqref{average_notation}, it is not difficult to check that the following   hold:
\begin{align}
&\quad\widehat{\mathbf{J}}_{\boldsymbol{\phi}^\nu}\mathbf{D}^\nu=\mathbf{1}\otimes\overline{\mathbf{D}}_{\boldsymbol{\phi}^\nu},
\label{D_avg_stackingMat}
\\
&\widehat{\mathbf{J}}_{\boldsymbol{\phi}^{\nu+1}} \widehat{\mathbf{W}}^{\nu:l}=\widehat{\mathbf{J}}\,\widehat{\boldsymbol{\Phi}}^l=\widehat{\mathbf{J}}_{\boldsymbol{\phi}^{l}}.
\label{colStochW_avgMat_eq}
%\\
%&\mathbf{W}^{\nu:l}=\big(\boldsymbol{\Phi}^{\nu+1}\big)^{-1}\mathbf{A}^{\nu:l}\,\boldsymbol{\Phi}^l.
%\label{W_mat_A_relat2}
\end{align}

The dynamics of the consensus disagreement  $e^\nu$ boils down to studying the decay  of $\|\mathbf{W}^{\nu:l}-\mathbf{J}_{\boldsymbol{\phi}^l}\|_2$ (this will be clear in the proof of Proposition \ref{consVar_props} below). The following lemma shows that $\mathbf{W}^{\nu:l}$ converges geometrically to $\mathbf{J}_{\boldsymbol{\phi}^l}$, as $\nu\to \infty$.

\begin{lemma}[\cite{book-CIME}-Lemma 4.13, Ch. 3.4.2.5]
\label{W_mat_err_decay} 
Let $\{\mathcal{G}^\nu\}_\nu$ be a sequence of digraphs satisfying Assumption \ref{B_strongly_connectivity}; let $\{\mathbf{A}^\nu\}_\nu$ be a sequence of matrices  satisfying Assumption \ref{A_matrix_Assumptions}; and let $\{\mathbf{W}^\nu\}_\nu$ be the sequence of matrices defined in \eqref{W_mat_A_relat}. Then, there holds
\begin{equation}
%\label{W_mat_err_decay_ineq}
\norm{\mathbf{W}^{\nu:l}-\mathbf{J}_{\boldsymbol{\phi}^l}}_2\leq c_W\,(\rho)^{\nu-l+1},\quad \forall \nu\geq l, \,\nu,l\in \mathbb{N}_+,
\nonumber
\end{equation}
where  $c_W>0$ is a (proper) constant, and  $\rho\in(0,1)$ is defined as \begin{equation}
\rho=\left(1-\tilde{\kappa}^{-(I-1)B}\right)^{\frac{1}{B(I-1)}}<1,\label{eq:rho}\end{equation} with 
 $\tilde{\kappa}=\kappa^{IB+1}/I$; and $\kappa$ is defined in Assumption \ref{A_matrix_Assumptions}.  
 
 Furthermore,   the sequence $\{\phi_i^\nu\}_\nu$   satisfies
\begin{equation}
\begin{aligned}
\underline{\epsilon}_\phi\triangleq\inf_{\nu\in\mathbb{N}_+}\left(\min_{1\leq i\leq I}\phi^\nu_i\right)\geq\kappa^{IB},\quad\text{and}\quad 
 \bar{\epsilon}_\phi\triangleq\sup_{\nu\in\mathbb{N}_+}\left(\max_{1\leq i\leq I}\phi^\nu_i\right) \leq I-(I-1)\kappa^{IB}.\end{aligned}
\label{delta_def}
\end{equation}
If all the matrices $\mathbf{A}^\nu$ are doubly-stochastic, then $\underline{\epsilon}_\phi=\bar{\epsilon}_\phi=1$.
\end{lemma}

\noindent \textbf{2) Proof of $\lim_{\nu\rightarrow\infty}e^\nu=0$.}
  {Using \eqref{eq:consensus_error}, we can write
\begin{equation}\begin{aligned}
e^\nu \overset{(a)}{\leq} K \norm{\mathbf{D}^\nu-\mathbf{1}\otimes\overline{\mathbf{D}}^\nu}_F 
 &\overset{(b)}{\leq} K \norm{\mathbf{D}^\nu-\mathbf{1}\otimes\overline{\mathbf{D}}_{\boldsymbol{\phi}^\nu}}_F+  K\,\sqrt{I}\norm{\overline{\mathbf{D}}^\nu-\overline{\mathbf{D}}_{\boldsymbol{\phi}^\nu}}_F\\
&\overset{(c)}{\leq}  2 K\norm{\mathbf{D}^\nu-\mathbf{1}\otimes\overline{\mathbf{D}}_{\boldsymbol{\phi}^\nu}}_F,
\label{Dcons_Dphi_consErr_Bound}
\end{aligned}\end{equation}
where (a) follows from the equivalence of norms; (b)  is due to the triangle inequality;  and in (c) we used $\sum_{i=1}^I a_i\leq \sqrt{I}||\mathbf{a}||$, with $\mathbf{a}=(a_i)_{i=1}^I\in\mathbb{R}^I$. }

{The following proposition concludes the proof of statement (a), proving that $\|\mathbf{D}^\nu-\mathbf{1}\otimes\overline{\mathbf{D}}_{\boldsymbol{\phi}^\nu}\|_F$ is square summable, along with some additional properties on related quantities.}

\begin{proposition}  \it
\label{consVar_props}
{In the above setting, there hold:\vspace{-0.1cm}
%\begin{enumerate}[label=(\alph*)]
%\item Asymptotic consensus:
\begin{align}
& \quad\lim_{\nu\rightarrow\infty} ~||\mathbf{D}^\nu-\mathbf{1}\otimes \overline{\mathbf{D}}_{\boldsymbol{\phi}^\nu}||_F=0,
\label{D_consensus}
\\
&\lim_{\nu\rightarrow\infty}\, \sum_{t=1}^\nu ~||\mathbf{D}^t- \mathbf{1}\otimes\overline{\mathbf{D}}_{\boldsymbol{\phi}^t}||_F^2<\infty;
\label{D_consError_summable}\\
%\end{align}
%and\vspace{-0.3cm}
%\begin{align}
& \lim_{\nu\rightarrow\infty}\, \sum_{t=1}^\nu ~||\mathbf{U}^t- \mathbf{1}\otimes\overline{\mathbf{U}}_{\boldsymbol{\phi}^t} ||_F^2<\infty.
\label{U_err_sqsum_bound}
%\\
%&~~ \textcolor{red}{\lim_{\nu\rightarrow\infty} \,\sum_{t=1}^\nu ~||\mathbf{D}^t-\mathbf{D}^{t-1}||_F^2<\infty.\textit{**It seems we never use this result**}}
%\label{deltaD_error_summable}
\end{align}}
\end{proposition}
\begin{proof}
To prove \eqref{D_consensus}, let us first expand $\mathbf{D}^\nu-\mathbf{1}\otimes \overline{\mathbf{D}}_{\boldsymbol{\phi}^\nu}$ as follows: for any $\nu\geq 1$, \vspace{-0.2cm}
\begin{equation}
\begin{aligned}
\mathbf{D}^\nu & \,\overset{\eqref{D_weighted_avg_stackingMat}}{=}\,\widehat{\mathbf{W}}^{\nu-1}\mathbf{U}^{\nu-1}=\widehat{\mathbf{W}}^{\nu-1}\mathbf{D}^{\nu-1}+ \widehat{\mathbf{W}}^{\nu-1}\left(\mathbf{U}^{\nu-1}-\mathbf{D}^{\nu-1}\right)
\\
& \, {=}\,\widehat{\mathbf{W}}^{\nu-1:0} \mathbf{D}^0+\sum_{t=0}^{\nu-1}\widehat{\mathbf{W}}^{\nu-1:t}  \left(\mathbf{U}^{t}-\mathbf{D}^{t}\right),
\end{aligned}
\label{rewrite_1}
\end{equation}
where the last equality   follows from  induction and  the definition of $\widehat{\mathbf{W}}^{\nu:l}$ [cf. \eqref{P_def}]. Similarly, we expand the subtrahend as  
\begin{equation}
\begin{aligned}
\mathbf{1}\otimes\overline{\mathbf{D}}_{\boldsymbol{\phi}^\nu}& \overset{\eqref{D_avg_stackingMat}}{=}\widehat{\mathbf{J}}_{\boldsymbol{\phi}^\nu}\mathbf{D}^\nu
\overset{\eqref{rewrite_1}}{=}\widehat{\mathbf{J}}_{\boldsymbol{\phi}^\nu} \left(\widehat{\mathbf{W}}^{\nu-1:0} \mathbf{D}^0+\sum_{t=0}^{\nu-1}\widehat{\mathbf{W}}^{\nu-1:t}  \left(\mathbf{U}^{t}-\mathbf{D}^{t}\right)\right)
\\
& \overset{\eqref{colStochW_avgMat_eq}}{=}\widehat{\mathbf{J}}\left(\mathbf{D}^0+ \sum_{t=0}^{\nu-1}\widehat{\boldsymbol{\Phi}}^t\left(\mathbf{U}^t-\mathbf{D}^t\right)\right).
\end{aligned}
\label{rewrite_2}
\end{equation}
 Subtracting \eqref{rewrite_2}  from \eqref{rewrite_1} and using \eqref{U_update_stackingMat}, yields\vspace{-0.2cm}
%\begin{equation}
%\begin{aligned}
%\mathbf{D}^\nu-\mathbf{1}\otimes\overline{\mathbf{D}}_{\boldsymbol{\phi}^\nu} 
%= \left(\widehat{\mathbf{W}}^{\nu-1:0}-\widehat{\mathbf{J}}\right) \mathbf{D}^0
%+\sum_{t=0}^{\nu-1}\gamma^{t}\left(\widehat{\mathbf{W}}^{\nu-1:t} -\widehat{\mathbf{J}}_{\phi^t} \right)\left(\widetilde{\mathbf{D}}^t-\mathbf{D}^{t}\right),
%\end{aligned}
%\label{rewrite_3}
%\end{equation}
%for all $\nu\geq 1$, where (a) is due to \eqref{U_update_stackingMat}; by \eqref{rewrite_3} and triangle inequality, we get
\begin{equation}
\begin{aligned}
\norm{\mathbf{D}^\nu-\mathbf{1}\otimes\overline{\mathbf{D}}_{\boldsymbol{\phi}^\nu} }_F\leq %&\norm{ \left(\widehat{\mathbf{W}}^{\nu-1:0}-\widehat{\mathbf{J}}\right) \mathbf{D}^0}_F
%+\sum_{t=0}^{\nu-1}\norm{\gamma^{t} \left(\widehat{\mathbf{W}}^{\nu-1:t}-\widehat{\mathbf{J}}_{\phi^t} \right)\left(\widetilde{\mathbf{D}}^t-\mathbf{D}^{t}\right)}_F
%\\
%\leq
&\norm{\widehat{\mathbf{W}}^{\nu-1:0}-\widehat{\mathbf{J}}}_2\norm{\mathbf{D}^0}_F
+\sum_{t=0}^{\nu-1}\gamma^{t}\norm{\widehat{\mathbf{W}}^{\nu-1:t}-\widehat{\mathbf{J}}_{\phi^t}}_2\norm{\widetilde{\mathbf{D}}^t-\mathbf{D}^{t}}_F
\\
\overset{(a)}{\leq}\, & c_1(\rho)^\nu+c_2~\sum_{t=0}^{\nu-1} ~\gamma^t~(\rho)^{\nu-t}
\stackrel[\nu\rightarrow\infty]{(b)}{\longrightarrow} 0,
\end{aligned}
\label{lim_diff_D}
\end{equation}
for some finite   constants $c_1, c_2>0$, where (a) is due to  Lemma \ref{W_mat_err_decay}  and the boundedness of  $\{||\widetilde{\mathbf{D}}^\nu-\mathbf{D}^{\nu}||\}_\nu$; and (b) follows from  Lemma \ref{sequence_convg_props}(a) in Appendix \ref{Prelim_results}. 

Let us now proceed to prove \eqref{D_consError_summable}. Using \eqref{lim_diff_D}, we have \vspace{-0.2cm}
\begin{equation}
\begin{aligned}
&\lim_{\nu\rightarrow\infty}\,\sum_{t=1}^\nu\norm{\mathbf{D}^t-\mathbf{1}\otimes \overline{\mathbf{D}}_{\boldsymbol{\phi}^t}}_F^2
\leq  \lim_{\nu\rightarrow\infty}\,\sum_{t=1}^\nu \left(c_1(\rho)^t+c_2~\sum_{l=0}^{t-1} ~\gamma^l~(\rho)^{t-l}\right)^2
\\
&\overset{(a)}{\leq}  \frac{2c_1^2}{1-(\rho)^2} + 2\,c_2^2 \lim_{\nu\rightarrow\infty}\,\sum_{t=1}^\nu\sum_{l=0}^{t-1}\sum_{k=0}^{t-1} \gamma^l\gamma^k(\rho)^{t-k}(\rho)^{t-l}
\\
&\overset{(b)}{\leq} \frac{2c_1^2}{1-(\rho)^2} + c_2^2 \lim_{\nu\rightarrow\infty}\,\sum_{t=1}^\nu\sum_{l=0}^{t-1} (\gamma^l)^2(\rho)^{t-l}\sum_{k=0}^{t-1}(\rho)^{t-k}
+ c_2^2 \lim_{\nu\rightarrow\infty}\,\sum_{t=1}^\nu\sum_{k=0}^{t-1} (\gamma^k)^2(\rho)^{t-k}\sum_{l=0}^{t-1}(\rho)^{t-l}
\\
&\leq \frac{2c_1^2}{1-(\rho)^2}+\frac{2c_2^2}{1-\rho}\lim_{\nu\rightarrow\infty}\,\sum_{t=1}^\nu\sum_{l=0}^{t-1} (\gamma^l)^2(\rho)^{t-l}
\,\overset{(c)}{<}  \infty,
\end{aligned}
%\label{cons_err_sq_summability_ineq}
\nonumber
\end{equation}
where in (a) and (b) we used  $(a+b)^2\leq 2(a^2+b^2)$ and $ab\leq (a^2+b^2)/2$, respectively, and (c) is due to  Lemma \ref{sequence_convg_props}(b) (cf. Appendix  \ref{Prelim_results}).

We prove now  \eqref{U_err_sqsum_bound}. Using  \eqref{average_notation} and  \eqref{U_update_stackingMat}, we get\vspace{-0.2cm}
\begin{equation}
\begin{aligned}
\norm{\mathbf{U}^t- \mathbf{1}\otimes\overline{\mathbf{U}}_{\boldsymbol{\phi}^t} }_F^2
&= \norm{\gamma^t\Big(\widetilde{\mathbf{D}}^t-\mathbf{1}\otimes\dfrac{1}{I}\sum_{i=1}^I\phi_i^t\, \widetilde{\mathbf{D}}_{(i)}^t\Big)+ \left(1-\gamma^t\right)\left(\mathbf{D}^t- \mathbf{1}\otimes\overline{\mathbf{D}}_{\boldsymbol{\phi}^t}\right)}_F^2
\\
& \leq 2 \left(\gamma^t\right)^2\Big\|\widetilde{\mathbf{D}}^t-\mathbf{1}\otimes\dfrac{1}{I}\sum_{i=1}^I\phi_i^t\,\widetilde{\mathbf{D}}_{(i)}^t\Big\|_F^2+2 \norm{\mathbf{D}^t- \mathbf{1}\otimes\overline{\mathbf{D}}_{\boldsymbol{\phi}^t}}_F^2,
\end{aligned}
\label{U_cons_error}
\end{equation}
where in the last inequality we used  Jensen's inequality and $(1-\gamma^t)\leq 1$. Therefore,
\begin{equation}
\begin{aligned}
& \lim_{\nu\rightarrow \infty}\,\sum_{t=1}^\nu \norm{\mathbf{U}^t- \mathbf{1}\otimes\overline{\mathbf{U}}_{\boldsymbol{\phi}^t} }_F^2
 \leq  \lim_{\nu\rightarrow \infty} 2\,c_3 \sum_{t=1}^\nu \left(\gamma^t\right)^2+\lim_{\nu\rightarrow \infty} 2\sum_{t=1}^\nu \norm{\mathbf{D}^t- \mathbf{1}\otimes\overline{\mathbf{D}}_{\boldsymbol{\phi}^t}}_F^2
\,\overset{(a)}{<}\infty,
\end{aligned}
\nonumber
\end{equation}
where $c_3$ is a positive finite constant, %such that $||\widetilde{\mathbf{D}}^t-\mathbf{1}\otimes\frac{1}{I}\sum_{i=1}^I\phi_i^\nu\widetilde{\mathbf{D}}_{(i)}^t||^2\leq c_3$ (such a constant exists since $\mathcal{D}$ and $\{\phi_i^\nu\}_\nu$ are bounded; see Assumption \ref{Problem_Assumptions}3 and Remark \ref{A_mat_err_decay}); 
and (a) follows from Assumption \ref{assumption:step-size} and \eqref{D_consError_summable}.
\begin{comment}
\textcolor{red}{We  finally prove \eqref{deltaD_error_summable}. Using the Jensen inequality and  \eqref{D_bar_update},   \eqref{deltaD_error_summable} can be bounded as 
\begin{equation}
\begin{aligned}
\lim_{\nu\rightarrow \infty}\,\sum_{t=1}^\nu \norm{\mathbf{D}^t-\mathbf{D}^{t-1}}_F^2
%& \overset{(a)}{\leq}  3 \sum_{t=1}^\nu \norm{\mathbf{D}^t-\mathbf{1}\otimes\overline{\mathbf{D}}_{\boldsymbol{\phi}^t}}_F^2
%+3\sum_{t=1}^\nu \norm{\mathbf{D}^{t-1}-\mathbf{1} \otimes\overline{\mathbf{D}}^{t-1}}_F^2
%\\
%&\quad+3I\sum_{t=1}^\nu \norm{\overline{\mathbf{D}}_{\boldsymbol{\phi}^t}-\overline{\mathbf{D}}^{t-1}}_F^2
%\\
%&\overset{(b)}{=}\, 
& \leq \,3\,\lim_{\nu\rightarrow \infty}\,\sum_{t=1}^\nu \norm{\mathbf{D}^t-\mathbf{1} \otimes\overline{\mathbf{D}}_{\boldsymbol{\phi}^t}}_F^2+3\,\lim_{\nu\rightarrow \infty} \,\sum_{t=1}^\nu \norm{\mathbf{D}^{t-1}-\mathbf{1} \otimes\overline{\mathbf{D}}_{\boldsymbol{\phi}^{t-1}}}_F^2
\\
&\quad +\frac{3}{I}\,\lim_{\nu\rightarrow \infty}\, \sum_{t=1}^\nu (\gamma^{t-1})^2\norm{\sum_{i=1}^I \phi_i^{t-1}(\widetilde{\mathbf{D}}_{(i)}^{t-1}-\mathbf{D}_{(i)}^{t-1})}_F^2 <\infty,
\end{aligned}
\nonumber
\end{equation}
%where (a) is due to Jensen inequality, and we used  \eqref{D_bar_update} in (b).
 where the last inequality follows   from  \eqref{D_consError_summable} and the boundedness of $\{\phi_i^\nu\}_\nu$ and $\{\widetilde{\mathbf{D}}_{(i)}^\nu-\mathbf{D}_{(i)}^\nu\}_\nu$.}
\end{comment}
\end{proof} 

 \begin{remark}[On $L_{\nabla{X_i}}(\overline{\mathbf{U}}_{\boldsymbol{\phi}^\nu})$ and $L_{\nabla{X_i}}(\mathbf{U}_{(i)}^{\nu})$]\label{remark_on_L_X}  Recall that $\nabla_{X_i} f_i(\overline{\mathbf{U}}_{\boldsymbol{\phi}^\nu} ,\bullet)$ is Lipschitz continuous on $\mathcal X_i$, with constant $L_{\nabla{X_i}}(\overline{\mathbf{U}}_{\boldsymbol{\phi}^\nu})$ (cf. Assumption \ref{Problem_Assumptions}2).
Since $||\mathbf{U}_{(i)}^\nu-\overline{\mathbf{U}}_{\boldsymbol{\phi}^\nu} ||_F\underset{\nu\rightarrow\infty}{\longrightarrow}0$ [cf. \eqref{U_err_sqsum_bound}, Proposition \ref{consVar_props}] and $L_{\nabla{X_i}}(\mathbf{D})$ is continuous, we have 
{\begin{equation}
\Big\lvert L_{\nabla{X_i}}(\overline{\mathbf{U}}_{\boldsymbol{\phi}^\nu} )-L_{\nabla{X_i}}(\mathbf{U}_{(i)}^{\nu})\Big\rvert\underset{\nu\rightarrow\infty}{\longrightarrow}0,\quad i=1,2,\ldots I.
\label{Lip_continuity}
\end{equation}}
%Finally, since $G$ is convex (and thus locally Lipschitz continuous) and  $\mathcal{D}$ is bounded, $G$ is  Lipschitz continuity  on $\mathcal{D}$, with some Lipschitz constant $L_G>0$.
\end{remark}

\subsubsection{Step 2--Boundedness of the iterates}\label{subsection-proof-step2} We show   that 
the sequence $\big\{\big(\mathbf{D}^\nu,\mathbf{X}^\nu\big)\big\}_\nu$  generated by the $D^4L$ Algorithm is bounded [statement (b-i)].   We prove the result 
only for  $\tilde{h}_i$ given by  \eqref{eq:htilde2}; the proof can be easily tailored to the other choice    of $\tilde{h}_i$. 
If the sets $\mathcal X_i$ are bounded [Assumption \ref{Problem_Assumptions}5(i)], the result follows readily. Therefore, we consider next the setting under \ref{Problem_Assumptions}5(ii). 

Throughout the proof, we will use the following properties of  $\tilde{f}_i$ and $\tilde{h}_i$.
\begin{remark}
		\label{surrogate_props}
		The surrogate functions $\tilde{f}_i$  and $\tilde{h}_i$  as in Assumption \ref{Ass-surrogates} have the  following properties: for all $i=1,2,\ldots,I$,
		\begin{enumerate}[label=(\alph*)]
			\item $\tilde{f}_i(\bullet;\mathbf{D},\mathbf{X}_i)$ is  strongly convex on $\mathcal{D}$, uniformly with respect to  $(\mathbf{D},\mathbf{X}_i)\in \mathcal{D}\times \mathcal{X}_i$, with constant  $\tau_{D,i}^\nu>0$; and $\nabla_{D}\tilde{f}_i(\mathbf{D};\mathbf{D},\mathbf{X}_i)=\nabla_D f_i(\mathbf{D},\mathbf{X}_i)$, for all $(\mathbf{D},\mathbf{X}_i)\in\mathcal{D}\times \mathcal{X}_i$.
			\item  $\tilde{h}_i(\bullet;\mathbf{D},\mathbf{X}_i)$ is   strongly convex on $\mathcal{X}_i$, uniformly with respect to  $(\mathbf{D},\mathbf{X}_i)\in \mathcal{D}\times \mathcal{X}_i$, with constant   $\tau_{X,i}^\nu>0$; and    $\nabla_{X_i}\tilde{h}_i(\mathbf{X}_i;\mathbf{D},\mathbf{X}_i)=\nabla_{X_i} f_i(\mathbf{D},\mathbf{X}_i)$, for all $(\mathbf{D},\mathbf{X}_i)\in\mathcal{D}\times \mathcal{X}_i$.
		\end{enumerate}\end{remark}

  By  the optimality of $\mathbf{X}_i^{\nu+1}$ in \eqref{eq:Xupdate},  there exist $\boldsymbol{\Xi}^0_i\in\partial_{X_i} g_i(\mathbf{X}^0_i)$ and $\boldsymbol{\Xi}^{\nu+1}_i\in\partial_{X_i} g_i(\mathbf{X}_i^{\nu+1})$ such that 
\begin{equation}\nonumber
\begin{aligned}
0&\leq \left\langle\nabla_{X_i} \tilde{h}_i(\mathbf{X}_i^{\nu+1};\mathbf{U}_{(i)}^\nu,\mathbf{X}_i^\nu)+\boldsymbol{\Xi}_i^{\nu+1},\mathbf{X}_i^0-\mathbf{X}_i^{\nu+1}\right\rangle \\ &= \left\langle\boldsymbol{\Xi}_i^{\nu+1}-\boldsymbol{\Xi}^0_i+\tau_{X,i}^\nu(\mathbf{X}_i^{\nu+1}-\mathbf{X}^0_i),\mathbf{X}^0_i-\mathbf{X}_i^{\nu+1}\right\rangle
\\
&\quad +\left\langle\nabla_{X_i} f_i(\mathbf{U}_{(i)}^\nu,\mathbf{X}_i^\nu)-\nabla_{X_i} f_i(\mathbf{U}_{(i)}^\nu,\mathbf{X}_i^0),\mathbf{X}^0_i-\mathbf{X}_i^{\nu+1}\right\rangle
\\
&\quad-\left\langle\tau_{X,i}^\nu(\mathbf{X}_i^\nu-\mathbf{X}^0_i),\mathbf{X}^0_i-\mathbf{X}_i^{\nu+1}\right\rangle
\\
&\quad+\left\langle\nabla_{X_i} f_i(\mathbf{U}_{(i)}^\nu,\mathbf{X}_i^0)+\boldsymbol{\Xi}^0_i,\mathbf{X}^0_i-\mathbf{X}_i^{\nu+1}\right\rangle.
\end{aligned}
\end{equation}
Using  Remark \ref{surrogate_props}(b) and the $\mu_i$-strongly convexity of $g_i$'s, we obtain
\begin{equation}
\label{X_iterate_bound0}
\begin{aligned}
 (\tau_{X,i}^\nu+\mu_i)\norm{\mathbf{X}_i^{\nu+1}-\mathbf{X}^0_i}_F^2
 &\leq\left\langle\tau_{X,i}^\nu\mathbf{X}_i^\nu-\nabla_{X_i} f_i(\mathbf{U}_{(i)}^\nu,\mathbf{X}_i^\nu),\mathbf{X}_i^{\nu+1}-\mathbf{X}^0_i\right\rangle
\\
&\quad -\left\langle\tau_{X,i}^\nu\mathbf{X}^0_i-\nabla_{X_i} f_i(\mathbf{U}_{(i)}^\nu,\mathbf{X}_i^0),\mathbf{X}_i^{\nu+1}-\mathbf{X}^0_i\right\rangle
\\
&\quad -\left\langle\underbrace{\nabla_{X_i} f_i(\mathbf{U}_{(i)}^\nu,\mathbf{X}_i^0)+\boldsymbol{\Xi}^0_i}_{\triangleq\mathbf{Z}_i^\nu\in\partial_{X_i}U(\mathbf{U}_{(i)}^\nu,\mathbf{X}^0_i)},\mathbf{X}_i^{\nu+1}-\mathbf{X}^0_i\right\rangle.
\end{aligned}
\end{equation}
Define $\Upsilon_i^\nu(\mathbf{X}_i)\triangleq \tau_{X,i}^\nu\mathbf{X}_i-\nabla_{X_i} f_i(\mathbf{U}_{(i)}^\nu,\mathbf{X}_i)$ and rewrite \eqref{X_iterate_bound0} as  
\begin{equation}
\label{X_iterate_bound1}
(\tau_{X,i}^\nu+\mu_i)\,||\mathbf{X}_i^{\nu+1}-\mathbf{X}^0_i||_F 
 \leq \norm{\Upsilon_i^\nu(\mathbf{X}^\nu_i)-\Upsilon_i^\nu(\mathbf{X}^0_i)}_F + \norm{\mathbf{Z}_i^\nu}_F.
\end{equation}

Since $\mathcal D$ is compact and $\mathbf{X}_i^0$ is given, we have $||\mathbf{Z}_i^\nu||_F\leq B_Z$,  for all $i$, $\nu\geq 1$, and  some finite  $B_Z>0$. Let us bound next   $||\Upsilon_i^\nu(\mathbf{X}^\nu_i)-\Upsilon_i^\nu(\mathbf{X}^0_i)||_F$. We write   
\begin{equation}\label{eq:Upsilon_bound}
\begin{aligned}
\norm{\Upsilon_i^\nu(\mathbf{X}^\nu_i)-\Upsilon_i^\nu(\mathbf{X}^0_i)}_F^2&=(\tau_{X,i}^\nu)^2\norm{\mathbf{X}^\nu_i-\mathbf{X}^0_i}_F^2
 +\norm{\nabla_{X_i} f_i(\mathbf{U}_{(i)}^\nu,\mathbf{X}_i^\nu)-\nabla_{X_i} f_i(\mathbf{U}_{(i)}^\nu,\mathbf{X}_i^0)}_F^2
\\
&\quad-2\tau_{X,i}^\nu\left\langle\nabla_{X_i} f_i(\mathbf{U}_{(i)}^\nu,\mathbf{X}_i^\nu)-\nabla_{X_i} f_i(\mathbf{U}_{(i)}^\nu,\mathbf{X}_i^0),\mathbf{X}^\nu_i-\mathbf{X}^0_i\right\rangle
\\
&\overset{(a)}{\leq}(\tau_{X,i}^\nu)^2\norm{\mathbf{X}^\nu_i-\mathbf{X}^0_i}_F^2
\\
&\quad+\left(1-\frac{2\tau_{X,i}^\nu}{L_{\nabla X_i}(\mathbf{U}_{(i)}^\nu)}\right)\norm{\nabla_{X_i} f_i(\mathbf{U}_{(i)}^\nu,\mathbf{X}_i^\nu)-\nabla_{X_i} f_i(\mathbf{U}_{(i)}^\nu,\mathbf{X}_i^0)}_F^2,
\end{aligned}  
\end{equation}
where in (a) we used  the $1/L_{\nabla X_i}(\mathbf{U}_{(i)}^\nu)$-co-coercivity of $\nabla _{X_i}f_i(\mathbf{U}_{(i)}^\nu,\bullet)$ [due to the convexity of $f_i(\mathbf{U}_{(i)}^\nu,\bullet)$ and the $L_{\nabla X_i}(\mathbf{U}_{(i)}^\nu)$-Lipschitianity of $\nabla _{X_i}f_i(\mathbf{U}_{(i)}^\nu,\bullet)$ \cite[Prop.12.60]{RockWets98}], i.e., \begin{equation}\nonumber\begin{array}{ll}\left\langle\nabla_{X_i} f_i(\mathbf{U}_{(i)}^\nu,\mathbf{X}_i)-\nabla_{X_i} f_i(\mathbf{U}_{(i)}^\nu,\mathbf{Y}_i),\mathbf{X}_i-\mathbf{Y}_i\right\rangle\geq \\ \hspace{2cm} \dfrac{1}{L_{\nabla X_i}(\mathbf{U}_{(i)}^\nu)}\cdot \norm{\nabla_{X_i} f_i(\mathbf{U}_{(i)}^\nu,\mathbf{X}_i)-\nabla_{X_i} f_i(\mathbf{U}_{(i)}^\nu,\mathbf{Y}_i)}_F^2,\qquad \forall \mathbf{X}_i, \mathbf{Y}_i\in \mathcal{X}_i.\end{array}\end{equation}
 Note that  $||\Upsilon_i^\nu(\mathbf{X}^\nu_i)-\Upsilon_i^\nu(\mathbf{X}^0_i)||_F\leq\tau_{X,i}^\nu||\mathbf{X}^\nu_i-\mathbf{X}^0_i||_F$  as long as $\tau_{X,i}^\nu\geq \frac{1}{2}L_{\nabla X_i}(\mathbf{U}_{(i)}^\nu)$, which is satisfied under \ref{freeVars_assumptions}1. Therefore, we can bound  \eqref{X_iterate_bound1} as
\begin{equation}
\label{X_iterate_bound2}
(\tau_{X,i}^\nu+\mu_i)||\mathbf{X}_i^{\nu+1}-\mathbf{X}^0_i||_F
\leq  \tau_{X,i}^\nu\norm{\mathbf{X}^\nu_i-\mathbf{X}^0_i}_F+ B_Z. 
\end{equation}
%Note that the subgradient set $\partial_{X_i}U(\mathbf{U}_{(i)}^\nu,\mathbf{X}^0_i)$ is uniformly bounded,  i.e., $||\mathbf{Z}_i^\nu||_F\leq B_Z$,  for all $i$, $\nu$ and  some finite  $B_Z>0$.

%To complete the proof, it is sufficient to show 
We can now prove   that, starting from $\mathbf{X}_i^0$, the iterates $\mathbf{X}_i^\nu$ stays in the ball  $\mathcal B_i(R_i,\mathbf{X}^0_i)\triangleq \{\mathbf{X}_i\in \mathbb{R}^{K\times n_i}\,:\,\|\mathbf{X}_i-\mathbf{X}^0_i\|_F\leq R_i\}$, for all $\nu\geq 1$, where $R_i\geq{B_Z}/{\mu_i}$. Let us prove it by induction. Evidently $\mathbf{X}^0_i\in\mathcal B_i(R_i,\mathbf{X}^0_i)$. Let $\mathbf{X}_i^\nu\in \mathcal B_i(R_i,\mathbf{X}^0_i)$; by \eqref{X_iterate_bound2}, we get
\begin{equation}\nonumber
\begin{aligned}
||\mathbf{X}_i^{\nu+1}-\mathbf{X}^0_i||_F& \,\leq\,\frac{\tau_{X,i}^\nu}{\tau_{X,i}^\nu+\mu_i}\,||\mathbf{X}_i^\nu-\mathbf{X}^0_i||_F
+ \frac{B_Z}{\tau_{X,i}^\nu+\mu_i} \,\leq\,R_i,
\end{aligned}
\end{equation}
where the second inequality is due to  $\mathbf{X}_i^\nu\in \mathcal B_i(R,\mathbf{X}^0_i)$ and $R\geq{B_Z}/{\mu}$. Hence  $\mathbf{X}_i^{\nu+1}\in \mathcal B_i(R_i,\mathbf{X}^0_i)$. Therefore,  $\mathbf{X}_i^{\nu}\in \mathcal B_i(R_i,\mathbf{X}^0_i)$, for all $\nu\geq 0$. Since  $\mathcal{D}$ is bounded (cf. Assumption \ref{Problem_Assumptions}3), it follows that $(\mathbf{D}_{(i)}^\nu,\mathbf{X}_i^\nu)\in \mathcal{D}\times \mathcal B_i(R_i,\mathbf{X}^0_i)$, for all $\nu\geq 0$. $\hfill \square$ 
%\vspace{-0.6cm}%\end{proof}

\begin{remark}[On the Lipschitz continuity of $\nabla f_i$'s]
\label{Lipschitz_continuity_remark} 
Since $f_i$ is $\mathcal{C}^2$, a direct consequence of  the boundedness of $\big\{\big(\mathbf{D}^\nu,\mathbf{X}^\nu\big)\big\}_\nu$ is that $\nabla f_i$  [the gradient of $f_i$ with respect to $(\mathbf{D},\mathbf{X}_i)$] is Lipschitz continuous  on  $\mathcal{D}\times \mathcal B_i(R_i,\mathbf{X}^0_i)$, that is,   there exists some positive finite constant   $L_{\nabla,i}$ such that 
\begin{equation}
\label{Lip_G}
||\nabla f_i(\mathbf{D},\mathbf{X}_i)-\nabla f_i(\mathbf{D}',\mathbf{X}'_i)||_F\leq L_{\nabla,i}\,||(\mathbf{D},\mathbf{X}_i)-(\mathbf{D}',\mathbf{X}_i')||_F,
\end{equation}
for all $(\mathbf{D},\mathbf{X}_i),(\mathbf{D}',\mathbf{X}_i')\in\mathcal{D}\times \mathcal B_i(R_i,\mathbf{X}^0_i)$, and    $i=1,2,\ldots,I$. We define $L_\nabla\triangleq \max_i L_{\nabla,i}$. 

The above result also  implies that  $\nabla_D F:\mathcal{D}\times(\mathcal{X}_1\times \cdots\times \mathcal{X}_I)\rightarrow\mathcal{D}$ [cf.  \eqref{eq:P1}] is Lipschitz continuous on $\mathcal{D}\times (\mathcal B_1(R_1,\mathbf{X}^0_1)\times \cdots \times \mathcal B_I(R_I,\mathbf{X}^0_I))$,  with constant ${L}_{\nabla_D}\triangleq I\cdot L_\nabla$. 
\end{remark}

{
\begin{remark}[On the Lipschitz continuity of $\nabla_D\tilde{f}_i$ and $\nabla_{X_i}\tilde{h}_i$]
	\label{Lipschitz_continuity_surrogates_remark} 
	  $\nabla_D \tilde{f}_i: \mathcal{D}\times \mathbb{R}^{K\times n_i}\times \mathbb{R}^{K\times n_i}\rightarrow \mathbb{R}^{M\times K}$  %[the gradient of $\tilde{f}_i$ with respect to first argument] 
	  and $\nabla_{X_i} \tilde{h}_i: \mathbb{R}^{K\times n_i}\times \mathcal{D}\times \mathbb{R}^{K\times n_i}\rightarrow \mathbb{R}^{K\times n_i}$  %[the gradient of $\tilde{h}_i$ with respect to first argument] 
	  are Lipschitz continuous  on  $\mathcal{D}\times \mathcal{D} \times \mathcal B_i(R_i,\mathbf{X}^0_i)$ and $\mathcal B_i(R_i,\mathbf{X}^0_i) \times \mathcal{D}\times \mathcal B_i(R_i,\mathbf{X}^0_i)$, respectively, with constants $\tilde{L}^D_{\nabla,i}$ and $\tilde{L}^X_{\nabla,i}$. Let us denote $\tilde{L}^D_{\nabla}\triangleq \max_{i}~\tilde{L}^D_{\nabla,i}$ and $\tilde{L}^X_{\nabla}\triangleq \max_{i}~\tilde{L}^X_{\nabla,i}$.
\end{remark}}

   \subsubsection{Step 3--Decrease of $\{U(\overline{\mathbf{D}}^\nu,\mathbf{X}^\nu)\}_\nu$}\label{subsection-proof-step3}
 We study here the  properties of $\{U(\overline{\mathbf{D}}^\nu,\mathbf{X}^\nu)\}_\nu$ showing, in  particular, that it is  convergent  [statement (b-ii)].   
 
 We begin with the following intermediate result. { \begin{proposition}\label{prop_descent}  
 Consider  the setting of Theorem \ref{th:conver}(b); there exist positive constants $s_X$, $\bar{\tau}_D$, $c_7, c_8,$ and a sufficiently large $\bar{\nu}\in \mathbb{N}_+$   such that   the following holds: for all $\nu\geq \bar{\nu}$, 
\begin{equation}
\begin{aligned}
U(\overline{\mathbf{D}}_{\boldsymbol{\phi}^{\nu+1}} ,\mathbf{X}^{\nu+1})\leq U(\overline{\mathbf{D}}_{\boldsymbol{\phi}^{\bar{\nu}}},\mathbf{X}^{\bar{\nu}})
  -\sum_{l=\bar{\nu}}^{\nu}Y^l+ \sum_{l=\bar{\nu}}^{\nu} W^l +E^{\nu,\bar{\nu}},
\end{aligned}
\label{descent_concise}\vspace{-0.3cm}
\end{equation}
where  
\begin{align}
\label{Y_sequence} 
& Y^l  \triangleq   s_X\left(||\mathbf{X}^{l+1}-\mathbf{X}^{l}||_F-\frac{Z^l}{2s_X} \right)^2+\dfrac{\bar{\tau}_D}{I}\,\gamma^l\left(|| \widetilde{\mathbf{D}}^l-\mathbf{D}^l||_F-\frac{I\bar{\epsilon}_\phi}{2\,\bar{\tau}_D}T^l\right)^2,\\
& Z_l  \triangleq c_7\sum_{t=l}^{\infty}\gamma^t(\rho)^{t-l}+L_X\|\mathbf{U}^l-\mathbf{1}\otimes \overline{\mathbf{U}}_{\boldsymbol{\phi}^l}\|_F\label{Zl},\\
& W^l  \triangleq   \frac{I\,\bar{\epsilon}_\phi^2}{4\bar{\tau}_D}\gamma^l\left(T^l\right)^2 +\frac{1}{4\,s_X} \left( Z^l\right)^2+L_G\gamma^l\norm{\mathbf{D}^l-\mathbf{1}\otimes\overline{\mathbf{D}}_{\boldsymbol{\phi}^l}}_F,\\
& E^{\nu,\bar{\nu}}  \triangleq \left[\frac{{c}_6}{1-\rho}((\rho)^{\bar{\nu}}-(\rho)^{\nu+1})+\frac{c_7 B_X}{(1-\rho)^2}\right]\cdot \left(\max_{ t\geq \bar{\nu}}\gamma^t\right),\\
&\{T^\nu\}_\nu\,\,\text{ is such that }\,\,\sum_{\nu=1}^\infty \left(T^\nu\right)^2<\infty, \label{eq:Tl_sequence_summability}
\end{align}
 with $\rho\in (0,1)$ and  $\bar{\epsilon}_\phi$   defined in (\ref{eq:rho}) and (\ref{delta_def}), respectively.  
\end{proposition}
\begin{proof}
See Appendix \ref{app_proof_prop_descent}	
\end{proof} }

Note that the sequences $\{Z^l\}_l$, $\{W^l\}_l$, and $\{E^{\nu,\bar{\nu}}\}_{\nu,\bar{\nu}}$ above satisfy
\begin{align}
 &~~ \lim_{\nu\rightarrow\infty}\sum_{l=\bar{\nu}}^{\nu}( Z^l)^2<\infty,\label{Z_l_sum_square}\smallskip
 \\
 &\quad \lim_{\nu\rightarrow\infty}\sum_{l=\bar{\nu}}^{\nu} W^l<\infty,\label{W_l_sum}
  \\
 & \lim_{\bar{\nu}\rightarrow\infty}\underset{{\triangleq E^{\infty,\bar{\nu}}<\infty}}{\underbrace{\left(\lim_{\nu\rightarrow\infty}E^{\nu,\bar{\nu}}\right)}}=0;\label{E_sequence}
\end{align}
where \eqref{Z_l_sum_square} follows from %the bound $(a+b)^2\leq 2(a^2+b^2)$,  
\eqref{U_err_sqsum_bound}  [cf. Prop.~\ref{consVar_props}], Assumption \ref{assumption:step-size}, and Lemma \ref{sequence_convg_props}(b) [cf. \eqref{bounded_sum3}] in Appendix \ref{subsec_sequences};   \eqref{W_l_sum} is a consequence of  $\lim_{\nu\rightarrow\infty}\sum_{l=\nu_1}^{\nu}\gamma^l(T^l)^2<\infty$ [due to (\ref{eq:Tl_sequence_summability})], \eqref{D_consError_summable}  [cf. Prop. \ref{consVar_props}], and \eqref{Z_l_sum_square}; and eq. \eqref{E_sequence} is proved by inspection.

It follows from \eqref{descent_concise}, \eqref{Z_l_sum_square}-\eqref{E_sequence} {and the lower-boundedness of $U$ (due to Assumption \ref{Problem_Assumptions}1)} that $\{U(\overline{\mathbf{D}}^{{\nu}},\mathbf{X}^{{\nu}})\}_{\nu}$ is convergent. Indeed, taking the limsup  of the LHS of \eqref{descent_concise} and using \eqref{W_l_sum} and \eqref{E_sequence}, we get
{\begin{equation}
\begin{aligned}
-\infty<\limsup_{\nu\rightarrow\infty} U(\overline{\mathbf{D}}_{\boldsymbol{\phi}^{\nu+1}} ,\mathbf{X}^{\nu+1})
&\leq U(\overline{\mathbf{D}}_{\boldsymbol{\phi}^{\bar{\nu}}},\mathbf{X}^{\bar{\nu}})+\sum_{l=\bar{\nu}}^\infty W^l+E^{\infty,\bar{\nu}}<\infty.
\end{aligned}
\nonumber
\end{equation}}
Taking now the liminf of the RHS of the above inequality with respect to $\bar{\nu}$  while  using  \eqref{E_sequence} and $\lim_{\bar{\nu}\rightarrow\infty}\sum_{l=\bar{\nu}}^\infty W^l=0$, yields
{\begin{align}
\nonumber
-\infty<\limsup_{\nu\rightarrow\infty}\, U(\overline{\mathbf{D}}_{\boldsymbol{\phi}^{\nu+1}} ,\mathbf{X}^{\nu+1})\leq \liminf_{\bar{\nu}\rightarrow\infty}\, U(\overline{\mathbf{D}}_{\boldsymbol{\phi}^{\bar{\nu}}},\mathbf{X}^{\bar{\nu}})<\infty,
%\label{sandwich}
\end{align}
which implies the convergence of $\{U(\overline{\mathbf{D}}_{\boldsymbol{\phi}^\nu} ,\mathbf{X}^\nu)\}_\nu$ to a finite value, and}
\begin{align}
&~\lim_{\nu\rightarrow\infty}\sum_{l=\bar{\nu}}^{\nu} \left(||\mathbf{X}^{l+1}-\mathbf{X}^{l}||_F-\frac{Z^l}{2s_X}\right)^2<\infty,
\label{finite_sq_sum_X}
\\
&\lim_{\nu\rightarrow\infty} \sum_{l=\bar{\nu}}^{\nu}\gamma^l\left(|| \widetilde{\mathbf{D}}^l-\mathbf{D}^l||_F-\frac{I\bar{\epsilon}_\phi}{2\bar{\tau}_D}T^l\right)^2 <\infty.
\label{finite_sq_sum_D}
\end{align}

  {Finally, we deduce that $\{U(\overline{\mathbf{D}}^\nu,\mathbf{X}^\nu)\}_\nu$  converges to the same limit point of $\{U(\overline{\mathbf{D}}_{\boldsymbol{\phi}^\nu} ,\mathbf{X}^\nu)\}_\nu$, due to i)  $\sqrt{I}\norm{\overline{\mathbf{D}}^\nu-\overline{\mathbf{D}}_{\boldsymbol{\phi}^\nu}}_F
\leq  \norm{\mathbf{D}^\nu-\mathbf{1}\otimes\overline{\mathbf{D}}_{\boldsymbol{\phi}^\nu}}_F\overset{\eqref{D_consensus}}{\underset{\nu\rightarrow\infty}{\longrightarrow}} 0$; ii) the   continuity of   $U$; and iii) the boundedness of    $\big\{\big(\mathbf{D}^\nu,\mathbf{X}^\nu\big)\big\}_\nu$     [cf. Sec. \ref{subsection-proof-step2}]. %(and so are    $\{\overline{\mathbf{D}}^\nu\}_\nu$ and $\{\overline{\mathbf{D}}_{\boldsymbol{\phi}^\nu}\}_\nu$),  imply convergence of   . 
%In fact, suppose this is not the case; then,  there would exist 
%Since otherwise, there would exist some  $\epsilon>0$ and a subsequence    $\{U(\overline{\mathbf{D}}^\nu,\mathbf{X}^\nu)\}_{\nu\in \mathcal{N}}$, for some infinite   index set $\mathcal{N}\subseteq\mathbb{N}_+$, such that $|U(\overline{\mathbf{D}}^\nu,\mathbf{X}^\nu)-U(\overline{\mathbf{D}}_{\boldsymbol{\phi}^\nu},\mathbf{X}^\nu)|>\epsilon$, for all  $\nu\in\mathcal{N}$. But this is in contradiction with continuity of  $U$ and $\|(\overline{\mathbf{D}}^\nu,\mathbf{X}^\nu)-(\overline{\mathbf{D}}_{\boldsymbol{\phi}^\nu} ,\mathbf{X}^\nu)\|_F\rightarrow 0$ as $\mathcal{N}\ni\nu\rightarrow\infty$.
}
This concludes the proof. $\hfill \square$% of the second claim in statement (b-ii) of Theorem \ref{th:conver}.

\subsubsection{Step 4--Vanishing X-stationarity}\label{subsection-proof-step4} Building on the results in the previous step, we prove here   $\lim_{\nu\rightarrow\infty}\Delta_X(\overline{\mathbf{D}}^\nu,\mathbf{X}^\nu)=0$ [statement (b-iii)]. 
  {For notational simplicity, we will use the shorthand  
 $\widehat{\mathbf{X}}_i^\nu\triangleq \widehat{\mathbf{X}}_i(\overline{\mathbf{D}}^\nu,\mathbf{X}^\nu)$ and  $\widehat{\mathbf{X}}^\nu\triangleq \widehat{\mathbf{X}}(\overline{\mathbf{D}}^\nu,\mathbf{X}^\nu)$, with $\widehat{\mathbf{X}}_i(\overline{\mathbf{D}}^\nu,\mathbf{X}^\nu)$ and $\widehat{\mathbf{X}}(\overline{\mathbf{D}}^\nu,\mathbf{X}^\nu)$ defined in (\ref{X_avg_def}). % $\widehat{\mathbf{X}}^\nu\triangleq \widehat{\mathbf{X}}(\overline{\mathbf{D}}^\nu,\mathbf{X}^\nu)$. %(cf. \eqref{X_avg_def}).
  }
  
  { Using  the equivalence of norms and the triangle inequality, we have  $$\Delta_{X}(\overline{\mathbf{D}}^{\nu},\mathbf{X}^{\nu})\leq K_{X}\|\widehat{\mathbf{X}}^{\nu}-\mathbf{X}^{\nu}\|_{F}\leq K_{X}\underset{\texttt{{term}\,I}}{\underbrace{\|\mathbf{X}^{\nu+1}-\mathbf{X}^{\nu}\|_{F}}}+K_{X}\underset{\texttt{{term}\,II}}{\underbrace{\|\mathbf{X}^{\nu+1}-\widehat{\mathbf{X}}^{\nu}\|_{F}}},$$ for some   $K_X>0$. The rest of the proof consists in showing that \texttt{term I} and  \texttt{term II} above are asymptotically vanishing. }
  
 \noindent $\bullet$ \textbf{On \texttt{term I}:} {\texttt{Term I} can be bounded as\begin{equation}
\frac{1}{2}||\mathbf{X}^{\nu+1}-\mathbf{X}^{\nu}||_F^2\leq \left(||\mathbf{X}^{\nu+1}-\mathbf{X}^{\nu}||_F-\frac{Z^\nu}{2s_X} \right)^2+ \left(\frac{Z^\nu}{2s_X} \right)^2,
\label{X_bound_squaresum_rate}
\end{equation}
for all $\nu\in\mathbb{N}_+$, where the inequality follows from    $\frac{1}{2}a^2\leq (a-b)^2+ b^2,$ with $a,b\in\mathbb{R}$. This,   together with  \eqref{Z_l_sum_square} and \eqref{finite_sq_sum_X},   yields 
\begin{equation}
\sum_{\nu=0}^{\infty}||\mathbf{X}^{\nu+1}-\mathbf{X}^{\nu}||_F^2<\infty\quad \Longrightarrow\quad 
%\label{X_consecutive_sqrSummable}
%\end{equation}
%implying that 
%\begin{equation}
\label{lim_DeltaX_0}
\lim_{\nu\rightarrow\infty} ||\mathbf{X}^{\nu+1}-\mathbf{X}^\nu||_F=0.
\end{equation}}

 \noindent $\bullet$ \textbf{On \texttt{term II}:} 
{ Invoking the   optimality of  $\widehat{\mathbf{X}}_i^\nu$ [cf. \eqref{X_avg_def}]  and $\mathbf{X}^{\nu+1}_i$ [cf. \eqref{eq:Xupdate}], yields
\begin{equation*}
\begin{aligned}
\left\langle\nabla_{X_i} f_i(\overline{\mathbf{D}}^\nu,\mathbf{X}_i^\nu)+\hat{\tau}_X(\widehat{\mathbf{X}}_i^\nu-\mathbf{X}_i^\nu),\mathbf{X}_i^{\nu+1}-\widehat{\mathbf{X}}_i^\nu\right\rangle+g_i(\mathbf{X}_i^{\nu+1})-g_i(\widehat{\mathbf{X}}_i^\nu)\geq 0,
\\
\left\langle\nabla_{X_i} \tilde{h}_i(\mathbf{X}_i^{\nu+1}; {\mathbf{U}}_{(i)}^\nu,\mathbf{X}_i^\nu),\widehat{\mathbf{X}}_i^\nu-\mathbf{X}_i^{\nu+1}\right\rangle+g_i(\widehat{\mathbf{X}}_i^\nu)-g_i(\mathbf{X}_i^{\nu+1})\geq 0.
\end{aligned}
%\label{Xavg_Xnew_bound0}
\end{equation*}}

{Summing the two inequalities above and using   Remark \ref{surrogate_props}(b), lead to
\begin{equation}
\begin{aligned}
\hat{\tau}_X||\mathbf{X}_i^{\nu+1}-\widehat{\mathbf{X}}_i^\nu||^2_F \leq & \left\langle \hat{\tau}_X\left(\mathbf{X}_i^{\nu+1}-\mathbf{X}_i^{\nu}\right)+\nabla_{X_i} f_i(\overline{\mathbf{D}}^\nu,\mathbf{X}_i^\nu)-\nabla_{X_i} f_i(\mathbf{U}_{(i)}^\nu,\mathbf{X}_i^\nu), \mathbf{X}_i^{\nu+1}-\widehat{\mathbf{X}}_i^\nu\right\rangle
\\
 &+\left\langle \nabla_{X_i} \tilde{h}_i(\mathbf{X}_i^{\nu}; {\mathbf{U}}_{(i)}^\nu,\mathbf{X}_i^\nu)-\nabla_{X_i} \tilde{h}_i(\mathbf{X}_i^{\nu+1}; {\mathbf{U}}_{(i)}^\nu,\mathbf{X}_i^\nu), \mathbf{X}_i^{\nu+1}-\widehat{\mathbf{X}}_i^\nu\right\rangle.
\end{aligned}
\label{Xavg_Xnew_bound1}
\end{equation}}

{Using the  $\tilde{L}_{\nabla,  i}^X$-Lipschitz continuity of $\nabla_{X_i} \tilde{h}_i$ [cf. Remark \ref{Lipschitz_continuity_surrogates_remark}]  and the $L_{\nabla,i}$-Lipschitz continuity of $\nabla f_i$, and \eqref{U_update}, it is not difficult to show that \eqref{Xavg_Xnew_bound1} implies
\begin{equation}
\begin{aligned}
&||\mathbf{X}_i^{\nu+1}-\widehat{\mathbf{X}}_i^\nu||_F
\\
%& \leq   \frac{\tilde{L}_{\nabla,  i}^X+\hat{\tau}_{X}}{\hat{\tau}_{X}}\norm{\mathbf{X}_i^{\nu+1}-\mathbf{X}_i^{\nu}}_F+\frac{L_{\nabla,i}}{\hat{\tau}_{X}}\norm{\overline{\mathbf{D}}^\nu-\mathbf{U}_{(i)}^\nu}_F
%\\
& \leq  % \frac{\tilde{L}_{\nabla,  i}^X+\hat{\tau}_{X}}{\hat{\tau}_{X}}\norm{\mathbf{X}_i^{\nu+1}-\mathbf{X}_i^{\nu}}_F+\frac{L_{\nabla,i}}{\hat{\tau}_{X}}\norm{\overline{\mathbf{D}}^\nu-\mathbf{D}_{(i)}^\nu}_F+\frac{L_{\nabla,i}}{\hat{\tau}_{X}}\norm{{\mathbf{D}}_{(i)}^\nu-\mathbf{U}_{(i)}^\nu}_F
%\\
%& = 
 \frac{\tilde{L}_{\nabla,  i}^X+\hat{\tau}_{X}}{\hat{\tau}_{X}}\norm{\mathbf{X}_i^{\nu+1}-\mathbf{X}_i^{\nu}}_F +\frac{L_{\nabla,i}}{\hat{\tau}_{X}}\norm{\overline{\mathbf{D}}^\nu-\mathbf{D}_{(i)}^\nu}_F+ \frac{L_{\nabla,i} } {\hat{\tau}_{X}}\gamma^\nu \norm{\widetilde{\mathbf{D}}_{(i)}^\nu-\mathbf{D}_{(i)}^\nu}_F,
\\
& \leq  \frac{\tilde{L}_{\nabla,  i}^X+\hat{\tau}_{X}}{\hat{\tau}_{X}}\norm{\mathbf{X}_i^{\nu+1}-\mathbf{X}_i^{\nu}}_F +\frac{L_{\nabla,i}}{\hat{\tau}_{X}}\norm{\overline{\mathbf{D}}^\nu-\mathbf{D}_{(i)}^\nu}_F+\frac{L_{\nabla,i}}{\hat{\tau}_{X}}\gamma^\nu\, B_D,
\end{aligned}
\label{Xavg_Xnew_bound2}
\end{equation} for some $B_D>0$, where in the last inequality we used the  boundedness of $\{||\widetilde{\mathbf{D}}^\nu-\mathbf{D}^\nu||_F\}_\nu$.
Eq. \eqref{Xavg_Xnew_bound2} together with \eqref{lim_DeltaX_0},  Theorem \ref{th:conver}(a), and $\gamma^\nu\overset{\nu\rightarrow\infty}{\longrightarrow}0$ (cf. Assumption \ref{assumption:step-size}), yield   
 \begin{equation}
  	\lim_{\nu\to \infty} \|\mathbf{X}_i^{\nu+1}-\widehat{\mathbf{X}}_i^\nu\|_F=0.\label{eq:X_minus_x_hat}
  \end{equation}}
  
%Finally, using  the equivalence of norms and the triangle inequality, we have  $$\Delta_X(\overline{\mathbf{D}}^\nu,\mathbf{X}^\nu)\leq K_X||\widehat{\mathbf{X}}^\nu-\mathbf{X}^\nu||_F \leq K_X||\mathbf{X}^{\nu+1}-\widehat{\mathbf{X}}^\nu ||_F+K_X||\mathbf{X}^{\nu+1}-\mathbf{X}^{\nu}||_F,$$ for some   $K_X>0$, which together with  \eqref{lim_DeltaX_0} and (\ref{eq:X_minus_x_hat}), imply  
%\begin{equation}
%\lim_{\nu\rightarrow\infty}\Delta_X({\mathbf{D}},\mathbf{X}) =0.
%\label{lim_DeltaX}
%\end{equation}
This concludes the proof of statement (b-iiii). $\hfill \square$

\subsubsection{Step 5--Vanishing liminf  D-stationarity}\label{subsection-proof-step5} {We prove  $\liminf_{\nu\rightarrow\infty}\Delta_D(\overline{\mathbf{D}}^\nu,\mathbf{X}^\nu)=0$ [statement (b-iv)] and $\{(\overline{\mathbf{D}}^\nu,\mathbf{X}^\nu)\}_\nu$   has at least one  limit point which is a stationary solution of    \ref{eq:P1}.
For notational simplicity,   we will use the shorthand   $\widehat{\mathbf{D}}^\nu\triangleq \widehat{\mathbf{D}}(\overline{\mathbf{D}}^\nu,\mathbf{X}^\nu)$, with $\widehat{\mathbf{D}}(\overline{\mathbf{D}}^\nu,\mathbf{X}^\nu)$ defined in   \eqref{D_avg_def}.}
  
{We begin bounding  $\Delta_D(\overline{\mathbf{D}}^\nu,\mathbf{X}^\nu)$ as \begin{equation*}
\Delta_{D}(\overline{\mathbf{D}}^{\nu},\mathbf{X}^{\nu})\leq K_{D}\|\widehat{\mathbf{D}}^{\nu}-\overline{\mathbf{D}}^{\nu}\|_{F}\leq K_{D}\underset{\texttt{{term}\,I}}{\underbrace{\|\widetilde{\mathbf{D}}_{(i)}^{\nu}-{\mathbf{D}}_{(i)}^{\nu}\|_{F}}}+K_{D}\underset{\texttt{{term}\,II}}{\underbrace{\|\widetilde{\mathbf{D}}_{(i)}^{\nu}-\widehat{\mathbf{D}}^{\nu}\|_{F}}}+K_{D}\|{\mathbf{D}}_{(i)}^{\nu}-\overline{\mathbf{D}}^{\nu}\|_{F},  \end{equation*} for some $K_D\!>\!0$. Note that $\|{\mathbf{D}}_{(i)}^{\nu}-\overline{\mathbf{D}}^{\nu}\|_{F}\!\to \!0$ [Theorem \ref{th:conver}(a)]. We show next that $\liminf_{\nu\to \infty}$ $\|\widetilde{\mathbf{D}}_{(i)}^{\nu}-{\mathbf{D}}_{(i)}^{\nu}\|_{F}=0$ and $\|\widetilde{\mathbf{D}}_{(i)}^{\nu}-\widehat{\mathbf{D}}^{\nu}\|_{F}\to 0$, which  proves $\liminf_{\nu\rightarrow\infty}$ $\Delta_D(\overline{\mathbf{D}}^\nu,\mathbf{X}^\nu)=0$.  } 
  
  \noindent $\bullet$ \textbf{On \texttt{term I}:}  {Similarly to the derivations of \eqref{X_bound_squaresum_rate},  there holds 
\begin{equation}
\frac{\gamma^\nu}{2}||\widetilde{\mathbf{D}}^\nu-\mathbf{D}^\nu||_F^2\leq \gamma^\nu\left(|| \widetilde{\mathbf{D}}^\nu-\mathbf{D}^\nu||_F-\frac{I\bar{\epsilon}_\phi}{2\,\bar{\tau}_D}T^\nu\right)^2+\left(\frac{I\bar{\epsilon}_\phi}{2\,\bar{\tau}_D}\right)^2\gamma^\nu(T^\nu)^2,
\end{equation}
for all $\nu\in\mathbb{N}_+$. By eq. \eqref{finite_sq_sum_D} and $\sum_{\nu=0}^\infty \left(T^\nu\right)^2<\infty$ [cf. \eqref{eq:Tl_sequence_summability}], we have
\begin{equation}
\sum_{\nu=0}^{\infty}\gamma^\nu||\widetilde{\mathbf{D}}^\nu-\mathbf{D}^\nu||_F^2<\infty\quad \overset{\text{Assumption \ref{assumption:step-size}}}{\Longrightarrow}\quad \liminf_{\nu\rightarrow\infty}||\widetilde{\mathbf{D}}^\nu-\mathbf{D}^\nu||_F=0.
\label{D_consecutive_sqrSummable}
\end{equation}
%implying    (by Assumption \ref{assumption:step-size})
%\begin{equation}
%\label{liminf_D_zero}
%\end{equation}
}
 
\noindent $\bullet$ \textbf{On \texttt{term II}:} { Using the optimality of $\widehat{\mathbf{D}}^\nu$ [cf. \eqref{D_avg_def}]   and $\widetilde{\mathbf{D}}_{(i)}^\nu$ [cf. \eqref{D_tilde_subproblem}], yields \begin{equation*}
\begin{aligned}
&\left\langle\nabla_D F(\overline{\mathbf{D}}^\nu,\mathbf{X}^\nu)+\hat{\tau}_D(\widehat{\mathbf{D}}^\nu-\overline{\mathbf{D}}^\nu),\widetilde{\mathbf{D}}_{(i)}^\nu-\widehat{\mathbf{D}}^\nu\right\rangle+G(\widetilde{\mathbf{D}}_{(i)}^\nu )-G(\widehat{\mathbf{D}}^\nu)\geq 0,
\\
&\left\langle\nabla_D \tilde{f}_i(\widetilde{\mathbf{D}}_{(i)}^\nu; {\mathbf{D}}_{(i)}^\nu,\mathbf{X}_i^\nu)+I\cdot
 {\boldsymbol{\Theta}}_{i}^\nu-\nabla_D f_i(\mathbf{D}^{\nu}_{(i)},\mathbf{X}^{\nu}_i),\widehat{\mathbf{D}}^\nu-\widetilde{\mathbf{D}}_{(i)}^\nu\right\rangle+G(\widehat{\mathbf{D}}^\nu)-G(\widetilde{\mathbf{D}}_{(i)}^\nu )\geq 0.
\end{aligned}
%\label{Davg_Dnew_bound0}
\end{equation*}
Summing the two inequalities above   and using Remark \ref{surrogate_props}(a),   yields
\begin{equation}
\begin{aligned}
&\hat{\tau}_D||\widetilde{\mathbf{D}}_{(i)}^\nu-\widehat{\mathbf{D}}^\nu ||^2_F
\\
& \leq %&\left\langle \hat{\tau}_D(\widetilde{\mathbf{D}}_{(i)}^\nu-\overline{\mathbf{D}}^\nu) + \nabla_D F(\overline{\mathbf{D}}^\nu,\mathbf{X}^\nu)-I \cdot\widetilde{\boldsymbol{\Theta}}_{i}^\nu ,\widetilde{\mathbf{D}}_{(i)}^\nu-\widehat{\mathbf{D}}^\nu\right\rangle
%\\
%&-\left\langle \nabla_D \tilde{f}_i(\widetilde{\mathbf{D}}_{(i)}^\nu; {\mathbf{D}}_{(i)}^\nu,\mathbf{X}_i^\nu)-\nabla_D \tilde{f}_i({\mathbf{D}}_{(i)}^\nu; {\mathbf{D}}_{(i)}^\nu,\mathbf{X}_i^\nu)  ,\widetilde{\mathbf{D}}_{(i)}^\nu-\widehat{\mathbf{D}}^\nu\right\rangle
%\\
%=
  \left\langle \hat{\tau}_D(\widetilde{\mathbf{D}}_{(i)}^\nu-\mathbf{D}_{(i)}^\nu) +\hat{\tau}_D(\mathbf{D}_{(i)}^\nu-\overline{\mathbf{D}}^\nu)  + \nabla_D F(\overline{\mathbf{D}}^\nu,\mathbf{X}^\nu)- \nabla_D F(\overline{\mathbf{D}}_{\boldsymbol{\phi}^{\nu}},\mathbf{X}^\nu) ,\widetilde{\mathbf{D}}_{(i)}^\nu-\widehat{\mathbf{D}}^\nu\right\rangle
\\
&\quad +\left\langle  \nabla_D F(\overline{\mathbf{D}}_{\boldsymbol{\phi}^{\nu}},\mathbf{X}^\nu)-I \cdot {\boldsymbol{\Theta}}_{i}^\nu ,\widetilde{\mathbf{D}}_{(i)}^\nu-\widehat{\mathbf{D}}^\nu\right\rangle
\\
&\quad -\left\langle \nabla_D \tilde{f}_i(\widetilde{\mathbf{D}}_{(i)}^\nu; {\mathbf{D}}_{(i)}^\nu,\mathbf{X}_i^\nu)-\nabla_D \tilde{f}_i({\mathbf{D}}_{(i)}^\nu; {\mathbf{D}}_{(i)}^\nu,\mathbf{X}_i^\nu)  ,\widetilde{\mathbf{D}}_{(i)}^\nu-\widehat{\mathbf{D}}^\nu\right\rangle.
\end{aligned}
\label{Davg_Dnew_bound1}
\end{equation}

Using the  $\tilde{L}_{\nabla,  i}^D$-Lipschitz continuity of $\nabla_D \tilde{f}_i$ [cf. Remark \ref{Lipschitz_continuity_surrogates_remark}] and the $L_{\nabla_D}$-Lipschitz continuity of $\nabla_D F$ [cf. Remark \ref{Lipschitz_continuity_remark}], it is not difficult to check that \eqref{Davg_Dnew_bound1} implies  \begin{equation}
\begin{aligned}
\left\|\widetilde{\mathbf{D}}_{(i)}^\nu-\widehat{\mathbf{D}}^\nu \right\|_F%\leq %& \frac{\tilde{L}_{\nabla,  i}^D+\hat{\tau}_D}{\hat{\tau}_D}\|\widetilde{\mathbf{D}}_{(i)}^\nu-\mathbf{D}_{(i)}^\nu\|_F+\frac{I}{\hat{\tau}_D}\left\|\widetilde{\boldsymbol{\Theta}}_{i}^\nu-\frac{1}{I}\nabla_D F(\overline{\mathbf{D}}_{\boldsymbol{\phi}^{\nu}},\mathbf{X}^\nu)\right\|_F
%\\
%& +\|\overline{\mathbf{D}}^\nu-{\mathbf{D}}_{(i)}^\nu\|_F+\frac{L_{\nabla_D}}{\hat{\tau}_D}\norm{\overline{\mathbf{D}}^\nu-\overline{\mathbf{D}}_{\boldsymbol{\phi}^{\nu}}}
%\\
\leq & \frac{\tilde{L}_{\nabla,  i}^D+\hat{\tau}_D}{\hat{\tau}_D}\left\|\widetilde{\mathbf{D}}_{(i)}^\nu-\mathbf{D}_{(i)}^\nu\right\|_F +\left\|\overline{\mathbf{D}}^\nu-{\mathbf{D}}_{(i)}^\nu\right\|_F+\frac{L_{\nabla_D}}{\hat{\tau}_D\sqrt{I}}\norm{{\mathbf{D}}^\nu-\mathbf{1}\otimes\overline{\mathbf{D}}_{\boldsymbol{\phi}^{\nu}}}\\& +\frac{I}{\hat{\tau}_D}\left\| {\boldsymbol{\Theta}}_{i}^\nu-\frac{1}{I}\nabla_D F(\overline{\mathbf{D}}_{\boldsymbol{\phi}^{\nu}},\mathbf{X}^\nu)\right\|_F,
\end{aligned}
\label{Davg_Dnew_bound2}
\end{equation}
for all $i=1,\ldots,I$. Since $\liminf_{\nu\to\infty} \|\widetilde{\mathbf{D}}_{(i)}^\nu-\mathbf{D}_{(i)}^\nu\|_F=0$ [cf. \eqref{D_consecutive_sqrSummable}], $\|\overline{\mathbf{D}}^\nu-{\mathbf{D}}_{(i)}^\nu\|_F\to 0$ [Theorem \ref{th:conver}(a)], and $\norm{{\mathbf{D}}^\nu-\mathbf{1}\otimes\overline{\mathbf{D}}_{\boldsymbol{\phi}^{\nu}}}\to 0$ [cf. (\ref{D_consensus})], to prove   $\liminf_{\nu\to \infty} \|\widetilde{\mathbf{D}}_{(i)}^\nu-\widehat{\mathbf{D}}^\nu \|_F=0$, it is sufficient to show that the last term on the RHS of the above inequality is asymptotically vanishing, which is done in the lemma below.

\begin{lemma}[Vanishing gradient-tracking error]
	\label{consVar_props2}
	In the setting above, there holds:
	\begin{equation}
\sum_{\nu=0}^{\infty}\Big\|\boldsymbol{ {\Theta}}^\nu-\mathbf{1}\otimes\frac{1}{I} \nabla_D F(\overline{\mathbf{D}}_{\boldsymbol{\phi}^{\nu}} ,\mathbf{X}^\nu)\Big\|_F^2<\infty.
\label{squaresummable_tracking_err}
	\end{equation}
	%implying
	%	\begin{equation}
	%	\lim_{\nu\rightarrow\infty}\,\Big\|\boldsymbol{\widetilde{\Theta}}^\nu-\mathbf{1}\otimes\frac{1}{I}\sum_{i=1}^I\nabla_D f_i(\overline{\mathbf{D}}_{\boldsymbol{\phi}^{\nu}},\mathbf{X}_i^\nu)\Big\|_F=0.
	%	\label{vanishing_tracking_err}
	%	\end{equation}
\end{lemma}
\begin{proof}
See Sec. \ref{proof_lemma_tracking_vanishing}.	
\end{proof}}

%\begin{equation}
%\lim\inf_{\nu\rightarrow\infty}||\boldsymbol{\Delta}_D^\nu||=0.
%\label{Delta_D_liminf}
%\end{equation}

By $\liminf_{\nu\rightarrow\infty}\Delta_D(\overline{\mathbf{D}}^\nu,\mathbf{X}^\nu)=0$, it follows that there exists an infinte subset $\mathcal{N}\subseteq \mathbb{N}_+$ such that     $\lim_{\mathcal{N}\ni\nu\rightarrow\infty}\Delta_D(\overline{\mathbf{D}}^\nu,\mathbf{X}^\nu)=0$. Since   $\{(\mathbf{D}^\nu,\mathbf{X}^\nu)\}_\nu$ is bounded [cf. Sec. \ref{subsection-proof-step2}], it has a convergent subsequence $\{(\mathbf{D}^\nu,\mathbf{X}^\nu)\}_{\nu\in \mathbb{N}'}$, with  ${\mathbb{N}'}\subseteq \mathbb{N}$; let $(\overline{\mathbf{D}}^\infty,\mathbf{X}^\infty)$  denote its limit point. Then, it must be $\lim_{\mathcal{N}'\ni\nu\rightarrow\infty}\Delta_D(\overline{\mathbf{D}}^\nu,\mathbf{X}^\nu)=0$. Combining this result with $\lim_{\nu\rightarrow\infty}\Delta_X(\overline{\mathbf{D}}^\nu,\mathbf{X}^\nu)=0$ (cf. Sec. \ref{subsection-proof-step4}), one can conclude  $\lim_{\mathcal{N}'\ni\nu\rightarrow\infty}\Delta^\nu=0$; hence,   $(\overline{\mathbf{D}}^\infty,\mathbf{X}^\infty)$ is a stationary solution of Problem~\ref{eq:P1}.

\subsubsection{Step 6--Vanishing D-stationarity}
\label{subsection-proof-step6} 
{Finally, we prove  $\lim_{\nu\rightarrow\infty}\Delta_D(\overline{\mathbf{D}}^\nu,\mathbf{X}^\nu)=0$ [statement (b')]. In view of the results already proved in Step 5, it is sufficient to show that $\limsup_{\nu\rightarrow\infty}|| \widetilde{\mathbf{D}}^\nu-\mathbf{D}^\nu||_F=0$. }
  
  \noindent \textbf{1) Preliminaries:}  We begin introducing the following preliminary results. \vspace{-0.2cm}
	{\begin{proposition}
	\it
	\label{sol_map_prop_part2} In the setting of Theorem  \ref{th:conver}(a),  the following hold for $\widetilde{\mathbf{D}}^{\nu}$ [cf. \eqref{D_tilde_subproblem}] and $\mathbf{X}^{\nu}$  [cf. \eqref{eq:Xupdate}]:\smallskip 	
\\
\noindent (a) There exists some constant $L_D>0$ and sequence  $\{\tilde{T}^\nu\}_\nu$, with   $\lim_{\nu\rightarrow\infty} \tilde{T}^\nu=0$, such that,  for any   $\nu_1, \nu_2\in \mathbb{N}_+$,
		%\label{D_operator_err}
		\begin{equation}
		\label{D_hat_Lipschitz_err}
		||\widetilde{\mathbf{D}}^{\nu_2}- \widetilde{\mathbf{D}}^{\nu_1}||_F\leq L_D\Big(||\overline{\mathbf{D}}_{\boldsymbol{\phi}^{\nu_2}}- \overline{\mathbf{D}}_{\boldsymbol{\phi}^{\nu_1}}||_F+||\mathbf{X}^{\nu_2}-\mathbf{X}^{\nu_1}||_F\Big)+\tilde{T}^{\nu_1}+\tilde{T}^{\nu_2};
		\end{equation}
	\noindent (b) There exist some constants $0<p_X<1$ and $q_X>0$, and a sufficiently large $\nu_X\in \mathbb{N}_+$ such that, for all $\nu\geq\nu_X$,
	\begin{equation}
	\begin{aligned}
	&||\mathbf{X}^{\nu+1}-\mathbf{X}^\nu||_F\leq p_X||\mathbf{X}^\nu-\mathbf{X}^{\nu-1}||_F+q_X||\mathbf{U}^\nu-\mathbf{U}^{\nu-1}||_F.
	\end{aligned}
	\label{DeltaX_bound}
	\end{equation}
\end{proposition}}

\begin{proof} See Appendix~\ref{app_proof_prop_delta_d_zero}\end{proof} 

 \noindent \textbf{2) Proof of $\limsup_{\nu\rightarrow\infty}|| \widetilde{\mathbf{D}}^\nu-\mathbf{D}^\nu||_F=0$.} 
For notational simplicity, let us define $\Delta\widetilde{\mathbf{D}}^\nu\triangleq \widetilde{\mathbf{D}}^\nu-\mathbf{D}^\nu$.
{Suppose by contradiction that $\limsup_{\nu\rightarrow\infty}~||\Delta\widetilde{\mathbf{D}}^\nu||_F>0$;  since $\liminf_{\nu\rightarrow\infty}~||\Delta\widetilde{\mathbf{D}}^\nu||_F$ $=0$ [cf. (\ref{D_consecutive_sqrSummable})],   there exists $\delta>0$ such that $||\Delta\widetilde{\mathbf{D}}^\nu||_F> 2\delta$ and $||\Delta\widetilde{\mathbf{D}}^{\nu'}||_F< \delta$ for infinitely many $\nu,\nu'\in\mathbb{N}_+$.} Therefore, one can find an infinite
 subset of indices, denoted by $\mathcal{K}$, having the following properties:
 for any $\nu\in\mathcal{K}$, there exists an index $i_\nu>\nu$ such that 
\begin{align}
\label{bounds_divergent_seq1}
&||\Delta\widetilde{\mathbf{D}}^\nu||_F<\delta, \qquad ||\Delta\widetilde{\mathbf{D}}^{i_\nu}||_F>2\delta,
\\
\label{bounds_divergent_seq2}
\delta\leq &||\Delta\widetilde{\mathbf{D}}^j||_F\leq 2\delta, \qquad\nu<j<i_\nu.
\end{align}
Let $\nu_2$ be a sufficiently large integer such that \eqref{descent_concise} holds   and %, satisfying $\nu_2\geq \nu_1$ ($\nu_1$ is defined in the the following of \eqref{one-step-descent})  and 
$T^{\nu}<\frac{2\bar{\tau}_D\,\delta}{I\bar{\epsilon}_\phi}$, for all $\nu\geq\nu_2$ [such $\nu_2$ exists, due to \eqref{eq:Tl_sequence_summability}]. Note that there exists a $\bar{\delta}>0$ such that $\delta-\frac{I\bar{\epsilon}_\phi}{2\bar{\tau}_D} T^\nu\geq \bar{\delta}$, for all  $\nu\geq\nu_2$. Choose  $\mathcal{K}\ni\nu\geq \nu_2$; using \eqref{descent_concise}, with $\nu=i_\nu$ and $\bar{\nu}=\nu+1$, yields\vspace{-0.2cm}
\begin{equation}
\begin{aligned}
U(\overline{\mathbf{D}}_{\boldsymbol{\phi}^{i_\nu+1}},\mathbf{X}^{i_\nu+1}) \leq & \,
U(\overline{\mathbf{D}}_{\boldsymbol{\phi}^{\nu+1}} ,\mathbf{X}^{\nu+1})-c_{8}\, \left(\bar{\delta}\right)^2\sum_{l=\nu+1}^{i_\nu}\gamma^l
 +\sum_{l=\nu+1}^{i_\nu} W^l+E^{i_\nu,\nu+1},
\end{aligned}
\label{ineq_4_contra}
\end{equation}
for some  finite constant $c_{8}>0$. Using the convergence of $\lbrace U(\overline{\mathbf{D}}_{\boldsymbol{\phi}^\nu} ,\mathbf{X}^\nu)\rbrace_\nu$, $\sum_{l=1}^{\infty} W^l<\infty$ [cf. \eqref{W_l_sum}], and $\lim_{\mathcal{K}\ni\nu\rightarrow\infty} E^{i_\nu,\nu+1}=0$ [cf. \eqref{E_sequence}], inequality \eqref{ineq_4_contra} implies
\begin{equation}
\label{lim_partial_sum_gamma_zero}
\lim_{\mathcal{K}\ni\nu\rightarrow\infty}\sum_{l=\nu+1}^{i_\nu}\gamma^l=0.
\end{equation}

We show next that \eqref{lim_partial_sum_gamma_zero} leads to a contradiction.

{ It follows from \eqref{bounds_divergent_seq1} and \eqref{bounds_divergent_seq2} that, for all $\mathcal{K}\ni\nu\geq \nu_2$,
\begin{equation}
\begin{aligned}
\delta< & ||\Delta\widetilde{\mathbf{D}}^{i_\nu}||_F -||\Delta\widetilde{\mathbf{D}}^\nu||_F
\\
\overset{(a)}{\leq} &|| \Delta\widetilde{\mathbf{D}}^{i_\nu}-\Delta\widetilde{\mathbf{D}}^\nu  ||_F
=||\widetilde{\mathbf{D}}^{i_\nu}-\mathbf{D}^{i_\nu}-\widetilde{\mathbf{D}}^\nu+\mathbf{D}^\nu||_F
\\
\overset{(b)}{\leq} &\|\widetilde{\mathbf{D}}^{i_\nu}-\widetilde{\mathbf{D}}^\nu\|+\|\mathbf{D}^\nu\pm\mathbf{1}\otimes\overline{\mathbf{D}}_{\boldsymbol{\phi}^{\nu}} \pm\mathbf{1}\otimes\overline{\mathbf{D}}_{\boldsymbol{\phi}^{i_\nu}}-\mathbf{D}^{i_\nu}\|_F
\\
 \overset{\eqref{D_hat_Lipschitz_err}}{\,\,\leq}  & \left(L_D+\sqrt{I}\right)\, ||\overline{\mathbf{D}}_{\boldsymbol{\phi}^{i_\nu}}- \overline{\mathbf{D}}_{\boldsymbol{\phi}^\nu} ||_F+ \widetilde{E}^{i_\nu,\nu},
\end{aligned}
\label{lower_bound_contra3}
\end{equation}
with 
\begin{equation}
\begin{aligned}
\widetilde{E}^{i_\nu,\nu}&\triangleq L_D||\mathbf{X}^{i_\nu}-\mathbf{X}^\nu||_F+||\mathbf{D}^\nu- \mathbf{1}\otimes\overline{\mathbf{D}}_{\boldsymbol{\phi}^\nu} ||_F +||\mathbf{D}^{i_\nu}-\mathbf{1}\otimes\overline{\mathbf{D}}_{\boldsymbol{\phi}^{i_\nu}}||_F+\tilde{T}^{\nu}+\tilde{T}^{i_\nu},
\end{aligned}
\nonumber
\end{equation}
where in (a) we  used the reverse triangle inequality (i.e. $||\mathbf{A}||_F-||\mathbf{B}||_F\leq ||\mathbf{A}-\mathbf{B}||_F,\forall \mathbf{A},\mathbf{B}\in\mathbb{R}^{MI\times K}$); and  in (b) we add/subtracted some dummy terms and used the triangle inequality.

We prove next that  $\lim_{\nu\rightarrow\infty}~\widetilde{E}^{i_\nu,\nu}=0$. Clearly,  if $\lim_{\nu\rightarrow\infty}||\mathbf{X}^{i_\nu}-\mathbf{X}^\nu||_F=0$, then \eqref{D_consensus} [cf.  Proposition~\ref{consVar_props}]  and $\tilde{T}^\nu\overset{\nu\rightarrow\infty}{\longrightarrow} 0$  [cf. Proposition \ref{sol_map_prop_part2}(a)] imply $\lim_{\nu\rightarrow\infty}~\widetilde{E}^{i_\nu,\nu}=0$. 
 It is then sufficient to show $\lim_{\nu\rightarrow\infty}||\mathbf{X}^{i_\nu}-\mathbf{X}^\nu||_F=0$.

First, we bound $||\mathbf{X}^{i_\nu}-\mathbf{X}^\nu||_F$ properly. %Under Assumption \ref{freeVars_assumptions}3 and Assumption \ref{Problem_Assumptions}5(ii),   \eqref{DeltaX_bound} [cf. Proposition \ref{sol_map_prop_part2} (b)] holds   with $p_X<1$, for all  $\mathcal{K}\ni\nu\geq \nu_2$; note that constructing a sequence $\{\tau_{X,i}^\nu\}_\nu$ satisfying   Assumption \ref{freeVars_assumptions}3 is feasible given that Assumption \ref{Problem_Assumptions}5(ii) holds (i.e. $g_i$'s are strongly convex).
} 
Summing \eqref{DeltaX_bound} from  $\nu>\nu_2$ to $i_\nu-1$, yields\vspace{-0.1cm}
\begin{equation}
\begin{aligned}
&\sum_{t=\nu}^{i_\nu-1}||\mathbf{X}^{t+1}-\mathbf{X}^t||_F\leq p_X\sum_{t=\nu}^{i_\nu-1}||\mathbf{X}^t-\mathbf{X}^{t-1}||_F+q_X\sum_{t=\nu}^{i_\nu-1}||\mathbf{U}^t-\mathbf{U}^{t-1}||_F,
\end{aligned}
\label{DeltaX_bound2}
\end{equation}
implying
\begin{equation}
||\mathbf{X}^{i_\nu}-\mathbf{X}^{\nu}||_F \leq \sum_{t=\nu}^{i_\nu-1}||\mathbf{X}^{t+1}-\mathbf{X}^t||_F \leq \frac{p_X}{1-p_X} ||\mathbf{X}^{\nu}-\mathbf{X}^{\nu-1}||_F
 +\frac{q_X}{1-p_X}\sum_{t=\nu}^{i_\nu-1}||\mathbf{U}^t-\mathbf{U}^{t-1}||_F.
\nonumber
\end{equation}
Since  $\lim_{\nu\rightarrow\infty}||\mathbf{X}^{\nu}-\mathbf{X}^{\nu-1}||_F=0$ [cf.~\eqref{lim_DeltaX_0}], it follows from the above inequality that $\lim_{\mathcal{K}\ni\nu\rightarrow\infty}||\mathbf{X}^{i_\nu}-\mathbf{X}^{\nu}||_F=0$, if $\lim_{\mathcal{K}\ni\nu\rightarrow\infty}\sum_{t=\nu}^{i_\nu-1}||\mathbf{U}^t-\mathbf{U}^{t-1}||_F=0$, which is proved next.
Rewrite first  $||\mathbf{U}^t-\mathbf{U}^{t-1}||_F$ as   
\begin{equation}
\begin{aligned}
&||\mathbf{U}^t-\mathbf{U}^{t-1}||_F \leq ||\mathbf{U}^t-\mathbf{1}\otimes\overline{\mathbf{U}}_{\boldsymbol{\phi}^t} ||_F+||\mathbf{U}^{t-1}-\mathbf{1}\otimes\overline{\mathbf{U}}_{\boldsymbol{\phi}^{t-1}}||_F+\sqrt{I}||\overline{\mathbf{U}}_{\boldsymbol{\phi}^t} -\overline{\mathbf{U}}_{\boldsymbol{\phi}^{t-1}}||_F
\\
&\qquad \overset{\eqref{D_bar_eq_U_bar}-\eqref{D_bar_update}}{=}||\mathbf{U}^t-\mathbf{1}\otimes\overline{\mathbf{U}}_{\boldsymbol{\phi}^t} ||_F+||\mathbf{U}^{t-1}-\mathbf{1}\otimes\overline{\mathbf{U}}_{\boldsymbol{\phi}^{t-1}}||_F
+\frac{\gamma^{t}}{\sqrt{I}}\left\Vert\sum_{i=1}^I\phi_i^t\left(\widetilde{\mathbf{D}}_{(i)}^{t}-\mathbf{D}_{(i)}^{t}\right)\right\Vert_F.
\end{aligned}
\nonumber
\end{equation}
 Since $||\sum_{i=1}^I\phi_i^t\,(\widetilde{\mathbf{D}}_{(i)}^{t}-\mathbf{D}_{(i)}^{t})||_F$ is bounded (due to $\phi_i^t\leq \bar{\epsilon}_\phi$ [cf. \eqref{delta_def}] and compactness of  $\mathcal{D}$) and $\lim_{\mathcal{K}\ni\nu\rightarrow\infty}\sum_{t=\nu}^{i_\nu}\gamma^t=0$ [cf. \eqref{lim_partial_sum_gamma_zero}], there holds  $\lim_{\mathcal{K}\ni\nu\rightarrow\infty}\sum_{t=\nu}^{i_\nu}{\gamma^{t}}||\sum_{i=1}^I\widetilde{\mathbf{D}}_{(i)}^{t}-\mathbf{D}_{(i)}^{t}||_F=0$. Therefore, to prove  $\lim_{\mathcal{K}\ni\nu\rightarrow\infty}\sum_{t=\nu}^{i_\nu-1}||\mathbf{U}^t-\mathbf{U}^{t-1}||_F=0$, it is sufficient to show that   $\lim_{\mathcal{K}\ni\nu\rightarrow\infty}\sum_{t=\nu}^{i_\nu}||\mathbf{U}^t-\mathbf{1}\otimes\overline{\mathbf{U}}_{\boldsymbol{\phi}^t} ||_F=0$ [which implies also $\lim_{\mathcal{K}\ni\nu\rightarrow\infty}\sum_{t=\nu}^{i_\nu}||\mathbf{U}^{t-1}-\mathbf{1}\otimes\overline{\mathbf{U}}_{\boldsymbol{\phi}^{t-1}}||_F=0$, due to $\lim_{\nu\rightarrow\infty}||\mathbf{U}^{\nu}-\mathbf{1}\otimes\overline{\mathbf{U}}_{\boldsymbol{\phi}^\nu} ||_F=0$, see \eqref{U_err_sqsum_bound}].
By \eqref{U_cons_error},  the boundedness of $\{\widetilde{\mathbf{D}}^\nu\}_\nu$, and \eqref{lim_partial_sum_gamma_zero}, it is sufficient  to show that $\lim_{\mathcal{K}\ni\nu\rightarrow\infty}\sum_{t=\nu}^{i_\nu}||\mathbf{D}^t-\mathbf{1}\otimes\overline{\mathbf{D}}_{\boldsymbol{\phi}^t}||_F=0$. We have
\begin{equation}
\begin{aligned}
&\lim_{\mathcal{K}\ni\nu\rightarrow\infty}\sum_{l=\nu}^{i_\nu}||\mathbf{D}^l-\mathbf{1}\otimes\overline{\mathbf{D}}_{\boldsymbol{\phi}^l} ||_F
\overset{\eqref{lim_diff_D}} {\leq} \lim_{\mathcal{K}\ni\nu\rightarrow\infty}\sum_{l=\nu}^{i_\nu}\left(
c_1(\rho)^l+c_2~\sum_{t=0}^{l-1} ~\gamma^t~(\rho)^{l-t}\right)
\\
&\quad= c_2~ \lim_{\mathcal{K}\ni\nu\rightarrow\infty}\sum_{l=\nu}^{i_\nu}
\sum_{t=0}^{l-1} ~\gamma^t~(\rho)^{l-t} =c_2~ \lim_{\mathcal{K}\ni\nu\rightarrow\infty}\sum_{t=0}^{i_\nu-1}
\sum_{l=\max(t+1,\nu)}^{i_\nu} ~\gamma^t~(\rho)^{l-t}
\\
&\quad= c_2~ \lim_{\mathcal{K}\ni\nu\rightarrow\infty}\sum_{t=0}^{\nu-1}
\sum_{l=\nu}^{i_\nu} ~\gamma^t~(\rho)^{l-t}+c_2~ \lim_{\mathcal{K}\ni\nu\rightarrow\infty}\sum_{t=\nu}^{i_\nu-1}
\sum_{l=t+1}^{i_\nu} ~\gamma^t~(\rho)^{l-t}
\\
&\quad\leq c_2~ \lim_{\mathcal{K}\ni\nu\rightarrow\infty}\sum_{t=0}^{\nu-1}
\sum_{l=\nu}^{i_\nu} ~\gamma^t~(\rho)^{l-t}+\frac{c_2}{1-\rho}~ \underbrace{\lim_{\mathcal{K}\ni\nu\rightarrow\infty}\sum_{t=\nu}^{i_\nu-1}~\gamma^t}_{\qquad\quad=0\quad\mathrm{by~}\eqref{lim_partial_sum_gamma_zero}} \\
&\quad\leq \frac{c_2}{1-\rho}~ \lim_{\mathcal{K}\ni\nu\rightarrow\infty}~\sum_{t=0}^{\nu-1}\gamma^t~(\rho)^{\nu-t}
\overset{\eqref{zero_lim}}{=}0.
\end{aligned}
\nonumber
\end{equation}
%where (a) follows from  \eqref{zero_lim} [cf. Lemma \ref{sequence_convg_props}]. 
This proves
$
\lim_{\nu\rightarrow\infty}||\mathbf{X}^{i_\nu}-\mathbf{X}^\nu||_F=0$ and thus   $\lim_{\nu\rightarrow\infty}~\widetilde{E}^{i_\nu,\nu}=0$.

%We can finally complete the proof showing that (\ref{lim_partial_sum_gamma_zero}) leads to a contradiction. % that Next we show that, $\lim_{\mathcal{K}\ni\nu\rightarrow\infty}\sum_{l=\nu+1}^{i_\nu}\gamma^l=0$ is in contradiction with $\limsup_{\nu\rightarrow\infty}~||\widetilde{\mathbf{D}}^\nu -\mathbf{D}^\nu||>0$.

We can now prove that  \eqref{lim_partial_sum_gamma_zero} leads to a contradiction. Since  $\widetilde{E}^{i_\nu,\nu}\overset{\nu\rightarrow\infty}{\longrightarrow}0$, there exists a sufficiently large integer $\nu_3\in \mathcal K$, such that  $\nu_3>\nu_2$   and  $\widetilde{E}^{i_\nu,\nu}<\delta$, for all  $\nu>\nu_3$. Define  $\delta'$ such that  $0<\delta'\leq \delta-\widetilde{E}^{i_\nu,\nu}$. Using  \eqref{D_bar_update} and $\mathbf{1}^\intercal\boldsymbol{\phi}^\nu=I$,
 \eqref{lower_bound_contra3} implies
\begin{equation}
\begin{aligned}
\dfrac{\delta'}{(L_D+1)\sqrt{I}}&< \sum_{t=\nu}^{i_\nu-1}\gamma^t~||\Delta\widetilde{\mathbf{D}}^t||_F
\overset{\eqref{bounds_divergent_seq1}-\eqref{bounds_divergent_seq2}}{\leq} 2\delta\,\sum_{t=\nu}^{i_\nu-1}\gamma^t,
\end{aligned}
\label{lower_bound_contra4}
\end{equation}
for all $\mathcal K\ni\nu>\nu_3$.  Equation \eqref{lower_bound_contra4}   contradicts  \eqref{lim_partial_sum_gamma_zero}. Hence, it must be    $\limsup_{\nu\rightarrow\infty}~||\Delta\widetilde{\mathbf{D}}^\nu||_F=0$, and thus  
%\begin{equation}
%\label{DeltaD_zero_lim}
$\lim_{\nu\rightarrow\infty}~||\widetilde{\mathbf{D}}^\nu-\mathbf{D}^\nu||_F=0.\hfill \square$
%\end{equation}
%{With a similar reasoning of \eqref{liminf_DeltaD}, \eqref{DeltaD_zero_lim} implies 
%\begin{equation}
%\lim_{\nu\rightarrow\infty}\Delta_D(\overline{\mathbf{D}}^\nu,\mathbf{X}^\nu)=0.
%\label{Delta_D_lim}
%\end{equation}}

\subsection{Proof of Theorem \ref{th:rate}}
\label{th:rate_proof}

\noindent {\textbf{(a) Rate of consensus error.} Fix $\theta\in (0,1)$. 
Combining   \eqref{Dcons_Dphi_consErr_Bound} and \eqref{lim_diff_D}, we obtain 
\begin{equation}
\begin{aligned}
e^\nu
\leq & %c_8~\sum_{t=0}^{\nu-1} ~\gamma^t~\rho^{\nu-t}=
  c_8~\sum_{l=1}^{\nu} ~\gamma^{\nu-l}~(\rho)^{l}
= c_8~\sum_{l=1}^{\floor{(1-\theta)\nu}} ~\gamma^{\nu-l}~(\rho)^{l}+c_8~\sum_{l=\floor{(1-\theta)\nu}+1}^{\nu} ~\gamma^{\nu-l}~(\rho)^{l}
\\
\leq & c_8 ~\gamma^{\nu-\floor{(1-\theta)\nu}}~\sum_{l=0}^{\floor{(1-\theta)\nu}}~(\rho)^{l}+c_8~\gamma^{0}~\sum_{l=\floor{(1-\theta)\nu}+1}^{\nu} ~(\rho)^{l},\\
\overset{(a)}{=} & c_8 ~\gamma^{\ceil{\theta\nu}}~\dfrac{1-(\rho)^{\floor{(1-\theta)\nu}+1}}{1-\rho}+c_8~\gamma^{0}~(\rho)^{\floor{(1-\theta)\nu}+1}\cdot\dfrac{1-(\rho)^{\ceil{\theta\nu}}}{1-\rho}\\
    %c_8 ~\gamma^{\ceil{(1-\theta)\nu}}\frac{1-\rho^{\floor{\theta\nu}+1}}{1-\rho}+c_8~\gamma^{0}\rho^{\floor{\theta\nu}+1}~\frac{1-\rho^{\ceil{(1-\theta)\nu}}}{1-\rho}
 {\leq} & c_9\left(\gamma^{\ceil{\theta\nu}}+  (\rho)^{(1-\theta)\nu}\right)\overset{(b)}{\leq} c_9\left(\gamma^{\ceil{\theta\nu}}+\left((\rho)^{\frac{1-\theta}{\theta}}\right)^{\ceil{\theta\nu}-1}\right)\overset{(c)}{=}\mathcal{O}\left(\gamma^{\ceil{\theta \nu}}\right),
\end{aligned}
\label{lim_diff_D5}
\end{equation}}
for some positive constants $c_8$ and $c_9$, where in (a) we used $\nu-\floor{(1-\theta)\nu}=\ceil{\theta\nu}$;   (b) follows from    $x\geq \ceil{x}-1,$ $x\in\mathbb{R}$; and in (c) we used   ${\tilde{\rho}}^{\nu}=o(\gamma^{\nu})$, for any $\tilde{\rho}\in(0,1)$. This proves statement (a). 

\noindent {\textbf{(b) Rate of optimization errors.}	In the following we will use the shorthand:  
 $\widehat{\mathbf{X}}_i^\nu\triangleq \widehat{\mathbf{X}}_i(\overline{\mathbf{D}}^\nu,\mathbf{X}^\nu)$ and  $\widehat{\mathbf{X}}^\nu\triangleq \widehat{\mathbf{X}}(\overline{\mathbf{D}}^\nu,\mathbf{X}^\nu)$, with $\widehat{\mathbf{X}}_i(\overline{\mathbf{D}}^\nu,\mathbf{X}^\nu)$ and $\widehat{\mathbf{X}}(\overline{\mathbf{D}}^\nu,\mathbf{X}^\nu)$ defined in (\ref{X_avg_def}); and  $\widehat{\mathbf{D}}^\nu\triangleq \widehat{\mathbf{D}}(\overline{\mathbf{D}}^\nu,\mathbf{X}^\nu)$, with $\widehat{\mathbf{D}}(\overline{\mathbf{D}}^\nu,\mathbf{X}^\nu)$  defined in   \eqref{D_avg_def}.}

\noindent { \textbf{1) Proof of \eqref{Rate_TeX}:}  We begin bounding  $\Delta_X(\overline{\mathbf{D}}^\nu,\mathbf{X}^\nu)$ as    \begin{equation}\frac{1}{2}(\Delta_X(\overline{\mathbf{D}}^\nu,\mathbf{X}^\nu))^2\overset{(a)}{\leq }\frac{ {K_1}}{2}||\widehat{\mathbf{X}}^\nu-\mathbf{X}^\nu||^2_F \leq {K_1}||\mathbf{X}^{\nu+1}-\widehat{\mathbf{X}}^\nu ||^2_F+{K}_1||\mathbf{X}^{\nu+1}-\mathbf{X}^{\nu}||^2_F,\label{eq:bound_delta_X_square}\end{equation} 
where (a) holds by equivalence of the norms with ${K}_1$ being  a proper positive constant. 

 By squaring both sides of \eqref{Xavg_Xnew_bound2} and using $\frac{1}{n}(\sum_{i=1}^n a_i)^2\leq \sum_{i=1}^n a_i^2,\forall a_i\in\mathbb{R}$ (by Jensen inequality), the first term on the RHS of (\ref{eq:bound_delta_X_square}) can be bounded as  
\begin{equation}
\begin{aligned}
&\frac{1}{3K_2}||\mathbf{X}_i^{\nu+1}-\widehat{\mathbf{X}}_i^\nu||^2_F
\leq\norm{\mathbf{X}_i^{\nu+1}-\mathbf{X}_i^{\nu}}^2_F +\norm{\overline{\mathbf{D}}^\nu-\mathbf{D}_{(i)}^\nu}^2_F+(\gamma^\nu)^2,
\end{aligned}
\label{Xavg_Xnew_bound2_recall}
\end{equation} 
for some positive constant $K_2>0$. Summing  \eqref{Xavg_Xnew_bound2_recall} over  $i=1,\ldots,I$, yields
\begin{equation}
\begin{aligned}
\frac{1}{3K_2}||\mathbf{X}^{\nu+1}-\widehat{\mathbf{X}}^\nu||^2_F
\leq & \norm{\mathbf{X}^{\nu+1}-\mathbf{X}^{\nu}}^2_F +\norm{\mathbf{D}^\nu-\mathbf{1}\otimes\overline{\mathbf{D}}^\nu}^2_F+I(\gamma^\nu)^2
\\
\leq & \norm{\mathbf{X}^{\nu+1}-\mathbf{X}^{\nu}}^2_F +4\norm{\mathbf{D}^\nu-\mathbf{1}\otimes\overline{\mathbf{D}}_{\boldsymbol{\phi}^\nu}}^2_F+I(\gamma^\nu)^2.
\end{aligned}
\label{Xavg_Xnew_bound2_recall2}
\end{equation} 
Finally, combining  (\ref{eq:bound_delta_X_square}) and \eqref{Xavg_Xnew_bound2_recall2}, yields
\begin{equation}
\begin{aligned}
\frac{1}{2{K_1}}(\Delta_X(\overline{\mathbf{D}}^\nu,\mathbf{X}^\nu))^2\leq (3K_2+1) 
\norm{\mathbf{X}^{\nu+1}-\mathbf{X}^{\nu}}^2_F +12K_1\norm{\mathbf{D}^\nu-\mathbf{1}\otimes\overline{\mathbf{D}}_{\boldsymbol{\phi}^\nu}}^2_F+3K_2I(\gamma^\nu)^2.
\end{aligned}
\label{Xavg_Xnew_bound2_recall3}
\end{equation} 
It follows from   \eqref{Xavg_Xnew_bound2_recall3} together with \eqref{lim_DeltaX_0}, \eqref{D_consError_summable}   and Assumption \ref{assumption:step-size} $$\sum_{\nu=0}^{\infty}(\Delta_X(\overline{\mathbf{D}}^\nu,\mathbf{X}^\nu))^2<\infty.$$ By the definition of $T_{X,\epsilon}$, there holds  
\begin{equation}
T_{X,\epsilon} \epsilon^2\leq \sum_{\nu=0}^{T_{X,\epsilon}}(\Delta_X(\overline{\mathbf{D}}^\nu,\mathbf{X}^\nu))^2<\infty,
\nonumber 
\end{equation}
which proves \eqref{Rate_TeX}.}
%\begin{equation}
%{T}^X_\epsilon= \mathcal{O}\left(\frac{1}{\epsilon^{2}}\right).
%\label{Tx_o}
%\end{equation}

\noindent { \textbf{2) Proof of \eqref{Rate_TeD}:} 
  Following the same approach as above, we can bound $\Delta_D(\overline{\mathbf{D}}^\nu,\mathbf{X}^\nu)$ as
 \begin{equation}
\frac{1}{3}(\Delta_D(\overline{\mathbf{D}}^\nu,\mathbf{X}^\nu))^2\leq \frac{K_3}{3}||\widehat{\mathbf{D}}^\nu-\overline{\mathbf{D}}^\nu||^2_F \leq K_3||\widehat{\mathbf{D}}^\nu-   \widetilde{\mathbf{D}}_{(i)}^\nu||^2_F+K_3||\widetilde{\mathbf{D}}_{(i)}^\nu-{\mathbf{D}}_{(i)}^\nu||^2_F+K_3||{\mathbf{D}}_{(i)}^\nu-\overline{\mathbf{D}}^\nu||^2_F, \label{eq:bound_delta_D_square}\end{equation} for some $K_3>0$. Using \eqref{Davg_Dnew_bound2}, the first term on the RHS of (\ref{eq:bound_delta_D_square}) can be bounded as
\begin{equation}
\begin{aligned}
  \frac{1}{4 K_4}\|\widetilde{\mathbf{D}}_{(i)}^\nu-\widehat{\mathbf{D}}^\nu \|^2_F \leq & \,\|\widetilde{\mathbf{D}}_{(i)}^\nu-\mathbf{D}_{(i)}^\nu\|^2_F+\left\|{\boldsymbol{\Theta}}_{(i)}^\nu-\frac{1}{I}\nabla_D F(\overline{\mathbf{D}}_{\boldsymbol{\phi}^{\nu}},\mathbf{X}^\nu)\right\|^2_F+\|\overline{\mathbf{D}}^\nu-{\mathbf{D}}_{(i)}^\nu\|_F^2
\\
&  +\norm{{\mathbf{D}}^\nu-\mathbf{1}\otimes\overline{\mathbf{D}}_{\boldsymbol{\phi}^{\nu}}}^2
\end{aligned}
\label{Davg_Dnew_bound2_recall}
\end{equation}
with some constant $K_4>0$. Summing \eqref{Davg_Dnew_bound2_recall} over   $i=1,\ldots,I$,   we get
\begin{equation}
\begin{aligned}
&\frac{1}{4 K_4}\|\widetilde{\mathbf{D}}^\nu-\mathbf{1}\otimes\widehat{\mathbf{D}}^\nu \|^2_F
 \\
\leq & \|\widetilde{\mathbf{D}}^\nu-\mathbf{D}^\nu\|^2_F+\left\| {\boldsymbol{\Theta}}^\nu-\mathbf{1}\otimes\frac{1}{I}\nabla_D F(\overline{\mathbf{D}}_{\boldsymbol{\phi}^{\nu}},\mathbf{X}^\nu)\right\|^2_F+\norm{\mathbf{D}^\nu-\mathbf{1}\otimes\overline{\mathbf{D}}^\nu}^2_F+I\norm{\mathbf{D}^\nu-\mathbf{1}\otimes\overline{\mathbf{D}}_{\boldsymbol{\phi}^\nu}}^2_F
\\
\leq & \|\widetilde{\mathbf{D}}^\nu-\mathbf{D}^\nu\|^2_F+\left\| {\boldsymbol{\Theta}}^\nu-\mathbf{1}\otimes\frac{1}{I}\nabla_D F(\overline{\mathbf{D}}_{\boldsymbol{\phi}^{\nu}},\mathbf{X}^\nu)\right\|^2_F+(4+I)\norm{\mathbf{D}^\nu-\mathbf{1}\otimes\overline{\mathbf{D}}_{\boldsymbol{\phi}^\nu}}^2_F.
\end{aligned}
\label{Davg_Dnew_bound2_recall2}
\end{equation}

Summing  (\ref{eq:bound_delta_D_square}) over $i=1,\ldots, I$ and using  (\ref{Davg_Dnew_bound2_recall2}), yields 
 
% ; summing this inequality over all $i=1,\ldots,I,$ and using \eqref{Dcons_Dphi_consErr_Bound_0}, we get $\frac{I}{3}(\Delta_D(\overline{\mathbf{D}}^\nu,\mathbf{X}^\nu))^2\leq K_3 ||   \widetilde{\mathbf{D}}^\nu-\mathbf{1}\otimes \widehat{\mathbf{D}}^\nu||^2_F+K_3||\widetilde{\mathbf{D}}^\nu-{\mathbf{D}}^\nu||^2_F+4 K_3||\mathbf{D}^\nu-\mathbf{1}\otimes\overline{\mathbf{D}}_{\boldsymbol{\phi}^\nu}||^2_F$ which together with \eqref{Davg_Dnew_bound2_recall2} gives 
\begin{equation}
\begin{aligned}
& \frac{I}{3 K_3}(\Delta_D(\overline{\mathbf{D}}^\nu,\mathbf{X}^\nu))^2
\\
\leq &  (4K_4+1) \|\widetilde{\mathbf{D}}^\nu-\mathbf{D}^\nu\|^2_F+4K_4\left\|{\boldsymbol{\Theta}}^\nu-\mathbf{1}\otimes\frac{1}{I}\nabla_D F(\overline{\mathbf{D}}_{\boldsymbol{\phi}^{\nu}},\mathbf{X}^\nu)\right\|^2_F
\\
& +4((4+I)K_4+1)\norm{\mathbf{D}^\nu-\mathbf{1}\otimes\overline{\mathbf{D}}_{\boldsymbol{\phi}^\nu}}^2_F.
\end{aligned}
\label{Davg_Dnew_bound2_recall2_2}
\end{equation}

It follows from   \eqref{Davg_Dnew_bound2_recall2_2}  together with    \eqref{D_consecutive_sqrSummable},   \eqref{squaresummable_tracking_err}, and \eqref{D_consError_summable} $$\sum_{\nu=0}^{\infty}\gamma^\nu(\Delta_D(\overline{\mathbf{D}}^\nu,\mathbf{X}^\nu))^2<\infty.$$ By definition of $T_{D,\epsilon}$ and non-increasing property of $\{\gamma^\nu\}_\nu$, we get
%Given   $\epsilon>0$, define    $T^D_\epsilon\triangleq\min\{\nu\in \mathbb{N}_+\,:\, ||\Delta\mathbf{D}^\nu||\leq \epsilon\}$, thus we have
\begin{equation}
\gamma^{ T_{D,\epsilon}}\,T_{D,\epsilon} \,\epsilon^2\leq \sum_{\nu=0}^{T_{D,\epsilon}}\gamma^\nu(\Delta_D(\overline{\mathbf{D}}^\nu,\mathbf{X}^\nu))^2<\infty.
\label{D_bound_squaresum_rate}
\end{equation}
Using $\gamma^\nu=K/\nu^p$, with some constant $K>0$ and $p\in(1/2,1)$,   \eqref{D_bound_squaresum_rate} provides the desired result as in \eqref{Rate_TeD}.
%\begin{equation}
%{T}^D_\epsilon= \mathcal{O}\left(\frac{1}{\epsilon^{2/(1-p)}}\right).
%\end{equation}}
\hfill$\square$}

\subsection{Miscellanea results}\label{Prelim_results}
This section contains some miscellanea  results used in the  proofs of Theorems \ref{th:conver} and \ref{th:rate}. 

\subsubsection{Sequence properties} \label{subsec_sequences} The following lemma summarizes some summability properties of suitably chosen sequences, which appear in some of the proofs.  
\begin{lemma}%[{\citep[Lemma 7]{Nedic2010}}]
\label{sequence_convg_props}
Given the sequences $\{a^\nu\}_\nu$ and  $\{b^\nu\}_\nu$, and a scalar   $\lambda\in [0,1)$,   the following  hold:
\begin{enumerate}[label=(\alph*)]
\item If $\lim_{\nu\rightarrow\infty}a^\nu=0$, then,
\begin{equation}
\label{zero_lim}
\lim_{\nu\rightarrow\infty}\sum_{t=1}^\nu a^t(\lambda)^{\nu-t}=0.
\end{equation}
\item If $\lim_{\nu\rightarrow\infty}\sum_{t=1}^\nu\left(a^t\right)^2<\infty$ and  $\lim_{\nu\rightarrow\infty}\sum_{t=1}^\nu\left(b^t\right)^2<\infty$, then
\begin{align}
&~~~\lim_{\nu\rightarrow\infty}\sum_{l=1}^\nu\sum_{t=1}^l a^t b^l(\lambda)^{l-t}<\infty,\label{bounded_sum1}
\\
&~~\lim_{\nu\rightarrow\infty}\sum_{l=1}^\nu\sum_{t=1}^l \left(a^t\right)^2(\lambda)^{l-t}<\infty,
\label{bounded_sum2}
\\
&\lim_{\nu,\nu'\rightarrow\infty}\sum_{l=1}^\nu\left(\sum_{t=l}^{\nu'}a^t(\lambda)^{t-l}\right)^2<\infty.
\label{bounded_sum3}
\end{align}
\end{enumerate}
\end{lemma}
\begin{proof}
For the proof of (a) and  \eqref{bounded_sum1}-\eqref{bounded_sum2} in (b), see \citep[Lemma 7]{Nedic2010}. We prove next \eqref{bounded_sum3}. Expand the LHS of \eqref{bounded_sum3} as:
\begin{equation}
\begin{aligned}
\lim_{\nu,\nu'\rightarrow\infty}&\sum_{l=1}^\nu\left(\sum_{t=l}^{\nu'}a^t(\lambda)^{t-l}\right)^2=\lim_{\nu,\nu'\rightarrow\infty}\sum_{l=1}^\nu\sum_{t=l}^{\nu'}\sum_{k=l}^{\nu'}a^ta^k(\lambda)^{t-l}(\lambda)^{k-l}
\\
&\leq \lim_{\nu,\nu'\rightarrow\infty} \sum_{l=1}^\nu\sum_{t=l}^{\nu'}\sum_{k=l}^{\nu'}\frac{(a^t)^2+(a^k)^2}{2}(\lambda)^{t-l}(\lambda)^{k-l}
=\lim_{\nu,\nu'\rightarrow\infty} \sum_{l=1}^\nu\sum_{t=l}^{\nu'}(a^t)^2(\lambda)^{t-l}\sum_{k=l}^{\nu'}(\lambda)^{k-l},
\end{aligned}
\nonumber
\end{equation}
where the inequality is due to $a\cdot b\leq (a^2+b^2)/2$. Using the bound on the sum of the geometric series, the above inequality yields
\begin{equation}
\begin{aligned}
\lim_{\nu,\nu'\rightarrow\infty}&\sum_{l=1}^\nu\left(\sum_{t=l}^{\nu'}a^t(\lambda)^{t-l}\right)^2
\leq \frac{1}{1-\lambda}\lim_{\nu,\nu'\rightarrow\infty}\sum_{l=1}^\nu\sum_{t=l}^{\nu'}(a^t)^2(\lambda)^{t-l}
  = \frac{1}{1-\lambda}\lim_{\nu'\rightarrow\infty} \sum_{t=1}^{\nu'}\lim_{\nu\rightarrow\infty}\sum_{l=1}^{\min(\nu,t)}(a^t)^2(\lambda)^{t-l}
\\
&\leq \frac{1}{1-\lambda}\lim_{\nu'\rightarrow\infty}\sum_{t=1}^{\nu'}(a^t)^2 \sum_{l=1}^t(\lambda)^{t-l}
\leq \frac{1}{(1-\lambda)^2}\lim_{\nu'\rightarrow\infty}\sum_{t=1}^{\nu'}(a^t)^2<\infty.
\end{aligned}
\nonumber
\end{equation}\vspace{-0.6cm}
\end{proof}

\smallskip
%\textcolor{red}{PROBABLY TO BE REMOVED
%\begin{lemma}[\citep{Bertsekas_NeuroDynamicProg}, Lemma 3.4, p. 121]
%\label{descent_2_conv_seq_lemma}
%Let $\{a^\nu\}_\nu$, $\{b^\nu\}_\nu$, $\{c^\nu\}_\nu$ be three sequences such that $b^\nu\geq 0$ for all $\nu$. Suppose that
%\begin{equation}
%a^{\nu+1}\leq a^\nu-b^\nu+c^\nu,\quad \forall \nu=0,1,\ldots
%\end{equation}
%and $\sum_{\nu=0}^\infty c^\nu<\infty$. Then either $a^\nu\rightarrow -\infty$ or else $\{a^\nu\}_\nu$ converges to a finite value and $\sum_{\nu=0}^\infty b^\nu<\infty$.
%\end{lemma}
%}

\subsubsection{On the properties of the best-response map}\label{subsec_best-response_properties} Some key properties of the best-response maps defined in   \eqref{D_tilde_subproblem} and \eqref{eq:Xupdate} are summarized and proved  next.
{\begin{proposition}
\it
\label{sol_map_prop} Let   {$\big\{\big(\mathbf{D}^\nu,\mathbf{X}^\nu\big)\big\}_\nu$}  be the sequence generated by the $D^4L$ Algorithm, in the setting of Theorem \ref{th:conver}(a). 
Given the solution maps defined in   \eqref{D_tilde_subproblem} and \eqref{eq:Xupdate}, the following hold:\smallskip 
 
\noindent (a) There exist some constants $s_D>0$ and $\eta>0$, and a sequence  $\{T^\nu\}_\nu$, with   $\sum_{\nu=1}^\infty \left(T^\nu\right)^2<\infty$, such that: for all $\nu\geq 1$,
\begin{equation}
\label{best-response_D_part_optimality} 
\begin{aligned}
&\Big\langle\nabla_D F(\overline{\mathbf{D}}_{\boldsymbol{\phi}^\nu} ,\mathbf{X}^{\nu}), \sum_{i=1}^I \phi_i^\nu\left(\widetilde{\mathbf{D}}^{\nu}_{(i)}-\mathbf{D}_{(i)}^{\nu}\right)\Big\rangle
+\sum_{i=1}^I\phi_i^\nu\left(G(\widetilde{\mathbf{D}}^{\nu}_{(i)})-G(\mathbf{D}_{(i)}^{\nu})\right)
\\
&\quad \leq -s_D\Big(||\widetilde{\mathbf{D}}^\nu-\mathbf{D}^\nu||_F-\frac{I\bar{\epsilon}_\phi}{2\,s_D}T^\nu\Big)^2
+\eta ~||\widetilde{\mathbf{D}}^{\nu}-\mathbf{D}^{\nu}||_F\sum_{t=1}^{\nu}(\rho)^{\nu-t}||\mathbf{X}^t-\mathbf{X}^{t-1}||_F+\frac{I^2\,\bar{\epsilon}_\phi^2}{4\,s_D}\left(T^\nu\right)^2,
\end{aligned}
\end{equation}
where $\rho\in (0,1)$ and $\bar{\epsilon}_\phi$ are defined in (\ref{eq:rho}) and (\ref{delta_def}), respectively;

\noindent (b) There exist finite constants $s_X>0$ and $L_X>0$, such that:  for all  $\nu\geq 1$,
\begin{equation}
\label{best-response_X_part_optimality}
\begin{aligned}
&\sum_{i=1}^I\left\langle\nabla_{X_i} f_i (\overline{\mathbf{D}}_{\boldsymbol{\phi}^{\nu+1}} ,\mathbf{X}_i^\nu), \mathbf{X}_i^{\nu+1}-\mathbf{X}_i^\nu \right\rangle
+\sum_{i=1}^I \left(g_i(\mathbf{X}_i^{\nu+1})-g_i(\mathbf{X}_i^\nu)\right)
\\
&\qquad\qquad\qquad \leq-\sum_{i=1}^I \tau_{X,i}^\nu ||\mathbf{X}_i^{\nu+1}-\mathbf{X}_i^\nu||_F^2+L_{X}||\mathbf{U}^\nu-\mathbf{1}\otimes\overline{\mathbf{U}}_{\boldsymbol{\phi}^\nu} ||_F~||\mathbf{X}^{\nu+1}-\mathbf{X}^\nu||_F.
\end{aligned}
\end{equation}
%
%\textcolor{red}{\noindent (c) There exists some constant $L_D>0$, such that,  for any integers pair $\nu_1$ and $\nu_2$,
%%\label{D_operator_err}
%\begin{equation}
%\label{D_hat_Lipschitz_err_old}
%||\widehat{\mathbf{D}}^{\nu_2}- \widehat{\mathbf{D}}^{\nu_1}||_F\leq L_D\Big(||\mathbf{D}^{\nu_2}- \mathbf{D}^{\nu_1}||_F+||\mathbf{X}^{\nu_2}-\mathbf{X}^{\nu_1}||_F+\sum_{i=1}^I|\tau_{D,i}^{\nu_2}-\tau_{D,i}^{\nu_1}|\Big);
%\end{equation}}
%
%{\noindent (c) There exists some constant $L_D>0$, such that,  for any integers pair $\nu_1$ and $\nu_2$,
%	%\label{D_operator_err}
%	\begin{equation}
%	\label{D_hat_Lipschitz_err}
%	||\widetilde{\mathbf{D}}^{\nu_2}- \widetilde{\mathbf{D}}^{\nu_1}||_F\leq L_D\Big(||\mathbf{D}^{\nu_2}- \mathbf{D}^{\nu_1}||_F+||\mathbf{X}^{\nu_2}-\mathbf{X}^{\nu_1}||_F+\sum_{i=1}^I|\tau_{D,i}^{\nu_2}-\tau_{D,i}^{\nu_1}|\Big);
%	\end{equation}}
%
%\noindent (d) There exist some constants $p_X, q_X>0$ and a sufficiently large $\nu_X$ such that, for all $\nu\geq\nu_X$,
%\begin{equation}
%\begin{aligned}
%&||\mathbf{X}^{\nu+1}-\mathbf{X}^\nu||_F\leq p_X||\mathbf{X}^\nu-\mathbf{X}^{\nu-1}||_F+q_X||\mathbf{U}^\nu-\mathbf{U}^{\nu-1}||_F.
%\end{aligned}
%\label{DeltaX_bound}
%\end{equation}
%$\quad$ Furthermore, if Assumption \ref{freeVars_assumptions}3 holds, then $p_X<1$.
\end{proposition}}

\begin{proof}
%[Proof of Proposition \ref{sol_map_prop}]
\textbf{(a)} It follows from the optimality of $\widetilde{\mathbf{D}}^\nu_{(i)}$ [cf.  \eqref{D_tilde_subproblem}] and convexity of $G$ that
{\begin{equation}
\label{eq:D_min_principle}
\begin{aligned}
&\left\langle \nabla_D \tilde{f}_i(\widetilde{\mathbf{D}}^\nu_{(i)};\mathbf{D}^\nu_{(i)},\mathbf{X}^\nu_i)+I\boldsymbol{\Theta}_{(i)}^\nu-\nabla_D f_i(\mathbf{D}_{(i)},\mathbf{X}_i),\mathbf{D}_{(i)}^\nu- \widetilde{\mathbf{D}}^\nu_{(i)} \right\rangle
+G(\mathbf{D}_{(i)}^\nu)-G(\widetilde{\mathbf{D}}_{(i)}^\nu)\geq 0.
\end{aligned}
\end{equation}
Adding and subtracting inside the first term $\sum_{j}\nabla_D f_j(\overline{\mathbf{D}}_{\boldsymbol{\phi}^\nu} ,\mathbf{X}_j^\nu)$   and using $\nabla_D \tilde{f}_i(\mathbf{D}_{(i)}^\nu;\mathbf{D}_{(i)}^\nu, \mathbf{X}_i^\nu)=\nabla_D {f}_i(\mathbf{D}_{(i)}^\nu, \mathbf{X}_i^\nu)$ [cf. Remark \ref{surrogate_props}], inequality \eqref{eq:D_min_principle} becomes
\begin{equation}\hspace{-0.3cm}
\begin{aligned}
&\left\langle\nabla_D \tilde{f}_i(\widetilde{\mathbf{D}}^\nu_{(i)};\mathbf{D}^\nu_{(i)},\mathbf{X}^\nu_i)-\nabla_D \tilde{f}_i(\mathbf{D}_{(i)}^\nu;\mathbf{D}_{(i)}^\nu, \mathbf{X}_i^\nu), \widetilde{\mathbf{D}}^\nu_{(i)}-\mathbf{D}_{(i)}^\nu \right\rangle
\\
&+\Big\langle I\cdot\boldsymbol{{\Theta}}^\nu_{(i)}-\sum_{j=1}^I\nabla_D f_j(\overline{\mathbf{D}}_{\boldsymbol{\phi}^\nu} ,\mathbf{X}_j^\nu), \widetilde{\mathbf{D}}^\nu_{(i)}-\mathbf{D}_{(i)}^\nu\Big\rangle
\\
&+\Big\langle\sum_{j=1}^I\nabla_D f_j(\overline{\mathbf{D}}_{\boldsymbol{\phi}^\nu} ,\mathbf{X}_j^\nu), \widetilde{\mathbf{D}}^\nu_{(i)}-\mathbf{D}_{(i)}^\nu \Big\rangle
+G(\widetilde{\mathbf{D}}_{(i)}^\nu)-G(\mathbf{D}_{(i)}^\nu)\leq 0.
\end{aligned}
\nonumber
\end{equation}}
Invoking the uniform strongly convexity of $\tilde{f}_i(\bullet;\mathbf{D}^\nu_{(i)},\mathbf{X}^\nu_i)$, the definition of ${\boldsymbol{\Theta}}^\nu_{(i)}$ in \eqref{Theta_update},  and recalling  that $\nabla_D F(\overline{\mathbf{D}}_{\boldsymbol{\phi}^\nu} ,\mathbf{X}^\nu)=\sum_{j} \nabla_D f_j(\overline{\mathbf{D}}_{\boldsymbol{\phi}^\nu} ,\mathbf{X}_j^\nu)$, we get
\begin{equation}
\begin{aligned}
&\left\langle \nabla_D F(\overline{\mathbf{D}}_{\boldsymbol{\phi}^\nu} ,\mathbf{X}^\nu), \widetilde{\mathbf{D}}^\nu_{(i)}-\mathbf{D}_{(i)}^\nu\right\rangle
+G(\widetilde{\mathbf{D}}^\nu_{(i)})-G(\mathbf{D}_{(i)}^\nu)
\\
&\qquad\qquad\leq - \tau_{D,i}^\nu~\left\|\widetilde{\mathbf{D}}^\nu_{(i)}-\mathbf{D}_{(i)}^\nu\right\|^2
+ I \left\|\boldsymbol{{\Theta}}^\nu_{(i)}-\frac{1}{I}\sum_{j=1}^I\nabla_D f_j(\overline{\mathbf{D}}_{\boldsymbol{\phi}^\nu} ,\mathbf{X}_j^\nu)\right\|_F \norm{\widetilde{\mathbf{D}}^\nu_{(i)}- \mathbf{D}_{(i)}^\nu}_F.
\end{aligned}
%\label{descentIneq_D_i}
\nonumber
\end{equation}
Multiplying both side of the above inequality by the  positive quantities $\phi_i^\nu$   and summing over $i=1,2,\ldots,I$ while using    $\phi_i^\nu\leq \bar{\epsilon}_\phi$ [cf. (\ref{delta_def})], yields 
\begin{equation}\hspace{-0.2cm}
\begin{aligned}
&\Big\langle\nabla_D F(\overline{\mathbf{D}}_{\boldsymbol{\phi}^\nu} ,\mathbf{X}^\nu),\sum_{i=1}^I  \phi_i^\nu\left(\widetilde{\mathbf{D}}^\nu_{(i)}-\mathbf{D}_{(i)}^\nu\right)\Big\rangle
+\sum_{i=1}^I\phi_i^\nu\left(G(\widetilde{\mathbf{D}}^\nu_{(i)})- G(\mathbf{D}^\nu_{(i)})\right)
\\
&\qquad\qquad \leq  -s_D||\widetilde{\mathbf{D}}^\nu- \mathbf{D}^\nu||^2
+I\bar{\epsilon}_\phi\underbrace{\left\|\boldsymbol{{\Theta}}^\nu-\mathbf{1}\otimes\frac{1}{I}\sum_{i=1}^I\nabla_D f_i(\overline{\mathbf{D}}_{\boldsymbol{\phi}^\nu} ,\mathbf{X}_i^\nu)\right\|_F}_{\text{gradient tracking error}}\norm{\widetilde{\mathbf{D}}^\nu- \mathbf{D}^\nu}_F,
\end{aligned}
\label{descentIneq_D_i_sum}
\end{equation}
where  $s_D$ is any positive constant such that $s_D\leq \min_{i,\nu}\phi_i^\nu\tau_{D,i}^\nu$ [note that such a constant exists because  $\phi_i^\nu\geq\underline{\epsilon}_\phi$, with $\underline{\epsilon}_\phi>0$ defined in (\ref{delta_def}), and all $\tau_{D,i}^\nu$ are uniformly bounded away from zero--see Assumption \ref{freeVars_assumptions}1].

Now let us bound the \emph{gradient tracking error} term in \eqref{descentIneq_D_i_sum}. Using  \eqref{Theta_update_stackingMat} recursively,   $\boldsymbol{{\Theta}}^{\nu}$ can be rewritten as
\begin{equation}
\begin{aligned}
\boldsymbol{{\Theta}}^{\nu}
&=\widehat{\mathbf{W}}^{\nu-1:0}\,\boldsymbol{{\Theta}}^{0}+\sum_{t=1}^{\nu-1}\widehat{\mathbf{W}}^{\nu-1:t}\left(\widehat{\boldsymbol{\Phi}}^t\right)^{-1}\left(\mathbf{G}^t-\mathbf{G}^{t-1}\right)
+\left(\widehat{\boldsymbol{\Phi}}^\nu\right)^{-1}\left(\mathbf{G}^{\nu}-\mathbf{G}^{\nu-1}\right).\end{aligned}
\label{theta_expansion}
\end{equation}
Using  the definition of $\mathbf{G}^\nu$ [cf.~\eqref{phi_notation}] and  $\widehat{\mathbf{J}}$ [cf. \eqref{J_def}], write
\begin{equation}
%\label{avg_eq_expr1}
\mathbf{1}\otimes\frac{1}{I}\sum_{i=1}^I\nabla_D f_i(\mathbf{D}_{(i)}^\nu,\mathbf{X}_i^\nu)=\widehat{\mathbf{J}}\mathbf{G}^\nu = \widehat{\mathbf{J}}\mathbf{G}^0+\sum_{t=1}^{\nu}\widehat{\mathbf{J}}\left(\mathbf{G}^t-\mathbf{G}^{t-1}\right),\nonumber 
\end{equation}
%and RHS of above equality can be rewritten in form of a telescopic series
%\begin{equation}
%\label{avg_eq_expr2}
%\widehat{\mathbf{J}}\mathbf{G}^\nu=\widehat{\mathbf{J}}\mathbf{G}^0+\sum_{t=1}^{\nu}\widehat{\mathbf{J}}\left(\mathbf{G}^t-\mathbf{G}^{t-1}\right).
%\end{equation}
which, using  ${\boldsymbol{\Theta}}^0=\mathbf{G}^0$, leads to the following expansion for $\mathbf{1}\otimes\frac{1}{I}\sum_{i=1}^I\nabla_D f_i(\overline{\mathbf{D}}_{\boldsymbol{\phi}^\nu} ,\mathbf{X}_i^\nu)$: \begin{equation}
\begin{aligned}
\mathbf{1}\otimes\frac{1}{I}\sum_{i=1}^I\nabla_D f_i(\overline{\mathbf{D}}_{\boldsymbol{\phi}^\nu} ,\mathbf{X}_i^\nu)
&=\widehat{\mathbf{J}}\,\boldsymbol{\widetilde{\Theta}}^{0}+\sum_{t=1}^{\nu}\widehat{\mathbf{J}}\left(\mathbf{G}^t-\mathbf{G}^{t-1}\right)
\\
&\quad+\mathbf{1}\otimes\frac{1}{I}\sum_{i=1}^I\left(\nabla_D f_i(\overline{\mathbf{D}}_{\boldsymbol{\phi}^\nu} ,\mathbf{X}_i^\nu)-\nabla_D f_i(\mathbf{D}_{(i)}^\nu,\mathbf{X}_i^\nu)\right).
\end{aligned}
\label{nablaD_expansion}
\end{equation}
Using  \eqref{theta_expansion} and \eqref{nablaD_expansion}, the  \emph{gradient tracking error} term in \eqref{descentIneq_D_i_sum} can be upper bounded as   
%\begin{equation}
%\begin{aligned}
%&\boldsymbol{\widetilde{\Theta}}^\nu-\mathbf{1}\otimes\frac{1}{I}\sum_{i=1}^I\nabla_D f_i(\overline{\mathbf{D}}_{\boldsymbol{\phi}^\nu} ,\mathbf{X}_i^\nu)
%\\
%&=\left(\widehat{\mathbf{W}}^{\nu-1:0}-\widehat{\mathbf{J}}\right)\boldsymbol{\widetilde{\Theta}}^{0}
%+\sum_{t=1}^{\nu-1}\left(\widehat{\mathbf{W}}^{\nu-1:t}\left(\widehat{\boldsymbol{\Phi}}^t\right)^{-1}- \widehat{\mathbf{J}}\right)\left(\mathbf{G}^t-\mathbf{G}^{t-1}\right)
%\\
%&\quad+\left(\left(\widehat{\boldsymbol{\Phi}}^\nu\right)^{-1}-\widehat{\mathbf{J}}\right)\left(\mathbf{G}^\nu-\mathbf{G}^{\nu-1}\right)
 %-\mathbf{1}\otimes\frac{1}{I}\sum_{i=1}^I\left(\nabla_D f_i(\overline{\mathbf{D}}_{\boldsymbol{\phi}^\nu} ,\mathbf{X}_i^\nu)-\nabla_D f_i(\mathbf{D}_{(i)}^\nu,\mathbf{X}_i^\nu)\right),
%\end{aligned}
%\nonumber
%\end{equation}
%for all $\nu\geq 2$, where together with triangle and Schwartz inequalities, it further implies
\begin{equation}\label{theta_tracking_bound1}
\begin{aligned}
&\Big\|\boldsymbol{ {\Theta}}^\nu-\mathbf{1}\otimes\frac{1}{I}\sum_{i=1}^I\nabla_D f_i(\overline{\mathbf{D}}_{\boldsymbol{\phi}^\nu} ,\mathbf{X}_i^\nu)\Big\|_F
\\
& \overset{(a)}{\leq} \Big\|\widehat{\mathbf{W}}^{\nu-1:0}-\widehat{\mathbf{J}}\Big\|_2~\Big\|\boldsymbol{{\Theta}}^{0}\Big\|_F
+\frac{1}{\underline{\epsilon}_\phi}\sum_{t=1}^{\nu-1}\Big\|\widehat{\mathbf{W}}^{\nu-1:t}-\widehat{\mathbf{J}}_{\phi^t}\Big\|_2~\Big\|\mathbf{G}^t-\mathbf{G}^{t-1}\Big\|_F
\\
&\quad +\Big\|\left(\widehat{\boldsymbol{\Phi}}^\nu\right)^{-1}-\widehat{\mathbf{J}}\Big\|_2\norm{\mathbf{G}^\nu-\mathbf{G}^{\nu-1}}_F
+\frac{1}{\sqrt{I}}\,\sum_{i=1}^I\Big\|\nabla_D f_i(\overline{\mathbf{D}}_{\boldsymbol{\phi}^\nu} ,\mathbf{X}_i^\nu)-\nabla_D f_i(\mathbf{D}_{(i)}^\nu,\mathbf{X}_i^\nu)\Big\|_F\\
& \overset{(b)}{\leq}   c_4\,(\rho)^\nu+c_5\,L_{\nabla} \sum_{t=1}^{\nu}(\rho)^{\nu-t}\left(\norm{\mathbf{D}^t-\mathbf{D}^{t-1}}_F +\norm{\mathbf{X}^t-\mathbf{X}^{t-1}}_F\right)
 +L_{\nabla}\Big\|\mathbf{D}^\nu-\mathbf{1} \otimes\overline{\mathbf{D}}_{\boldsymbol{\phi}^\nu} \Big\|_F
 \\
&  \overset{(c)} {=} T^\nu+c_5\,L_{\nabla}\sum_{t=1}^{\nu}(\rho)^{\nu-t}\norm{\mathbf{X}^t-\mathbf{X}^{t-1}}_F,
\end{aligned}
\end{equation}
for some positive finite constants $c_4$ and $c_5$, where   in (a) we used the lower bound $\phi^\nu_i\geq \underline{\epsilon}_\phi$  [cf. \eqref{delta_def}] and   $\widehat{\mathbf{J}}_{\phi^t}=\widehat{\mathbf{J}}\,\widehat{\boldsymbol{\Phi}}^t$ [cf.~\eqref{colStochW_avgMat_eq}]; and in  (b) we used \eqref{phi_notation}, (\ref{Lip_G}) (cf. Remark \ref{Lipschitz_continuity_remark}), and Lemma \ref{W_mat_err_decay}; and in (c) we defined  %), and finally the bound $||\mathbf{x}||_1\leq\sqrt{I}||\mathbf{x}||_2$, ........; and in (c) we defined 
%with $\mathbf{x}\in\mathbb{R}^I$,
%\begin{equation}
%\sum_{i=1}^I a_i\leq\sqrt{I}||\mathbf{a}||_2,\quad \mathbf{a} \triangleq (a_1,a_2,\ldots,a_I),
%\label{jensen_sqrtI}
%\end{equation}
%$\sum_{i=1}^I||\mathbf{D}_{(i)}^\nu-\overline{\mathbf{D}}_{\boldsymbol{\phi}^\nu} ||\leq \sqrt{I}||\mathbf{D}^\nu-\mathbf{1} \otimes\overline{\mathbf{D}}_{\boldsymbol{\phi}^\nu} ||$,
 %\eqref{theta_tracking_bound1} becomes
%\begin{equation}
%\begin{aligned}
%&\Big\|\boldsymbol{\widetilde{\Theta}}^\nu-\mathbf{1}\otimes\frac{1}{I}\sum_{i=1}^I\nabla_D f_i(\overline{\mathbf{D}}_{\boldsymbol{\phi}^\nu} ,\mathbf{X}_i^\nu)\Big\|_F
%\\
%& \leq  c_4\,\rho^\nu+c_5\,L_{\nabla} \sum_{t=1}^{\nu}\rho^{\nu-t}\left(\norm{\mathbf{D}^t-\mathbf{D}^{t-1}}_F +\norm{\mathbf{X}^t-\mathbf{X}^{t-1}}_F\right)
 %+L_{\nabla}||\mathbf{D}^\nu-\mathbf{1} \otimes\overline{\mathbf{D}}_{\boldsymbol{\phi}^\nu} ||_F
%\\
%& = T^\nu+c_5\,L_{\nabla}\sum_{t=1}^{\nu}\rho^{\nu-t}\norm{\mathbf{X}^t-\mathbf{X}^{t-1}}_F,
%\end{aligned}
%\label{Theta_error_expansion}
%\end{equation}
%for some positive finite constants $c_4$ and $c_5$, where 
$T^\nu$  as
\begin{equation}
\label{T_nu}
T^\nu \triangleq   c_4\,(\rho)^\nu+c_5\,L_{\nabla}\sum_{t=1}^{\nu}(\rho)^{\nu-t}\norm{\mathbf{D}^t-\mathbf{D}^{t-1}}_F+L_{\nabla}\norm{\mathbf{D}^\nu-\mathbf{1} \otimes\overline{\mathbf{D}}_{\boldsymbol{\phi}^\nu} }_F.
\end{equation}
Substituting \eqref{theta_tracking_bound1} into \eqref{descentIneq_D_i_sum} %leads, after some basic manipulation, to the desired result \eqref{best-response_D_part_optimality}, with % 
yields\vspace{-0.2cm}
\begin{equation}
\begin{aligned}
\Big\langle\nabla_D F(\overline{\mathbf{D}}_{\boldsymbol{\phi}^\nu} ,&\mathbf{X}^\nu),\sum_{i=1}^I  \phi_i^\nu\left(\widetilde{\mathbf{D}}^\nu_{(i)}-\mathbf{D}_{(i)}^\nu\right)\Big\rangle
+\sum_{i=1}^I\phi_i^\nu\left(G(\widetilde{\mathbf{D}}^\nu_{(i)})- G(\mathbf{D}^\nu_{(i)})\right)
\\
&\leq -s_D ||\widetilde{\mathbf{D}}^\nu- \mathbf{D}^\nu||_F^2+I\bar{\epsilon}_\phi T^\nu ||\widetilde{\mathbf{D}}^\nu- \mathbf{D}^\nu||_F
\\
&\quad+I\,\bar{\epsilon}_\phi \,c_5\, L_{\nabla}\,||\widetilde{\mathbf{D}}^\nu- \mathbf{D}^\nu||_F\sum_{t=1}^{\nu}(\rho)^{\nu-t}\norm{\mathbf{X}^t-\mathbf{X}^{t-1}}_F
\\
&=-s_D\left(||\widetilde{\mathbf{D}}^\nu- \mathbf{D}^\nu||_F-\frac{I\,\bar{\epsilon}_\phi}{2 \,s_D}T^\nu\right)^2+\frac{I^2\,\bar{\epsilon}_\phi^2}{4 \,s_D}\left(T^\nu\right)^2
\\
&\quad+\eta ~||\widetilde{\mathbf{D}}^\nu- \mathbf{D}^\nu||_F\,\sum_{t=1}^{\nu}(\rho)^{\nu-t}\norm{\mathbf{X}^t-\mathbf{X}^{t-1}}_F,
\end{aligned}
\nonumber
\end{equation}
with $\eta=I\bar{\epsilon}_\phi c_5L_{\nabla}$.

To complete the proof, we need to show that   $\sum_{\nu=1}^\infty \left(T^\nu\right)^2<\infty$. Note that  the first term on the RHS of \eqref{T_nu}  is square summable, and so is the third one, due to  Proposition \ref{consVar_props} [cf. \eqref{D_consError_summable}]. Invoking Jensen's inequality, it is sufficient to show that the second  term on RHS of \eqref{T_nu} is square summable. Following the same approach  used to prove  \eqref{D_consError_summable}, we have\vspace{-0.2cm}
\begin{equation}
\begin{aligned}
&\lim_{\nu\rightarrow\infty}\sum_{t=1}^\nu\left(\sum_{l=1}^{t}(\rho)^{t-l}\norm{\mathbf{D}^l-\mathbf{D}^{l-1}}_F\right)^2
\\
&\leq \lim_{\nu\rightarrow\infty}\sum_{t=1}^\nu \sum_{l=1}^{t} \sum_{k=1}^{t} (\rho)^{t-l}(\rho)^{t-k}\norm{\mathbf{D}^l-\mathbf{D}^{l-1}}_F\norm{\mathbf{D}^k-\mathbf{D}^{k-1}}_F
\\
&\overset{(a)}{\leq}\frac{1}{1-\rho}\lim_{\nu\rightarrow\infty}\sum_{t=1}^\nu \sum_{l=1}^{t}  (\rho)^{t-l}(\gamma^l)^2 \norm{\widetilde{\mathbf{D}}^{l-1}-\mathbf{D}^{l-1}}_F^2
\,\overset{(b)}{<}\infty,
\end{aligned}
\nonumber
\end{equation}
where (a) follows  from $ab\leq(a^2+b^2)/2$, and (b) is due to  Lemma \ref{sequence_convg_props} and Assumption \ref{Problem_Assumptions}3.
Hence $\sum_{\nu=1}^\infty \left(T^\nu\right)^2<\infty$.

\noindent \textbf{(b)} We prove this statement using the definition \eqref{eq:ftilde2} of  $\tilde{f}_i$; the same conclusion holds   also using the alternative choice \eqref{eq:ftilde1} of $\tilde{f}_i$; the proof is thus omitted.   Invoking the  optimality of $\mathbf{X}^{\nu+1}$ [cf. \eqref{eq:Xupdate}] together with the  convexity of $g_i$, yield\vspace{-0.2cm}
%\begin{equation}
%\left\langle\nabla_{X_i} \tilde{h}_i(\mathbf{X}_i^{\nu+1};\mathbf{U}_{(i)}^\nu,\mathbf{X}_i^\nu),\mathbf{X}_i-\mathbf{X}_i^{\nu+1}\right\rangle+g_i(\mathbf{X}_i)-g_i(\mathbf{X}_i^{\nu+1})\geq 0,
%\nonumber
%\end{equation}
%for all $\mathbf{X}_i\in\mathcal{X}_i$. Set $\mathbf{X}_i=\mathbf{X}_i^\nu$ and let us add and subtract some dummy terms (inside the inner product), then it is not difficult to show that
\begin{equation}
\begin{aligned}
&\left\langle\nabla_{X_i} \tilde{h}_i(\mathbf{X}_i^{\nu+1};\mathbf{U}_{(i)}^\nu,\mathbf{X}_i^\nu)-\nabla_{X_i} \tilde{h}_i(\mathbf{X}_i^\nu;\mathbf{U}_{(i)}^\nu,\mathbf{X}_i^\nu),\mathbf{X}^\nu_i- \mathbf{X}_i^{\nu+1}\right\rangle
\\
&+\left\langle\nabla_{X_i} \tilde{h}_i(\mathbf{X}_i^\nu;\mathbf{U}_{(i)}^\nu,\mathbf{X}_i^\nu)-\nabla_{X_i} \tilde{h}_i(\mathbf{X}_i^\nu;\overline{\mathbf{U}}_{\boldsymbol{\phi}^\nu} ,\mathbf{X}_i^\nu),\mathbf{X}^\nu_i- \mathbf{X}_i^{\nu+1}\right\rangle
\\
&+\left\langle\nabla_{X_i} \tilde{h}_i(\mathbf{X}_i^\nu;\overline{\mathbf{U}}_{\boldsymbol{\phi}^\nu} ,\mathbf{X}_i^\nu),\mathbf{X}^\nu_i- \mathbf{X}_i^{\nu+1}\right\rangle+g_i\left(\mathbf{X}^\nu_i\right)- g_i( \mathbf{X}_i^{\nu+1})\geq 0.
\end{aligned}
%\label{X_opt_cond_addSubstractTerms}
\nonumber\vspace{-0.2cm}
\end{equation}
Using Remark \ref{surrogate_props} %(a) ($\tau^\nu_{X,i}$-strongly convexity of $\tilde{h}_i(\bullet;\mathbf{U}_{(i)}^\nu,\mathbf{X}_i^\nu)$),
%using  $\nabla_{X_i}\tilde{h}_i(\mathbf{X}_i^\nu;\mathbf{U}_{(i)}^\nu,\mathbf{X}_i^\nu)=\nabla_{X_i}{f}_i(\mathbf{U}_{(i)}^\nu,\mathbf{X}_i^\nu)$,  $\nabla_{X_i}\tilde{h}_i(\mathbf{X}_i^\nu;\overline{\mathbf{U}}_{\boldsymbol{\phi}^\nu} ,\mathbf{X}_i^\nu)=\nabla_{X_i}{f}_i(\overline{\mathbf{U}}_{\boldsymbol{\phi}^\nu} ,\mathbf{X}_i^\nu)$,
 %Remark \ref{surrogate_props} (b),
and $\overline{\mathbf{U}}_{\boldsymbol{\phi}^\nu} =\overline{\mathbf{D}}_{\boldsymbol{\phi}^{\nu+1}} $ [cf. \eqref{D_bar_eq_U_bar}], we obtain\vspace{-0.1cm}
\begin{equation}
%\label{eq:nabla_x_last_ineq}
\begin{aligned}
&\left\langle\nabla_{X_i}f_i(\overline{\mathbf{D}}_{\boldsymbol{\phi}^{\nu+1}} ,\mathbf{X}_i^\nu),\mathbf{X}_i^{\nu+1}-\mathbf{X}_i^\nu\right\rangle
+g_i(\mathbf{X}_i^{\nu+1})-g_i\left(\mathbf{X}_i^\nu \right)
\\
&\qquad\qquad \leq -\tau^\nu_{X,i}\norm{\mathbf{X}_i^{\nu+1}-\mathbf{X}_i^\nu}_F^2 +\left\langle\nabla_{X_i}f_i(\overline{\mathbf{U}}_{\boldsymbol{\phi}^\nu} ,\mathbf{X}_i^\nu)-\nabla_{X_i}f_i(\mathbf{U}_{(i)}^\nu,\mathbf{X}_i^\nu),\mathbf{X}_i^{\nu+1}-\mathbf{X}_i^\nu\right\rangle,
\end{aligned}\nonumber \vspace{-0.1cm}
\end{equation}
which, together with  \eqref{Lip_G}, leads to the desired result \eqref{best-response_X_part_optimality}, with $L_{X}=\sqrt{I}L_{\nabla}$.\smallskip
\end{proof}

 %where the last equality is due to Proposition~\ref{consVar_props} [cf.~\eqref{D_consensus}].

%We prove next statements (b) and (b') throughout the following steps: {1) we first show that  $\{U(\overline{\mathbf{D}}^\nu,\mathbf{X}^\nu)\}_\nu$ is convergent,  $\liminf_{\nu\rightarrow\infty}\Delta_D(\overline{\mathbf{D}}^\nu,\mathbf{X}^\nu)=0$ and $\lim_{\nu\rightarrow\infty}\Delta_X(\overline{\mathbf{D}}^\nu,\mathbf{X}^\nu)=0$ under Assumptions \ref{Problem_Assumptions}, \ref{B_strongly_connectivity}, \ref{A_matrix_Assumptions}, \ref{freeVars_assumptions}1-\ref{freeVars_assumptions}2; 2) we then strengthen the results by proving  $\lim_{\nu\rightarrow\infty}\Delta_D(\overline{\mathbf{D}}^\nu,\mathbf{X}^\nu)=0$ under stronger assumptions \ref{Problem_Assumptions}1-\ref{Problem_Assumptions}4, \ref{Problem_Assumptions}5 (ii), \ref{B_strongly_connectivity}, \ref{A_matrix_Assumptions}, \ref{freeVars_assumptions}.} 

%and 3) finally we show that  the limit point of any subsequence   $\{(\overline{\mathbf{D}}^\nu ,\mathbf{X}^\nu)\}_{\nu\in \mathcal Z}$   satisfying $\lim_{\mathcal{Z}\ni \nu\rightarrow\infty}|| \widetilde{\mathbf{D}}^\nu-\mathbf{D}^\nu||_F=0$,  is a stationary solution of \ref{eq:P1}. 

\subsubsection{Proof of Proposition \ref{prop_descent}}\label{app_proof_prop_descent}
%Finally, since $G$ is convex (and thus locally Lipschitz continuous) and  $\mathcal{D}$ is bounded, $G$ is  Lipschitz continuity  on $\mathcal{D}$, with some Lipschitz constant $L_G>0$.
%We begin  studying the descent properties of $U$ along the sequence   $\{(\overline{\mathbf{D}}_{\boldsymbol{\phi}^\nu} ,\mathbf{X}^{\nu})\}_\nu$. 
We begin observing that\vspace{-0.4cm} \begin{equation}
G(\overline{\mathbf{D}}_{\boldsymbol{\phi}^{\nu+1}} ) \leq %G(\overline{\mathbf{D}}_{\boldsymbol{\phi}^\nu} )+\frac{\gamma^\nu}{I} \sum_{i=1}^I\phi_i^\nu\left(G(\widetilde{\mathbf{D}}^\nu_{(i)})-G(\overline{\mathbf{D}}_{\boldsymbol{\phi}^\nu} )\right)
   G(\overline{\mathbf{D}}_{\boldsymbol{\phi}^\nu} )+\frac{\gamma^\nu}{I} \sum_{i=1}^I\phi_i^\nu\left(G(\widetilde{\mathbf{D}}^\nu_{(i)})-G(\mathbf{D}_{(i)}^\nu)\right)
 +\frac{\gamma^\nu}{I} \sum_{i=1}^I\phi_i^\nu\left(G(\mathbf{D}_{(i)}^\nu)-G(\overline{\mathbf{D}}_{\boldsymbol{\phi}^\nu} )\right),\label{G_bound}
\end{equation}
  due to \eqref{D_bar_update} and the convexity of $G$ [together with $\mathbf{1}^\intercal \boldsymbol{\phi}^\nu=1$]. %; and in (b) we add/subtracted a $\phi_i^\nu$-weighted sum of $G(\mathbf{D}_{(i)}^\nu)$ for all $i=1,2,\ldots,I$.

 Invoking the descent lemma for  $f_i(\overline{\mathbf{U}}_{\boldsymbol{\phi}^\nu} ,\bullet)$  and using $\overline{\mathbf{U}}_{\boldsymbol{\phi}^\nu} =\overline{\mathbf{D}}_{\boldsymbol{\phi}^{\nu+1}} $ [cf. \eqref{D_bar_eq_U_bar}], we get: for sufficiently large  $\nu$, say $\nu\geq \nu_0$,  \vspace{-0.2cm}
\begin{equation}
\label{descent_X}
\begin{aligned}
&U(\overline{\mathbf{D}}_{\boldsymbol{\phi}^{\nu+1}} ,\mathbf{X}^{\nu+1})
\\
&\leq\sum_{i=1}^I \left\{f_i(\overline{\mathbf{D}}_{\boldsymbol{\phi}^{\nu+1}} ,\mathbf{X}_i^\nu)+ g_i(\mathbf{X}_i^{\nu})\right\}+ G(\overline{\mathbf{D}}_{\boldsymbol{\phi}^{\nu+1}} )+\sum_{i=1}^I \left\{g_i(\mathbf{X}_i^{\nu+1})-g_i(\mathbf{X}_i^{\nu})\right\}
\\
&\quad+\sum_{i=1}^I\,\left\langle\nabla_{X_i} f_i(\overline{\mathbf{D}}_{\boldsymbol{\phi}^{\nu+1}} ,\mathbf{X}_i^{\nu}),\mathbf{X}_i^{\nu+1}-\mathbf{X}_i^\nu\right\rangle
+\frac{1}{2}\sum_{i=1}^I L_{\nabla{X_i}}(\overline{\mathbf{U}}_{\boldsymbol{\phi}^\nu} )\norm{\mathbf{X}_i^{\nu+1}-\mathbf{X}_i^\nu}_F^2
\\
&\overset{(a)}{\leq}  U(\overline{\mathbf{D}}_{\boldsymbol{\phi}^{\nu+1}} ,\mathbf{X}^{\nu})-\sum_{i=1}^I\left\lbrace\left(\tau^\nu_{X,i}-\frac{1}{2}L_{\nabla{X_i}}(\overline{\mathbf{U}}_{\boldsymbol{\phi}^\nu} )\right)\norm{\mathbf{X}_i^{\nu+1}-\mathbf{X}_i^{\nu}}_F^2\right\rbrace
\\
&\quad +L_X ||\mathbf{U}^\nu-\mathbf{1}\otimes\overline{\mathbf{U}}_{\boldsymbol{\phi}^\nu} ||_F~||\mathbf{X}^{\nu+1}-\mathbf{X}^\nu||_F\\
& \overset{(b)}{\leq} U(\overline{\mathbf{D}}_{\boldsymbol{\phi}^{\nu+1}} ,\mathbf{X}^{\nu})-{s_X}||\mathbf{X}^{\nu+1}-\mathbf{X}^{\nu}||_F^2
+L_{X}||\mathbf{U}^\nu-\mathbf{1}\otimes\overline{\mathbf{U}}_{\boldsymbol{\phi}^\nu} ||_F~ ||\mathbf{X}^{\nu+1}-\mathbf{X}^\nu||_F,
\end{aligned}
\end{equation}
where in (a) we used Proposition \ref{sol_map_prop}(b); and in (b)  $s_X>0$ is a constant such that $\inf_{\nu\geq \nu_0}(\tau_{X,i}^\nu-\frac{1}{2}L_{\nabla{X_i}}(\overline{\mathbf{U}}_{\boldsymbol{\phi}^\nu} ))\geq s_X$, for all $i=1,\ldots, I$. Note that such a constant exists because of   \eqref{Lip_continuity} and Assumption \ref{freeVars_assumptions}1.

 {To upper bound $ U(\overline{\mathbf{D}}_{\boldsymbol{\phi}^{\nu+1}} ,\mathbf{X}^{\nu})$,} we apply the descent lemma to $F(\bullet,\mathbf{X}^\nu)$. Recalling  that $\nabla_D F(\bullet,\mathbf{X}^\nu)$ is Lipschitz continuous with constant {${L}_{\nabla_D}$} and using \eqref{D_bar_update},  we get\vspace{-0.2cm}
 \begin{equation}
\begin{aligned}
U(\overline{\mathbf{D}}_{\boldsymbol{\phi}^{\nu+1}} ,\mathbf{X}^{\nu})& %= F(\overline{\mathbf{D}}_{\boldsymbol{\phi}^{\nu+1}} ,\mathbf{X}^{\nu}) +\sum_{i=1}^I g_i(\mathbf{X}_i^{\nu}) +G(\overline{\mathbf{D}}_{\boldsymbol{\phi}^{\nu+1}} )
%\\
%& 
\leq F(\overline{\mathbf{D}}_{\boldsymbol{\phi}^\nu} ,\mathbf{X}^{\nu})+\frac{\gamma^\nu}{I} \Big\langle\nabla_D F(\overline{\mathbf{D}}_{\boldsymbol{\phi}^\nu} ,\mathbf{X}^\nu),\sum_{i=1}^I \phi_i^\nu\left(\widetilde{\mathbf{D}}_{(i)}^\nu-\mathbf{D}_{(i)}^\nu\right) \Big\rangle
\\
&\quad+\frac{{L}_{\nabla_D}}{2}\left(\frac{\gamma^\nu}{I}\right)^2\norm{\sum_{i=1}^I\phi_i^\nu\left(\widetilde{\mathbf{D}}^\nu_{(i)} -\mathbf{D}_{(i)}^\nu\right)}_F^2+\sum_{i=1}^I g_i(\mathbf{X}_i^{\nu})+  G(\overline{\mathbf{D}}_{\boldsymbol{\phi}^{\nu+1}} ),\\
%\end{aligned}
%\label{descent_D_0}
%\nonumber
%\end{equation}}
%and by the bound \eqref{G_bound}, and some further simplifications, it yields
%\textcolor{red}{\begin{equation}
%\begin{aligned}
%&U(\overline{\mathbf{D}}_{\boldsymbol{\phi}^{\nu+1}} ,\mathbf{X}^{\nu}) 
%\\
%& \overset{(a)}{\leq}  U(\overline{\mathbf{D}}_{\boldsymbol{\phi}^\nu} ,\mathbf{X}^{\nu})
%+\frac{\gamma^\nu}{I}\left[\Big\langle\nabla_D F(\overline{\mathbf{D}}_{\boldsymbol{\phi}^\nu} ,\mathbf{X}^\nu), \sum_{i=1}^I \phi_i^\nu\left( \widetilde{\mathbf{D}}_{(i)}^\nu-\mathbf{D}_{(i)}^\nu\right)\Big\rangle
%+\sum_{i=1}^I\phi_i^\nu\left(G(\widetilde{\mathbf{D}}^\nu_{(i)})-G(\mathbf{D}_{(i)}^\nu)\right)\right]
%\\
%&\quad+\frac{{L}_{\nabla_D}}{2}({\gamma^\nu})^2\,||\widetilde{\mathbf{D}}^\nu-\mathbf{D}^\nu||_F^2+\frac{\gamma^\nu}{I} \sum_{i=1}^I\phi_i^\nu\left(G(\mathbf{D}_{(i)}^\nu)-G(\overline{\mathbf{D}}_{\boldsymbol{\phi}^\nu} )\right)\\
%\end{aligned}
%\label{descent_D_1}
%\nonumber
%\end{equation} 
%where in (a) we used \eqref{G_bound}
 %and finally, by Proposition \ref{sol_map_prop} (a) and \textcolor{red}{Lipschitz continuity of $G$ on $\mathcal{D}$ [cf. Remark \ref{Lipschitz_continuity_remark}]}, we get
%\textcolor{red}{\begin{equation}
%\begin{aligned}
%U(\overline{\mathbf{D}}_{\boldsymbol{\phi}^{\nu+1}} ,\mathbf{X}^{\nu})
& \overset{(a)}{\leq}  \,  U(\overline{\mathbf{D}}_{\boldsymbol{\phi}^\nu} ,\mathbf{X}^{\nu})
-\frac{s_D\cdot \gamma^\nu}{I}\left(||\widetilde{\mathbf{D}}^\nu-\mathbf{D}^\nu||_F-\frac{I\bar{\epsilon}_\phi}{2s_D}T^\nu\right)^2
\\
&+\frac{\eta\cdot\gamma^\nu}{I} || \widetilde{\mathbf{D}}^\nu-\mathbf{D}^\nu||_F\sum_{t=1}^{\nu}(\rho)^{\nu-t}||\mathbf{X}^t-\mathbf{X}^{t-1}||_F+\frac{I\bar{\epsilon}_\phi^2 \gamma^\nu}{4s_D}\left(T^\nu\right)^2
\\
&+\frac{{L}_{\nabla_D}}{2}({\gamma^\nu})^2\,||\widetilde{\mathbf{D}}^\nu-\mathbf{D}^\nu||_F^2+\underbrace{\frac{\gamma^\nu}{I} \sum_{i=1}^I\phi_i^\nu\left(G(\mathbf{D}_{(i)}^\nu)-G(\overline{\mathbf{D}}_{\boldsymbol{\phi}^\nu} )\right)}_{\leq L_G\gamma^\nu\norm{\mathbf{D}^\nu-\mathbf{1}\otimes\overline{\mathbf{D}}_{\boldsymbol{\phi}^\nu} }_F},
%\\
%{\leq} &  U(\overline{\mathbf{D}}_{\boldsymbol{\phi}^\nu} ,\mathbf{X}^{\nu})
%-\frac{s_D\cdot \gamma^\nu}{I}\left(||\widetilde{\mathbf{D}}^\nu-\mathbf{D}^\nu||_F-\frac{I\bar{\epsilon}_\phi}{2s_D}T^\nu\right)^2
%\\
%&+\frac{\eta\cdot\gamma^\nu}{I} || \widetilde{\mathbf{D}}^\nu-\mathbf{D}^\nu||_F\sum_{t=1}^{\nu}(\rho)^{\nu-t}||\mathbf{X}^t-\mathbf{X}^{t-1}||_F+\frac{I\bar{\epsilon}_\phi^2 \gamma^\nu}{4s_D}\left(T^\nu\right)^2
%\\
%&+\frac{{L}_{\nabla_D}}{2}({\gamma^\nu})^2\,||\widetilde{\mathbf{D}}^\nu-\mathbf{D}^\nu||_F^2+L_G\gamma^\nu\norm{\mathbf{D}^\nu-\mathbf{1}\otimes\overline{\mathbf{D}}_{\boldsymbol{\phi}^\nu} }_F,
\end{aligned}
\label{descent_D}
\end{equation} 
where   in (a) we used \eqref{G_bound}
 and Proposition \ref{sol_map_prop}(a), and   the Lipschitz continuity of $G$, due to the convexity of $G$ ($G$ is thus locally Lipschitz continuous) and  the compactness of $\mathcal{D}$; we denoted by $L_G>0$ the   Lipschitz constant.

 Combining \eqref{descent_X} with \eqref{descent_D} and defining {$\bar{\tau}^\nu_D\triangleq s_D-\gamma^\nu\,\frac{ I{L}_{\nabla_D} }{2}$}, we get: for $\nu\geq \nu_0$,
\begin{equation}
\label{one-step-descent}
\begin{aligned}
U(\overline{\mathbf{D}}_{\boldsymbol{\phi}^{\nu+1}} ,\mathbf{X}^{\nu+1})
&\leq U(\overline{\mathbf{D}}_{\boldsymbol{\phi}^\nu} ,\mathbf{X}^{\nu})-s_X||\mathbf{X}^{\nu+1}-\mathbf{X}^{\nu}||_F^2+L_X||\mathbf{U}^\nu-\mathbf{1}\otimes\overline{\mathbf{U}}_{\boldsymbol{\phi}^\nu} ||_F ~||\mathbf{X}^{\nu+1}-\mathbf{X}^\nu||_F
\\
&\quad-\frac{\bar{\tau}^\nu_D\,\gamma^\nu}{I}\left(|| \widetilde{\mathbf{D}}^\nu-\mathbf{D}^\nu||_F-\frac{I\,\bar{\epsilon}_\phi}{2\,\bar{\tau}^\nu_D}\,T^\nu\right)^2+\frac{I\,\bar{\epsilon}_\phi^2\,\gamma^\nu}{4\,\bar{\tau}^\nu_D}\,\left(T^\nu\right)^2
\\
&\quad+\frac{\eta\,\gamma^\nu}{I}\, || \widetilde{\mathbf{D}}^\nu-\mathbf{D}^\nu||_F\underbrace{\sum_{t=1}^{\nu}(\rho)^{\nu-t}||\mathbf{X}^t-\mathbf{X}^{t-1}||_F}_{\leq {c}_5\,(\rho)^\nu+ {\rho}^{-1}\sum_{t=1}^{\nu}(\rho)^{\nu-t}||\mathbf{X}^{t+1}-\mathbf{X}^{t}||_F}\!\!\!\!\!\!\!\!\!+L_G\,\gamma^\nu\norm{\mathbf{D}^\nu-\mathbf{1}\otimes\overline{\mathbf{D}}_{\boldsymbol{\phi}^\nu} }_F,
\end{aligned}
\end{equation}
where $c_5\triangleq (\rho)^{-1}\,||\mathbf{X}^1-\mathbf{X}^0||_F$. Since $\gamma^\nu \downarrow
  0$, there exists an integer $\nu_1\geq \nu_0$  and some $\bar{\tau}_D$ such that   $\bar{\tau}^\nu_D\geq\bar{\tau}_D>0$, for all $\nu \geq \nu_1$. Let $\bar{\nu}$ be any integer $\bar{\nu}\geq \nu_1$. Then,  applying  \eqref{one-step-descent}  recursively on $\nu, \nu-1, \ldots , \bar{\nu}+1,\bar{\nu}$, and using the boundedness of $\{||\widetilde{\mathbf{D}}^\nu-\mathbf{D}^\nu||_F\}_\nu$, we obtain\vspace{-0.2cm}
\begin{equation}
\begin{aligned}
&U(\overline{\mathbf{D}}_{\boldsymbol{\phi}^{\nu+1}} ,\mathbf{X}^{\nu+1})
\\
&\leq U(\overline{\mathbf{D}}_{\boldsymbol{\phi}^{\bar{\nu}}},\mathbf{X}^{\bar{\nu}})-s_X\sum_{l=\bar{\nu}}^{\nu}||\mathbf{X}^{l+1}-\mathbf{X}^{l}||_F^2
+L_X\,\sum_{l=\bar{\nu}}^{\nu}||\mathbf{U}^l-\mathbf{1}\otimes\overline{\mathbf{U}}_{\boldsymbol{\phi}^l}||_F~||\mathbf{X}^{l+1}-\mathbf{X}^l||_F
\\
&\quad-\frac{\bar{\tau}_D}{I}\sum_{l=\bar{\nu}}^{\nu}\gamma^l\left(|| \widetilde{\mathbf{D}}^l-\mathbf{D}^l||_F-\frac{I\bar{\epsilon}_\phi}{2\,\bar{\tau}_D}T^l\right)^2+\frac{I\bar{\epsilon}_\phi^2}{4\,\bar{\tau}_D}\sum_{l=\bar{\nu}}^{\nu}\gamma^l\left(T^l\right)^2
\\
&\quad+c_6\,\sum_{l=\bar{\nu}}^{\nu}\gamma^l(\rho)^l+c_7\,\sum_{l=\bar{\nu}}^{\nu}\sum_{t=1}^{l}\gamma^l(\rho)^{l-t}\norm{\mathbf{X}^{t+1}-\mathbf{X}^t}_F+L_G\sum_{l=\bar{\nu}}^\nu\gamma^l\norm{\mathbf{D}^l-\mathbf{1}\otimes\overline{\mathbf{D}}_{\boldsymbol{\phi}^l} }_F,
\end{aligned}
\label{descent_ineq}
\end{equation}
for some finite constants $c_6, c_7>0$. Using the boundedness of   $\{||\mathbf{X}^{\nu+1}-\mathbf{X}^{\nu}||_F\}_\nu$ {(cf. Step 2), i.e.,  $||\mathbf{X}^{\nu+1}-\mathbf{X}^{\nu}||_F\leq B_X$, for all  $\nu$ and some   $B_X>0$, we can bound the double-sum term on the RHS of  \eqref{descent_ineq} as}
\begin{equation}
\begin{aligned}
&\sum_{l=\bar{\nu}}^{\nu}\sum_{t=1}^{l}\gamma^l(\rho)^{l-t} \norm{\mathbf{X}^{t+1}-\mathbf{X}^t}_F
=\sum_{l=1}^{\nu}\sum_{t=\max(\bar{\nu},l)}^{\nu}\gamma^t(\rho)^{t-l}\norm{\mathbf{X}^{l+1}-\mathbf{X}^l}_F
\\
&=\sum_{l=1}^{\bar{\nu}-1}\norm{\mathbf{X}^{l+1}-\mathbf{X}^l}_F\sum_{t=\bar{\nu}}^{\nu}\gamma^t(\rho)^{t-l}
+\sum_{l=\bar{\nu}}^{\nu}\norm{\mathbf{X}^{l+1}- \mathbf{X}^l}_F\sum_{t=l}^{\nu}\gamma^t(\rho)^{t-l}
\\
&\overset{(a)}{\leq} B_X\cdot\left(\max_{\bar{\nu}\leq t\leq \nu}\gamma^t\right)\cdot\frac{\left(1-(\rho)^{\bar{\nu}}\right)\left(1-(\rho)^{\nu-\bar{\nu}+1}\right) }{(1-\rho)^2}
+\sum_{l=\bar{\nu}}^{\nu}\norm{\mathbf{X}^{l+1}-\mathbf{X}^l}_F\sum_{t=l}^{\infty}\gamma^t(\rho)^{t-l}
\\
&\leq \frac{B_X}{(1-\rho)^2}~\left(\max_{ t\geq \bar{\nu}}\gamma^t\right) +\sum_{l=\bar{\nu}}^{\nu}\norm{\mathbf{X}^{l+1}-\mathbf{X}^l}_F\sum_{t=l}^{\infty}\gamma^t(\rho)^{t-l},
\end{aligned}
\label{error_sum_simplification}
\end{equation}
where in (a) we used the summability of $\sum_{t=l}^{\nu}\gamma^t(\rho)^{t-l}$, and the following bound
 $$\sum_{l=1}^{\bar{\nu}-1}\sum_{t=\bar{\nu}}^{{\nu}}\gamma^t (\rho)^{t-l}\leq \left(\max_{ t\geq \bar{\nu}}\gamma^t\right)\,\dfrac{\left(1-(\rho)^{\bar{\nu}}\right)\left(1-(\rho)^{\nu-\bar{\nu}+1}\right)}{(1-\rho)^2}.$$
Substituting \eqref{error_sum_simplification} in \eqref{descent_ineq}, yields\vspace{-0.3cm}
\begin{equation}
\begin{aligned}
U(\overline{\mathbf{D}}_{\boldsymbol{\phi}^{\nu+1}} ,\mathbf{X}^{\nu+1})&\leq U(\overline{\mathbf{D}}_{\boldsymbol{\phi}^{\bar{\nu}}},\mathbf{X}^{\bar{\nu}})-s_X\sum_{l=\bar{\nu}}^{\nu}\norm{\mathbf{X}^{l+1}-\mathbf{X}^{l}}_F^2
\\
&\quad+\sum_{l=\bar{\nu}}^{\nu}\underbrace{\left(c_7\sum_{t=l}^{\infty}\gamma^t(\rho)^{t-l}+L_X||\mathbf{U}^l-\mathbf{1}\otimes \overline{\mathbf{U}}_{\boldsymbol{\phi}^l}||_F\right)}_{\triangleq Z^l} \norm{\mathbf{X}^{l+1}-\mathbf{X}^l}_F
\\
&\quad -\frac{\bar{\tau}_D}{I}\sum_{l=\bar{\nu}}^{\nu}\gamma^l\left(|| \widetilde{\mathbf{D}}^l-\mathbf{D}^l||_F-\frac{I\bar{\epsilon}_\phi}{2\,\bar{\tau}_D}T^l\right)^2+\frac{I\bar{\epsilon}_\phi^2}{4\,\bar{\tau}_D}\sum_{l=\bar{\nu}}^{\nu}\gamma^l\left(T^l\right)^2
\\
&\quad+\underbrace{\left[\frac{{c}_6}{1-\rho}\left((\rho)^{\bar{\nu}}-(\rho)^{\nu+1}\right)+\frac{c_7 B_X}{(1-\rho)^2}\right]\cdot \left(\max_{ t\geq \bar{\nu}}\gamma^t\right)}_{\triangleq E^{\nu,\bar{\nu}}}+L_G\sum_{l=\bar{\nu}}^\nu\gamma^l\norm{\mathbf{D}^l-\mathbf{1}\otimes\overline{\mathbf{D}}_{\boldsymbol{\phi}^l} }_F\nonumber 
%\\
%&=U(\overline{\mathbf{D}}_{\boldsymbol{\phi}^{\bar{\nu}}},\mathbf{X}^{\bar{\nu}})
%  -\sum_{l=\bar{\nu}}^{\nu}Y^l+ \sum_{l=\bar{\nu}}^{\nu} W^l +E^{\nu,\bar{\nu}},
\end{aligned}
%\label{descent_concise}
\vspace{-0.3cm}
\end{equation}
which complete the proof. %where we denoted
%\begin{equation}
%\label{Y_sequence}
%Y^l \triangleq   s_X\left(||\mathbf{X}^{l+1}-\mathbf{X}^{l}||_F-\frac{Z^l}{2s_X} \right)^2+\dfrac{\bar{\tau}_D}{I}\,\gamma^l\left(|| \widetilde{\mathbf{D}}^l-\mathbf{D}^l||_F-\frac{I\bar{\epsilon}_\phi}{2\,\bar{\tau}_D}T^l\right)^2,
%\end{equation}
%\begin{equation}
%W^l\triangleq   \frac{I\bar{\epsilon}_\phi^2}{4\bar{\tau}_D}\gamma^l\left(T^l\right)^2 +\frac{1}{4s_X} \left( Z^l\right)^2+L_G\gamma^l\norm{\mathbf{D}^l-\mathbf{1}\otimes\overline{\mathbf{D}}_{\boldsymbol{\phi}^l} }_F.
%\nonumber
%\end{equation}
\hfill$\square$
\subsubsection{Proof of Lemma \ref{consVar_props2}}\label{proof_lemma_tracking_vanishing}

The result  follows by squaring both sides of  eq. \eqref{theta_tracking_bound1} [in the proof of Proposition \ref{sol_map_prop} (a)],   using $\sum_{\nu=1}^{\infty} (T^\nu)^2<\infty$ [cf. Proposition \ref{sol_map_prop} (a)], and 
	\begin{equation}
	\begin{aligned}
	&\lim_{\nu\rightarrow\infty}\sum_{t=1}^\nu\left(\sum_{l=1}^{t}(\rho)^{t-l}\norm{\mathbf{X}^l-\mathbf{X}^{l-1}}_F\right)^2
	\\
	&\leq \lim_{\nu\rightarrow\infty}\sum_{t=1}^\nu \sum_{l=1}^{t} \sum_{k=1}^{t} (\rho)^{t-l}(\rho)^{t-k}\norm{\mathbf{X}^l-\mathbf{X}^{l-1}}_F\norm{\mathbf{X}^k-\mathbf{X}^{k-1}}_F
	\\
	&\overset{(a)}{\leq}\frac{1}{1-\rho}\lim_{\nu\rightarrow\infty}\sum_{t=1}^\nu \sum_{l=1}^{t}  (\rho)^{t-l}  \norm{{\mathbf{X}}^{l}-\mathbf{X}^{l-1}}_F^2
	\,\overset{(b)}{<}\infty,
	\end{aligned}
	\nonumber
	\end{equation}
	where (a) follows  from $ab\leq(a^2+b^2)/2,\forall a,b\in\mathbb{R}$, and (b) is due to   \eqref{lim_DeltaX_0} and Lemma \ref{sequence_convg_props}. %Finally \eqref{vanishing_tracking_err} is a direct consequence of \eqref{squaresummable_tracking_err}.
 \hfill $\square$ 
 
 \subsubsection{Proof of Proposition \ref{sol_map_prop_part2}}\label{app_proof_prop_delta_d_zero} 
 
 \noindent  \textbf{(a)} Using the optimality of $\widetilde{\mathbf{D}}_{(i)}^\nu$ defined in \eqref{D_tilde_subproblem} %(at two different iterations $\nu_1$ and $\nu_2$) 
	together with convexity of $G$, yields
	\begin{equation}
	\nonumber 
	\begin{aligned}
	&\quad\left\langle \highlightBLUE{\nabla_D \tilde{f}_i(\widetilde{\mathbf{D}}^{\nu_1}_{(i)};\mathbf{D}^{\nu_1}_{(i)},\mathbf{X}^{\nu_1}_i)}+I\cdot
	 {\boldsymbol{\Theta}}_{(i)}^{\nu_1}-\nabla_D f_i(\mathbf{D}^{\nu_1}_{(i)},\mathbf{X}^{\nu_1}_i), \widetilde{\mathbf{D}}^{\nu_2}_{(i)}- \widetilde{\mathbf{D}}^{\nu_1}_{(i)} \right\rangle
	+G(\widetilde{\mathbf{D}}^{\nu_2}_{(i)})-G(\widetilde{\mathbf{D}}_{(i)}^{\nu_1})\geq 0,
	\\
	&\quad\left\langle \highlightRED{\nabla_D \tilde{f}_i(\widetilde{\mathbf{D}}^{\nu_2}_{(i)};\mathbf{D}^{\nu_2}_{(i)},\mathbf{X}^{\nu_2}_i)}+I\cdot
	 {\boldsymbol{\Theta}}_{(i)}^{\nu_2}-\nabla_D f_i(\mathbf{D}^{\nu_2}_{(i)},\mathbf{X}^{\nu_2}_i), \widetilde{\mathbf{D}}^{\nu_1}_{(i)}- \widetilde{\mathbf{D}}^{\nu_2}_{(i)} \right\rangle
	+G(\widetilde{\mathbf{D}}^{\nu_1}_{(i)})-G(\widetilde{\mathbf{D}}_{(i)}^{\nu_2})\geq 0.
	\end{aligned}
	\end{equation}
	Summing the two  inequalities above while adding/subtracting inside the inner product  $\nabla_D \tilde{f}_i(\widetilde{\mathbf{D}}^{\nu_2}_{(i)};\mathbf{D}^{\nu_1}_{(i)},\mathbf{X}_i^{\nu_1})$ and using  \eqref{Theta_update}, yield
	%\begin{equation}
	%\label{min_princ_1}
	%\nonumber
	%\begin{aligned}
	%&\left\langle \highlightBLUE{\nabla_D \tilde{f}_i(\widetilde{\mathbf{D}}^{\nu_1}_{(i)};\mathbf{D}^{\nu_1}_{(i)},\mathbf{X}^{\nu_1}_i)}-\!\highlightRED{\nabla_D \tilde{f}_i(\widetilde{\mathbf{D}}^{\nu_2}_{(i)};\mathbf{D}^{\nu_2}_{(i)},\mathbf{X}^{\nu_2}_i)},\widetilde{\mathbf{D}}_{(i)}^{\nu_2}- \widetilde{\mathbf{D}}^{\nu_1}_{(i)} \right\rangle
	%\\
	%&+ \sum_{j=1}^I\left\langle \nabla_D f_j(\mathbf{D}_{(i)}^{\nu_1},\mathbf{X}_j^{\nu_1})-\nabla_D f_j(\mathbf{D}_{(i)}^{\nu_2},\mathbf{X}_j^{\nu_2}),\widetilde{\mathbf{D}}_{(i)}^{\nu_2}- \widetilde{\mathbf{D}}^{\nu_1}_{(i)}  \right\rangle
	%\\
	%&-\left\langle\nabla_D f_i(\mathbf{D}_{(i)}^{\nu_1},\mathbf{X}_i^{\nu_1})-\nabla_D f_i(\mathbf{D}_{(i)}^{\nu_2},\mathbf{X}_i^{\nu_2}),\widetilde{\mathbf{D}}_{(i)}^{\nu_2}- \widetilde{\mathbf{D}}^{\nu_1}_{(i)}\right\rangle
	%\geq 0.
	%\end{aligned}
	%\end{equation}
	%Now let us add and subtract $\nabla_D \tilde{f}_i(\widetilde{\mathbf{D}}^{\nu_2}_{(i)};\mathbf{D}^{\nu_1}_{(i)},\mathbf{X}_i^{\nu_1})$ inside the inner product of above inequality, which leads to
	\begin{equation}\label{min_princ_2}
	\begin{aligned}
	%&\left\langle \highlightBLUE{\nabla_D \tilde{f}_i(\widetilde{\mathbf{D}}^{\nu_1}_{(i)};\mathbf{D}^{\nu_1}_{(i)},\mathbf{X}^{\nu_1}_i)}- \highlightGRAY{\nabla_D \tilde{f}_i(\widetilde{\mathbf{D}}^{\nu_2}_{(i)};\mathbf{D}^{\nu_1}_{(i)},\mathbf{X}^{\nu_1}_i)},\widetilde{\mathbf{D}}_{(i)}^{\nu_2}- \widetilde{\mathbf{D}}^{\nu_1}_{(i)} \right\rangle
	%\\
	&\!\left\langle \highlightGRAY{\nabla_D \tilde{f}_i(\widetilde{\mathbf{D}}^{\nu_2}_{(i)};\mathbf{D}^{\nu_1}_{(i)},\mathbf{X}^{\nu_1}_i)}\!-\!\highlightRED{\nabla_D \tilde{f}_i(\widetilde{\mathbf{D}}^{\nu_2}_{(i)};\mathbf{D}^{\nu_2}_{(i)},\mathbf{X}^{\nu_2}_i)},\widetilde{\mathbf{D}}_{(i)}^{\nu_2}- \widetilde{\mathbf{D}}^{\nu_1}_{(i)} \right\rangle
	\\
	&-\left\langle\nabla_D f_i(\mathbf{D}_{(i)}^{\nu_1},\mathbf{X}_i^{\nu_1})-\nabla_D f_i(\mathbf{D}_{(i)}^{\nu_2},\mathbf{X}_i^{\nu_2}),\widetilde{\mathbf{D}}_{(i)}^{\nu_2}- \widetilde{\mathbf{D}}^{\nu_1}_{(i)}\right\rangle
	+I \left\langle   {\boldsymbol{\Theta}}_{(i)}^{\nu_1}- {\boldsymbol{\Theta}}_{(i)}^{\nu_2}    ,\widetilde{\mathbf{D}}_{(i)}^{\nu_2}- \widetilde{\mathbf{D}}^{\nu_1}_{(i)}\right\rangle
	\\
	&\geq \left\langle \highlightBLUE{\nabla_D \tilde{f}_i(\widetilde{\mathbf{D}}^{\nu_1}_{(i)};\mathbf{D}^{\nu_1}_{(i)},\mathbf{X}^{\nu_1}_i)}- \highlightGRAY{\nabla_D \tilde{f}_i(\widetilde{\mathbf{D}}^{\nu_2}_{(i)};\mathbf{D}^{\nu_1}_{(i)},\mathbf{X}^{\nu_1}_i)}, \widetilde{\mathbf{D}}^{\nu_1}_{(i)}-\widetilde{\mathbf{D}}_{(i)}^{\nu_2} \right\rangle
	\geq \tau_{D,i}^{\nu_1}~||\widetilde{\mathbf{D}}_{(i)}^{\nu_2}- \widetilde{\mathbf{D}}^{\nu_1}_{(i)}||_F^2,
	\end{aligned}
	\end{equation}
	where the second inequality follows from the  $\tau_{D,i}^{\nu_1}$-strong convexity of $\tilde{f}_i(\bullet;\mathbf{D}^{\nu_1}_{(i)},\mathbf{X}^{\nu_1}_i)$ [cf. Remark \ref{surrogate_props}].  %in the first term of \eqref{min_princ_2}, it implies
	%\begin{equation}\label{min_princ_3}
	%\begin{aligned}
	%\tau_{D,i}^{\nu_1}~||\widetilde{\mathbf{D}}_{(i)}^{\nu_2}- \widetilde{\mathbf{D}}^{\nu_1}_{(i)}||_F^2
	%&\leq\!\left\langle \highlightGRAY{\nabla_D \tilde{f}_i(\widetilde{\mathbf{D}}^{\nu_2}_{(i)};\mathbf{D}^{\nu_1}_{(i)},\mathbf{X}^{\nu_1}_i)}\!-\!\highlightRED{\nabla_D \tilde{f}_i(\widetilde{\mathbf{D}}^{\nu_2}_{(i)};\mathbf{D}^{\nu_2}_{(i)},\mathbf{X}^{\nu_2}_i)},\widetilde{\mathbf{D}}_{(i)}^{\nu_2}- \widetilde{\mathbf{D}}^{\nu_1}_{(i)} \right\rangle
	%\\
	%&\quad-\left\langle\nabla_D f_i(\mathbf{D}_{(i)}^{\nu_1},\mathbf{X}_i^{\nu_1})-\nabla_D f_i(\mathbf{D}_{(i)}^{\nu_2},\mathbf{X}_i^{\nu_2}),\widetilde{\mathbf{D}}_{(i)}^{\nu_2}- \widetilde{\mathbf{D}}^{\nu_1}_{(i)}\right\rangle
	%\\
	%&\quad+ \sum_{j=1}^I\left\langle \nabla_D f_j(\mathbf{D}_{(i)}^{\nu_1},\mathbf{X}_j^{\nu_1})-\nabla_D f_j(\mathbf{D}_{(i)}^{\nu_2},\mathbf{X}_j^{\nu_2}),\widetilde{\mathbf{D}}_{(i)}^{\nu_2}- \widetilde{\mathbf{D}}^{\nu_1}_{(i)}  \right\rangle.
	%\end{aligned}
	%\end{equation}
	To bound the first term on the LHS of the above inequality, let us use the expression \eqref{eq:ftilde2} of $\tilde{f}_i$, and write \begin{equation}
	\label{min_princ_4}
	\begin{aligned}
	& \highlightGRAY{\nabla_D \tilde{f}_i(\widetilde{\mathbf{D}}^{\nu_2}_{(i)};\mathbf{D}^{\nu_1}_{(i)},\mathbf{X}^{\nu_1}_i)}\!-\!\highlightRED{\nabla_D \tilde{f}_i(\widetilde{\mathbf{D}}^{\nu_2}_{(i)};\mathbf{D}^{\nu_2}_{(i)},\mathbf{X}^{\nu_2}_i)}
	\\
	&\quad=\nabla_D f_i(\mathbf{D}_{(i)}^{\nu_1},\mathbf{X}_i^{\nu_1})-\nabla_D f_i(\mathbf{D}_{(i)}^{\nu_2},\mathbf{X}_i^{\nu_2})
	+\tau_{D,i}^{\nu_1}\left(\mathbf{D}_{(i)}^{\nu_2}-\mathbf{D}_{(i)}^{\nu_1}\right)+(\tau_{D,i}^{\nu_1}-\tau_{D,i}^{\nu_2})\left(\widetilde{\mathbf{D}}_{(i)}^{\nu_2}-\mathbf{D}_{(i)}^{\nu_2}\right).
	\end{aligned}
	\end{equation}
	Substituting \eqref{min_princ_4} in \eqref{min_princ_2},
	%and invoking  the Lipschitz continuity bound \eqref{Lip_G}, %(cf. Remark \ref{Lipschitz_continuity_remark}), 
	we get
	%\begin{equation}
	%\label{min_princ_5}
	%\begin{aligned}
	%\tau_{D,i}^{\nu_1}~\norm{\widetilde{\mathbf{D}}_{(i)}^{\nu_2}- \widetilde{\mathbf{D}}^{\nu_1}_{(i)}}_F^2
	%&\leq |\tau_{D,i}^{\nu_1}-\tau_{D,i}^{\nu_2}|\norm{\widetilde{\mathbf{D}}_{(i)}^{\nu_2}-\mathbf{D}_{(i)}^{\nu_2}}_F\norm{\widetilde{\mathbf{D}}_{(i)}^{\nu_2}- \widetilde{\mathbf{D}}^{\nu_1}_{(i)}}_F
	%\\
	%&\quad +\left(\tau_{D,i}^{\nu_1}+IL_\nabla\right)\norm{\mathbf{D}^{\nu_1}_{(i)}-\mathbf{D}^{\nu_2}_{(i)}}_F\norm{\widetilde{\mathbf{D}}_{(i)}^{\nu_2}- \widetilde{\mathbf{D}}^{\nu_1}_{(i)}}_F
	%\\
	%&\quad+L_\nabla\sum_{j=1}^I\norm{\mathbf{X}^{\nu_1}_j-\mathbf{X}_j^{\nu_2}}_F\norm{\widetilde{\mathbf{D}}_{(i)}^{\nu_2}- \widetilde{\mathbf{D}}^{\nu_1}_{(i)}}_F.
	%\end{aligned}
	%\end{equation}
	%By some simplification and using the bound $||\mathbf{x}||_1\leq\sqrt{I}||\mathbf{x}||_2$ with $\mathbf{x}\in\mathbb{R}^I$, \eqref{min_princ_5} gives
	\begin{equation}
	\begin{aligned}
	&\norm{\widetilde{\mathbf{D}}_{(i)}^{\nu_2}- \widetilde{\mathbf{D}}^{\nu_1}_{(i)}}_F
	\\
	\leq & \frac{|\tau_{D,i}^{\nu_1}-\tau_{D,i}^{\nu_2}|}{\tau_{D,i}^{\nu_1}}\norm{\widetilde{\mathbf{D}}_{(i)}^{\nu_2}-\mathbf{D}_{(i)}^{\nu_2}}_F
	+\norm{\mathbf{D}^{\nu_1}_{(i)}-\mathbf{D}^{\nu_2}_{(i)}}_F
+\frac{I}{\tau_{D,i}^{\nu_1}}\norm{  {\boldsymbol{\Theta}}_{(i)}^{\nu_1}- {\boldsymbol{\Theta}}_{(i)}^{\nu_2} }_F
	\\
\overset{(a)}{\leq} & \frac{|\tau_{D,i}^{\nu_1}-\tau_{D,i}^{\nu_2}|}{\tau_{D,i}^{\nu_1}}\norm{\widetilde{\mathbf{D}}_{(i)}^{\nu_2}-\mathbf{D}_{(i)}^{\nu_2}}_F
+\norm{\overline{\mathbf{D}}_{\boldsymbol{\phi}^{\nu_2}}-\overline{\mathbf{D}}_{\boldsymbol{\phi}^{\nu_1}}}_F+\norm{\mathbf{D}^{\nu_1}_{(i)}-\overline{\mathbf{D}}_{\boldsymbol{\phi}^{\nu_1}}}_F+\norm{\mathbf{D}^{\nu_2}_{(i)}-\overline{\mathbf{D}}_{\boldsymbol{\phi}^{\nu_2}}}_F
\\
&+\frac{I}{\tau_{D,i}^{\nu_1}}\norm{  {\boldsymbol{\Theta}}_{(i)}^{\nu_1}-\frac{1}{I}\sum_{i=1}^I\nabla_D f_i(\overline{\mathbf{D}}_{\boldsymbol{\phi}^{\nu_1}},\mathbf{X}_i^{\nu_1}) }_F	
\\
&+\frac{I}{\tau_{D,i}^{\nu_1}}\norm{ {\boldsymbol{\Theta}}_{(i)}^{\nu_2}-\frac{1}{I}\sum_{i=1}^I\nabla_D f_i(\overline{\mathbf{D}}_{\boldsymbol{\phi}^{\nu_2}},\mathbf{X}_i^{\nu_2}) }_F	
\\
&+\frac{{L}_{\nabla_D}}{\tau_{D,i}^{\nu_1}}\norm{(\overline{\mathbf{D}}_{\boldsymbol{\phi}^{\nu_1}},\mathbf{X}^{\nu_1})-(\overline{\mathbf{D}}_{\boldsymbol{\phi}^{\nu_2}},\mathbf{X}^{\nu_2})},
	\end{aligned}
	\label{Lip_bound_3}
	\end{equation}
where (a) holds by add/subtracting average quantities $\frac{1}{I}\sum_{i=1}^I\nabla_D f_i(\overline{\mathbf{D}}_{\boldsymbol{\phi}^{\nu}},\mathbf{X}_i^{\nu})$ and $\overline{\mathbf{D}}_{(i)}^\nu$, triangle inequality, and invoking  the Lipschitz continuity bound \eqref{Lip_G}. 	
	By the compactness of $\mathcal D$, we have  $||\widetilde{\mathbf{D}}_{(i)}^{\nu_2}-\mathbf{D}_{(i)}^{\nu_2}||_F\leq B_D$, for some finite $B_D>0$. Furthermore,  by  Assumption \ref{freeVars_assumptions}, $\tau_{D,i}^{\nu}$ is convergent to some $\tau_{D,i}^{\infty}>0$  and there exists a sufficiently small  $\tilde{s}_D>0$  such that   $\tilde{s}_D\leq\tau_{D,i}^{\nu}\leq \tilde{s}_D^{-1}$, for all $\nu$ and $i$. Thus \eqref{Lip_bound_3} gives
	\begin{equation}
	%\label{Lip_bound_4}
	\begin{aligned}
	\norm{\widetilde{\mathbf{D}}_{(i)}^{\nu_2}- \widetilde{\mathbf{D}}^{\nu_1}_{(i)}}_F
	&\leq \left(\frac{{L}_{\nabla_D}}{\tilde{s}_D}+1\right)\norm{(\overline{\mathbf{D}}_{\boldsymbol{\phi}^{\nu_1}},\mathbf{X}^{\nu_1})-(\overline{\mathbf{D}}_{\boldsymbol{\phi}^{\nu_2}},\mathbf{X}^{\nu_2})}+\tilde{T}_i^{\nu_1}+\tilde{T}_i^{\nu_2}
	\end{aligned}
		\label{Lip_bound_4}
	\end{equation}
	with $\tilde{T}_i^{\nu}\triangleq \frac{I}{\tilde{s}_D}|| \widetilde{\boldsymbol{\Theta}}_{i}^{\nu}-\frac{1}{I}\sum_{i=1}^I\nabla_D f_i(\overline{\mathbf{D}}_{\boldsymbol{\phi}^{\nu}},\mathbf{X}_i^{\nu})||_F	+||\mathbf{D}^{\nu}_{(i)}-\overline{\mathbf{D}}_{\boldsymbol{\phi}^{\nu}}||_F+\frac{B_D}{\tilde{s}_D}|\tau_{D,i}^{\nu}-\tau_{D,i}^{\infty}|$. Convergence of $\tau_{D,i}^\nu$ (Assumption \ref{freeVars_assumptions}2), Lemma \ref{consVar_props2} and Proposition \ref{consVar_props} yield $\tilde{T}_i^\nu\rightarrow 0$ as $\nu\rightarrow\infty$. Summing \eqref{Lip_bound_4} leads to the desired result, with $L_D=I(\frac{{L}_{\nabla_D}}{\tilde{s}_D}+1)$ and $\tilde{T}^\nu\triangleq \sum_{i=1}^I\tilde{T}_i^\nu$.
	\\
	It is clear that the claim also holds when   \eqref{eq:ftilde1} is chosen for $\tilde{f}_i$ [specifically, trivial extension is to modify the RHS of \eqref{min_princ_4}  and following a similar steps as in the rest of the proof]; we omit further details.

\noindent \textbf{(b)}
We prove \eqref{DeltaX_bound} when $\tilde{h}_i$ is given by  \eqref{eq:htilde2}; we leave the proof  under \eqref{eq:htilde1}   to the reader, since it is almost identical to that under \eqref{eq:htilde2}.  

Invoking optimality of each  $\mathbf{X}_i^\nu$ and $\mathbf{X}_i^{\nu+1}$ defined in \eqref{eq:Xupdate}  while using the strong convexity of $\tilde{h}_i(\bullet;\mathbf{U}_{(i)}^\nu,\mathbf{X}_i^\nu)$ and $g_i$'s,  it is not difficult to show that the following holds:
\begin{equation}
%\label{DeltaX_bound0}
\nonumber
\begin{aligned}
&(\tau_{X,i}^\nu+\mu_i)\norm{\mathbf{X}_i^{\nu+1}-\mathbf{X}_i^\nu}_F^2\leq
\\
&\qquad\qquad\left\langle\nabla_{X_i} \tilde{h}_i(\mathbf{X}_i^{\nu};\mathbf{U}_{(i)}^\nu,\mathbf{X}_i^\nu)-\nabla_{X_i} \tilde{h}_i(\mathbf{X}_i^\nu;\mathbf{U}_{(i)}^{\nu-1},\mathbf{X}_i^{\nu-1}),\mathbf{X}_i^\nu-\mathbf{X}_i^{\nu+1}\right\rangle.
\end{aligned}
\end{equation}
Using \eqref{eq:htilde2}, the definition  $\Upsilon_i^\nu(\mathbf{X}_i)\triangleq \tau_{X,i}^\nu\mathbf{X}_i-\nabla_{X_i} f_i(\mathbf{U}_{(i)}^\nu,\mathbf{X}_i)$, and  the Cauchy-Schwarz inequality, the above inequality yields
\begin{equation}
\begin{aligned}
(\tau_{X,i}^\nu+\mu_i)\norm{\mathbf{X}_i^{\nu+1}-\mathbf{X}_i^\nu}_F^2
\leq & \norm{ \Upsilon_i^\nu(\mathbf{X}^\nu_i)-\Upsilon_i^\nu(\mathbf{X}^{\nu-1}_i)}_F\norm{\mathbf{X}_i^{\nu+1}-\mathbf{X}_i^\nu}_F
\\
& +\norm{\nabla_{X_i}f_i(\mathbf{U}_{(i)}^\nu,\mathbf{X}_i^{\nu-1})-\nabla_{X_i}f_i(\mathbf{U}_{(i)}^{\nu-1},\mathbf{X}_i^{\nu-1})}_F\norm{\mathbf{X}_i^{\nu+1}-\mathbf{X}_i^\nu}_F
\\
&+|\tau_{X,i}^\nu-\tau_{X,i}^{\nu-1}|\norm{\mathbf{X}_i^\nu-\mathbf{X}_i^{\nu-1}}_F\norm{\mathbf{X}_i^{\nu+1}-\mathbf{X}_i^\nu}_F.
\end{aligned}
\label{DeltaX_bound4}
\end{equation}
Following the same steps used to prove \eqref{eq:Upsilon_bound}, %(see proof of Lemma  \ref{Bounded_X}), 
it is not difficult to check that, under Assumption \ref{freeVars_assumptions}1,    $||\Upsilon_i^\nu(\mathbf{X}^\nu_i)-\Upsilon_i^\nu(\mathbf{X}^0_i)||_F\leq\tau_{X,i}^\nu||\mathbf{X}^\nu_i-\mathbf{X}^0_i||_F$. Using in \eqref{DeltaX_bound4}  this bound   together with the  Lipschitz continuity of $\nabla_{X_i}f_i(\bullet,\mathbf{X}_i^{\nu-1})$ [cf. Remark \ref{Lipschitz_continuity_remark}]  and summing over $i$, yield\vspace{-0.3cm}
%\begin{equation}
%\begin{aligned}
%&(\tau_{X,i}^\nu+\mu)\norm{\mathbf{X}_i^{\nu+1}-\mathbf{X}_i^\nu}_F^2
%\\
%&\leq\big(\tau_{X,i}^\nu+|\tau_{X,i}^\nu-\tau_{X,i}^{\nu-1}|\big)\norm{\mathbf{X}_i^\nu-\mathbf{X}_i^{\nu-1}}_F\norm{\mathbf{X}_i^{\nu+1}-\mathbf{X}_i^\nu}_F
%\\
%&\quad+L_\nabla \norm{\mathbf{U}_{(i)}^\nu-\mathbf{U}_{(i)}^{\nu-1}}_F\norm{\mathbf{X}_i^{\nu+1}-\mathbf{X}_i^\nu}_F.
%\end{aligned}
%\label{DeltaX_bound5}
%\end{equation}
%By summing \eqref{DeltaX_bound5} for all $i=1,2,\ldots,I$, we get
\begin{equation}
\begin{aligned}
&||\mathbf{X}^{\nu+1}-\mathbf{X}^\nu||_F^2=\sum_{i=1}^I||\mathbf{X}_i^{\nu+1}-\mathbf{X}_i^\nu||_F^2
\\
&\leq\sum_{i=1}^I\frac{\tau_{X,i}^\nu+|\tau_{X,i}^\nu-\tau_{X,i}^{\nu-1}|}{\tau_{X,i}^\nu+\mu_i}\,||\mathbf{X}_i^\nu-\mathbf{X}_i^{\nu-1}||_F~||\mathbf{X}_i^{\nu+1}-\mathbf{X}_i^\nu||_F
\\
&\quad+\sum_{i=1}^I\frac{L_\nabla}{\tau_{X,i}^\nu+\mu_i}\,||\mathbf{U}_{(i)}^\nu-\mathbf{U}_{(i)}^{\nu-1}||_F~||\mathbf{X}_i^{\nu+1}-\mathbf{X}_i^\nu||_F.
\end{aligned}
\label{DeltaX_bound6}
\end{equation}
Define  \begin{equation}
p_X^\nu\triangleq \max_i\,\frac{\tau_{X,i}^\nu + |\tau_{X,i}^\nu-\tau_{X,i}^{\nu-1}|}{\tau_{X,i}^\nu+\mu_i},\quad\text{and}\quad  q_{X}^\nu\triangleq\max_i\,\frac{L_\nabla}{\tau_{X,i}^\nu+\mu_i}.
\label{p_X_q_X_def}
\end{equation}
%and  $p_X^\nu=\max_i~p_{X,i}^\nu$ and $q_X^\nu=\max_i~q_{X,i}^\nu$. 
Note that   $0<p_{X}^\nu,q_{X}^\nu<\infty$. Then, \eqref{DeltaX_bound6} becomes\vspace{-0.2cm}
\begin{equation}
\begin{aligned}
||\mathbf{X}^{\nu+1}-\mathbf{X}^\nu||_F^2&\leq p_X^\nu\sum_{i=1}^I||\mathbf{X}_i^\nu-\mathbf{X}_i^{\nu-1}||_F~||\mathbf{X}_i^{\nu+1}-\mathbf{X}_i^\nu||_F
+q_X^\nu\sum_{i=1}^I||\mathbf{U}_{(i)}^\nu-\mathbf{U}_{(i)}^{\nu-1}||_F~||\mathbf{X}_i^{\nu+1}-\mathbf{X}_i^\nu||_F
\\
& \leq p_X^\nu||\mathbf{X}^\nu-\mathbf{X}^{\nu-1}||_F~||\mathbf{X}^{\nu+1}-\mathbf{X}^\nu||_F
+q_X^\nu||\mathbf{U}^\nu-\mathbf{U}^{\nu-1}||_F~||\mathbf{X}^{\nu+1}-\mathbf{X}^\nu||_F,
\end{aligned}
\nonumber
\end{equation}
where in the last inequality we used $\sum_{i}a_ib_i\leq||\mathbf{a}||\cdot||\mathbf{b}||$. Therefore,
\begin{equation}
||\mathbf{X}^{\nu+1}-\mathbf{X}^\nu||_F\leq p_X^\nu||\mathbf{X}^\nu-\mathbf{X}^{\nu-1}||_F+q_X^\nu||\mathbf{U}^\nu-\mathbf{U}^{\nu-1}||_F.
\label{DeltaX_bound_proof}
\end{equation}
If, in addition,  Assumption \ref{freeVars_assumptions}2 holds, then  it follows from  \eqref{p_X_q_X_def}  that there exists a sufficiently large $\nu_X>0$ such that  $p_X^\nu\in(0,\delta_0)$,   for all $\nu\geq\nu_X$ and some   $\delta_0\in(0,1)$. \hfill$\square$

%\section{Table of Notation}
%\label{Table_of_Notation_sec}
%We summarize below the main notations used in this paper:\vspace{-.5cm}
%{
%	\subsection*{Summary of notations}
%	We summarize the main notations used in this paper in below table:}
\end{appendix}

\clearpage

 \bibliography{References_main,References_scutari,References_others}

\end{document}